\documentclass{amsart}

\usepackage{Preset}
\usepackage[makeroom]{cancel}
\usepackage{nccmath}
\usepackage{tikz-cd}
\usepackage{mathrsfs}
\usepackage[latin1]{inputenc}
\usepackage[T1]{fontenc}
\usepackage{tikz}
\usepackage[colorlinks=true,linkcolor=blue!50!black,anchorcolor=red,citecolor=blue,filecolor=black,menucolor=black,runcolor=black,urlcolor=black]{hyperref}

\usetikzlibrary{decorations.pathmorphing}
\usetikzlibrary{shapes,arrows,graphs}
\usetikzlibrary{positioning}
\setlist[enumerate,1]{label={\textnormal{(\arabic*)}}}

\DeclareFontFamily{U}{mathb}{\hyphenchar\font45}
\DeclareFontShape{U}{mathb}{m}{n}{ <5> <6> <7> <8> <9> <10> gen * mathb <10.95> mathb10 <12> <14.4> <17.28> <20.74> <24.88> mathb12 }{}
\DeclareSymbolFont{mathb}{U}{mathb}{m}{n}
\DeclareFontSubstitution{U}{mathb}{m}{n}
\DeclareMathSymbol{\selfmap}{3}{mathb}{"FD}

\title[Rigidity of the attractor of neutral renormalization]{Rigidity of the attractor of neutral renormalization}
\author{Dzmitry Dudko}  
\email{dzmitry.dudko@stonybrook.edu}  
\author{Willie Rush Lim}
\email{willie\_rush\_lim@brown.edu}
\author{Mikhail Lyubich}  
\email{mlyubich@stonybrook.edu}  
\date{}

\begin{document}

\begin{abstract}
    We prove combinatorial rigidity for the full renormalization attractor of neutral quadratic polynomials. 
    More precisely, any two bi-infinite renormalization towers with the same combinatorics are conformally conjugate on neighborhoods of their Mother Hedgehogs.
    The rigidity theorem includes arbitrary irrational combinatorics and respective parabolic enrichments. The proof relies on a comprehensive analysis of neutral cascades, i.e. transcendental dynamical systems arising from the rescaled limits of the first return maps of neutral quadratic polynomials. 
    We establish uniform butterfly bounds for neutral cascades and prove that any two combinatorially equivalent neutral cascades are affinely conjugate.
    We also prove rigidity for parabolic towers with equivalent backward combinatorics.
\end{abstract}

\maketitle

\setcounter{tocdepth}{1}
\tableofcontents

\begingroup
\renewcommand{\addcontentsline}[3]{}


\setcounter{equation}{0}
\renewcommand{\theequation}{\arabic{equation}}

\section*{Prologue}

Rigidity is a fundamental phenomenon in Geometry and Dynamics.
Roughly speaking, it suggests that a weak (topological- or combinatorial-like) equivalence between two objects of a certain class can be promoted to a stronger  (geometric- or analytic-like) equivalence. A classical example, inspiring
numerous follow-up activities, is the Mostow Rigidity Theorem from the 1960s: 
two compact hyperbolic manifolds of dimension at least $3$ with isomorphic fundamental groups are in fact the same up to hyperbolic isometry. 

In the early 1980s, Thurston introduced rigidity ideas into Holomorphic Dynamics by establishing the combinatorial rigidity for complex postcritically finite rational maps: combinatorial equivalence between such maps implies that they are in fact equivalent up to a M\"obius conjugacy (with a minor and well-understood exception for flexible Latt\`es examples). This was the last missing ingredient for the 
Monotonicity of Entropy Conjecture for the \emph{real} Logistic Family 
\begin{equation}
\label{eq:LogFamily}
    x\mapsto \lambda x(1-x) \colon [0,1]\selfmap ,\qquad \lambda\in [0,4].
\end{equation}
The Combinatorial Theory for this family
based on kneading sequences, was designed by Milnor and Thurston in the mid-1970s. Even though the Logistic Family has been initially studied within the framework of Real Dynamics, 
the Thurston Rigidity Theorem, as well as numerous further results, 
are based upon complex methods.

At about the same time as Thurston proved his Rigidity Theorem,   
Douady and Hubbard initiated  an extensive  exploration of the Complex Quadratic Family 
(the complexification of the Logistic Family), 
which led them to formulation of the celebrated {\em MLC Conjecture} 
(``the Mandelbrot set is Locally Connected''). 
It was realized soon that this conjecture is equivalent 
to the {\em Combinatorial Rigidity of all (non-hyperbolic) complex polynomials}.
And a little later (after Yoccoz's results circa 1990)  these conjectures were tightly linked to
even more  fundamental theory: {\em Dynamical Renormalization}. 

Renormalization ideas were introduced into Dynamics in the mid-1970s following the discovery of 
{\em Feigenbaum Universality}: 
the period-doubling cascade in the Logistic Family~\eqref{eq:LogFamily} appeared to 
exhibit a geometrically universal scenario for transition from regular to chaotic dynamics. 
Soon thereafter, many other universal scaling phenomena were discovered in Low-Dimensional Dynamics. Over time, it became apparent that Universality is only one of the many important 
manifestations of Renormalization Theory;
in particular, Rigidity phenomenon was recognized as a conceptual foundation for the theory. 

Here is a brief summary of the Renormalization framework. Instead of iterating a given dynamical system, one can ``renormalize space-time'' by considering first return maps 
to smaller scales and  then iterating 
the corresponding renormalization operators on an infinite-dimensional space of maps. Initial \emph{a priori} bounds (precompactness of the renormalization orbits) 
are needed to establish the existence of the renormalization attractor. 
This naturally leads to 
a Combinatorial Rigidity Problem for this attractor. 
It is a prerequisite for analyzing the geometric nature of convergence toward the attractor and ultimately
for establishing the Hyperbolicity of the renormalization.
It governs the multi-scale geometric structure
in the original space of non-hyperbolic systems
that has numerous deep applications.

 Feigenbaum-Coullet-Tresser Universality itself was extensively explored  and confirmed in the 1980--90s  
 in the framework of Douady-Hubbard's quadratic-like (ql) maps with bounded combinatorics
 in the work of Sullivan \cite{SulICM,Sul92}, McMullen \cite{McM96}, 
 and the third author of this paper \cite{L99}. 
 In particular, the Universality Conjecture was proven for real ql maps of bounded combinatorial type. 
 
 By the end of the 1990s, the scope of the initial bounded-type renormalization program was extended to include all maps in the Logistic Family,
 leading to the \emph{Regular vs. Stochastic} Dichotomy in~\cite{L02}: with respect to the Lebesgue measure, the dynamics of almost every map in~\eqref{eq:LogFamily} 
 is governed by either an attracting cycle (regular case) or an ergodic absolutely continuous invariant measure with positive characteristic exponent (stochastic case).
 In this development, various rigidity results played a crucial role; see \S\ref{sss:ql-maps} and~\S\ref{sss:puzzle} for more details.

The Neutral Family
\begin{equation}
\label{eq:Famneutral} z\mapsto e^{2\pi i \theta} z+z^2, \qquad \theta\in \R/\Z
\end{equation}
represents 
the main cardioid of the Mandelbrot set $\MM$. Active study of~\eqref{eq:Famneutral} began in the 1980s, although many aspects of the theory of Neutral Dynamics can be traced back to Siegel's 1942 work on neutral germs $z\mapsto e^{2\pi i \theta} z+O(z^2)$, 
which presented the first advance in the Small Divisors Problem. 
Siegel's work is now recognized as the simplest prototypical model for 
the  Kolmogorov--Arnold--Moser (KAM) Theory.

Along with many analogies,
there are fundamental differences between the Logistic 
and Neutral 
families. Together, however, it appears that these two families model most of the geometric features of the Mandelbrot set $\MM$.  
So, we believe that 
the Full Renormalization Theory for  $\MM$ is some kind of interpolation 
between the Logistic and Neutral theories. We refer to a recent survey~\cite{DudICM} for more details on the current status of the Renormalization Theory for $\MM$.

In fact,  the contemporary Neutral Renormalization Theory consists of two parts: 
Siegel (bounded-type) and Near-Parabolic (high-type). Both theories have a long history 
and by now are essentially complete.  
The Siegel case~\cite{McM98,DLS} is non-perturbative, 
but it belongs to the \emph{JLC regime} 
when the Julia set is locally connected.
The Near-Parabolic case deals with 
the non-JLC regime and relies on the perturbative theory of Inou-Shishikura \cite{IS}, which has led to many spectacular applications; see~\S\ref{sss:Neutral:types} for more details.

The Unified Renormalization Theory for the Neutral Family~\eqref{eq:Famneutral} was put forward by the first and the third authors in~\cite{DLy22}, where uniform pseudo-Siegel bounds were established for all maps in~\eqref{eq:Famneutral}. They imply (by~\cite{DLy26a}) the existence of the renormalization attractor for \eqref{eq:Famneutral} under the sector renormalization.

In this paper, we establish the {\bf Rigidity Theorem} for the 
{\bf full renormalization attractor} in the Neutral Family;
see Theorems~\ref{main-thm:rigidity} and~\ref{main-thm:cascades}. This is the first instance when such a theorem is proven in the non-perturbative  non-JLC regime. 
The Rigidity Theorem is established 
in the framework of Transcendental Dynamics; see~\S\ref{ss:outline} for more details. 
Some applications of the theorem are summarized in~\S\ref{ss:intro:appl}.

Let us finally remark that many global properties of the Neutral Family  
are closely linked to those of (\emph{uni}-)\emph{critical circle maps}. The Renormalization Theory for the latter was fully developed by de Faria, de Melo \cite{dF99,dFdM00}, and Yampolsky \cite{Y01,Y03a,Y03b} by the early 2000s. 
In the bounded-type case, the Douady-Ghys surgery \cite{D87,G84}, introduced in the 1980s, allows one to transfer \emph{a priori} bounds from circle maps to neutral maps. More generally, Neutral Complex Dynamics is widely regarded as a natural 
(albeit, much more  delicate)   counterpart 
to Circle Dynamics, 
with the latter providing valuable  
conceptual and technical insights into the former.

\vspace{0.3cm}

\endgroup

\section{Introduction}
\label{sec:intro}

The introduction is organized as follows. Theorem~\ref{main-thm:rigidity} describes the conformal rigidity of the attractor of sector renormalization and is stated in~\S\ref{ss:intro:main:thm} after recalling the necessary background. In \S\ref{ss:intro.TD}, we discuss the global transcendental dynamical systems called neutral cascades that are induced by bi-infinite towers in the attractor. Theorem~\ref{main-thm:cascades} states the affine rigidity of the attractor in the framework of these neutral cascades; this is the main viewpoint in the paper. Theorem~\ref{main-thm:backward-rigidity} states the rigidity for backward parabolic towers and serves as a preparation for uniform hyperbolicity; see~\S\ref{sss:unif_hyperb} for details. Further applications are summarized in the rest of~\S\ref{ss:intro:appl}. Lastly, \S\ref{ss:history} provides detailed historical background.

\subsection{Sector renormalization}\label{ss:SectRenorm}

A key development in neutral dynamics
has been the theory of \emph{sector renormalization}.
For a neutral holomorphic germ $g(z) = e^{2\pi i \theta}z + O(z^2)$, one can construct a fundamental sector near the fixed
point at 0 and conformally glue its sides using the dynamics; see Figure \ref{fig:sector-renormalization-local} for an
illustration.
The first return map then produces a new neutral holomorphic germ $\Rsec g$. 
The procedure can be iterated to a sequence of renormalized maps $(\Rsec)^n g$.

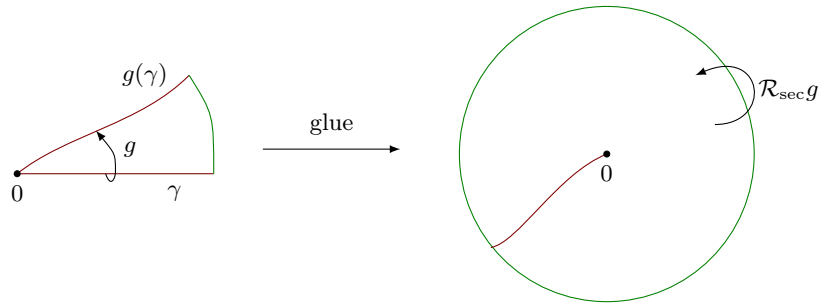
\begin{figure}
    \centering
    \begin{tikzpicture}[scale=1.3]
    \draw[white] (-4.1,0)--(-4.08,0);
    \draw[red!50!black] (-1,0) -- (-3,0) .. controls (-2.5,0.4) and (-1.75,0.5) .. (-1.25,1);
    \draw[green!50!black] (-1.25,1) .. controls (-1,0.6) .. (-1,0);
    \filldraw[color=black] (-3,0) circle (0.03);
    \node at (-3,-0.2) {\small $0$};
    \node at (-1.4,-0.2) {\small $\gamma$};
    \node at (-1.7,1) {\small $g(\gamma)$};
    \draw[-latex] (-2.1,0) .. controls (-2.05,-0.15) and (-2,-0.05) .. (-2,0.02) .. controls (-2,0.2) .. (-2.2,0.44);
    \node at (-1.85,0.25) {\small $g$};

    \draw[-latex] (-0.5,0.25) -- (0.9,0.25);
    \node at (0.2,0.5) {\small glue};
        
    \draw[green!50!black] (3,0.2) circle [radius = 1.5cm];
    \draw[red!50!black] (3,0.2) .. controls (2.5,0) and (2.1,-0.7) .. (1.82,-0.75);
    \filldraw[color=black] (3,0.2) circle (0.03);
    \node at (3,0) {\small $0$};

    \draw[-latex] (4.1,0.5) .. controls (4.7,0.5) and (4.6,1.3) .. (3.9,1);
    \node at (4.85,0.85) {\small $\Rsec g$};
\end{tikzpicture}
    \caption{A schematic picture of sector renormalization}
    \label{fig:sector-renormalization-local}
\end{figure}

The concept 
of sector renormalization appeared in the work of Douady and Ghys \cite{D87}. 
In his pioneering work \cite{Yoc95}, Yoccoz carefully studied iterations of sector renormalizations and proved the optimality of the Brjuno condition for analytic linearization.

In this paper, we will consider neutral quadratic polynomials
\[
    f_\theta(z) := e^{2\pi i \theta} z + z^2.
\]
They 
form the simplest family of non-linear maps admitting a neutral fixed point.
Recently, the first and third authors \cite{DLy22} proved uniform \emph{a priori bounds} for all neutral quadratic polynomials in the following sense.
For every irrational $\theta$, there exists a closed uniform quasidisk $\hat{Z}_\theta$, called a \emph{pseudo-Siegel disk}, on which $f_\theta$ is injective, contains the fixed point $0$, the critical point and the whole postcritical set.
In the bounded-type regime, $\hat{Z}_\theta$ is constructed by filling in the parabolic fjords of the Siegel disk $Z_\theta$.
An immediate application is the existence of the \emph{Mother Hedgehog} $H_\theta$ for $f_\theta$ for every irrational $\theta$, that is, a completely invariant full continuum containing both the neutral fixed point and the critical point.

In the subsequent work \cite{DLy26a}, it was further shown that such pseudo-Siegel bounds provide uniform bounds on sector renormalization and allow one to upgrade Yoccoz's local sectors to sectors of uniform size that well capture the critical orbit. 
More precisely, for every irrational $\theta$, one can construct sector renormalizations $(\Rsec)^n (f_{\theta})$, $n \geq 1$, that admit uniform pseudo-Siegel disks. 
This implies that the space of forward renormalization orbits 
\begin{equation}
    \label{eq:intro.renorm.attr}
\towsec := \left\{ \seq{ (\mathcal{R}_{\text{sec}} )^{m+n} f_\theta }_{n \geq 0} \: : \: m \geq 0, \; \theta \text{ is irrational}\right\}
\end{equation}
is pre-compact.

Every tower $\fbold = \seq{f_n}_{n\geq 0}$ in $\overline{\towsec}$ admits a sequence of pre-renormalizations $f_0^{[n]}$ in the dynamical plane of $f_0$ that project to $f_n$ under the gluing of an appropriate sector. 
If $f_0$ is parabolic (e.g. $z\mapsto z+z^2$), then $f_0^{[1]}$ will be a Lavaurs map of $f_0$ coming from parabolic implosion phenomena; otherwise, $f_0^{[1]}$ is a genuine iterate of $f_0$. 
Each $\fbold$ also admits a unique \emph{Mother Hedgehog} $H(\fbold)$ that is completely invariant under the semigroup generated by $f_0^{[n]}$, $n\geq 0$.

\subsection{The renormalization attractor}\label{ss:intro:main:thm}
The pre-compactness of $\towsec$ implies that renormalization orbits converge to a \emph{renorm-attractor} $\attr$.
Every element of $\attr$ can be described as a bi-infinite renormalization tower $\fboldbar = \seq{f_n}_{n\in \Z}$ where $\Rsec f_n = f_{n+1}$ for all $n \in \Z$.
For every $n \in \Z$, the dynamical plane of $f_n$ houses a unique Mother Hedgehog $H_n(\fboldbar)$.

The combinatorics of a bi-infinite tower $\fboldbar = \seq{f_n}_{n\in \Z} \in \attr$ can be described by an element $\tttheta(\fboldbar) = \seq{ (\varepsilon_n, \abar_n) }_{n \in \Z}$ of the compact bi-infinite sequence space 
    \[
        \TheBiCpt = 
        \left( \{-,+\} \times \{2,3,4,\ldots,\infty\} 
        \right)^{\Z}.
    \]
The symbols record the orientation and the return times at every renormalization level; the value $\infty$ records parabolic bifurcation.
The space $\TheBiCpt$ can be interpreted as the natural extension of an appropriate compactification of the space of irrationals 
\[
    \Irrat = \left( -\frac{1}{2},\frac{1}{2} \right) \backslash \Q.
\]
Our first main theorem states that these combinatorial data completely determine the geometry.

\begin{thmx}[Conformal Rigidity of bi-infinite towers]
    \label{main-thm:rigidity}
    Any two combinatorially equivalent bi-infinite renormalization towers $\fboldbar = \seq{f_n}_{n \in \Z}$ and $\gboldbar = \seq{g_n}_{n \in \Z}$ in the renorm-attractor $\attr$ are conformally conjugate: for all $n \in \Z$, $f_n$ and $g_n$ are conformally conjugate on some definite neighborhoods of $H_n(\fboldbar)$ and $H_n(\gboldbar)$ respectively. 
\end{thmx}

A more precise version of this theorem can be found in Theorem \ref{main-thm:rigidity-precise}.

\begin{remark}[Combinatorial data vs. multiplier data]\label{rem:comb.eq vs mult}
    Combinatorial equivalence is stronger than having the same sequence of multipliers.
    This relation is strict due to the parabolic enrichment phenomenon.
    However, for irrational combinatorics, the two properties are equivalent.
\end{remark}

\begin{remark}[No hybrid equivalence for forward towers]\label{rem:intro.forw.towers}
    For combinatorially equivalent forward towers in $\overline{\towsec}$, we do not expect quasiconformal conjugacy on the neighborhood of the Mother Hedgehog. This is one fundamental difference from all other renormalization theories developed in the non-perturbative regime: bounded-type Siegel maps \cite{McM98,AL22}, ql maps~(\S\ref{sss:ql-maps}) and critical circle maps with arbitrary irrational rotation number~(\S\ref{sss:circle-maps}). The reason is that the local QC geometry near the $\alpha'$-points (the iterated preimages of the fixed point $0$) encodes the anti-renormalization history, potentially making the corresponding local geometries of the two forward-equivalent towers incompatible; see Remark~\ref{rem:anti.renorm.history} for details. For forward-equivalent towers, we will develop the notion of QC Thurston equivalence; see~\S\ref{sss:qc.Thurst.eq} and Theorem \ref{thm:qc-thurston-equivalence}.
\end{remark}

\begin{remark}[On the notion of Rigidity] Note that the terminology in this paper and in other related dynamical contexts is slightly inconsistent. 
Combinatorial Rigidity (e.g., for quadratic polynomials) usually refers to the assertion that a combinatorial equivalence can be promoted to the strongest possible equivalence (e.g. a M\"obius conjugacy). By contrast, the terms QC, Conformal, Affine, $C^{1+\alpha}$- Rigidity emphasize the nature of the resulting conjugacy. 
\end{remark}

Theorem~\ref{main-thm:rigidity} implies that the renorm-attractor admits a topological horseshoe structure. 
Let
\begin{equation}
    \label{eq:dfn.Hors}
\Hors :=\attr/_{\sim_\text{conf}}
\end{equation}
be the space of bi-infinite renormalization towers in $\attr$ up to conformal conjugacy 
on definite neighborhoods of their Mother Hedgehogs.
This is the {\em Full Renormalization Horseshoe} for the Neutral Family. 
By definition, the \emph{quotient topology} on $\Hors$ is defined so that the quotient map $\attr\to \Hors$ is continuous. 

\begin{corx}
\label{main-corollary}
    $\Rsec: \Hors \to \Hors$ is conjugate to the shift map on $\TheBiCpt$.
\end{corx}

Theorem~\ref{main-thm:cascades} below provides another viewpoint on $\attr$ as the space of neutral cascades with some natural topology. 

In an upcoming paper, 
we will refine Corollary~\ref{main-corollary} by constructing a canonical cylinder renormalization operator 
$\renorm_{\textnormal{cyl}}$  whose action on the corresponding  attractor $\attr(\renorm_{\textnormal{cyl}})$
is conjugate to the action of $\Rsec$ on $\Hors$; see~\S\ref{sss:RRcyl}. 

\subsection{Transcendental dynamics of neutral cascades} \label{ss:intro.TD} Transcendental dynamics allows us to 
deal with 
the crucial technical problem that sector renormalization acts on maps that are \emph{not} branched coverings.

For any irrational $\theta$, every rescaled limit of the quadratic dynamical systems $\{f_\theta^k \}_{k \geq 1}$ about the critical value of $f_\theta$ is a global transcendental dynamical system called a \emph{neutral cascade}.
Formally, it is a commutative semigroup $\Fbold = (\Fbold^P)_{P \in \Tbold}$ of $\sigma$-proper maps onto $\C$.
The parametrizing semigroup $\Tbold$ is totally ordered and has an infinite generating set $\{Q_{[n]}\}_{n\in\Z}$. 
The set of relations is determined by a unique element $\tttheta(\Fbold)$ of $\TheBiCpt$ called the \emph{combinatorics} of the neutral cascade.

The complete geometry of a bi-infinite tower is recorded in a single neutral cascade.
More precisely, for every bi-infinite tower $\fboldbar=\seq{f_n}_{n\in\Z}$ in $\attr$, there exists a neutral cascade $\Fbold = (\Fbold^P)_{P \in \Tbold}$ with the same combinatorics such that, 
after some conformal change of coordinates, each of $\Fbold^{[n]} = \Fbold^{Q_{[n]}}$ is an analytic extension of $f_n$.
See Figure \ref{fig:transfer-to-cascade}.


Theorem \ref{main-thm:rigidity} is deduced from the following transcendental rigidity.

\begin{thmx}[Affine Rigidity of neutral cascades]
    \label{main-thm:cascades}
    Any two neutral cascades $\Fbold = (\Fbold^P)_{P \in \Tbold}$ and $\Gbold = (\Gbold^P)_{P \in \Tbold}$ with the same combinatorics are affinely conjugate: there exists an affine map $h: \C \to \C$ such that $h \circ \Fbold^P = \Gbold^P \circ h$ for all $P \in \Tbold$.
\end{thmx}

Analogous results are known in the bounded-type regime and form an important part of its theory; see \S~\ref{sss:Neutral:types}. 
Theorem \ref{main-thm:cascades} gives us a complete dynamical universality of various asymptotic self-similarities (without exponential convergence) taking into account all possible combinatorics and parabolic enrichments.

Our Rigidity Theorem also implies that various important  dynamical sets 
(such as Mother Hedgehogs and pseudo-Siegel disks)
depend uniformly continuously on the combinatorics (cf. Proposition \ref{prop:uniform-continuity}).
This continuity property will be used in the upcoming
construction of the cylinder renormalization operator across the Full Renormalization Horseshoe.

We would like to emphasize a deep  connection between
the Renormalization Theory for neutral quadratic polynomials and 
the Renormalization Theory for analytic critical circle maps.
Our Rigidity Theorem is the neutral counterpart of the Rigidity Theorem for critical circle maps proven by Yampolsky \cite{Y01}.
A key technical ingredient of Yampolsky's proof is \emph{complex beau bounds} in the form of \emph{butterflies} \cite{Y99}.
Similarly, in the course of proving Theorem \ref{main-thm:rigidity}, 
we obtain \emph{hoglet butterfly bounds} (Figure~\ref{fig:pseudo-butterfly}) associated to bi-infinite towers, 
overcoming a new fundamental technical problem, 
the {\em Cremer phenomenon}, that does not appear in the theory of critical circle maps.
A more detailed summary is given in \S\ref{ss:outline}.

\subsection{Parabolic backward-infinite towers}

The techniques developed for the cascade rigidity also yield a new type of rigidity theorem for parabolic towers.

\begin{thmx}[Backward rigidity of parabolic towers]
    \label{main-thm:backward-rigidity}
    Consider two bi-infinite renormalization towers $\fboldbar = \seq{f_n}_{n \in \Z}$ and $\gboldbar = \seq{g_n}_{n \in \Z}$ in $\attr$ with combinatorics 
    $\tttheta(\fboldbar) = \seq{ (\varepsilon_n, \abar_n) }_{n \in \Z}$
    and $\tttheta(\gboldbar)=\seq{ (\varepsilon'_n, \abar'_n) }_{n \in \Z}$ where
\[
    \abar_1 = \abar'_1 = \infty
    \qquad \text{ and } \qquad (\varepsilon'_n, \abar'_n) = (\varepsilon_n, \abar_n) \quad \text{ for all } n \leq 1.
\]
    Then, for all $n \leq 0$, $f_n$ and $g_n$ are conformally conjugate on some definite neighborhoods of $H_n(\fboldbar)$ and $H_n(\gboldbar)$ respectively. 
\end{thmx}

This theorem essentially states that all possible neutral enrichments of parabolic maps in the renorm-attractor also belong to the attractor.
It is a backward parabolic analog of Theorem \ref{main-thm:cascades}
needed for  the upcoming proof of  hyperbolicity of the renorm-attractor 
(compare~\S\ref{sss:unif_hyperb}). 

\subsection{Applications}\label{ss:intro:appl} 

Let us outline a number of applications 
(immediate and potential) of the results of this paper.

\subsubsection{On the boundedness of limiting parabolic basins} 
Our results, in particular the boundedness of \emph{lakes} (Corollary \ref{cor:bounded-lakes}), 
imply the boundedness of the immediate parabolic basins for parabolic maps in $\Hors$.
To the best of our knowledge, this property was not known before.

\subsubsection{On the cylinder renormalization operator}\label{sss:RRcyl}

Our results also imply the {\em uniform continuity of the Mother Hedgehog} of neutral cascades on the combinatorics 
(Proposition \ref{prop:uniform-continuity}).
(Note that such a property, in the Inou-Shishikura class, 
played an important role in the constructions of Julia sets of positive area.) 

In an upcoming paper, we will use this continuity property to construct a cylinder renormalization operator $\renorm_{\textnormal{cyl}}$ across $\Hors$ with the property  that the associated quotient map $\attr(\renorm_{\textnormal{cyl}})\to \Hors$ in~\eqref{eq:dfn.Hors} is a conjugacy. 
Unlike sector renormalization acting on conformal classes, 
$\renorm_{\textnormal{cyl}}$ acts on affine classes 
(i.e.  on normalized maps). 
This operator is similar to Yampolsky's cylinder renormalization operator for critical circle maps \cite{Y03a,Y03b}. 
We will discuss this construction a little further in \S\ref{ss:cylinder}.

\subsubsection{On the uniform hyperbolicity}\label{sss:unif_hyperb} The present work was designed as a foundation 
for the uniform hyperbolicity of the attractor $\Hors$ for the cylinder renormalization. In \S\ref{ss:uniform-hyperbolicity},
we will outline our strategy for the upcoming proof.  
A key step  is a general formulation of Theorem \ref{main-thm:backward-rigidity}
(see Theorem \ref{thm:backward-rigidity-general}) for backward parabolic towers with a priori bounds.

\subsubsection{On the canonicity of the renorm-attractor} Rigidity of the renorm-attractor $\attr$ justified in this paper implies that $\attr$ is canonical in the sense that the same attractor arises for any natural degree two semi-global family with a neutral cycle.  Indeed, by \cite{DLL25,DLu25}, sector renormalization can be initiated from the neutral family arising from the boundary of any disjoint-type hyperbolic component.  Outside the Sierpi\'nski case, the classical Douady-Hubbard straightening theory for ql maps is not available in this setting. We expect that the Rigidity Theorem, together with the hyperbolicity of the cylinder renormalization operator discussed in~\S\ref{sss:RRcyl} and~\S\ref{sss:unif_hyperb}, will provide an adequate substitute.

\subsection{Historical remarks}
\label{ss:history}
As explained in the Prologue, 
the Neutral Renormalization Theory has a long history 
intimately related to other major renormalization themes in Complex Dynamics. 
Here we give a more detailed account on this historical development 
noting parallels and differences between different theories.

\subsubsection{Neutral Family}
\label{sss:Neutral:types}

The Renormalization Theory of the Neutral Family $\{f_\theta\}_{\theta \in \Irrat}$ 
has been developed through two complementary regimes.

In the \emph{bounded-type} regime, governed by Siegel phenomenon, Douady-Ghys' quasiconformal surgery \cite{D87,G84} shows that $f_\theta$ admits a quasiconformal Siegel disk whose boundary contains the critical point.
Stirnemann \cite{St94} gave a computer-assisted proof of a golden-mean renormalization fixed point, and McMullen \cite{McM98} proved existence and rigidity of periodic-type neutral cascades and exponential convergence of renormalizations toward them.
Avila and the third author \cite[\S4]{AL22} established quasiconformal rigidity of bounded-type Siegel maps, with regularity improvable to $C^{1+\alpha}$ via McMullen's method.
Gaidashev and Yampolsky \cite{GY22} gave a computer-assisted proof 
of the golden mean hyperbolicity via almost commuting pairs.
The first and third authors, jointly with Nikita Selinger \cite{DLS},
introduced \emph{Pacman renormalization operator} --- a branched-covering version of sector renormalization --- and proved hyperbolicity of the renormalization periodic points, with applications to self-similarity of the Mandelbrot set near Siegel parameters.
It followed up with  the work  \cite{DLy23} that described  the {\em puzzle structure} 
for periodic-type neutral cascades to construct the first examples  of satellite infinitely renormalizable parameters
of bounded type for which the  MLC holds.
Moreover, for some of these parameters, the Julia set has positive area. 

On the other hand,
in the \emph{high-type} regime, every renormalization of $f_\theta$ is uniformly close to being simply parabolic. 
In a breakthrough work \cite{IS},
Inou and Shishikura  constructed a renorm-invariant class of maps that contains renormalizations of high-type neutral quadratic polynomials and proved uniform hyperbolicity of the corresponding near-parabolic renorm-horseshoe. 
These results found  numerous important applications that were previously inaccessible, 
including the existence of positive area Julia sets \cite{BC12, AL22}, 
an essentially complete description, in this class, 
of  the measure-theoretic and topological properties of the postcritical set of $f_\theta$ \cite{Che13, Che18, Che19, SY24, Che25}, 
and MLC at some satellite infinitely  renormalizable parameters (of unbonded type) \cite{CS15}. 
Note that non-locally connected Julia sets (the \emph{non-JLC} phenomenon) 
arise exactly due to near-parabolic effects.

\subsubsection{Quadratic-like renormalization}
\label{sss:ql-maps}

In a seminal paper \cite{DH85}, 
Douady and Hubbard  introduced the class of quadratic-like (ql) maps 
and the associated renormalization,
as a tool for explaining  presence of small copies of the Mandelbrot set $\MM$ within itself. 

In an address to the ICM-86 in Berkeley,
Dennis Sullivan \cite{SulICM} laid down an inspiring program  suggesting 
the Feigenbaum-Coullet-Tresser Universality  \cite{F1,F2,TC,CT}
(particularly, its dynamical part) should be built upon the foundation of
the Douady-Hubbard ql renormalization combined with ideas of Teichm\"uller Theory. 
The key analytic input to this approach is \emph{a priori} (\emph{beau}) \emph{bounds}, 
which would give precompactness for the family of renormalizations.
In this direction
Sullivan \cite{Sul92,dMvS93} proved such bounds for real infinitely renormalizable ql maps of bounded type, 
showed convergence of renormalizations to an attractor, 
and proved combinatorial rigidity of the attractor.

Assuming \emph{a priori bounds}, McMullen \cite{McM96} 
then treated the general complex combinatorics of bounded type: 
he proved $C^{1+\alpha}$ rigidity of the postcritical set, showed exponential convergence of renormalizations, and established the transcendental cascade structure for bi-infinite towers.

Using a different method, the  third author established a priori bounds for a class of 
complex polynomials (that includes certain bounded and unbounded cases) \cite{LY97}.
Then he proved,  
assuming a priori bounds, hyperbolicity of renormalization 
for bounded type combinatorics \cite{L99}.
Applied to  real infinitely renormalizable maps of bounded type, 
it provided  Parameter Universality in the  Feigenbaum-Coullet-Tresser  Conjecture, 
completing this renormalization story.

In the real setting, unbounded combinatorics splits into two types.
The \emph{high-type} combinatorics is amenable to puzzle techniques; see \S\ref{sss:puzzle} below.
The \emph{essentially bounded} combinatorics degenerate via parabolic bifurcation, as in the neutral and the critical circle settings;
complex beau bounds here were established in \cite{LY97} and used by Hinkle \cite{Hin00} to prove rigidity of the corresponding bi-infinite towers.
Together, these gave the geometric foundation for the Real Renormalization Theory and its Full Renormalization Horseshoe \cite{L02}.

As was already mentioned, 
on the complex side,  first beau bounds  results appeared in \cite{Lyu97} for high-type maps; see \S\ref{sss:puzzle} below.
The general bounded combinatorics was later treated by Kahn \cite{K06} (primitive case) 
and the first and the third authors \cite{DLy26b} (satellite case), via Kahn's Near-Degenerate Regime.
Partial unification of these developments 
appeared \cite{KL08,KL09a}. 
The existence and hyperbolicity of the corresponding Renormalization Horseshoes then follows from \cite{L02,AL11}.

\subsubsection{Puzzle renormalization}
\label{sss:puzzle}
Yoccoz's puzzle techniques was introduced as a tool to get a geometric control of quadratic Julia sets
and then transfer it  to the parameter plane, 
yielding MLC at all parameters that are at most finitely renormalizable (see \cite{Hub93,Mil00}). 
This machinery was  then formulated by the third author in terms of {\em generalized renormalization} of
{\em generalized ql maps}, 
whose range of applications was  extended to the case of {\em infinitely renormalizable} maps.
It allowed him to control intermediate scales in between two consecutive ql renorm levels. 
In this paper we will refer to this machinery as {\em puzzle renormalization}.
In \cite{Lyu97,Lyu00}, it was used to prove the linear growth of the moduli in the Principal Nest of puzzle
and parapuzzle pieces
and deduce from it the  dynamical and parameter beau bounds for  high-type combinatorics
satisfying  the ``Secondary Limb Condition'' (SLC), accompanied with the corresponding Combinatorial Rigidity Theorem.  
This was the first  genuinely complex class 
of parameters for which  beau bounds and combinatorial rigidity were established. 

Since the SLC is satisfied for real maps,
these results covered real quadratic polynomials of high type,
leading to the Combinatorial Rigidity for \emph{all} real quadratic polynomials \cite{Lyu97}. 
Moreover, according to \cite{Lyu00}, the Phase-Parameter relation on deep puzzle scales 
becomes {\em almost linear}, 
yielding the first part for the \emph{Regular vs. Stochastic} Dichotomy in the real quadratic family
(asserting that almost all real non-hyperbolic parameters that are at most finitely renormalizable are stochastic).
It followed with the construction of the Full Hyperbolic Renormalization Horseshoe for real ql maps \cite{L02}
implying the second part of the Dichotomy (asserting that the set of  
real infinitely renormalizable parameters has zero measure).

\subsubsection{Critical circle maps}
\label{sss:circle-maps}

Renormalization for analytic critical circle maps is combinatorially similar to neutral renormalization, though technically it is more conveniently formulated via real-analytic commuting pairs $(f_-,f_+)$ rather than single maps.

Herman and \'{S}wi\k{a}tek \cite{H87,Sw87}  proved \emph{real a priori bounds}, yielding quasisymmetric conjugacy for any two critical circle maps with the same irrational rotation number.
In the 1990s, de Faria, de Melo, and Yampolsky \cite{dF99,dFdM00,Y99} obtained \emph{complex butterfly bounds} for commuting pairs, giving:
\begin{enumerate}
    \item[(a)] QC+$C^{1+\alpha}$ rigidity \cite{dFdM00,Y99,KY06};
    \item[(b)] exponential convergence of renormalizations to an attractor \cite{dFdM99};
    \item[(c)] rigidity of the attractor \cite{Y01},  a prototype of Theorem \ref{main-thm:rigidity}.
\end{enumerate}
Together, (b) and (c) give dynamical universality.
An analog of (a) is known for bounded-type Siegel maps \cite{AL22}; in general, it may not hold for semi-global neutral maps with arbitrary irrational rotation number; see Remark \ref{rem:intro.forw.towers}.

Yampolsky's cylinder renormalization operator for Siegel maps
(with roots, in the parabolic and near-parabolic  settings, in the work of \'{E}calle, Voronin and Douady)  
was shown to be uniformly hyperbolic \cite{Y03a,Y03b}. 
It implies parameter universality conjectured by physicists \cite{FKS,ORSS}.

Finally, the second author \cite{Lim23a,Lim23b,Lim24} recently developed the Renormalization Theory of bounded-type critical quasicircle maps, allowing imbalanced criticalities and extending both the critical circle maps (balanced criticality) and the Siegel maps (zero inner criticality) settings.
In \cite{Lim24}, the periodic-type cascade structure for such maps was analyzed to establish hyperbolicity of renormalization periodic points.

\subsubsection{Summary}

As we see, the current paper on the Rigidity of Neutral Towers 
(and its projected continuation on the Hyperbolicity of the Full Neutral Renormalization Horseshoe) is deeply related to all earlier one-dimensional renormalization theories 
(for logistic and ql, critical circle, Siegel, and near-parabolic maps) elaborated above. 

What makes this paper substantially different is the need to handle the non-JLC regime in the non-perturbative setting. 
What makes this possible is the systematic use of Transcendental Dynamics, as summarized in~\S\ref{ss:outline}, combined with pseudo-Siegel a priori bounds~\cite{DLy22}.

\subsection{Acknowledgements}
The first author was partially supported by the Simons Fellows grant and by the NSF grants DMS 2055532 and DMS 2246485. The second author is grateful to the Institute for Mathematical Sciences at Stony Brook for their hospitality. The third author thanks the NSF for their continuing support.


\section{Outline, notations, and conventions}
\label{sec:guide}

This section offers a guide on how to read this paper.  Subsection~\S\ref{ss:outline} summarizes main ingredients in the proof of the rigidity theorem; see Figure~\ref{fig:roadmap}.  Subsection~\ref{ss:organization} briefly describes the organization of the paper; see Figure~\ref{fig:dependency}. Notations and conventions are summarized in~\S\ref{ss:notation}.

\subsection{Summary of the main steps and tools towards the rigidity}
\label{ss:outline} Most of our analysis is carried out within the framework of Transcendental Dynamics and centers around Theorem~\ref{main-thm:cascades}. To deal with various non-JLC type geometric subtleties (cf. Remarks~\ref{rem:intro.forw.towers},~\ref{rem:anti.renorm.history}), we develop a ladder of QC Thurston equivalence: 
first on the pseudo-Siegel disk of forward towers (see~\S\ref{sss:qc.Thurst.eq}), 
then on the top pseudo-Siegel half-planes $\hat{\Zbold}$ of neutral cascades (see~\S\ref{sss:TD.outline}), 
then on the  pseudo-bubbles $\widehat \Bubb_P$ of external cascades, 
and finally on the butterfly structures. 
The latter is included in what we refer to as the Hoglet Butterfly Bounds; see~\S\ref{sss:outl.HogletBounds}. They allow us to apply the pullback argument to construct a QC conjugacy on butterflies for combinatorially equivalent neutral cascades; see~\S\ref{sss:outl.PA}. This leads to an affine conjugacy for combinatorially equivalent neutral cascades.

\subsubsection{QC Thurston equivalence on pseudo-Siegel disks}\label{sss:qc.Thurst.eq} In Theorem \ref{thm:qc-thurston-equivalence}, we prove that any two combinatorially equivalent towers $\fbold = \seq{f_n}_{n\geq 0}$ and $\gbold = \seq{g_n}_{n\geq 0}$ in $\overline{\towsec}$ are \emph{QC Thurston equivalent} (on their pseudo-Siegel disks): there exists a sequence of uniformly quasiconformal maps $h_n: \D \to \D$, $n \geq 0$ such that each $h_n$
\begin{itemize}
    \item maps the pseudo-Siegel disks in the dynamical plane of $f_n$ to those in the dynamical plane of $g_n$, 
    \item is a conjugacy between the Mother Hedgehog in the dynamical plane of $f_n$ to that in the dynamical plane of $g_n$, conformal in the interior, and
    \item is equivariant under the sector renormalization change of variables.
\end{itemize}
This is proven using a variant of the Pullback Argument on the triangulated pseudo-Siegel disks: by lifting triangulations under sectorial anti-renormalization, we obtain equivariance on the Mother Hedgehog; see Figure~\ref{fig:triangulation}.

\subsubsection{Transcendental Dynamics}\label{sss:TD.outline} Neutral transcendental cascades $\Fbold = (\Fbold^P)_{P \in \Tbold}$ are formalized in~\S\ref{sec:transcendental-dynamics}. They arise as rescaled limits of neutral quadratic polynomials. This viewpoint justifies the transcendental $\sigma$-proper structure for neutral cascades. 

Every neutral cascade $\Fbold $ can also be obtained from a bi-infinite tower
$\seq{f_n}_{n\in \Z} \in \overline{\towsec}$ by linearizing the renormalization change of coordinates so that the dynamical plane of $f_n$ ``linearly embeds'' into $f_{n-1}$; see~Figure~\ref{fig:transfer-to-cascade}. Every $\Fbold$ has the associated \emph{Mother Hedgehog} $\Hbold = \Hbold(\Fbold)$. It is the unique $\Fbold$-invariant closed subset of $\C$ that contains the closure of the critical orbit of $\Fbold$ and such that $\Hbold \cup \{\infty\}$ is a full compact subset of the Riemann sphere. 
Filling in the parabolic fjords of $\Hbold$ at all levels gives us the \emph{top pseudo-Siegel half-plane} $\hat{\Zbold}$, which is uniformly qc equivalent to the standard lower half-plane.

Geometric bounds for $\Hbold$ and $\hat{\Zbold}$ are obtained by transferring the corresponding bounds from bi-infinite towers $\seq{f_n}_{n\in \Z}$ to neutral cascades $\Fbold$; see again~Figure~\ref{fig:transfer-to-cascade}. Consequently, QC Thurston equivalence of forward-equivalent towers (cf. \S\ref{sss:qc.Thurst.eq}) induces QC Thurston equivalence between combinatorially equivalent neutral cascades on their top pseudo-Siegel half-planes.

The top pseudo-Siegel half-plane $\hat{\Zbold}$ cannot be iterated forward. For instance, in the Cremer case, $\hat{\Zbold}$ is not contained in the domain of $\Fbold^P$ for any $P \in \Tbold$. Our primary interest lies in the preimages of $\hat{\Zbold}$ under $\Fbold^P$, which are called \emph{pseudo-bubbles}. 
Each pseudo-bubble encloses a \emph{hoglet}, the corresponding preimage of $\Hbold$.

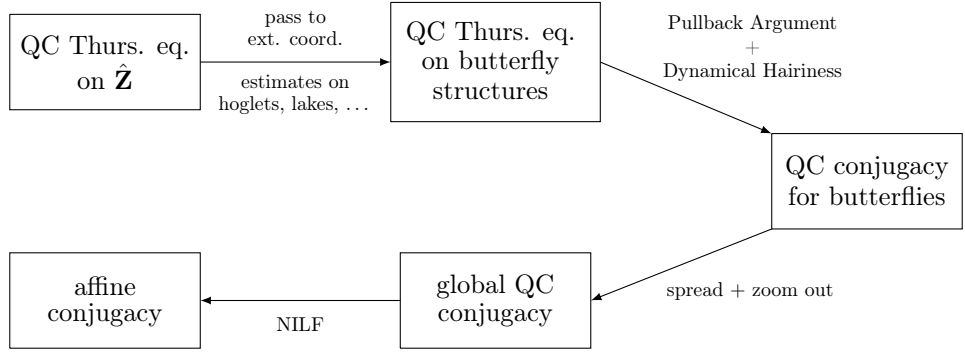
\begin{figure}
    \centering
    \begin{tikzpicture}[scale=0.63]
        \draw (-2,3.5) -- (2,3.5) -- (2,1.5) -- (-2,1.5) -- (-2,3.5);
        \node at (0,2.8) {QC Thurs. eq.};
        \node at (0,2.2) {on $\hat{\Zbold}$};

        \draw[-latex] (2,2.5) -- (6,2.5);
        \node at (4,3.4) {\scalebox{0.75}{pass to}};
        \node at (4,3) {\scalebox{0.75}{ext. coord.}};
        \node at (4,2.1) {\scalebox{0.75}{estimates on}};
        \node at (4,1.6) {\scalebox{0.75}{hoglets, lakes, \dots}};

        \draw (6,3.75) -- (10.4,3.75) -- (10.4,1.25) -- (6,1.25) -- (6,3.75);
        \node at (8.1,3.1) {QC Thurs. eq.};
        \node at (8.1,2.5) {on butterfly};
        \node at (8.1,1.9) {structures};

        \draw[-latex] (10.4,2.5) -- (14,1);
        \node at (13.6,3.3) {\scalebox{0.75}{Pullback Argument}};
        \node at (13.6,2.8) {\scalebox{0.75}{+}};
        \node at (13.6,2.3) {\scalebox{0.75}{Dynamical Hairiness}};

        \draw (14,1) -- (18,1) -- (18,-1) -- (14,-1) -- (14,1);
        \node at (16,0.3) {QC conjugacy};
        \node at (16,-0.3) {for butterflies};

        \draw[-latex] (14,-1) -- (10.2,-2.5);
        \node at (13.55,-2.3) {\scalebox{0.75}{spread $+$ zoom out}};

        \draw (6.2,-3.5) -- (10.2,-3.5) -- (10.2,-1.5) -- (6.2,-1.5) -- (6.2,-3.5);
        \node at (8.2,-2.2) {global QC};
        \node at (8.2,-2.8) {conjugacy};

        \draw[-latex] (6.2,-2.5) -- (2,-2.5);
        \node at (4.1,-3) {\scalebox{0.75}{NILF}};

        \draw (-2,-3.5) -- (2,-3.5) -- (2,-1.5) -- (-2,-1.5) -- (-2,-3.5);
        \node at (0,-2.2) {affine};
        \node at (0,-2.8) {conjugacy};
    \end{tikzpicture}
    \caption{Roadmap of the proof of Theorem \ref{main-thm:cascades}, see~\S\ref{ss:outline}. 
    A key step is to prepare the Pullback Argument~(\S\ref{sss:outl.HogletBounds},~\S\ref{sss:outl.PA}).}
    \label{fig:roadmap}
\end{figure}

\subsubsection{External coordinates}\label{sss:outl.extern.coord}
To control the hoglets and pseudo-bubbles, 
we transfer these objects to \emph{external coordinates} relative to $\Hbold$
(which can be viewed as an inverse to the Douady-Ghys surgery).
It allows us to apply methods of real one-dimensional dynamics.
More precisely, we uniformize the complement of $\Hbold$ to the upper half-plane $\UHP$ via some normalized Riemann mapping $\Psi$.
In doing so, we obtain a semigroup
\[
    \Fext^{P}(z) = \Psi \circ \Fbold^{P} \circ \Psi^{-1}(z) \qquad \text{ for all } P \in \Tbold.
\]
which we call an \emph{external cascade} $\Fext$.
This external cascade extends continuously to a semigroup of self-homeomorphisms of $\R$ and admits Real Bounds (Proposition \ref{prop:R-apb-transcendental}) analogous to analytic critical circle maps.

External coordinates convert the subtle dynamics of $\Fbold$ on the Mother Hedgehog $\Hbold$ to the well-understood dynamics of $\Fext$ on $\R$. Non-JLC-type structures of $\Hbold$ are encoded in ``parabolic fjords" of $\Fext$ attached to  $\R$; 
the angular size of these fjords can be explicitly estimated 
(cf. Figure~\ref{fig:angular-control-fjords} and  Lemma~\ref{lem:fjord-angle-control}).

The downside of the  external coordinates transfer 
is that it break the direct dynamical relationship between the hoglets and the Mother Hedgehog. This is relevant for the NILF Theorem (cf. \S\ref{sss:outl.NIFL}). 

\subsubsection{Hoglets, pseudo-bubbles, and lakes}\label{sss:outl.lakes} As noted in~\S\ref{sss:TD.outline}, our primary interest lies in understanding the dynamics of hoglets and their enclosing pseudo-bubbles. Hoglets are dynamically meaningful objects, but they can exhibit subtle non-JLC-type geometry. We bypass this issue by obtaining uniform geometric control of pseudo-bubbles. Pseudo-bubbles inherit the QC Thurston equivalence from the top pseudo-Siegel half-plane $\hat{\Zbold}$. 

A separate result by the first and the second authors~\cite{DLi26} implies that pseudo-bubbles are uniformly QC and have uniform angular control and uniformly bounded size relative to the relevant renormalization scale (Proposition \ref{prop:bubble-bounds}). 
These properties are referred to as Pseudo-Bubble Bounds throughout this paper. 
See also Figure~\ref{fig:pseudo-bubble-bounds} for an illustration. 
Together with Real Bounds and the angular control of fjords, they imply the disjointness of primary pseudo-bubbles (Lemma~\ref{lem:primary-bubble-disjoint}).

We also establish geometric control for hoglet chains; see Figures~\ref{fig:bubble-chain-01},~\ref{fig:bubble-chain-parabolic}, \ref{fig:lake-boundary}. This leads to the boundedness of \emph{lakes}, 
i.e. the preimages of the upper half-plane under $\Fext$ (Corollary~\ref{cor:bounded-lakes}).

Along the lines, we justify in~\S\ref{ss:alpha-points} the existence of {\em alpha-points} within hoglets and, in~\S\ref{ss:alpha'.hyperb}, establish the hyperbolicity of the $\alpha'$-Dynamics associated to sectorial towers. The alpha-points are the roofs of hoglets. They are also cut-points of the finite-time escaping set of $\Fbold$; compare with~\cite{DLy23}, where such a property was established and widely used for periodic combinatorics. For sectorial towers, the $\alpha'$-hyperbolicity implies that the local geometry at $\alpha'$ points (i.e. preimages of $0$) is stable under a small perturbation. This is one of the ingredients for our 
Hyperbolicity Program~(\S\ref{ss:uniform-hyperbolicity}). 

\subsubsection{Hoglet butterfly bounds}\label{sss:outl.HogletBounds} In~\S\ref{sec:butterfly}, we assemble various components to construct hoglet butterflies together with geometric bounds required for the Pullback Argument. The boundedness of butterflies relative to the combinatorial scale follows from the boundedness of lakes. As Figure~\ref{fig:pseudo-butterfly} illustrates, two wings of a butterfly are anchored to three primary hoglets which are enclosed in pseudo-bubbles. QC Thurston equivalence on these pseudo-bubbles eventually leads to QC Thurston equivalence of butterflies for combinatorially equivalent neutral cascades. The Pullback Argument for butterflies~(\S\ref{sss:outl.PA}) is then run relative to ``dynamically invariant'' hoglets rather than pseudo-bubbles.

To simplify the analysis, we construct butterflies relative to the Sectorial return times. With additional care, we could also justify butterflies relative to Commuting Pair return times as illustrated in Figure~\ref{ss:regularity-of-butterflies}. Refer to Remark~\ref{rem:sector.return.time} for details.

Parabolic limits require extra treatments. For instance, the non-escaping set of a butterfly may contain a parabolic basin. To ensure that the non-escaping set ends up being nowhere dense, we introduce the concept of enriched butterflies (cf. Definition \ref{def:enriched-butterfly}) by adding deeper level maps.

\subsubsection{Hairiness and NILF}\label{sss:outl.NIFL} In Section~\ref{sec:hairiness-NILF}, we prove a number of expansion-type results, two of which are the following.
\begin{itemize}
    \item Dynamical Hairiness Theorem \ref{thm:no-wandering-domains}: iterated preimages of $\Hbold$ under $\Fbold$ are dense in the plane.
    \item \emph{NILF} Theorem \ref{thm:NILF}: the Julia set of $\Fbold$ supports No Invariant Line Field.
\end{itemize}
These theorems are stated in terms of the original transcendental dynamics $\Fbold$, although most of the analysis is carried out in external coordinates. The Hairiness Theorem implies that $\Fbold$ has no wandering domains and the Fatou set, if non-empty, comes from the lift of Siegel disks (the interior of $\Hbold$). The NILF Theorem implies that neutral cascades have trivial QC deformation theory. Therefore, any global QC conjugacy is necessarily affine; see Figure~\ref{fig:roadmap}.

Both the Hairiness and NILF Theorems rely on the property that a point in the upper half-plane outside of hoglets has infinitely many iterates with definite hyperbolic expansion relative to the scale; see Remark~\ref{rem:expanding.iterates} for more details. The NILF Theorem also uses the uniform non-linearity illustrated in Figure~\ref{fig:non-linearity} and the area-zero result by the second author~\cite{Lim26b} for boundaries of hoglets (compare with \cite{Che13,Che19}). 
The area-zero property implies the absence of invariant line fields on the boundaries of hoglets. This separate treatment is necessary because external coordinates break a direct relation between hoglets and the Mother Hedgehog; see~\S\ref{sss:outl.extern.coord}.

\subsubsection{Pullback Argument and conclusions}\label{sss:outl.PA} Consider two combinatorially equivalent neutral cascades $\Fbold$ and $\Gbold$. On every renormalization level, they possess QC Thurston equivalent butterflies with uniform geometric bounds. By running Pullback Argument and using the Hairiness Theorem, we construct QC conjugacy with uniform dilatation. Since butterflies can be selected on an arbitrary shallow level, we can construct a global QC conjugacy that is conformal on the Fatou set. 
By the NILF Theorem, the QC conjugacy is in fact affine. This establishes Theorem~\ref{main-thm:cascades}. The proof is explained in Section \ref{sec:qc-rigidity}.

A quantitative version of Theorem \ref{main-thm:rigidity} (Theorem \ref{main-thm:rigidity-precise}) 
follows from Theorem \ref{main-thm:cascades}. 
Corollary \ref{main-corollary} follows from the continuity of combinatorial data across $\overline{\towsec}$ 
and is explained at the beginning of Section \ref{sec:horseshoe}. The proof of Theorem \ref{main-thm:backward-rigidity} goes along the same lines after one particular modification. Instead of Mother Hedgehogs, we work with Mother Flowers for parabolic maps, which are independent of the forward enrichment combinatorics. 
By filling-in parabolic fjords, Mother Flowers thicken to uniformly quasiconformal pseudo-Siegel flowers.
Details are supplied in Section \ref{sec:backward-rigidity}.

\subsection{Organization}
\label{ss:organization}

The logic of this paper follows the dependency map displayed in Figure \ref{fig:dependency}. Most sections in the paper have brief summaries at their beginning.

Sections \ref{sec:combinatorics} and \ref{sec:pseudo-Siegel} establish the combinatorial and analytic foundations for the  sector renormalization, including the combinatorial space $\TheBiCpt$, pseudo-Siegel disks, and Mother Hedgehogs. 
These are based on previous works \cite{DLy22,Lim26a,DLi26,DLy26a,Lim26b}.
The genuinely new content begins in Section \ref{sec:qc-thurston-equivalence} where we prove the QC Thurston equivalence of combinatorially equivalent forward towers of sector renormalizations.

Most of the content of this paper is dedicated to the dynamics of neutral cascades.
Section \ref{sec:transcendental-dynamics} discusses the construction of neutral cascades and their dynamical sets such as pseudo-Siegel half-planes and Mother Hedgehogs.
Here, we also establish Real Bounds for external cascades.
In Section \ref{sec:bounds}, we establish Pseudo-Bubble Bounds and use them to control the geometry of principal pseudo-bubbles 
and of lakes.
Section \ref{sec:hairiness-NILF} discusses the global qualitative behavior of neutral cascades:
here Dynamical Hairiness Theorem \ref{thm:no-wandering-domains} and NILF Theorem \ref{thm:NILF} are proven.
Then, Section \ref{sec:butterfly} discusses the construction and the uniform geometric and dynamical properties of hoglet butterflies as well as enriched butterflies.

Section \ref{sec:qc-rigidity} applies all of these ingredients to finally prove Theorems \ref{main-thm:rigidity} and \ref{main-thm:cascades}.
Section \ref{sec:backward-rigidity} discusses the construction of Mother Flowers replacing Mother Hedgehogs to prove Theorem \ref{main-thm:backward-rigidity}.
The paper is concluded  in Section \ref{sec:horseshoe} 
with a discussion of the Full Renormalization Horseshoe; 
this includes a strategy for a proof of uniform hyperbolicity, which we aim to carry through in follow-up papers.

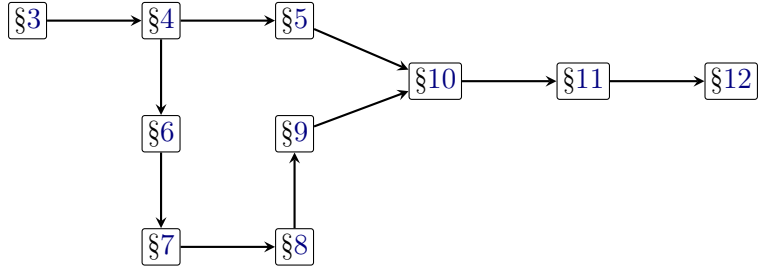
\begin{figure}
    \centering
    \begin{tikzpicture}[
        >=stealth,
        node distance=1cm and 1.25cm,
        sec/.style={
            draw,
            rectangle,
            rounded corners=1pt,
            minimum width=0.45cm,
            minimum height=0.45cm,
            inner sep=2pt,
            font=\large
        },
        arrow/.style={
            ->,
            thick
        }
    ]

    \node[sec] (s3) {\S\ref{sec:combinatorics}};
    \node[sec, right=of s3] (s4) {\S\ref{sec:pseudo-Siegel}};
    \node[sec, right=of s4] (s5) {\S\ref{sec:qc-thurston-equivalence}};
    \node[sec, below=of s4] (s6) {\S\ref{sec:transcendental-dynamics}};
    \node[sec, below=of s6] (s7) {\S\ref{sec:bounds}};
    \node[sec, right=of s7] (s8) {\S\ref{sec:hairiness-NILF}};
    \node[sec, above=of s8] (s9) {\S\ref{sec:butterfly}};
    \node[sec, right=of s5, yshift=-0.8cm] (s10) {\S\ref{sec:qc-rigidity}};
    \node[sec, right=of s10] (s11) {\S\ref{sec:backward-rigidity}};
    \node[sec, right=of s11] (s12) {\S\ref{sec:horseshoe}};

    \draw[arrow] (s3) -- (s4);
    \draw[arrow] (s4) -- (s5);
    
    \draw[arrow] (s4) -- (s6);
    \draw[arrow] (s6) -- (s7);
    \draw[arrow] (s7) -- (s8);
    \draw[arrow] (s8) -- (s9);
    
    \draw[arrow] (s9) -- (s10);
    \draw[arrow] (s5) -- (s10);
    
    \draw[arrow] (s10) -- (s11);
    \draw[arrow] (s11) -- (s12);
    \end{tikzpicture}
    \caption{Dependency structure of this paper.}
    \label{fig:dependency}
\end{figure}

\subsection{Notation and conventions}
\label{ss:notation}

For any open topological disk $U \subsetneq \C$, 
\begin{itemize}
    \item denote by $\partial^c U$ the Carath\'eodory boundary of $U$;
    \item denote by $\textnormal{rad}(U,u)$ the conformal radius of $U$ about a point $u$ in $U$.
\end{itemize}

For any open proper sub-domain $U \subsetneq \C$ of the complex plane, 
\begin{itemize}
    \item denote by $\dist_{U}(x,y)$ the distance between two points $x$ and $y$ in $U$ with respect to the Poincar\'e hyperbolic metric of $U$;
    \item denote by $\D_U(x,r)$ the open round disk centered at a point $x$ in $U$ with radius $r>0$ with respect to the Poincar\'e hyperbolic metric of $U$.
\end{itemize}

For any subset $X$ of the closed upper half plane $\overline{\UHP}$, we define
\begin{itemize}
    \item the real boundary: $\parR X = \partial X \cap \R$,
    \item the essential real boundary: $\parRess X = \overline{\textnormal{int}_{\R}(\parR X)}$,
    \item the upper-half boundary: $\parH X = \partial X \cap \UHP$.
\end{itemize}

For any compact subset $S$ of $\C$ and any constant $\delta > 0$, denote
\[
\nbh_{\delta}(S) := \{ z \in \C \: | \: \dist(z, S) < \delta \, \diam(S) \}.
\]

\subsubsection{On half-Poincar\'e domains}
\label{sss:poincare}

\begin{definition}
    For any interval $I \subset \R$ and any angle $\tau \in (0,\pi)$, we define the \emph{half-Poincar\'e domain} $\poincare_{\tau}(I) \subset \UHP$ with \emph{base} $I$ and \emph{angle} $\tau$ as follows.
    Equip the open domain $\C \backslash ( \overline{\R \backslash I} )$ with the hyperbolic metric; in it, $I$ is a hyperbolic geodesic. 
    Given any angle $\tau \in (0, \pi)$, we define the Poincar\'e neighborhood $P_\tau(I)$ of $I$ with angle $\tau>0$ to be the set of points in $\C \backslash ( \overline{\R \backslash I} )$ that are of hyperbolic distance less than $\log\cot\frac{\tau}{4}$ away from $I$. 
    The upper part is the domain
\[
    \poincare_\tau(I) := P_\tau(I) \cap \UHP.
\]
    The boundary of $\poincare_\tau(I)$ is the union of the interval $I$ and the unique circular arc in $\UHP$ that has the same endpoints as $I$ and meets $\R$ with external angle $\tau$.
\end{definition}

\begin{lemma}[Schwarz Lemma]
    Consider a univalent map $h: \UHP \to \UHP$ that extends continuously to a homeomorphism $h: I \to J$ between two compact real intervals $I$ and $J$.
    Then, for any angle $\tau \in (0,\pi)$,
    \[
        h\left( \poincare_\tau (I) \right) \subset \poincare_\tau (J).
    \]
\end{lemma}

Half-Poincar\'e domains will be used from Section \ref{sec:hairiness-NILF} onwards.

\subsubsection{On the special constant $\threshold$}
\label{sss:threshold}

Throughout this paper, the notation $\threshold \in \N$ will be reserved for the combinatorial threshold that governs when a renormalization level is declared \emph{bounded} or \emph{near-parabolic}.
A related constant is 
\[
\threshold' := \lfloor e^{\sqrt{\log\threshold}} \rfloor
\]
which is explicitly used in the construction of pseudo-Siegel disks (cf. \S\ref{ss:pseudo-siegel}). 
Throughout this paper, $\threshold$ will be assumed to be a fixed sufficiently high number.
Constants that depend only on $\threshold$ will be called $\threshold$-\emph{uniform}.
Constants that are independent of any parameter including $\threshold$ will be called \emph{uniform} or \emph{universal}.
We will denote 
\begin{itemize}
    \item $g \succ h$ if there exists a universal constant $k \in (0,1)$ such that $g \geq k h$;
    \item $g \asymp h$ if $g \succ h$ and $h \succ g$;
    \item $g = O(1)$ if $1 \succ g$;
    \item $g \Msucc h$ if there exists an $\threshold$-uniform constant $k \in (0,1)$ such that $g \geq k h$;
    \item $g \Masymp h$ if $g \Msucc h$ and $h \Msucc g$;
    \item $g = O_{\scriptscriptstyle \threshold}(1)$ if $1 \Msucc g$.
\end{itemize}
Whenever we say that 
\begin{center}
    ``property X holds if $\threshold$ is sufficiently high``
\end{center}
we mean that there exists a uniform constant $\threshold_0 \in \N$ such that when $\threshold \geq \threshold_0$, property X holds.

\subsubsection{On combinatorics}
\label{sss:notation-for-combinatorics}

We denote
\begin{itemize}
    \item $\Irrat = \left( -\frac{1}{2},\frac{1}{2} \right) \backslash \Q$;
    \item $\Sigma^{\N} = ( \{-,+\} \times \N_{\geq 2} )^{\N}$;
    \item $\TheCpt = ( \{-,+\} \times \overline{\N}_{\geq 2} )^{\N}$, the parabolic compactification of irrationals;
    \item $\TheBiCpt = ( \{-,+\} \times \overline{\N}_{\geq 2} )^{\Z}$, the bi-infinite version of $\TheCpt$;
    \item $\gauss =$ the Gauss-like map $\Irrat \to \Irrat$ where $\gauss(\theta) \equiv -\frac{1}{\theta}$ (mod $1$);
    \item $\mathfrak{X} = $ the map $\Sigma^{\N} \to \Irrat$ described in (\ref{eqn:symbolic-irrational});
    \item $\bar{\mu} =$ the rotation number map $\TheCpt \to \T$ described in Proposition \ref{prop:rotation-number-map};
    \item $\shift = $ standard shift map on either $\TheCpt$ or $\TheBiCpt$;
    \item $\textnormal{proj} = $ the projection map $\TheBiCpt \to \TheCpt, \langle (\varepsilon_n, \abar_n) \rangle_{n \in \Z} \mapsto \langle (\varepsilon_n, \abar_n) \rangle_{n \geq 1}$.
\end{itemize}
For $\theta \in \Irrat$ and $n \geq 0$, denote
\begin{itemize}
    \item $ b_n(\theta) = a_n + \frac{1+\varepsilon_n \varepsilon_{n+1}}{2}$ where $\seq{(\varepsilon_n,\abar_n)}_{n\geq 1} = \mathfrak{X}^{-1}(\theta)$;
    \item $q_{[n]} = q_{[n]}(\theta)$, the sequence described in Proposition \ref{prop:q[n]}.
\end{itemize}
For $\ttheta = \seq{(\varepsilon_n,\abar_n)}_{n\geq 1} \in \TheCpt$, denote
\begin{itemize}
    \item $b_n = b_n(\ttheta) := \abar_n + \frac{1+\varepsilon_n \varepsilon_{n+1}}{2}$ for $n \geq 0$;
    \item $\Kont = \Kont_{\ttheta}$, the continuant group of $\ttheta$ (cf. Definition \ref{def:time-group}), an abelian group with generators $\qq_{[n]}(\ttheta)$, $n \geq 0$, equipped with the time order ``$<$'';
    \item $\Time = \Time_{\ttheta} := \{ p \in \Kont_{\ttheta} \: : \: p > 0 \}$, the time semigroup of $\ttheta$.
\end{itemize}
For $\tttheta = \seq{(\varepsilon_n,\abar_n)}_{n\in \Z} \in \TheBiCpt$, denote
\begin{itemize}
    \item $b_n = b_n(\tttheta) := \abar_n + \frac{1+\varepsilon_n \varepsilon_{n+1}}{2}$ for $n \in \Z$;
    \item $\Kbold = \Kbold_{\tttheta}$, the continuant group of $\tttheta$ (cf. \S~\ref{sss:continuant-group}), an abelian group with generators $Q_{[n]}(\ttheta)$, $n \in \Z$, equipped with the time order ``$<$'';
    \item $\Tbold = \Tbold_{\tttheta} := \{ P \in \Kbold_{\ttheta} \: : \: P > 0 \}$, the time semigroup of $\tttheta$;
    \item $\{ Q_n = Q_n(\tttheta) \}_{n \in \Z}$, the slow return times of $\tttheta$, another set of generators of $\Kbold_{\tttheta}$ with relations $\{a_n=a_n(\tttheta)\}_{n\in\Z} \in \TheBiCpt$ (c.f \S\ref{sss:conversion});
    \item $\Tbold_{\Arch, n} = \Tbold_{\Arch, n, \tttheta}$, the sub-semigroup of $\Tbold$ consisting of elements that are not infinitely larger than $Q_n$ for $n \in \Z$ (cf. Section \ref{sec:hairiness-NILF}).
\end{itemize}

\subsubsection{On dynamical objects related to $\overline{\towsec}$}
\label{sss:notation-for-towsec}

The space $\overline{\towsec}$ will be introduced in \S\ref{ss:renormalization-towers}.
Every element $\fbold$ of $\overline{\towsec}$ is an infinite forward tower $\fbold = \seq{f_n}_{n\geq 0}$ of sector renormalizations of holomorphic maps with a neutral fixed point at $0$.
It comes with renormalization combinatorics $\ttheta = \ttheta(\fbold)$ which is an element of $\TheCpt$.
The properties of the objects listed below are described in Theorem \ref{thm:renorm-limits}.
For every $n \geq 0$, denote:
\begin{itemize}
    \item $\ttheta_n = \shift(\ttheta_n)$ and $\theta_n = \bar{\mu}(\ttheta_n)$, the rotation number of $f_n$ at $0$;
    \item the continuant group $\Kont^{n} = \Kont_{\ttheta_n}$ of $\ttheta_n$ with generators $\qq^n_{[m]} = \qq_{[m]}(\ttheta_n)$, $m \geq 0$;
    \item the time semigroup $\Time^n = \Time_{\ttheta_n}$ of $\ttheta_n$;
    \item the $n$\textsuperscript{\textnormal{th}} semigroup $\mathcal{F}_n = \{ f_n ^p \}_{p \in \Time^n}$ parametrized by $\Time^n$;
    \item the nest of pseudo-Siegel pinched disks $\hat{Z}_n^{[m]}=\hat{Z}_n^{[m]}(\fbold)$, $m \geq -1$ where each $\hat{Z}_n^{[m]}$ is almost invariant under $f_n^{[m+1]} := f_n^{\qq^n_{[m+1]}} \in \mathcal{F}_n$;
    \item the Mother Hedgehog $H_n=H_n(\fbold)$ completely invariant under $\mathcal{F}_n$;
    \item the unique bi-infinite critical orbit $\{v^n_{p} = v^n_{p}(\fbold)\}_{p \in \Kont^{n}}$ in $H_n$ where $v^n_{0}$ is the unique critical value of $f_n$;
    \item the postcritical set $P_n = P_n(\fbold) := \overline{ \big\{ v^n_{p} \big\}_{p \in \Tbold^n \cup \{0\}} }$ of $\mathcal{F}_n$, which is the boundary of $H_n$;
    \item a nest of renormalization sectors $S_n^j = S_n^j(\fbold)$, $j \geq 1$ where
    \[
        S_n := S_n^1 \supset S_n^2 \supset S_n^3 \supset \ldots;
    \]
    \item a conformal gluing map 
    \[
        \psi_n = \psi_{n,\fbold}: S_n \to \D
    \]
    for the top sector $S_n$ that projects $\hat{Z}_n^{[m-1]}$, $S_n^{m+1}$, and $f_n^{[m]}$ 
    to $\hat{Z}_{n+1}^{[m-2]}$, $S_{n+1}^{m}$, and $f_{n+1}^{[m-1]}$ respectively for every $m \geq 1$.
\end{itemize}
In Section \ref{sec:transcendental-dynamics}, we will also introduce the renorm-attractor $\attr$, a bi-infinite version of $\overline{\towsec}$.
Every element $\fboldbar$ of $\attr$ is a bi-infinite tower $\seq{f_n}_{n\in\Z}$ of sector renormalizations of holomorphic maps with a neutral fixed point at $0$.
It comes with renormalization combinatorics $\tttheta = \tttheta(\fboldbar)$, an element of $\TheBiCpt$, and a corresponding continuant group $\Tbold = \Tbold_{\tttheta}$.
For every $n \in \Z$, we have an element $\fbold_n = \seq{ f_{n+k} }_{k\geq 0}$ of $\overline{\towsec}$ which comes with analogous objects $\ttheta_n$, $\theta_n$, $\Time^{n}$, $\{ \qq^n_{[m]} \}_{m \geq 0}$, $\mathcal{K}^n$, $\mathcal{F}_n$, $\{ \hat{Z}_n^m \}_{m \geq -1}$, $H_n$, $\{ v^n_p \}_{p \in \Time^{n}}$, $P_n$, $\{S_n^j\}_{j \geq 1}$, and $\psi_n$.

\subsubsection{On neutral cascades}
\label{sss:cascade}

In Section \ref{sec:transcendental-dynamics}, we will introduce the notion of a \emph{neutral cascade} $\Fbold = (\Fbold^P: \Dom(\Fbold^P) \to \C )_{P \in \Tbold}$ associated to an element $\fboldbar = \seq{ f_n }_{n \in \Z}$ of $\attr$.
It comes with
\begin{itemize}
    \item $\tttheta = \tttheta(\Fbold)$, an element of $\TheBiCpt$ which determines the continuant group $\Kbold = \Kbold_{\tttheta}$ and the time semigroup $\Tbold = \Tbold_{\tttheta}$ parametrizing $\Fbold$;
    \item $\Ibold^{\leq P} = \C \backslash \Dom(\Fbold^P)$, the time $P \in \Tbold$ escaping set of $\Fbold$;
    \item $\{ \Sbold^m \}_{m \in \Z}$, the nest of infinite sectors described in Proposition \ref{prop:trans-sectors};
    \item $\hat{\Zbold}^{[m]} = $ the level $m$ pseudo-Siegel pinched half-plane of $\Fbold$ for $m \in \Z$;
    \item $\hat{\Zbold} = \cup_m \hat{\Zbold}^{[m]}$, the top pseudo-Siegel half-plane;
    \item $\Hbold = \cap_m \hat{\Zbold}^{[m]}$, the Mother Hedgehog of $\Fbold$;
    \item $\{ C_P \}_{P \in \Kbold}$, the bi-infinite critical orbit of $\Fbold$ on the boundary of $\Hbold$;
    \item $\Psi: \C \backslash \Hbold \to \UHP$, the normalized Riemann mapping described in \S\ref{ss:trans-external-coordinates};
    \item $\Bbold_P = $ the primary hoglet of generation $P \in \Tbold$ (cf. \S\ref{ss:bubbles-and-pseudo-bubbles});
    \item $\hat{\Bbold}_P =$ the primary pseudo-bubble of generation $P \in \Tbold$ (cf. \S\ref{ss:bubbles-and-pseudo-bubbles});
    \item $\aalpha_P =$ the alpha-point of $\Bbold_P$ for $P \in \Tbold$ (cf. \S\ref{ss:alpha-points}).
\end{itemize}
Via the Riemann mapping $\Psi$, $\Fbold$ is transformed to an \emph{external cascade} $\Fext = (\Fext^P: \Dom(\Fext^P) \to \UHP)_{P \in \Tbold}$.
In external coordinates, we denote
\begin{itemize}
    \item $\{ \crit_P = \Psi(C_P) \}_{P \in \Tbold}$, the bi-infinite ``critical`` orbit located along $\R$;
    \item $\tiling_Q =$ the tiling of $\R$ by $\{\crit_{-P}\}_{0 < P \leq Q}$ for any $Q \in \Tbold$;
    \item $\Fjord^{[n]} = $ the union of fjords of level $n \in \Z$;
    \item $\Fjord = \cup_n \Fjord_n$, the union of all fjords;
    \item $\fjord_n = $ the principal level $n$ fjord (cf. Definition \ref{def:parabolic-fjord});
    \item $\Iext^{\leq P} = \Psi(\Ibold^{\leq P})$, $\Bubb_P = \Psi(\Bbold_P)$, $\hat{\Bubb}_P = \Psi(\hat{\Bubb}_P)$, and $\alpha_P = \Psi(\aalpha_P)$ for $P \in \Tbold$;
    \item $\lake_P = $ the primary lake of generation $P \in \Tbold$ with $0$ on its real boundary;
    \item $\Bchain_n =$ the (enriched) chain of hoglets $\Gamma_{n,j}$ for any $n \in \Z$ (cf. (\ref{eqn:hoglet-chain-old}));
    \item $\hat{\Bchain}_n =$ the pseudo-bubble version of $\Bchain_n$;
    \item $\beta_n =$ the unique fixed point of $\Fext^{Q_n}: \lake_{Q_n} \cup \partial^c \lake_{Q_n} \to \overline{\UHP}$ for any $n \in \Z$;
    \item $\Dext_n =$ the immediate parabolic basin of $\beta_n$ for any parabolic depth $n \in \Z$.
\end{itemize}


\section{The combinatorics of sector renormalization}
\label{sec:combinatorics}

In this section, we summarize the combinatorial foundations of sector renormalization. All the proofs can be found in \cite{Lim26a}. The key notations introduced in this section are listed in \S\ref{sss:notation-for-combinatorics}. Let us provide a brief summary.

We denote by $\Irrat$ the set of normalized irrational numbers; see~\eqref{eq:irrat.dfn}. 
In \S\ref{ss:sect.renorm}, we demonstrate that for every $\theta\in \Irrat$, the combinatorial data of the rigid rotation $\rotate_\theta$ under sector renormalization can be uniquely encoded by a sequence $\seq{ (\varepsilon_n, \abar_n) }_{n\geq 1}\in \The$. Theorem~\ref{thm:X.theta} asserts that this encoding naturally gives us the modified continued fraction expansion~\eqref{eqn:symbolic-irrational}.

The set of irrational rotation numbers $\Irrat$ has a parabolic compactification $\TheCpt$, see~\eqref{eq:dfn.ovlTheta}, which satisfies the associated universal property stated in Theorem~\ref{thm:ovlTheta:univers}. In ~\S\ref{ss:natural-extension}, we introduce the natural extension $\TheBiCpt$ of $\TheCpt$ with respect to sector renormalization. Associated to every element $\tttheta$ of $\TheBiCpt$ is the \emph{time semigroup} $\Tbold_{\tttheta}$ and the full \emph{continuant group} $\Kbold_{\tttheta} = \Tbold_{\tttheta} \cup \{0\} \cup -\Tbold_{\tttheta}$. In the irrational case, $\Tbold_{\tttheta}$ can be thought as a natural parametrization of the associated cascades of translations; see~\S\ref{sss:irrat.cascade}. In general, $\Tbold_{\tttheta}$ will be used in Section \ref{sec:transcendental-dynamics} to parametrize transcendental neutral cascades $\Fbold$.

\subsection{Sector renormalization of irrational rotations}
\label{ss:sect.renorm}

Denote
\begin{equation}
\label{eq:irrat.dfn}
    \Irrat := \left( -\frac{1}{2},\frac{1}{2} \right) \backslash \Q.
\end{equation}

For $\theta \in \Irrat$, denote
\[
    \varepsilon(\theta) := \frac{\theta}{|\theta|} \in \{-1,+1\}, \quad \quad 
    \abar(\theta) := \left\lfloor \frac{1}{|\theta|} \right\rfloor \in \N_{\geq 2},
\]
and let
\[
    \rotate_\theta : \bar{\D} \to \bar{\D}, \qquad z \mapsto e^{2\pi i \theta} z
\]
be the rigid rotation by angle $\theta$ on the closed unit disk $\bar{\D} \subset \C$.
The sector renormalization of $\rotate_\theta$ is constructed as follows.

Let $S$ be a sector on $\bar{\D}$ bounded by two arcs $\gamma$ and $\rotate_\theta(\gamma)$ where $\gamma$ is a radial arc from $0$ to a point on the unit circle.
The power map $\psi(z) = z^{1/|\theta|}$ glues the two sides of the sector $S$ together, sending $S$ onto $\bar{\D}$.
The first return map of $\rotate_\theta$ back to $S$ is a piecewise continuous map consisting of a pair of iterates $(\rotate_\theta^{\abar}, \rotate_\theta^{\abar+1})$, where $\abar = \abar(\theta)$.
Under $\psi$, this pair projects onto a new rigid rotation $\rotate_{\gauss(\theta)}: \bar{\D} \to \bar{\D}$, called the sector renormalization of $\rotate_{\theta}$.
By elementary computation, the new rotation number $\gauss(\theta)$ is the unique irrational number in $\Irrat$ with
\[
    \gauss(\theta) \equiv -\frac{1}{\theta} \quad (\textnormal{mod }1).
\]
This procedure gives us a continuous infinite-to-one surjective self-map $\gauss: \Irrat \to \Irrat$.

Consider the infinite sequence space
\[
    \The := \left\{ \seq{ (\varepsilon_n, \abar_n) }_{n\geq 1} \: : \: \varepsilon_n \in \{-1,+1\}, \abar_n \in \N_{\geq 2} \right\}.
\]
For every $\sigma = \seq{(\varepsilon_n, \abar_n )}_{n\geq 1}$, we denote
\begin{equation}
   \label{eq:dfn.b_n}
    b_n = b_n(\sigma) := \abar_n + \frac{1 + \varepsilon_n \varepsilon_{n+1}}{2} \qquad \text{ for } n \geq 1
\end{equation}
and the associated modified continued fraction rel $\Big(-\frac 12, \frac 12\Big)$

\begin{align}
\label{eqn:symbolic-irrational}
    \mathfrak{X}( \sigma ) := \cfrac{1}{\varepsilon_1 b_1 - \cfrac{1}{\varepsilon_2 b_2 - \frac{1}{\varepsilon_3 b_3 - \ldots}}}.
\end{align}

\begin{theorem}\label{thm:X.theta}
    The function $\mathfrak{X}(\cdot)$ above is a well-defined homeomorphism from $\The$ onto $\Irrat$ with inverse
    \[
        \mathfrak{X}^{-1}(\theta) = \seq{ \big( \varepsilon( \gauss^{n-1}(\theta)),\;\abar(\gauss^{n-1}(\theta) ) \big) }_{n\geq 1}.
    \]
    The map $\mathfrak{X}$ is a conjugacy between the standard shift map $\shift: \The \to \The$ and the map $\gauss: \Irrat \to \Irrat$.
\end{theorem}

Fix $\theta \in \Irrat$ and denote $\theta_n := \gauss^n(\theta)$.
We will now describe a particular choice of the renormalization sector $S$ for $\rotate_\theta$ that is more appropriate for iteration.

For $k \in \Z$, denote 
\[
    v_k = v_k(\theta) := \rotate_\theta^k(1).
\]
For any two distinct integers $k,l \in \Z$, we denote by $\triangle_{\theta}(k,l)$ the closed radial sector bounded by the radial line segment from $0$ to $v_k$, the radial line segment from $0$ to $v_l$, and the shortest circular arc in $\partial \D$ joining $v_k$ and $v_l$.
We will in particular consider the sector
\[
    S_\theta := \triangle_{\theta}(-\abar-1,-\abar),
\]
where $\abar=\abar(\theta)$. 
This sector contains the point $v_0 = 1$. 

The first return map of $\rotate_{\theta}$ back to $S_{\theta}$ is the piecewise linear map
\[
    \textnormal{FRM}_{\theta}(z) = \begin{cases}
        \rotate_{\theta}^{\abar+1}(z) & \text{ if } z \in \triangle_{\theta}(-\abar-1,-2\abar-1) , \\
        \rotate_{\theta}^{\abar}(z) & \text{ if } z \in \triangle_{\theta}(-2\abar-1,-\abar) .
    \end{cases}
\]
See Figure \ref{fig:sector-rotation}.
The power map 
\[
    \psi_{\theta}: S_{\theta} \to \bar{\D}, \quad \psi_\theta(z) = z^{1/|\theta|}
\]
sends the interior of $S_{\theta}$ conformally onto the unit disk minus the radial slit 
\[
\gamma_\theta = \{ \arg z = - 2\pi \, \gauss(\theta) \},
\]
and it glues the two radial edges of $S_{\theta}$ together onto the slit. 
It also projects $\textnormal{FRM}_{\theta}(z)$ to the renormalization $\rotate_{\gauss(\theta)}$.
Since $\gamma_{\theta}$ is disjoint from the sector $S_{\theta_1}$, we obtain a well-defined nest of closed sectors 
\[
S_{\theta}^1 \ \supset \ S_{\theta}^2 \ \supset \ 
S_{\theta}^3 \ \supset \ S_{\theta}^4 \ \supset \ \ldots
\]
where 
    \[
    S_{\theta}^n := \psi_{\theta_0}^{-1} \circ \psi_{\theta_1}^{-1} \circ \ldots \circ \psi_{\theta_{n-2}}^{-1}(S_{\theta_{n-1}})
    \quad \text{ for } n\geq 2.
    \]
By design, the sector $S_{\theta}^n$ always contains the real interval $[0,1] \subset \R$.
Proposition \ref{prop:q[n]} below implies that the nested intersection $\cap_n S_{\theta}^n$ is indeed equal to $[0,1]$.

\begin{figure}    
    \centering
    \begin{tikzpicture}[scale=1]
        \filldraw[white,fill opacity=0.5,fill=green!25!white] (-3.5,0) -- (-1,0) arc (0:17:2.5cm) -- cycle;        
        \filldraw[white,fill opacity=0.5,fill=yellow!30!white] (-3.5,0) -- ({2.5*cos(-32)-3.5},{2.5*sin(-32)}) arc (-32:0:2.5cm) -- cycle;
        \draw [black,domain=0:360,smooth] plot ({2.5*cos(\x)-3.5}, {2.5*sin(\x)});
        \draw [blue,thick,domain=-32:17,-latex] plot ({1.25*cos(\x)-3.5}, {1.25*sin(\x)});
        \draw[black,line width=0.5pt] ({2.5*cos(-32)-3.5},{2.5*sin(-32)}) -- (-3.5,0) -- ({2.5*cos(17)-3.5},{2.5*sin(17)}); 
        \draw[gray!70!white,line width=0.5pt] (-3.5,0) -- (-1,0);

        \filldraw[white,fill opacity=0.5,fill=yellow!30!white] (0.5,0) -- ({2.5*cos(-15)+0.5},{2.5*sin(-15)}) arc (-15:17:2.5cm) -- cycle;   
        \filldraw[white,fill opacity=0.5,fill=green!25!white] (0.5,0) -- ({2.5*cos(-32)+0.5},{2.5*sin(-32)}) arc (-32:-15:2.5cm) -- cycle;
        \draw[black,line width=0.5pt] ({2.5*cos(17)+0.5},{2.5*sin(17)}) -- (0.5,0) -- ({2.5*cos(-32)+0.5},{2.5*sin(-32)});        
        \draw[gray!70!white,line width=0.5pt] ({2.5*cos(-15)+0.5},{2.5*sin(-15)}) -- (0.5,0);
        \draw [black,domain=-32:17] plot ({2.5*cos(\x)+0.5}, {2.5*sin(\x)});

        \filldraw (-1,0) circle (2pt);
        \filldraw ({2.5*cos(-49)-3.5},{2.5*sin(-49)}) circle (2pt);
        \filldraw ({2.5*cos(-98)-3.5},{2.5*sin(-98)}) circle (2pt);
        \filldraw ({2.5*cos(-147)-3.5},{2.5*sin(-147)}) circle (2pt);
        \filldraw ({2.5*cos(-196)-3.5},{2.5*sin(-196)}) circle (2pt);
        \filldraw ({2.5*cos(-245)-3.5},{2.5*sin(-245)}) circle (2pt);
        \filldraw ({2.5*cos(-294)-3.5},{2.5*sin(-294)}) circle (2pt);
        \filldraw ({2.5*cos(-343)-3.5},{2.5*sin(-343)}) circle (2pt);
        \filldraw ({2.5*cos(-392)-3.5},{2.5*sin(-392)}) circle (2pt);

        \node [black, font=\bfseries] at (1.75,2) {$\varepsilon =+1$, $\abar= 7$};
        \node [blue, font=\bfseries] at (-1.95,-0.3) {\small $\rotate_{\theta}$};
        \draw[yellow!30!black,line width=0.5pt,-latex] (1.5,0.5).. controls (0.75,1.1) and (-0.5,-0.7) .. (-0.95,-0.7) ;
        \draw[green!50!black,line width=0.5pt,-latex] (1.5,-0.8) .. controls (0.5,-1) and (0,0.2) .. (-0.9,0.4);
        \node [green!50!black, font=\bfseries] at (0.4,-0.9) {$\rotate^8_{\theta}$};
        \node [yellow!30!black, font=\bfseries] at (0.4,0.8) {$\rotate^7_{\theta}$};
        \node [black, font=\bfseries] at (-0.67,-0.1) {\small $v_{0}$};
        \node [black, font=\bfseries] at (-1.4,-2) {\small $v_{-1}$};
        \node [black, font=\bfseries] at (-0.75,0.85) {\small $v_{-7}$};
        \node [black, font=\bfseries] at (-1,-1.45) {\small $v_{-8}$};

        \filldraw[white,fill opacity=0.5,fill=green!25!white] (-3.5,-5.5) -- (-1,-5.5) arc (0:-17:2.5cm) -- cycle;        
        \filldraw[white,fill opacity=0.5,fill=yellow!30!white] (-3.5,-5.5) -- (-1,-5.5) arc (0:32:2.5cm) -- cycle;
        \draw [black,domain=0:360,smooth] plot ({2.5*cos(\x)-3.5}, {2.5*sin(\x)-5.5});        
        \draw [blue,thick,domain=32:-17,-latex] plot ({1.25*cos(\x)-3.5}, {1.25*sin(\x)-5.5});
        \draw[black,line width=0.5pt] ({2.5*cos(32)-3.5},{2.5*sin(32)-5.5}) -- (-3.5,-5.5) -- ({2.5*cos(-17)-3.5},{2.5*sin(-17)-5.5});        
        \draw[gray!70!white,line width=0.5pt] ({2.5*cos(0)-3.5},{2.5*sin(0)-5.5}) -- (-3.5,-5.5);

        \filldraw[white,fill opacity=0.5,fill=yellow!30!white] (0.5,-5.5) -- ({2.5*cos(-17)+0.5},{2.5*sin(-17)-5.5}) arc (-17:15:2.5cm) -- cycle;        
        \filldraw[white,fill opacity=0.5,fill=green!25!white] (0.5,-5.5) -- ({2.5*cos(15)+0.5},{2.5*sin(15)-5.5}) arc (15:32:2.5cm) -- cycle;
        \draw[black,line width=0.5pt] (0.5,-5.5) -- ({2.5*cos(-17)+0.5},{2.5*sin(-17)-5.5}) arc (-17:32:2.5cm) -- cycle;        
        \draw[gray!70!white,line width=0.5pt] ({2.5*cos(15)+0.5},{2.5*sin(15)-5.5}) -- (0.5,-5.5);
        
        \filldraw (-1,-5.5) circle (2pt);
        \filldraw ({2.5*cos(49)-3.5},{2.5*sin(49)-5.5}) circle (2pt);
        \filldraw ({2.5*cos(98)-3.5},{2.5*sin(98)-5.5}) circle (2pt);
        \filldraw ({2.5*cos(147)-3.5},{2.5*sin(147)-5.5}) circle (2pt);
        \filldraw ({2.5*cos(196)-3.5},{2.5*sin(196)-5.5}) circle (2pt);
        \filldraw ({2.5*cos(245)-3.5},{2.5*sin(245)-5.5}) circle (2pt);
        \filldraw ({2.5*cos(294)-3.5},{2.5*sin(294)-5.5}) circle (2pt);
        \filldraw ({2.5*cos(343)-3.5},{2.5*sin(343)-5.5}) circle (2pt);
        \filldraw ({2.5*cos(392)-3.5},{2.5*sin(392)-5.5}) circle (2pt);

        \node [black, font=\bfseries] at (1.75,-3.5) {$\varepsilon=-1$, $\abar = 7$};
        \node [blue, font=\bfseries] at (-1.95,-5.25) {$\rotate_{\theta}$};
        \draw[green!50!black,line width=0.5pt,-latex] (1.5,-4.7).. controls (0.75,-4.1) and (-0.5,-5.9) .. (-0.95,-5.9) ;
        \draw[yellow!30!black,line width=0.5pt,-latex] (1.5,-6) .. controls (0.2,-6.2) and (-0.2,-5) .. (-0.95,-4.8);
        \node [green!50!black, font=\bfseries] at (0.4,-4.45) {$\rotate^8_{\theta}$};
        \node [yellow!30!black, font=\bfseries] at (0.4,-6.2) {$\rotate^7_{\theta}$};
        \node [black, font=\bfseries] at (-0.68,-5.48) {\small $v_{0}$};
        \node [black, font=\bfseries] at (-1.45,-3.55) {\small $v_{-1}$};
        \node [black, font=\bfseries] at (-0.75,-6.4) {\small $v_{-7}$};
        \node [black, font=\bfseries] at (-1,-4) {\small $v_{-8}$};
\end{tikzpicture}
    \caption{Two examples of the first return map on the shaded sector $S_{\theta}$.}
    \label{fig:sector-rotation}
\end{figure}
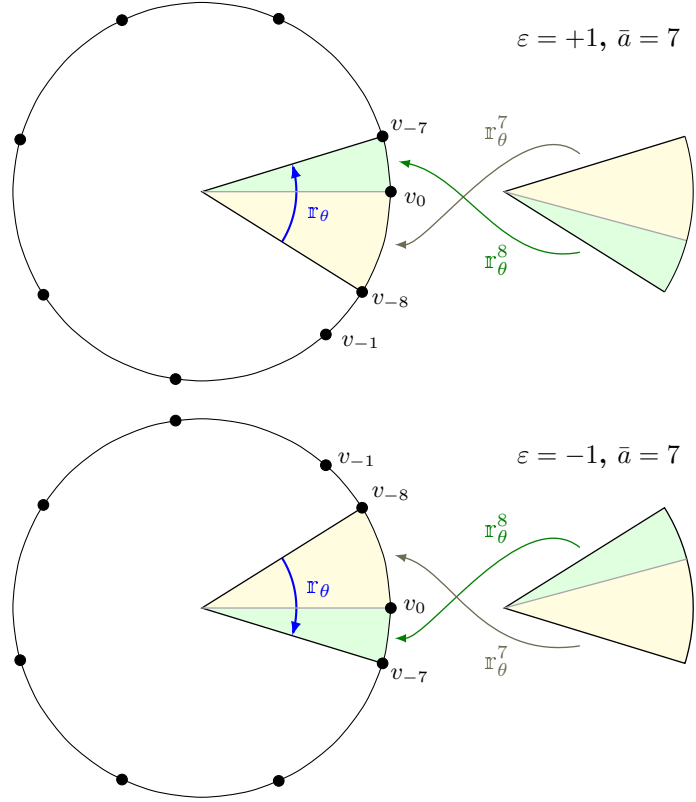

\begin{proposition}[First return times]
\label{prop:q[n]}
    Let $\sigma = \seq{ (\varepsilon_n, \abar_n)}_{n\geq 1} \in \The$ and $\theta = \mathfrak{X}(\sigma)$.
    Recall the notation $b_n$, $\theta_n$, and $v_n$ associated to $\sigma$ and $\theta$ introduced above.
    Consider the increasing sequence of integers $\{q_{[n]}=q_{[n]}(\sigma)\}_{n\geq 0}$ where
\[
    q_{[0]} = 1, \qquad q_{[1]} = b_1, 
\]
and recursively for $n\geq 2$,
\[
    q_{[n]} = b_n q_{[n-1]} - \varepsilon_{n-1} \varepsilon_n  q_{[n-2]}.
\]
    Then, for $n \geq 1$,
\begin{enumerate}
    \item $\rotate_{\theta}^{q_{[n]}}$ is the rigid rotation by angle $|\theta_0\ldots\theta_{n-1}|\theta_n$;
    \item $q_{[n]}$ is the unique smallest positive integer that satisfies 
    \[
    \left| [v_0, v_{q_{[n]}]}] \right| \leq \frac{1}{2} \left| [v_0, v_{q_{[n-1]}}] \right|,
    \]
    where $|\cdot|$ denotes the normalized angular measure on $\partial \D$;
    \item if we denote 
    \[
        \check{q}_{[n]} := q_{[n]} - \varepsilon_n \varepsilon_{n+1} q_{[n-1]},
    \]
    then the sector $S_{\theta}^n$ is equal to 
    $ \triangle_{\theta} (-q_{[n]}, -\check{q}_{[n]})$, and the first return map of $\rotate_\theta$ back to $S_{\theta}^n$ is the pair of iterates
    \[
        \rotate_\theta^{q_{[n]}}: 
        \triangle_{\theta} (-q_{[n]}, -q_{[n]} - \check{q}_{[n]})
        \to  
        \triangle_{\theta}(0,-\check{q}_{[n]})
    \]
    and
    \[ 
        \rotate_\theta^{\check{q}_{[n]}}: \triangle_{\theta}(-q_{[n]} - \check{q}_{[n]}, -\check{q}_{[n]}) 
        \to 
        \triangle_{\theta}(-q_{[n]}, 0). 
    \]
\end{enumerate}
\end{proposition}

Essentially, item (3) of the proposition above says that $\rotate_\theta^{q_{[n]}}|_{S_{\theta}^n}$ is the $n$\textsuperscript{th} pre-renormalization of $\rotate_\theta$:
\[
    \rotate_{\theta_n} = \rotate_\theta^{q_{[n]}}|_{S_{\theta}^n} \Big/ \rotate_\theta^{q_{[n-1]}}.
\]

\begin{example}
    The golden mean irrationals in $\Irrat$ are $\theta_{\textnormal{gm}} = \frac{3-\sqrt{5}}{2}$ and $-\theta_{\textnormal{gm}} = \frac{\sqrt{5}-3}{2}$. We have
    \begin{align*}
        \mathfrak{X}^{-1}(\theta_{\textnormal{gm}}) &= \seq{ (+,2), (+,2), (+,2), (+,2),\ldots },\\
        \mathfrak{X}^{-1}(-\theta_{\textnormal{gm}}) &= \seq{ (-,2), (-,2), (-,2), (-,2), \ldots }.
    \end{align*}
    The corresponding modified continued fraction expansions are
    \[
        \theta_{\textnormal{gm}} = \cfrac{1}{3 - \frac{1}{3 - \frac{1}{3 - \ldots}}} \quad \text{ and } \quad 
        -\theta_{\textnormal{gm}} = \cfrac{1}{-3 - \frac{1}{-3 - \frac{1}{-3 - \ldots}}}.
    \]
    Both $\theta_{\textnormal{gm}}$ and $-\theta_{\textnormal{gm}}$ have the same first return times, namely
    \[
        q_{[0]} = 1, \quad q_{[1]} = 3, \quad q_{[2]}=8, \quad q_{[3]} = 21, \quad q_{[4]} = 55, \quad q_{[5]} = 144, \quad \ldots.
    \]
\end{example}

\subsection{The parabolic compactification of irrationals}
\label{ss:compactification}

Let $\overline{\N} = \N \cup \{\infty\}$, where the addition of the point $\displaystyle \infty = \lim_{n \to \infty} n$ makes it a compactification of the discrete space $\N$.
Let 
\begin{equation}
    \label{eq:dfn.ovlTheta}
\overline{\Sigma} := \{-1,+1\} \times \overline{\N}_{\geq 2} \qquad \text{ and } \qquad
    \TheCpt := \overline{\Sigma}^{\N}.
\end{equation}
The infinite sequence space $\TheCpt$ is a Cantor set. The inclusion map $\iota: \The \xhookrightarrow[]{} \TheCpt$ has a dense image.
Through $\iota \circ \mathfrak{X}^{-1}$, the space $\TheCpt$ together with the shift map $\shift$ is a compactification of the dynamical system $(\Irrat,\gauss)$.

\begin{theorem}[Universal property of $\TheCpt$]\label{thm:ovlTheta:univers}
    The embedding 
    \[
    \bar{\iota} :=\iota \circ \mathfrak{X}^{-1}: \Irrat \to \TheCpt
    \]
    is the smallest compactification of $\Irrat$ satisfying the following properties.
    The dynamical system $(\Theta,\gauss)$ extends to a continuous self-map on $\TheCpt$, which is $(\TheCpt, \shift)$, and various natural embeddings of $\Irrat$ into some compact spaces extend continuously.
\end{theorem}

Refer to \cite[Theorem B]{Lim26a} for the details of these natural embeddings.
Roughly, this formal theorem states that $\TheCpt$ is the most appropriate compactification of the irrationals that takes into account parabolic implosion.
For this reason, $\TheCpt$ is called the \emph{parabolic compactification} of the space of irrationals $\Irrat$.

\begin{definition}
    An element $\ttheta = \seq{ (\varepsilon_n, \abar_n) }_{n\geq 1}$ of $\TheCpt$ is said to be \emph{irrational} if $\abar_n < \infty$ for all $n \geq 1$, \emph{enriched rational} otherwise.
\end{definition}

Here is another useful property of $\TheCpt$.

\begin{proposition}[The rotation number map]
\label{prop:rotation-number-map}
    Under $\bar{\iota}$, the embedding $\mu: \Irrat \to \T$, $ \theta \mapsto \theta \textnormal{ (mod } 1)$ into the circle $\T=\R/\Z$ extends to a surjective continuous map
\[
    \bar{\mu} : \TheCpt \to \T,
\]
    called the rotation number map.
    For every rational number $\theta$ in $\mathbb{T}$, $\bar{\mu}^{-1}(\theta)$ is homeomorphic to $\TheCpt$ itself.
\end{proposition}

Dynamically, the set $\bar{\mu}^{-1}(\theta)$ in Proposition~\ref{prop:rotation-number-map} represents all possible Lavaurs enrichment combinatorics of a map with a parabolic fixed point.

\begin{definition}
\label{def:time-group}
For every $\ttheta = \seq{ (\varepsilon_n, \abar_n) }_{n \geq 1} \in \TheCpt$, denote
\[
    b_n = b_n(\ttheta) = 
    \begin{cases}
        \abar_n + \cfrac{1+ \varepsilon_n \varepsilon_{n+1}}{2} & \text{ if } \abar_n < \infty, \\
        \infty & \text{ if } \abar_n = \infty.
    \end{cases}
\]
Define an additive abelian group $\Kont_{\ttheta}$, called the \emph{continuant group of $\ttheta$}, to be the group generated by the infinite sequence 
\[
\{\qq_{[n]} = \qq_{[n]}(\ttheta)\}_{n \geq 0}
\]
subject to the relations
\[
    \qq_{[n]} = b_n \qq_{[n-1]} - \varepsilon_{n-1} \varepsilon_n  \qq_{[n-2]} \qquad \text{ for all } n \geq 1 \text{ with } \abar_n < \infty.
\]
In the above, we set $\qq_{[-1]} = 0$ is the identity element.
Each $\qq_{[n]}$ will be referred to as the \emph{$n$\textsuperscript{th} return time} of $\ttheta$.
\end{definition}

In the special case where $\ttheta$ is irrational, $\qq_{[0]} \mapsto 1$ induces an isomorphism between $\Kont_{\ttheta}$ and $\Z$ that sends each $\qq_{[n]}$ to the number $q_{[n]}$ coming from Proposition \ref{prop:q[n]}.

\begin{proposition}[Time order]
\label{prop:chronological-order-01}
    For every $\ttheta = \seq{ (\varepsilon_n, \abar_n) }_{n \geq 1} \in \TheCpt$, the continuant group $\Kont_{\ttheta}$ admits a unique translation-invariant total order $<$, called the time order of $\Time_{\ttheta}$, with the property that for all $n \geq 0$,
    \begin{enumerate}
        \item $0 < \qq_{[n]} < \qq_{[n+1]}$,
        \item if $\abar_{n+1} = \infty$, then $k \qq_{[n]} < \qq_{[n+1]}$ for all $k \geq 1$.
    \end{enumerate}
\end{proposition}

\begin{definition}
    We define the \emph{time semigroup} of $\ttheta$ to be the commutative semigroup equal to the positive cone of $<$, i.e.
    \[
    \Time_{\ttheta} := \{P \in \Kont_{\ttheta} \: : \: P > 0 \},
    \]
\end{definition}

Unlike the continuant group, the set of generators of the time semigroup $\Time_{\ttheta}$ is more complicated, e.g.
    \[
        \{\qq_{[0]}\} \cup \{ \qq_{[n+1]} - \qq_{[n]} \}_{n \geq 0} \cup \bigcup_{\abar_{n+1} = \infty} \{ \qq_{[n+1]} - k \qq_{[n]} \}_{k \geq 2}.
    \]
In \S\ref{ss:renormalization-towers}, this time semigroup will be used to parametrize the semigroup of pre-renormalizations of neutral quadratic polynomials.

\subsection{The natural extension}
\label{ss:natural-extension}

Consider the bi-infinite sequence space
\[
    \TheBiCpt := \overline{\Sigma}^{\Z} = 
    \left\{ \seq{ (\varepsilon_n, \abar_n) }_{n\in \Z} \: : \: 
    \varepsilon_{n} \in \{-1,+1\},
    \abar_n \in \overline{\N}_{\geq 2} \text{ for all } n \right\},
\]
equipped with the infinite product topology. 
The projection map
\[
    \textnormal{proj}: \TheBiCpt \to \TheCpt, \quad 
    \seq{ \ldots, (\varepsilon_{-1}, \abar_{-1}), (\varepsilon_0, \abar_0); (\varepsilon_1, \abar_1), \ldots } \mapsto \seq{ (\varepsilon_n, \abar_n) }_{n\geq 1}
\]
is continuous and the shift map $\shift: \TheBiCpt \to \TheBiCpt$ given by
\[
\shift( \seq{ \ldots, (\varepsilon_{-1}, \abar_{-1}), (\varepsilon_0, \abar_0); (\varepsilon_1, \abar_1), \ldots } ) = \seq{ \ldots, (\varepsilon_{0}, \abar_{0}), (\varepsilon_1, \abar_1); (\varepsilon_2, \abar_2), \ldots }
\]
is a homeomorphism.
In other words, $\shift: \TheBiCpt \to \TheBiCpt$ is the natural extension of the shift map $\shift: \TheCpt \to \TheCpt$.

\begin{definition}
    An element $\tttheta = \seq{(\varepsilon_n, \abar_n)}_{n \in \Z}$ of $\TheBiCpt$ is called \emph{irrational} if $\abar_n < \infty$ for all $n$, \emph{enriched rational} if otherwise.
\end{definition}

\subsubsection{The continuant group}
\label{sss:continuant-group}

\begin{definition}
\label{def:time-group-bi-inf}
    For every $\tttheta = \seq{ (\varepsilon_n, \abar_n) }_{n \in \Z} \in \TheBiCpt$, 
    denote
\[
    b_n = b_n(\ttheta) = 
    \begin{cases}
        \abar_n + \cfrac{1+ \varepsilon_n \varepsilon_{n+1}}{2} & \text{ if } \abar_n < \infty, \\
        \infty & \text{ if } \abar_n = \infty.
    \end{cases}
\]
    Again, we define the \emph{continuant group} $\Kbold_{\tttheta}$ of $\tttheta$ to be the additive abelian group generated by the bi-infinite sequence of symbols $\{Q_{[n]}= Q_{[n]}(\tttheta)\}_{n\in\Z}$ subject to the relations
\[
    Q_{[n]} = b_n Q_{[n-1]} - \varepsilon_{n-1} \varepsilon_n  Q_{[n-2]} 
    \qquad \text{ for all }
    n \in \Z \text{ with } \abar_n < \infty.
\]
    We will call $Q_{[n]}$ the \emph{$n$\textsuperscript{th} return time} of $\tttheta$.
\end{definition}

In what follows, we will fix $\tttheta = \seq{ (\varepsilon_n, \abar_n) }_{n \in \Z} \in \TheBiCpt$.

\subsubsection{Conversion}
\label{sss:conversion}

The group $\Kbold_{\tttheta}$ admits another generating set 
\[
\{Q_n=Q_n(\tttheta)\}_{n\in\Z}
\]
and the relations amongst the $Q_n$'s are governed by a bi-infinite sequence 
\[
\{a_n = a_n(\tttheta)\}_{n\in\Z} \in \overline{\N}^{\Z}.
\]
Each $Q_n$ will be called the $n$\textsuperscript{th} \emph{slow return time} of $\tttheta$.
The conversion 
\[
(\varepsilon_n, \abar_n,Q_{[n]}) \leadsto (a_n,Q_n)
\]
is as follows.
\begin{itemize}
    \item Consider the strictly increasing sequence of integers $\{t_n\}_{n \in \Z}$ defined inductively as follows.
    First, we set $t_0 = 0$.
    Then, for $n \geq 1$,
    we set
\[
    t_n = \begin{cases}
        t_{n-1}+1 & \text{ if } \varepsilon_{n} \varepsilon_{n+1} = -,\\
        t_{n-1}+2 & \text{ if } \varepsilon_{n} \varepsilon_{n+1} = +,
    \end{cases}
\]
and 
\[
    t_{-n} = \begin{cases}
        t_{-n+1}-1 & \text{ if } \varepsilon_{-n+1} \varepsilon_{-n+2} = -,\\
        t_{-n+1}-2 & \text{ if } \varepsilon_{-n+1} \varepsilon_{-n+2} = +.
    \end{cases}
\]
    \item Consider the sequence of positive integers $\{a_n\}_{n\in \Z}$ where for all $n \in \Z$,
\begin{itemize}
    \item if $\varepsilon_n \varepsilon_{n+1} = -$, set $a_{t_n + 1} = \abar_{n+1}$;
    \item if $\varepsilon_n \varepsilon_{n+1} = +$, set $a_{t_n} = 1$ and $a_{t_n + 1} = \abar_{n+1}-1$.
\end{itemize}
\end{itemize}
Then, for all $n \in \Z$, we set
\[
    Q_{t_n} := Q_{[n]},
\]
and whenever $\varepsilon_n \varepsilon_{n+1} = +$, we set 
\[
    Q_{t_n-1} := Q_{[n]}-Q_{[n-1]}.
\]
The corresponding set of relations for this new generating set is simpler: for every $m \in \Z$ with $a_m < \infty$, we have
\[
    Q_m = a_m Q_{m-1} + Q_{m-2}.
\]

\subsubsection{The time semigroup}

Similar to Proposition \ref{prop:chronological-order-01}, we have:

\begin{proposition}[Time order]
\label{prop:time-order}
    The group $\Kbold_{\tttheta}$ admits a unique translation-invariant total order $<$, called the time order, such that for all $n \in \Z$,
    \begin{enumerate}
        \item $0 < Q_{[n]} < Q_{[n+1]}$,
        \item if $\abar_{n+1} = \infty$, then $k Q_{[n]} < Q_{[n+1]}$ for all $k \geq 1$.
    \end{enumerate}
\end{proposition}

\begin{definition}
    We define the \emph{time semigroup} of $\tttheta \in \TheBiCpt$ to be the totally ordered commutative semigroup $\Tbold_{\tttheta}$ equal the positive cone of $<$, i.e.
\[
    \Tbold_{\tttheta} := \{P \in \Kbold_{\underline{\ttheta}} \: : \: P > 0 \},
\] 
\end{definition}

Unlike the continuant group, the set of generators of the time semigroup $\Tbold_{\tttheta}$ is more complicated, e.g.
\begin{align*}
    \{ Q_{[n]}\}_{n \in \Z} \cup 
    \{Q_{[n+1]}-Q_{[n]}\}_{n\in\Z}
    \cup 
    \bigcup_{\abar_{n+1} = \infty} \{Q_{[n+1]}- k Q_{[n]}\}_{k \geq 1}.
\end{align*}
In terms of the slow return times, $\Tbold_{\tttheta}$ has a simpler generating set:
    \[
        \{Q_m\}_{m \in \Z} \cup \bigcup_{a_{m+1} = \infty} \{ Q_{m+1} - k Q_m \}_{m \in \Z}.
    \]
Starting from Section \ref{sec:transcendental-dynamics}, this time semigroup will be used to parametrize neutral cascades.
The slow return times will be used more prominently starting from Section \ref{sec:bounds}.

\subsubsection{The space order and topological cascades}

The group $\Kbold_{\tttheta}$ admits another natural order $\lhd$ that is most easily described in terms of the slow generators $\{Q_n\}_{n \in \Z}$.

\begin{proposition}[Space order]
\label{prop:space-order}
    The group $\Kbold_{\tttheta}$ admits a unique translation invariant total order $\lhd$, called the space order, that satisfies the following properties.
    \begin{enumerate}
        \item $\varepsilon_1 Q_n \lhd 0$ if $n$ is even, $0 \lhd \varepsilon_1 Q_n$ if $n$ is odd.
        \item For $n \in \Z$ with $a_{n+1} = \infty$, then $-kQ_n$ is always $\lhd$-between $0$ and $Q_{n-1}$ for all $k \geq 1$.
    \end{enumerate}
\end{proposition}

The next proposition asserts that small time implies large space.

\begin{proposition}[Proper discontinuity]
\label{prop:proper-discontinuity}
    For $P \in \Kbold_{\tttheta}$ and $n \in \Z$, 
    \[
        0 < |P| \leq Q_n 
        \quad \xRightarrow[]{\qquad} \quad
        P \textnormal{ is not } \lhd\textnormal{-between } Q_n \textnormal{ and } -Q_n.
    \]
\end{proposition}

The following theorem roughly says that $(\Kbold_{\tttheta},+,\lhd)$ is one-dimensional and can be optimally represented as an invertible real dynamical system.

\begin{theorem}[Invertible topological cascades]
\label{thm:topological-cascade}
    There exists a dynamical system
    \[
        A_{\tttheta} = \{ A_{\tttheta}^P : \R \to \R \}_{P \in \Kbold_{\tttheta}}
    \]
    parametrized by $\Kbold_{\tttheta}$ with the following properties.
    \begin{enumerate}
        \item Each map $A_{\tttheta}^P$ is a monotonically increasing self-homeomorphism of $\R$.
        \item The map $(\Kbold_{\tttheta}, \lhd) \to (\R,<)$, $P \mapsto A_{\tttheta}^P(0)$ is an order-preserving injective map with a dense image.
    \end{enumerate}
    Moreover,
    \begin{enumerate}[start=3]
        \item $A_{\tttheta}$ is unique up to topological conjugacy.
        \item If $\bar{a}_{n+1} = \infty$ for some $n$, then $A^{Q_{[n]}}$ admits fixed points.
        \item If $\tttheta$ is irrational, then each $A_{\tttheta}^P$ can be chosen to be a translation.
        Such a group of translations $A_{\tttheta}$ is unique up to linear conjugacy. 
    \end{enumerate}
\end{theorem}

The dynamical system $A_{\tttheta}$ is called an \emph{invertible topological cascade} with combinatorics $\tttheta$.
The non-invertible version is $(A_{\tttheta}^P)_{P \in \Tbold_{\tttheta}}$, which we simply refer to as a \emph{topological cascade}.
This will be realized by the action of the external cascade $\Fext$ on $\R$ introduced in \S\ref{ss:trans-external-coordinates}.

\subsubsection{Irrational cascades of translations}\label{sss:irrat.cascade} Let illustrate the dynamical meaning of the time semigroup $\Tbold_{\tttheta}$ in the case of an irrational $\tttheta \in \TheBiCpt$. 

In the irrational case, $\tttheta$ encodes a unique bi-infinite tower of rigid rotations $(\rotate_{\theta_n})_{n\in\Z}$.
As we lift to the universal cover, each $\rotate_{\theta_n}$ lifts to a translation $T^{[n]}(x) = x+ \tilde{\ell}_n$ that can be chosen such that
\[ 
\rotate_{\theta_n} \ = T^{[n]} \mod T^{[n-1]}.
\]
The numbers $(\tilde \ell_n)_{n \in \Z}$ are unique up to a simultaneous rescaling. 
The group of translations generated by $\{T^{[n]}\}_{n\in\Z}$ is isomorphic to $\Kbold_{\tttheta}$ via the identification $Q_{[n]} \mapsto T^{[n]}$;
hence, the elements of such a group can be written as $T^P$, $P \in \Kbold_{\tttheta}$ where $T^{Q_{[n]}} = T^{[n]}$.

The group $(T^P)_{P \in \Kbold_{\tttheta}}$ is the invertible topological cascade of translations mentioned in item (5) of Theorem~\ref{thm:topological-cascade}. 
The corresponding (non-invertible) cascade of translations $(T^P)_{P \in \Tbold_{\tttheta}}$ is the linear model for the dynamics on the Siegel set $\Zbold(\Fbold)$ of neutral cascades $\Fbold$ with combinatorics $\tttheta$; see Proposition \ref{prop:eventually-brjuno}.

With appropriate modifications, a similar combinatorial framework is also available for an enriched rational $\tttheta$ by taking $A_{\tttheta}^{Q_{[n]}}$ mod $A_{\tttheta}^{Q_{[n-1]}}$. 


\section{Sector renormalization and pseudo-Siegel bounds}
\label{sec:pseudo-Siegel}
 
In this section, we summarize the main results of the pseudo-Siegel theory for the Neutral Family
\[
f_\theta(z) = e^{2\pi i \theta} z + z^2, \qquad \theta \in \Irrat.
\]

It begins in \S\ref{ss:pseudo-siegel} with the structure of the Mother Hedgehog and the construction of pseudo-Siegel disks of $f_\theta$ following \cite{DLy22}.
In \S\ref{ss:real-bubble-bounds}, we recall Real Bounds of the external map associated to $f_\theta$ and Pseudo-Bubble Bounds of \cite{DLi26}, which give uniform geometric control on hoglets and pseudo-bubbles; see Figure~\ref{fig:pseudo-bubble-bounds}.
In \S\ref{ss:sector-renormalization}, we discuss the construction of sector renormalizations of $f_\theta$ that are compatible with pseudo-Siegel disks and Mother Hedgehogs following \cite{DLy26a}; see Figure~\ref{fig:sec-renorm}.

Finally, in~\S\ref{ss:renormalization-towers}, we introduce the compact space of sector renormalization towers $\overline{\towsec}$, see~\eqref{eq:dfn.Rsec}, and analyze parabolic limits in $\overline{\towsec}$. 
The combinatorial data provides a natural semi-conjugacy from $\overline{\towsec}$ to the parabolic compactification $\TheCpt$ of rotation numbers; see Proposition~\ref{prop:continuity-of-combinatorics}. 
We also state the area-zero theorem for the boundary of Mother Hedgehogs from~\cite{Lim26b}.

The key notations introduced in this section are listed in~\S\ref{sss:notation-for-towsec}.

\subsection{Pseudo-Siegel disks} 
\label{ss:pseudo-siegel}

Let
\[
    \Irrat_{\text{bdd}} := 
    \left\{ \theta \in \Irrat \: : \: \inf_n |\gauss^n(\theta)| > 0 \right\}
    = 
    \left\{ \theta \in \Irrat \: : \: \sup_n \abar(\gauss^n(\theta)) < \infty \right\}.
\]
Elements of $\Irrat_{\text{bdd}}$ are called \emph{bounded-type} irrationals. 
One special subset is the set $\Irrat_{\text{EGM}}$ of \emph{eventually golden mean} (EGM) irrationals, that is, the set of irrationals $\theta$ such that  $|\gauss^k(\theta)| = \theta_{\text{gm}}$ for some $k \geq 0$.

For any irrational $\theta \in \Irrat$, denote by 
\[
v_{\theta,-1}=-\frac{e^{2\pi i \theta}}{2}
\qquad \text{and} \qquad
v_{\theta,0} = f_{\theta}(v_{\theta,-1}) = -\frac{e^{4\pi i \theta}}{4}
\]
the critical point and the critical value of $f_\theta$ respectively.

\begin{definition}
    For $\theta \in \Irrat$, a subset $H$ of $\C$ is called a \emph{Mother Hedgehog} of $f_\theta$ if it is a full compact connected subset of $\C$ containing $0$ and $v_{\theta,-1}$ such that $f_\theta: H \to H$ is a homeomorphism.
\end{definition}

For $\theta \in \Irrat_{\text{bdd}}$, Douady-Ghys surgery \cite{D87,G84} implies that $f_\theta$ admits a Siegel disk $Z_\theta$ with quasiconformal boundary passing through the critical point. In this case, $f_\theta$ admits a unique Mother Hedgehog which is simply the closure of $Z_\theta$.

\begin{theorem}[\cite{DLy22,DLy26a}]
\label{thm:mother-hedgehog}
    For every $\theta \in \Irrat$, $f_\theta$ admits a unique Mother Hedgehog $H_\theta$. It has the following properties.
    \begin{enumerate}
        \item The boundary of $H_\theta$ is the postcritical set of $f_\theta$. 
        \item The critical value $v_{\theta,0}$ is an accessible point of $H_\theta$ from infinity and $H_\theta \backslash \{v_{\theta,0}\}$ is connected.
        \item $H_\theta$ is star-like, i.e. a bouquet of arcs called internal rays centered at the fixed point $0$. 
        It can be written as $H_\theta = \bigcup_{\phi \in Q_{\theta}} I_{\phi}$ where $Q_\theta$ is a subset of $\R/ \Z$ and $f_{\theta}$ sends each internal ray $I_{\phi}$ onto $I_{\phi+\theta}$.
        \item The internal ray of $H_\theta$ landing at $v_{\theta,0}$ is a uniform quasiarc. 
        \item $H_{\theta}$ depends uniformly continuously on $\theta \in \Irrat$ (as a subspace of $\TheCpt$) in the Hausdorff topology and the normalized Riemann mapping
        \[
            X_{\theta} : \C \backslash \overline{\D} \to \C \backslash H_{\theta}, \quad 
            \quad X(1) = v_{\theta,0}
        \]
        also depends continuously on $\theta$ in the compact-open topology.
    \end{enumerate}
\end{theorem}

\begin{remark}
    The bouquet presentation of $H_\theta$ as described in item (3) above is not unique when $H_\theta$ has interior.
\end{remark}

Observe that the Mother Hedgehog $H_\theta$ admits a unique bi-infinite critical orbit which we will denote from now on by
\[
    \ldots \, \xrightarrow[\quad]{f_\theta} \, v_{\theta,-2} \, \xrightarrow[\quad]{f_\theta} \, v_{\theta,-1} \, \xrightarrow[\quad]{f_\theta} \,
    v_{\theta,0} \, \xrightarrow[\quad]{f_\theta} \, v_{\theta,1} \, \xrightarrow[\quad]{f_\theta} \, v_{\theta,2} \, \xrightarrow[\quad]{f_\theta} \, \ldots.
\]

The proof of Theorem \ref{thm:mother-hedgehog} relies on pseudo-Siegel bounds. 
For $\theta \in \Irrat_{\text{bdd}}$, the quasiconformal dilatation of the Siegel disk $Z_{\theta}$ generally worsens as $\inf_n |\gauss^n(\theta)|$ gets closer to $0$. 
To rectify this issue, one can enlarge the Siegel disk to a \emph{pseudo-Siegel disk}, which is constructed as follows.

A pseudo-Siegel disk of $f_\theta$ is a closed quasidisk containing the postcritical set of $f_\theta$ on which $f_\theta$ is injective.
For $\theta \in \Irrat$, we will outline below the construction of a particular nested sequence of pseudo-Siegel disks
\[
    \ldots \subset \hat{Z}^{[2]}_\theta \subset \hat{Z}^{[1]}_\theta \subset \hat{Z}^{[0]}_\theta \subset \hat{Z}^{[-1]}_\theta \equiv \hat{Z}_\theta
\]
associated to $f_\theta$.

Consider the sequence $\mathfrak{X}^{-1}(\theta) = \seq{ ( \varepsilon_n, \abar_n ) }_{n \geq 1}$ as well as the corresponding sequence of first return times $q_{[n]} = q_{[n]}(\theta)$, $n \geq 0$ described in Section~\ref{sec:combinatorics}. 
Recall:   \[
        b_n := \abar_n + \frac{1+ \varepsilon_n \varepsilon_{n+1}}{2} \qquad
        \text{ and } \qquad
        \check{q}_{[n]} := q_{[n]} - \varepsilon_n \varepsilon_{n+1} q_{[n-1]}.
    \]
Also, denote $f_\theta^{[n]} \equiv f_\theta^{q_{[n]}}$.

Let us fix a sufficiently large constant $\threshold \in \N$.
This will be called the \emph{combinatorial threshold} for the pseudo-Siegel disks and it is independent of $\theta$.
For every integer $n \geq 0$, there will be two cases:
\begin{itemize}
    \item level $n$ is bounded: $\abar_{n+1}< \threshold$,
    \item level $n$ is near-parabolic (NP): $\abar_{n+1} \geq \threshold$.
\end{itemize}
We will also consider the number
\[
\threshold' := \lfloor e^{\sqrt{\log \threshold}}\rfloor,
\]
which is much smaller than $\threshold$. 
The construction of the pseudo-Siegel disks $\hat{Z}_{\theta}^{[m]}$ will technically depend on $\threshold$.

For every $n \geq 0$, set
\begin{equation}
\label{eqn:xn-yn}
    x_n = v_{\theta,-q_{[n]}}, \qquad 
    y_n = v_{\theta,-\check{q}_{[n]}}.
\end{equation}
For $n = 0$, both $x_0$ and $y_0$ are equal to the critical point $v_{\theta,-1}$ of $f_\theta$. 
For $n \geq 1$, the critical value $v_{\theta,0}$ is between $x_n$ and $y_n$ along the Carath\'eodory boundary $\partial^c H_\theta$ of $H_\theta$ and that $x_n$ and $y_n$ are identified under the iterate $f_{\theta}^{[n-1]}$.
When $\abar_{n+1} \geq \threshold$, we are also interested in the pre-critical points 
\[
x'_n = v_{\theta,-q_{[n+1]} + \threshold' q_{[n]}},\qquad
y'_n = v_{\theta,-\check{q}_{[n]} - \threshold' q_{[n]}}.
\]
on $\partial H_\theta$. 
It has the property that 
\[
(f_\theta^{[n]})^{b_{n+1}-2\threshold'-1}(x'_n) = y'_n \quad \text{ and } \quad (f_\theta^{[n]})^{\threshold'}(y'_n) = y_n,
\]
and along $\partial^c H_\theta$, these points are ordered as follows: 
\[
x_n < v_{\theta,0} < x'_n < y'_n < y_n.
\]
The sub-``interval`` $\left[x'_n, y'_n\right]$ of $\partial^c H_\theta$ is almost invariant under $f_\theta^{[n]}$.

For any NP level $n$, we define the \emph{principal level $n$ dam} $d^{[n]}_{\theta}$ to be the unique hyperbolic geodesic of $\RS \backslash H_{\theta}$ with endpoints $x'_n$ and $y'_n$ that does not separate $v_{\theta,0}$ from infinity.
In general, for $j \in \{0, 1,2,\ldots,q_{[n]}\}$, we define the $j$\textsuperscript{th} \emph{level $n$ dam} $d_{\theta,j}^{[n]}$ to be the unique connected component of $f_\theta^{-j}(d^{[n]}_{\theta})$ that has both endpoints on $\partial H_{\theta}$.
Observe that $d^{[n]}_{\theta,q_{[n]}}$ is combinatorially very close to $d^{[n]}_{\theta}$.

For every NP level $n$ and every $j \in \{0,1,\ldots,q_{[n]}-1\}$, the closure in $\C \backslash H_{\theta}$ of the disk enclosed by $d^{[n]}_{\theta,j}$ is called a \emph{parabolic fjord} of level $n$, which we will denote by $V^{[n]}_{\theta,j}$.
For every $m \geq 0$, we then define
    \[
        \hat{Z}_\theta^{[m-1]} := H_\theta \cup \bigcup_{\substack{ n\geq m, \\ \text{level } n \text{ is NP}}} \bigcup_{j=0}^{q_{[n]}-1} V^{[n]}_{\theta,j}.
    \]
From the construction, we have that for all $m \geq 0$,
\[
    \hat{Z}_\theta^{[m-1]} \supseteq \hat{Z}_\theta^{[m]}
\]
and equality holds if and only if level $m$ is bounded.

\begin{figure}
    \centering
\begin{tikzpicture}
    \node[anchor=south west,inner sep=0] (image) at (0,0) {\includegraphics[width=1\linewidth]{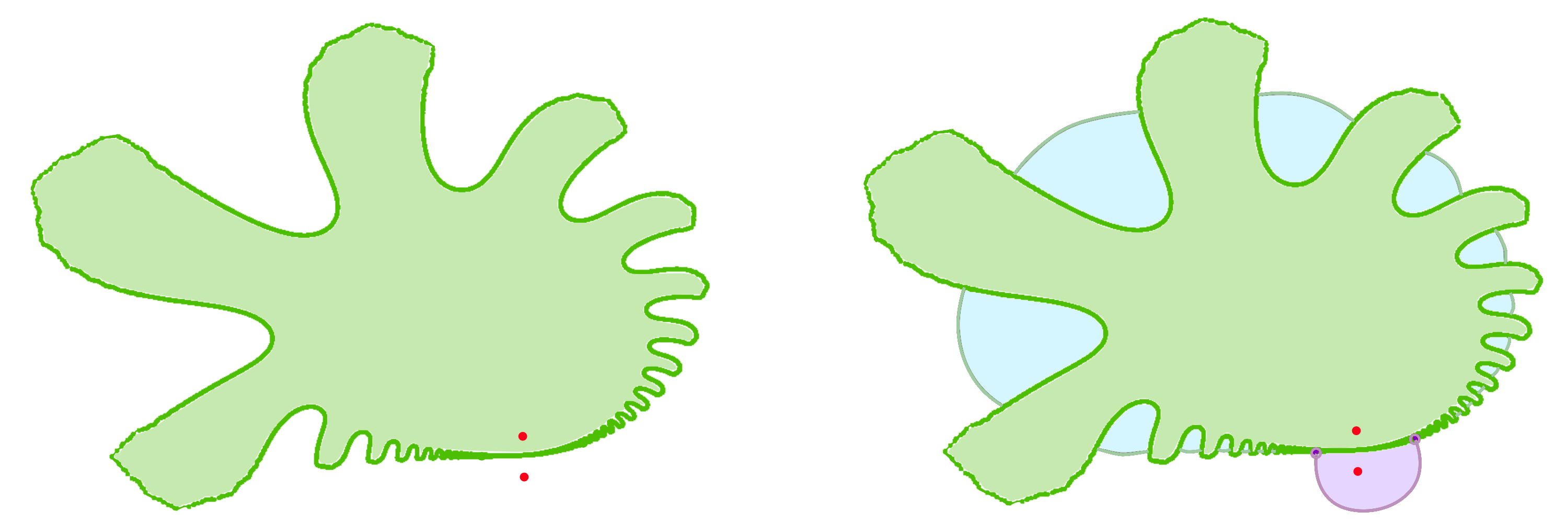}};
    \begin{scope}[
        x={(image.south east)},
        y={(image.north west)}
    ]
    
    \node [green!20!black, font=\bfseries] at (0.3,0.4) {$Z_\theta$};
    \node [green!20!black, font=\bfseries] at (0.83,0.4) {$\hat{Z}_\theta$};
    \end{scope}
\end{tikzpicture}
    \caption{The Siegel disk and a pseudo-Siegel disk (right) of $f_\theta$ when $\mathfrak{X}^{-1}(\theta) = \seq{ (+,100),(-,100),(+,2),(+,2), \ldots }$. The pseudo-Siegel disk $\wZ_\theta$ is obtained from $\overline Z_\theta$ by filling in level $1$ parabolic fjords (in blue) and one level $0$ parabolic fjord (in purple). The alpha and beta fixed points are marked in red. }
    \label{fig:PS-disk}
\end{figure}


\begin{theorem}[Pseudo-Siegel bounds \cite{DLy22, DLy26a}]
\label{thm:pseudo-siegel}
    For sufficiently high $\threshold \in \N$ and for all $\theta \in \Irrat$ and $m \geq -1$, the set $\hat{Z}_{\theta}^{[m]}$, $m \geq -1$ constructed above is a pseudo-Siegel disk of $f_\theta$. 
    They also satisfy the following properties.
    \begin{enumerate}
        \item Almost invariance: There exists a universal constant $C>0$ (independent of $\threshold$ and $\theta$) such that for all $m \geq 0$ and $j \in \{1,\ldots, q_{[m]}(\theta)\}$, the hyperbolic distance between every point on $d^{[m]}_{\theta,j}$ and the unique geodesic with the same endpoints as $d^{[m]}_{\theta,j}$ is bounded above by $C$.
        \item Uniform quasiconformality: There exists a constant $ K(\threshold)>0$ (independent of $\theta$) such that $K(\threshold) \to \infty$ as $\threshold \to \infty$ and $\hat{Z}_{\theta} = \hat{Z}_{\theta}^{[-1]}$ is $K(\threshold)$-quasiconformal.
        \item Uniform continuity: For every $m \geq -1$, the normalized Riemann mapping 
        \[
            X_{m,\theta}: \C \backslash \overline{\D} \to \RS \backslash \hat{Z}_{\theta}^{[m]}, \quad X_{m,\theta}(1) = v_{\theta,0}
        \]
        depends uniformly continuously on $\theta \in \Irrat$ (as a subspace of $\TheCpt$) in the $C^0$ topology. 
    \end{enumerate}
\end{theorem}

The proof of pseudo-Siegel bounds is achieved in the near-degenerate regime. 
In \cite{DLy22}, geodesic pseudo-Siegel disks are constructed (a slight variation of $\hat{Z}_{\theta}$ where all dams are instead hyperbolic geodesics) for EGM irrationals and are shown to be almost invariant and uniformly quasiconformal. 
As a corollary, the existence of the Mother Hedgehog $H_{\theta}$ across all $\theta$ was shown as it is equal to the non-escaping set of $f_{\theta}: \hat{Z}_{\theta}^{[-1]} \to \C$. 
The uniform continuity of pseudo-Siegel disks across all bounded-type rotation numbers was proven in the follow up work \cite{DLy26a} and this results in Theorem \ref{thm:mother-hedgehog}. 
In the next subsection, we will describe in detail the uniform bounds on sector renormalization obtained in \cite{DLy26a}.

\subsection{Real Bounds and Pseudo-Bubble Bounds}
\label{ss:real-bubble-bounds}

Consider the quadratic polynomial $f_\theta$ for some irrational $\theta \in \Irrat$.
In this subsection, we will summarize the key results of \cite{DLi26}.
Consider the pseudo-Siegel disks $\hat{Z}^{[m]}_{\theta}$, $m \geq -1$, the Mother Hedgehog $H_\theta$, and the bi-infinite critical orbit in $H_\theta$ by $\{v_{\theta,n}\}_{n\in\Z}$ introduced in the previous subsection.

Let
\[
    \Psi_{\theta} : \RS \backslash H_{\theta} \to \RS \backslash \overline{\D}
\]
be the unique Riemann mapping fixing $\infty$ and sending the critical value $v_{\theta,0}$ of $f_{\theta}$ to $1$.
This normalization makes sense because by Theorem \ref{thm:renorm-limits} (6), $\Psi_{\theta}$ extends uniquely and continuously on the bi-infinite critical orbit $\{v_{\theta,n}\}_{n \in \Z}$.

Let 
\[
    B_{\theta}^1 := \overline{f_{\theta_k}^{-1}(H_{\theta})\backslash H_{\theta}}
    \quad \text{ and } \quad
    \mathsf{B}_{\theta}^1 := \Psi_\theta(B_{\theta,1})
\]
Then, the external map
\[
    \mathsf{f}_\theta \equiv 
    \mathsf{f}_\theta^{\textnormal{ext}} 
    := \Psi_\theta \circ f_\theta \circ \Psi_\theta^{-1} : 
    \RS \backslash ( \overline{\D} \cup \mathsf{B}_{\theta}^1) \to \RS \backslash \overline{\D}
\]
is a well-defined holomorphic double covering map branched at $\infty$.
Denote 
\[
\mathsf{v}_{\theta,n} = \mathsf{v}_{\theta,n}^{\textnormal{ext}}
:= \Psi_{\theta}(v_{\theta,n}) \in \partial \D
\qquad \text{ for every }
n \in \Z.
\]
For $n \in \N$, consider the tiling $\mathfrak{D}_{\theta}^{[n]}$ of $\partial \D$ obtained by removing $\{ \mathsf{v}_{\theta,-k} \}_{0<k\leq q_{[n]}(\theta)}$. 
Then, $\mathfrak{D}_{\theta}^{[n+1]}$ is a refinement of $\mathfrak{D}_{\theta}^{[n]}$ and every tile in $\mathfrak{D}_{\theta}^{[n]}$ decomposes into at least two tiles in $\mathfrak{D}_{\theta}^{[n+1]}$.

\begin{theorem}[Real Bounds]
\label{thm:R-APB}
    Let $\theta \in \Theta$ and let $I$ be a tile in $\mathfrak{D}_{\theta}^{[n]}$ for some $n \in \N$.
    \begin{enumerate}
        \item If $J \in \mathfrak{D}_{\theta}^{[n]}$ is adjacent to $I$, then $|I| \asymp |J|$.
        \item Let $J_1,J_2,\ldots,J_k$ be all the tiles in $\mathfrak{D}_{\theta}^{[n+1]}$ that are contained in $I$, labelled in consecutive order, then
        \[
            |J_j| \asymp \frac{|I|}{\min\{j, k+1-j\}^2} \qquad \text{ for all } j \in \{1,\ldots,k\}.
        \]
    \end{enumerate}
    In particular, the maximal diameter of tiles in $\mathfrak{D}_{\theta}^{[n]}$ tends to zero uniformly exponentially fast as $n \to \infty$.
\end{theorem}

\begin{proof}
    In the bounded-type case, items (1) and (2) for $j \in \{1,k\}$ were proven in \cite[\S3]{DLi26} as a consequence of pseudo-Siegel bounds.
    Item (2) for all remaining $j$ follows from a general argument by Yoccoz for real maps with negative Schwarzian derivative; the proof can be found in \cite[Appendix B]{dFdM99}.
    Alternatively, (2) can also be proven directly using the Log Rule \cite{DLy22}.
    By continuity of Mother Hedgehogs (Theorem \ref{thm:mother-hedgehog} (5)), both (1) and (2) generally hold for any rotation number.
\end{proof}

As a consequence of Real Bounds, we have:

\begin{corollary}
    The external map $\mathsf{f}_\theta$ extends continuously to a unique circle homeomorphism $\mathsf{f}_{\theta}: \partial \D \to \partial \D$ with rotation number $\theta$. 
\end{corollary}

\begin{lemma}
\label{lem:qc-fjord-external}
    There exists a universal constant $K>1$ such that for every $\theta \in \Theta$ and $m \geq -1$, the image under $\Psi_{\theta}$ of the boundary of $\hat{Z}_{\theta}^{[m]}$ is a $K$-quasicircle.
\end{lemma}

\begin{proof}
    For bounded combinatorics, this was proven in \cite[Lemma 3.2]{DLi26} with uniform $K$. 
    The proof extends to all irrationals by the uniform continuity of Mother Hedgehogs (Theorem \ref{thm:mother-hedgehog} (5)).
    Indeed, by almost invariance, every dam is a uniform quasiarc close to a semicircle in external coordinates. 
    By Real Bounds, fjords are uniformly separated in external coordinates.
    These two ingredients imply the bounded turning property for the boundary of $\Psi_{\theta}(\partial \hat{Z}_{\theta}^{[m]})$.
\end{proof}

\begin{corollary}
\label{cor:qc-external-coords}
    There exists an $\threshold$-uniform constant $K>1$ such that for all $\theta \in \Theta$, 
    the Riemann mapping $\Psi_{\theta}: \RS \backslash \hat{Z}_{\theta}^{[-1]} \to \Psi_{\theta}(\RS \backslash \hat{Z}_{\theta}^{[-1]})$ extends to a global $K$-quasiconformal map of $\RS$. 
\end{corollary}

\begin{proof}
    This follows from the $\threshold$-uniform quasiconformality of the top pseudo-Siegel disk $\hat{Z}_{\theta}$ (Theorem \ref{thm:renorm-limits} (4)) and Lemma \ref{lem:qc-fjord-external}.
\end{proof}

Together, Corollary \ref{cor:qc-external-coords} and Real Bounds imply:

\begin{corollary}[Bounds on scaling ratio]
\label{cor:universal-scaling-law}
    There exist universal constants $C, \lambda_-, \lambda_+$ with $C>1$ and $1< \lambda_- \leq \lambda_+$ such that for all $n \in \N$,
    \begin{align*}
        C^{-1} \big| v_{\theta,q_{[n]}(\theta)} - v_{\theta,0} \big| 
        & \leq \big| v_{\theta,-q_{[n]}(\theta)} - v_{\theta,0} \big| \leq 
        C \big| v_{\theta,q_{[n]}(\theta)} - v_{\theta,0} \big|, \\
        C^{-1} \lambda_-^n 
        &\leq \big| v_{\theta,-q_{[n]}(\theta)} - v_{\theta,0} \big| \leq C \lambda_+^n.
    \end{align*}
\end{corollary}

\begin{definition}
    We call the set $B_{\theta}^1$ the \emph{hoglet of generation $1$} of $f_{\theta}$.
    It is contained in a \emph{pseudo-bubble} $\hat{B}_{\theta}^1$ of generation $1$, which is the closure of the connected component of $f_{\theta}^{-1}\big( \textnormal{int}\hat{Z}_{\theta}\big)$ that avoids the Mother Hedgehog $H_{\theta}$. 
        
    In general, a \emph{hoglet} $B$ of $f_{\theta}$ of generation $g \geq 2$ is the closure of a connected component of $f_{\theta}^{-(g-1)}(B_{\theta}^1)$, and it is contained in a \emph{pseudo-bubble} $\hat{B}$, which is the unique connected component of $f_{\theta}^{-(g-1)}(\hat{B}_{\theta}^1)$ containing $B$. 
    The \emph{root} of a hoglet (resp. pseudo-bubble) of generation $g\geq 1$ is the unique critical point of $f_{\theta}^{g}$ contained in it.
        
    We say that a hoglet (resp. pseudo-bubble) is \emph{primary} if it is rooted at $v_{\theta,-k} \in H_\theta$ where $k$ is the generation, \emph{secondary} if it is rooted at a critical point on a primary hoglet away from $H_\theta$.

    A \emph{hoglet} (resp. \emph{pseudo-bubble}) in the \emph{external coordinates} is the image of a hoglet (resp. pseudo-bubble) of $f_\theta$ under $\Psi_\theta$.
\end{definition}

\begin{remark}[Bubbles vs. Hoglets]
\label{rem:name.hoglet} 
In this paper, we use the term ``hoglet'' to emphasize that the associated objects need not be locally connected. The setting of this paper is conceptually different from those of~\cite{DLy22} and~\cite{DLi26}. In those papers, uniform geometric bounds are established by approximating all neutral quadratic polynomials by polynomials of eventually-golden-mean type; accordingly, the term ``bubbles'' is used there. These bubbles can exhibit geometry close to the non-locally connected regime, but they still remain locally connected. By contrast, the present paper concerns the geometric analysis of the limiting dynamical planes, where hoglets can actually be non-locally connected. We keep the name pseudo-bubble as they are locally connected.
\end{remark}

In \cite{DLi26}, we apply Real Bounds and tools from near-degenerate regime to obtain uniform bounds for pseudo-bubbles.
Much of the analysis was done in external coordinates.

\begin{theorem}[Pseudo-Bubble Bounds {\cite{DLi26}}]
\label{thm:bubble-bounds}
    There exist a universal constant $K> 1$ and an $\threshold$-uniform constant $K' > 1$ such that the following holds for any $\theta \in \Irrat$ and any hoglet $B$ of $f_{\theta}$.
    Denote $v_k = v_{\theta,k}$ and $\mathsf{v}_k = \Psi_{\theta}(v_k)$ for all $k \in \Z$, and $q_n = q_n(\theta)$ for all $n \geq 0$.
    Let $\hat{B}$ be the corresponding pseudo-bubble and let $\mathsf{B} = \Psi_\theta(B)$ and $\hat{\mathsf{B}} = \Psi_\theta(\hat{B})$ be the corresponding objects in external coordinates.
\begin{enumerate}
        \item Both $\hat{B}$ and $\hat{\mathsf{B}}$ are $K'$-quasidisks.
        \item Suppose $B$ is a primary hoglet of some generation $k \geq 1$.
        \begin{itemize}
            \item Size control: Let $n \in \N$ be the unique integer with $q_n < k \leq  q_{n+1}$. Then,
            \[
                K^{-1} |v_{-k} - v_{q_n-k}| \leq \textnormal{diam}(B) \leq \textnormal{diam}(\hat{B}) \leq K |v_{-k} - v_{q_n-k}|,
            \]
            and
            \[
                K^{-1} |\mathsf{v}_{-k} - \mathsf{v}_{q_n-k}| \leq \textnormal{diam}(\mathsf{B}) \leq \textnormal{diam}(\hat{\mathsf{B}}) \leq K |\mathsf{v}_{-k} - \mathsf{v}_{q_n-k}|.
            \]
            \item Angular control: With respect to the hyperbolic metric of $\C \backslash \overline{\D}$, every point in $\hat{\mathsf{B}} \backslash \{\mathsf{v}_{-k}\}$ is of distance at most $K$ from the radial ray emanating from $\mathsf{v}_{-k}$ to infinity.
        \end{itemize}
        \item If $B$ is not a primary hoglet, then with respect to the hyperbolic metric of $\C \backslash H_\theta$, the diameter of $\hat{B}$ is at most $K$.
    \end{enumerate}
    \end{theorem}

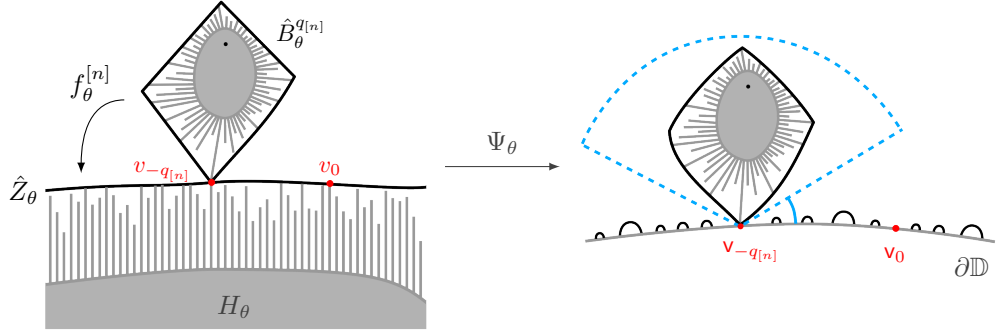
\begin{figure}
    \centering
    
    \begin{tikzpicture}
    \node[anchor=south west, inner sep=0] (image) at (0,0) {\includegraphics[width=1\linewidth]{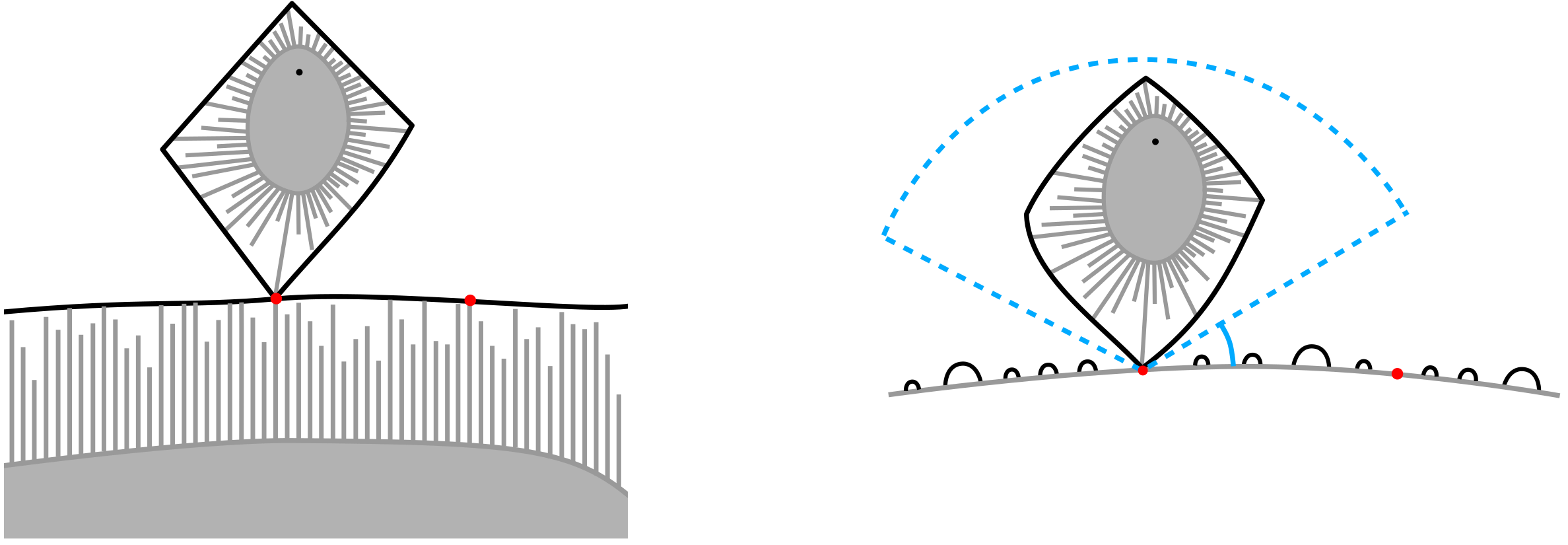}};
    \begin{scope}[
        x={(image.south east)},
        y={(image.north west)}
    ]
        \draw [-latex] (0.085,0.7) .. controls (0.05,0.7) and (0.04,0.6) .. (0.04,0.48);
        \node [black] at (0.05,0.75) {$f_\theta^{[n]}$};
        \draw [gray!50!black,-latex] (0.42,0.5) -- (0.54,0.5);
        \node [gray!50!black] at (0.2,0.08) {$H_\theta$};
        \node [red] at (0.3,0.495) {\scalebox{0.9}{$v_{0}$}};
        \node [red] at (0.125,0.485) {\scalebox{0.9}{$v_{-q_{[n]}}$}};
        \node [black] at (-0.02,0.43) {$\hat{Z}_{\theta}$};
        \node [black] at (0.27,0.9) {\scalebox{0.9}{$\hat{B}_{\theta}^{q_{[n]}}$}};
        \node [gray!50!black] at (0.97,0.2) {$\partial \D$};
        \node [gray!50!black] at (0.48,0.57) {$\Psi_\theta$};
        
        \node [red] at (0.89,0.235) {\scalebox{0.9}{$\mathsf{v}_{0}$}};
        \node [red] at (0.74,0.245) {\scalebox{0.9}{$\mathsf{v}_{-q_{[n]}}$}};
    \end{scope}
\end{tikzpicture}
    
    \caption{Bounds for principal pseudo-bubbles in the dynamical plane of $f_\theta$ (left) and in its external coordinate; see Theorem~\ref{thm:bubble-bounds}.}
    \label{fig:pseudo-bubble-bounds}
\end{figure}

See Figure \ref{fig:pseudo-bubble-bounds}.
Pseudo-Bubble Bounds will be vital in understanding the dynamics of geometric limits in later sections.
Here is the intuition.
When we zoom in about the critical value at an arbitrarily small scale, any visible primary pseudo-bubble has uniformly controlled regularity and diameter bounded above by the distance between its root and the critical value. (See Figure \ref{fig:bubbles-and-alphas}.)
Hence, the geometric limit of these pseudo-bubbles has uniformly controlled size, in particular bounded, and has uniform regularity.
This will be expanded in detail in \S\ref{ss:bubbles-and-pseudo-bubbles}.

\subsection{Sector renormalization}
\label{ss:sector-renormalization}

\begin{definition}
\label{def:sector}
    We define a (\emph{topological}) \emph{sector} $S$ to be an open Jordan domain in $\RS$ such that its boundary $\partial S$ is a concatenation of three arcs $\alpha, \alpha', \omega$, where $\alpha$ and $\alpha'$ are called the \emph{sides} of $S$ and $\omega$ is called the \emph{summit}\footnote{This is by analogy with the notion of the summit in a Saccheri quadrilateral.} of $S$. The \emph{vertex} of $S$ is the common endpoint of $\alpha$ and $\alpha'$.
    A \emph{conformal gluing map} for a sector $S$ is a univalent map $\phi: S \to \D$ with the following properties.
    \begin{itemize}
        \item There is a ray $\gamma$ in $\D$, called the \emph{slit} of $\phi$, that emanates from $0$ to a point on the boundary of $\D$ such that the image of $\phi$ is $\D \backslash \gamma$.
        \item The map $\phi$ extends continuously to a map from $\partial S$ onto $\partial \D \cup \gamma$. It sends the vertex of $S$ to $0$, maps each of the two sides of $S$ homeomorphically onto $\overline{\gamma}$, and maps the summit of $S$ onto $\partial \D$.
    \end{itemize}
\end{definition} 

Now, we are ready to describe sector renormalizations $(\Rsec)^n f_\theta$ of $f_\theta$, $\theta \in \Irrat$.
The theorem below states that they can be constructed to still admit uniform pseudo-Siegel disks and Mother Hedgehogs.

\begin{theorem}[{Sectorial bounds \cite{DLy26a}}]
\label{thm:sectorial-bounds}
    Pick any $\theta \in \Irrat$ and denote $f=f_{\theta}$, $H=H_{\theta}$, $\hat{Z}^{[m]} = \hat{Z}^{[m]}_\theta$ for all $m \geq -1$, and $\theta_n = \gauss^n(\theta)$ for all $n \geq 0$.
    There exists a nested sequence of sectors 
    \[
    S^1 \supset S^2 \supset S^3 \supset \ldots
    \]
    in the dynamical plane of $f$ such that the following properties hold for every $n \geq 1$.
\begin{enumerate}
    \item For any tip $z$ of $H$, denote by $I_z$ the internal ray in $H$ with endpoints $0$ and $z$.
    Consider the pre-critical points $x_n, y_n$ coming from (\ref{eqn:xn-yn})
    and let $z_n \in H$ be such that $\{z_n, f^{[n-1]}(z_n)\} = \{x_n,y_n\}$.
    The first side of $S^n$ is the concatenation of $I_{z_n}$ with an arc $\Gamma_n$ starting at $z_n$ which has diameter 
    \[
        \diam(\Gamma_n) \asymp |x_n - y_n|.
    \] 
    The second side of $S^n$ is the image of the first side under $f^{[n-1]}$, and $f^{[n]}(f^{[n-1]}(\Gamma_n))$ is a subarc of the arc $I_{f^{[n]}(f^{[n-1]}(z_n))}$.
    The summit of $S^n$ is a hyperbolic geodesic of $C \backslash H$ connecting the endpoints of $\Gamma_n$ and $f^{[n-1]}(\Gamma_n)$.
    \item Gluing $I_{z_n} \cup \Gamma_n$ with its image under $f^{[n-1]}$ gives us a conformal gluing map $\phi_n: (S^n,0) \to (\D,0)$ of $S^n$. 
    Under $\phi_n$, the action of $f^{[n]}$ on $S^n$ (modulo $f^{[n-1]}$) projects to a holomorphic map 
    \[
        (\Rsec)^n f := f_n: \Dom(f_n) \to \D,
    \]
    on some open domain $\Dom(f_n) \subset \D$ which has a neutral fixed point at $0$ with rotation number $\theta_n$.
    \item The map $\phi_n$ projects $H \cap \overline{S^n}$ to a set $H_n$, called the Mother Hedgehog of $f_n$, which satisfies \textnormal{(1)--(4)} in Theorem \ref{thm:mother-hedgehog}. 
    The intersection of $H_n$ and $\phi_n(\partial S^n)$ is the internal ray of $H_n$ ending at the critical point of $f_n$.
    \item For $m \geq -1$, the map $\phi_n$ projects $\hat{Z}^{[n+m]} \cap \overline{S^n}$ to a pseudo-Siegel disk $\hat{Z}^{[m]}_n$ of $f_n$ which satisfies the following properties:
    \begin{enumerate}[label=\textnormal{(\alph*)}]
        \item $\hat{Z}^{[m]}_n$ is almost invariant under the iterate $f^{[m+1]}_n := f_n^{q_{[m+1]}(\theta_n)}$;
        \item $\hat{Z}^{[-1]}_n$ and $\hat{Z}^{[0]}_n$ are $K$-quasidisks for some $\threshold$-uniform constant $K\geq 1$;
        \item there exist an $\threshold$-uniform constant $R$ and a universal constant $R'$ (independent of $\theta$, $n$, $m$) such that $0<R<R'<1$ and
            \[
            \D(0,R) \subset \hat{Z}^{[-1]}_n \subset \D(0,R').
            \]
    \end{enumerate}
    \item The arc $I_{z_n} \cup \Gamma_n$ can be chosen to depend uniformly continuously on $\theta \in \Irrat$ (as a subspace of $\TheCpt$). 
    Consequently, the gluing map $\phi_n$, the resulting map $f_n$, the pseudo-Siegel disks $\hat{Z}^{[m]}_n$, and the Mother Hedgehog $H_n$ of $f_n$ also depend uniformly continuously on $\theta$.
\end{enumerate}
\end{theorem}

From the theorem above, each $f_n$ also admits a nest of sectors 
\[
S_n^1 \supset S_n^2 \supset S_n^3 \supset \ldots
\qquad \text{ where }
S_n^m = \phi_n(S^{n+m}).
\]
The conformal gluing map $\phi_{n+1}\circ \phi_n^{-1} :S_n^1 \to \D$ projects the first return map of $f_n$ back to $S_n^1$ to the map $f_{n+1}$. See Figure \ref{fig:sec-renorm}.

\begin{lemma}
\label{lem:size-of-sectors}
    For every $\theta \in \Theta$ and $m \geq 1$, the conformal radius of the $m$\textsuperscript{th} renormalization sector $S^m$ of $f_\theta$ about its critical value $v_{\theta,0}$ satisfies
    \[
        \textnormal{rad}(S^m , v_{\theta,0}) \asymp |v_{\theta,q_{[m]}(\theta)}-v_{\theta,0}|.
    \]
\end{lemma}

\begin{proof}
    This follows immediately from Corollary \ref{cor:universal-scaling-law} as well as the construction of $S^m$ as described in Theorem \ref{thm:sectorial-bounds}.
%
\end{proof}

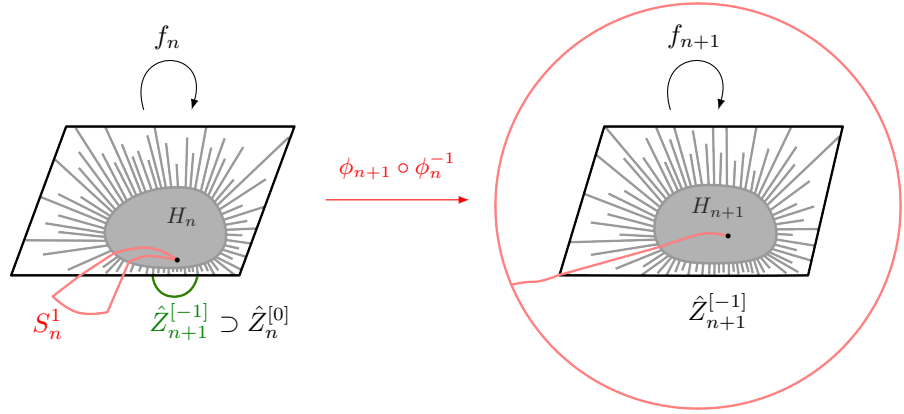
\begin{figure}
    \centering
    
    \begin{tikzpicture}
    \node[anchor=south west, inner sep=0] (image) at (0,0) {\includegraphics[width=1\linewidth]{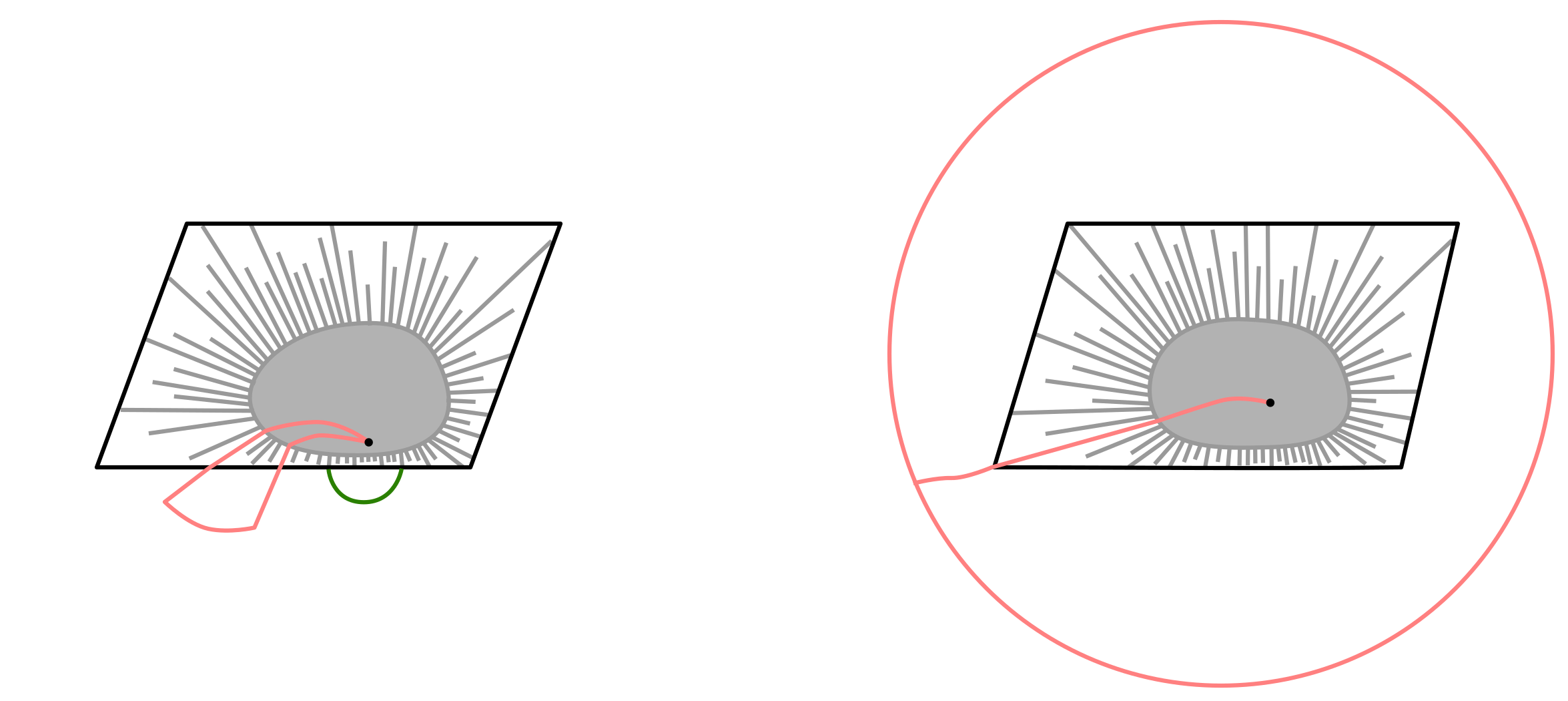}};
    \begin{scope}[
        x={(image.south east)},
        y={(image.north west)}
    ]
        \node [gray!30!black] at (0.24,0.47) {\scalebox{0.85}{$H_n$}};
        \node at (0.28,0.225) {\textcolor{green!50!black}{$\hat{Z}^{[-1]}_{n+1}$} $\supset\hat{Z}^{[0]}_n$ };
        \node [gray!30!black] at (0.8,0.49) {\scalebox{0.85}{$H_{n+1}$}};
        \node [black] at (0.8,0.25) {$\hat{Z}^{[-1]}_{n+1}$};
        
        \node [red] at (0.1,0.22) {$S_n^1$};
        \node [black] at (0.225,0.89) {$f_n$};
        \draw[-latex] (0.2,0.72) .. controls (0.18,0.87) and (0.27,0.87) .. (0.25,0.72);
        \node [black] at (0.775,0.89) {$f_{n+1}$};
        \draw[-latex] (0.75,0.72) .. controls (0.73,0.87) and (0.82,0.87) .. (0.80,0.72);
        \draw[red,-latex] (0.39,0.51) -- (0.54,0.51);
        \node [red] at (0.465,0.59) {\small $\phi_{n+1} \circ \phi_n^{-1}$};
    \end{scope}
\end{tikzpicture}
    
    \caption{Sector renormalization preserves pseudo-Siegel disks and Mother Hedgehogs; see Theorem~\ref{thm:sectorial-bounds}.}
    \label{fig:sec-renorm}
\end{figure}

\subsection{Renormalization towers} 
\label{ss:renormalization-towers}

Let us denote the collection of sector renormalizations of neutral quadratic polynomials by
\begin{align*}
    \orbsec := \{ \Rsec^n f_\theta \: : \: \theta \in \Irrat, n \geq 0 \},
\end{align*}
equipped with the topology of uniform convergence on compact subsets. 
It is forward invariant under $\Rsec$.
Let us also consider the infinite sequence space
\[
    \towsec = \left\{ \seq{\Rsec^n f }_{n\geq 0} \: : \: f \in \orbsec \right\},
\]
equipped with the topology where for every $k \geq 0$, the projection map $\towsec  \to \orbsec, \seq{\Rsec^n f}_{n\geq 0} \mapsto \Rsec^k f$ is continuous.

It was demonstrated in \cite{DLy26a} that renormalization sectors have uniformly bounded geometry.
This leads to an extension of Theorem \ref{thm:sectorial-bounds} below.

\begin{theorem}[\cite{DLy26a}]
\label{thm:renorm-limits}
    The spaces $\orbsec$ and $\towsec$ are pre-compact in the following sense. 
    For any infinite sequence of maps $g_1$, $g_2$, $ g_3$, $\ldots$ in $\orbsec$, there is a subsequence $N_k \to \infty$ such that for every $n \geq 0$, the following properties hold.
    \begin{enumerate}[label = \textnormal{(\arabic*)}]
        \item In $\TheCpt$, the rotation number of $g_{N_k}$ converges to an element $\ttheta = \seq{ (\varepsilon_m, \abar_m) }_{m \geq 1}$ of $\TheCpt$ as $k \to \infty$.
        \item The sequence of maps 
        \[
        f_{n,k} := \Rsec^n g_{N_k}: \Dom(f_{n,k}) \to \D
        \]
        converges to a holomorphic map $f_n : \Dom(f_n) \to \D$ uniformly on compact subsets of some open domain $\Dom(f_n) \subset \D$.
        \item The map $f_n$ admits a simple critical point $v_{-1}(g_n)$ and a neutral fixed point at $0$ with rotation number $\bar{\mu}(\ttheta_n)$ where $\ttheta_n = \shift^n(\ttheta)$. 
        If $f_n'(0) = 1$ or equivalently $\abar_{n+1} = \infty$, then $0$ is a simple parabolic fixed point of $f_n$.
        \item The domain $\Dom(f_n)$ contains a nested sequence of full connected compact sets 
        \[
            \ldots \subset \hat{Z}^{[1]}_n \subset \hat{Z}^{[0]}_n \subset \hat{Z}^{[-1]}_n = \hat{Z}_n
        \]
        called the \emph{pseudo-Siegel pinched disk} of $f_n$. 
        These have the following properties.
        \begin{enumerate}[label = \textnormal{(\alph*)}]
            \item Each $\hat{Z}^{[m]}_n$ is the Hausdorff limit of the level $m$ pseudo-Siegel disk of $f_{n,k}$ as $k \to \infty$.
            \item Both $\hat{Z}^{[-1]}_n$ and $\hat{Z}^{[0]}_n$ are closed $K$-quasidisks for some $\threshold$-uniform constant $K \geq 1$.
            \item There exist an $\threshold$-uniform constant $R$ and a universal constant $R'$ (independent of $\theta$, $n$, $m$) such that $0<R<R'<1$ and
            \[
            \D(0,R) \subset \hat{Z}^{[-1]}_n \subset \D(0,R').
            \]
        \end{enumerate} 

        \item Denote $f_n^{[0]} = f_n$. 
        Consider the time semigroup $\Time^n= \Time_{\ttheta_n}$ of $\ttheta_n$ and the corresponding return times $\qq_{n,[m]} = \qq_{[m]}(\ttheta_n) \in \Time^n$, $m \geq 0$. 
        The semigroup generated by the iterates of $f_{n,k}$ converges to a semigroup 
        \[
            \mathcal{F}_n = \{f_n^p \}_{p \in \Time^n}
        \]
        of commuting holomorphic maps parametrized by $\Time^n$ in the following sense.
        \begin{enumerate}[label = \textnormal{(\alph*)}]
            \item For every $m \geq 0$, the pre-renormalizations $f_{n,k}^{[m]}: \Dom(f_{n,k}^{[m]}) \to \D$ converge to a holomorphic map $f_n^{[m]} = f_n^{\qq_{n,[m]}}: \Dom(f_n^{[m]}) \to \D$ uniformly on compact subsets as $k \to \infty$. 
            In $\D \backslash \{0\}$, $\Dom(f_n^{[m]}) \backslash \{0\}$ is an open neighborhood of $\hat{Z}^{[m-1]}_n \backslash \{0\}$.
            Moreover, the map $f_n^{[m]}$ is injective on $\hat{Z}^{[m-1]}_n$, and $\hat{Z}^{[m-1]}_n$ is almost invariant under $f_n^{[m]}$. 
            \item If $\abar_{n+m+1} < \infty$, then $f_n^{[m+1]}$ is an iterate of $f_n^{[m]}$. 
            Else, $f_n^{[m+1]}$ is a ``Lavaurs map`` of $f_n^{[m]}$.
        \end{enumerate}
        
        \item The compact set
        \[
            H_n := \bigcap_{ m \geq -1} \hat{Z}^{[m]}_n
        \]
        satisfies the following properties.
        \begin{enumerate}[label = \textnormal{(\alph*)}]
            \item $H_n$ contains $0$ and the critical point of $f_n$.
            \item $H_n$ is the Hausdorff limit of the Mother Hedgehog of $f_{n,k}$, and the corresponding Riemann mapping $\C \backslash \overline{\D} \to \C \backslash H(f_{n,k})$ sending $1$ to the critical value $v_0(f_{n,k})$ converges uniformly on compact subsets to the Riemann mapping of the complement of $H_n$.
            \item Every map $\mathcal{F}_n$ restricts to a self-homeomorphism of $H_n$.
            \item $H_n$ satisfies the properties stated in Theorem \ref{thm:mother-hedgehog} (2)--(4) for $f_n$.
            \item The boundary of $H_n$ is equal to the \emph{postcritical set} $P_n$, that is, the closure of the orbit of the critical point of $f_n$ under the semigroup $\mathcal{F}_n$.
        \end{enumerate}

        \item 
        As $k \to \infty$, the nest of renormalization sectors $S^m(f_{n,k})$, $m \geq 1$ of $f_{n,k}$ described in Theorem \ref{thm:sectorial-bounds} converges to a nest of sectors 
        \[
        S_n^1 \supset S_n^2 \supset S_n^3 \supset \ldots
        \]
        in the Carath\'eodory topology of open disks with basepoint $0$.
        These sectors satisfy an analog of Theorem \ref{thm:sectorial-bounds} (1)--(4) for $f_n$. 
        In particular, for any $m \geq -1$ and $j \geq 1$, the conformal gluing map for the sector $S_n^j$ projects $f_n^{[m+j+1]}$ and $\hat{Z}^{[j+m]}_n$ into $f_{n+j}^{[m+1]}$ and $\hat{Z}^{[m]}_{n+j}$.
    \end{enumerate}
\end{theorem}

Note that the structure of the pseudo-Siegel pinched disks $\hat{Z}_n^{[m]}$ are similar to those for quadratic polynomials as described in the previous section.
Namely, each $\hat{Z}_n^{[m]}$ is the union of the Mother Hedgehog $H_n$ together with the level $s$ principal parabolic fjord and its preimages up to time $\qq^n_{[s]}$ for all $s \geq m$ with $\abar_{n+s+1} \geq \threshold$.

As we include the limiting maps $f_0$ and the limiting renormalization orbits $\seq{f_n}_{n\geq 0}$, we obtain compactifications $\overline{\orbsec}$ and $\overline{\towsec}$ of $\orbsec$ and $\towsec$ respectively. 
The space $\overline{\orbsec}$ is equipped with the topology of uniform convergence on compact subsets.
The space $\overline{\towsec}$ is equipped with the topology such that
\begin{itemize}
    \item the projection map
    \[
        \overline{\towsec} \to \overline{\orbsec}, \qquad \seq{f_n}_{n\geq 0} \mapsto f_0
    \]
    is continuous;
    \item the first pre-renormalization $\seq{f_n}_{n\geq 0} \mapsto f_0^{[1]}$ is continuous with respect to the topology of uniform convergence on compact subsets;
    \item the shift map (the renormalization operator)
    \begin{equation}
        \label{eq:dfn.Rsec}
        \Rsec: \overline{\towsec} \to \overline{\towsec}, \quad \seq{f_0,f_1,f_2,\ldots} \mapsto \seq{f_1,f_2,f_3,\ldots}
    \end{equation}
    is continuous.
\end{itemize}

Every forward tower $\fbold = \seq{f_n}_{n\geq 0}$ in $\overline{\towsec}$ comes with an associated combinatorics 
\[
    \ttheta(\fbold) = \seq{ (\varepsilon_n, \abar_n)}_{n\geq 1} \in \TheCpt.
\]
If $f_0$ has an irrationally indifferent fixed point, then $\ttheta(\fbold)$ represents the rotation number $f_0$.

\begin{proposition}
\label{prop:continuity-of-combinatorics}
    The combinatorial map 
    $\ttheta: \overline{\towsec} \to \TheCpt$
    is a semi-conjugacy between $\Rsec$ and the shift map $\shift$ on $\TheCpt$.
\end{proposition}

\begin{proof}
    The identity 
    \[
    \ttheta \circ \Rsec = \shift \circ \ttheta
    \]
    is obvious.
    To prove that $\ttheta$ is continuous, it is sufficient to show that both $\varepsilon_1(\fbold)$ and $\abar_1(\fbold)$ is continuous.
    
    Consider a sequence $\fbold_m = \seq{f_{n,m}}_{n\geq 0}$, $m = 1, 2,\ldots$ of forward towers in $\overline{\towsec}$ converging to $\fbold = \seq{f_n}_{n\geq 0}$.
    Let us write $\ttheta(\fbold_m) = \seq{(\varepsilon_{n,m},\abar_{n,m})}_{n \geq 1}$ and $\ttheta(\fbold) = \seq{(\varepsilon_n, \abar_n)}_{n \geq 1}$.
    Without loss of generality, we will assume $\varepsilon_1 = +$.
    Let us denote by $\theta_{n,m} = \bar{\mu}(\shift^n\ttheta(\fbold_m))$ the rotation number of $f_{n,m}$ and by $\theta_0 = \bar{\mu}(\shift^n\ttheta(\fbold))$ the rotation number of $f_n$.
    For all $n \geq 0$, when $m \to \infty$, since $f_{n,m} \to f_n$, then $\theta_{n,m} \to \theta_n$ as elements of $\T = \R/\Z$.
    
    Suppose first that $\abar_1 = \infty$.
    Since $\theta_{0,m} \to 0$ as $m \to \infty$, it must be the case that $\abar_{1,m}$ converges to $\infty$ as $m \to \infty$.
    It can be assumed (by slight perturbation) that each $f_{0,m}$ satisfies $\abar_{1,m} < \infty$ and $\abar_{1,m} \to \infty$ as $m \to \infty$.
    The map $f_{0,m}^{[1]}$ is the $b_{1,m}$\textsuperscript{th} iterate of $f_{0,m}$ for some $b_{1,m}=b_1(\ttheta(\fbold_m)) \in \{\abar_{1,m}, \abar_{1,m}+1\}$.
    The map $f_{0,m}^{[1]}$ converges to the top Lavaurs map $f_0^{[1]}$ of $f_0$.
    Therefore, from classical Lavaurs theory, we must have $\varepsilon_{1,m} = +$ for all sufficiently high $m$.

    Next, suppose instead that $\abar_1 < \infty$.
    There are two sub-cases.
    \begin{itemize}
        \item Suppose $\abar_2 < \infty$.
        Then, the rotation number $\theta_0 = \bar{\mu}(\ttheta)$ of $f_0$ is contained in the real interval $\left( \frac{1}{\abar_1+1}, \frac{1}{\abar_1} \right)$.
        Then, for all sufficiently high $m$, $\theta_{0,m}$ is in $\left( \frac{1}{\abar_1+1}, \frac{1}{\abar_1} \right)$ and thus $(\varepsilon_{1,m}, \abar_{1,m}) = (\varepsilon_1, \abar_1)$.
        \item Suppose $\abar_2 = \infty$.
        From the previous paragraph (applied to $f_{1,m}$), we already know that $\varepsilon_{2,m} = \varepsilon_{2}$ for all sufficiently high $m$ and $\abar_{2,m} \to \abar_2$ as $m \to \infty$.
        Since $\theta_{0,m} \to \theta_0$, then
        \[
            \frac{1}{\varepsilon_{1,m}(\abar_{1,m}+ \frac{1+\varepsilon_{1,m}\varepsilon_2}{2})} \xrightarrow[\quad]{} \frac{1}{\abar_{1,m}+ \frac{1+\varepsilon_2}{2}}
        \]
        as elements of $\T$.
        An elementary arithmetic exercise shows that $(\varepsilon_{1,m}, \abar_{1,m}) = (+,\abar_1)$ for all sufficiently high $m$.
    \end{itemize}
    This concludes the proof of continuity of $\varepsilon_1(\fbold)$ and $\abar_1(\fbold)$.
\end{proof}

A glossary of combinatorial and dynamical objects associated to elements $\fbold$ of $\overline{\towsec}$ is available in \S\ref{sss:notation-for-towsec} for the reader's convenience.
This includes the Mother Hedgehog $H_n(\fbold)$, the pseudo-Siegel disks $\hat{Z}_n^{[m]}(\fbold)$ for $m \geq -1$, the pre-renormalization semigroup $\mathcal{F}_n = \{f_n^p\}_{p \in \mathcal{T}^n}$, the renormalization sectors $S_n^j(\fbold)$, $j \geq 1$, etc.

\subsubsection{The Mother Hedgehog}

Let us discuss a few fundamental properties of Mother Hedgehog associated to forward towers in $\overline{\towsec}$.

\begin{theorem}[Zero area \cite{Lim26b}]
\label{thm:zero-area}
    For every $\fbold \in \overline{\towsec}$, the postcritical set $P_0(\fbold)$ has zero Lebesgue measure.
\end{theorem}

This theorem was proven by utilizing the pre-compactness of renormalizations and estimating the renormalization change of variables in logarithmic coordinates.
In the same paper, we also study the Brjuno condition.
Recall the rotation number map $\bar{\mu}: \TheCpt \to \T$ from Proposition \ref{prop:rotation-number-map}.

\begin{definition}
    We say that $\ttheta \in \TheCpt$ is \emph{Brjuno} if $\theta_n = \bar{\mu}(\shift^n(\ttheta)) \in \T$ is irrational for all $n \geq 0$ and, if normalized such that $|\theta_n| < \frac{1}{2}$, then the infinite sum
\[
    \log\frac{1}{|\theta_0|} + |\theta_0| \log\frac{1}{|\theta_1|} + |\theta_0\theta_1| \log\frac{1}{|\theta_2|} + |\theta_0\theta_1\theta_2| \log\frac{1}{|\theta_3|} +\ldots
\]
    converges to a finite number.
    More generally, we say that $\ttheta$ is \emph{eventually Brjuno} if there is some $m \geq 0$ such that $\shift^m(\ttheta)$ is Brjuno.
\end{definition}

From the definition above, every element of $\TheCpt$ that is eventually Brjuno but not Brjuno must be an enriched rational.
The following theorem is an extension of Yoccoz's theorem \cite{Yoc95} to renormalization towers $\overline{\towsec}$.

\begin{theorem}[The Brjuno condition \cite{Lim26b}]
\label{thm:eventually-brjuno}
    Let $\fbold = \seq{f_n}_{n\geq 0} \in \overline{\towsec}$ and let $\ttheta = \ttheta(\fbold)$ be its combinatorics.
    \begin{enumerate}
        \item The map $f_0$ admits a Siegel disk around $0$, which is equal to the interior of the Mother Hedgehog $H_0(\fbold)$, if and only if $\ttheta$ is Brjuno.
        \item If $\ttheta$ is eventually Brjuno, the interior of $H_0(\fbold)$ is non-empty and for sufficiently high $m \in \N$, the intersection of the interior of $H_0$ with the renormalization sector $S_0^m(\fbold)$ projects to the Siegel disk of $f_m$.
        \item If $\ttheta$ is not eventually Brjuno, the interior of $H_0(\fbold)$ is empty.
    \end{enumerate}
\end{theorem}

In \cite{Lim26b}, we also prove that in the Brjuno case, the conformal radius of the Siegel disk is uniformly comparable to $e^{-\textnormal{Br}(\theta)}$ where $\textnormal{Br}(\cdot)$ is the Brjuno-Yoccoz function.
This was done by applying estimates of the renormalization change of variables in logarithmic coordinates.

\subsubsection{Near-parabolic dynamics}

\begin{definition}
    We say that a map $f \in \overline{\orbsec}$ is \emph{near-parabolic} if the rotation number $\theta$ of $f$ at the fixed point $0$ satisfies $|\theta| < 1/\threshold$, (\emph{simply}) \emph{parabolic} if $\theta = 0$.
\end{definition}

\begin{definition}
    For any orbit $\fbold = \seq{f_n}_{n\geq 0} \in \overline{\towsec}$ with $f_n$ being near-parabolic for some $n \geq 0$, the set
\[
    \hat{V}_n = \hat{V}_n(\fbold) := \overline{ \hat{Z}_n^{[-1]}(\fbold) \backslash \hat{Z}_n^{[0]}(\fbold) },
\]
    which is non-empty, will be called the $n$\textsuperscript{th} \emph{top regularized fjord} of $\fbold$.
\end{definition}

\begin{theorem}[Simple parabolic bifurcation]
\label{thm:beta-fixed-points}
    Let $\fbold = \seq{f_n}_{n\geq 0} \in \overline{\towsec}$ and let $\theta$ be the rotation number of $f_0$.
    \begin{enumerate}
        \item If $f_0$ is near-parabolic but not parabolic, then there exists a unique repelling fixed point $\beta$ in the interior of $\hat{V}_n$ and it satisfies $|\beta| \asymp |\theta|$.
        \item If $f_0$ is parabolic, then the parabolic fixed point $\beta = 0$ is simple and it is on $\partial \hat{Z}_0^{[0]}(\fbold) \cap \textnormal{int} \hat{Z}_0^{[-1]}(\fbold)$.
    \end{enumerate}
\end{theorem}

\begin{proof}
    It was proven in \cite{DLy26a} that there exists a unique repelling fixed point $\beta$ in the interior of $\hat{V}_n$ and pre-compactness of renormalizations implies the occurrence of simple parabolic bifurcation.
    The estimate on $|\beta|$ then follows from compactness considerations; see \cite{Lim26b} for details.
\end{proof}

\section{QC Thurston equivalence for forward towers}
\label{sec:qc-thurston-equivalence}

Recall from Remark~\ref{rem:intro.forw.towers} and Figure~\ref{fig:roadmap} that QC Thurston equivalence is central to this paper. In a certain sense, it substitutes the role of hybrid equivalence employed in other non-perturbative renormalization theories.

\begin{theorem}[QC Thurston equivalence]
\label{thm:qc-thurston-equivalence}
    There exists an $\threshold$-uniform constant $K>1$ such that for any two combinatorially equivalent renormalization towers $\fbold = \seq{f_n}_{n\geq 0}$ and $\gbold = \seq{g_n}_{n\geq 0}$ in $\overline{\towsec}$, there exist a sequence of $K$-quasiconformal maps $\{h_n : \D \to \D\}_{n \geq 0}$ such that for every $n \geq 0$,
    \begin{enumerate}
        \item $h_n$ sends the critical value of $f_n$ to the critical value of $g_n$;
        \item $h_n$ sends $\hat{Z}^{[m]}_n(\fbold)$ onto $\hat{Z}^{[m]}_n(\gbold)$ for every $m \geq -1$;
        \item $h_n$ is a conjugacy between the pre-renormalization semigroups $\mathcal{F}_n|_{H_n(\fbold)}$ and $\mathcal{G}_n|_{H_n(\gbold)}$ of $\fbold$ and $\gbold$ on their corresponding Mother Hedgehogs;
        \item $h_n$ is conformal on the interior of $H_n(\fbold)$;
        \item for $m \geq 1$, $\psi_{n,\gbold} \circ h_n = h_{n+1} \circ \psi_{n,\fbold}$ on the renormalization sector $S_n(\fbold)$.
    \end{enumerate}
\end{theorem}

The proof of Theorem \ref{thm:qc-thurston-equivalence} has two essential ingredients.
The first is the natural dynamical triangulation of pseudo-Siegel disks arising from their compatibility with sector renormalization; see Figure \ref{fig:triangulation}.
The second is the uniformly quasiconformal tiling of the boundary of pseudo-Siegel disks.
Together, they allow us to apply a pullback argument to construct uniformly quasiconformal pseudo-conjugacies.
As we pass to a limit, we construct QC maps $h_n$'s that preserve pseudo-Siegel disks and are equivariant on the Mother Hedgehogs.
Lastly, we perform some soft surgery procedure to ensure conformality on the interior of the Mother Hedgehogs.

\subsection{Renormalization triangulations of pseudo-Siegel disks}
\label{ss:renorm-tiling}

Let us fix $\fbold = \seq{f_n}_{n\geq 0}$ in $\overline{\towsec}$ with combinatorics $\ttheta = \ttheta(\fbold) = \seq{ (\varepsilon_n, \abar_n)}_{n\geq 1}$. 
We will use the notation listed in \S\ref{sss:notation-for-towsec} associated to $\fbold$.
Below, we will describe a dynamical triangulation 
$\hat{\Delta}^{[m]}_n = \hat{\Delta}^{[m]}_n(\fbold)$ 
of the pseudo-Siegel pinched disks $\hat{Z}^{[m]}_n$ for every $n \geq 0$ and $m \geq -1$.
 
For every $m \geq 1$, denote 
\[
    \check{\qq}^n_{[m]} := \qq^n_{[m]} - \varepsilon_{n+m} \varepsilon_{n+m+1} \qq^n_{[m-1]}.
\]
For $p \in \Kont^{n}$, we will denote by $I^n_p$ the internal ray of the $n$\textsuperscript{th} Mother Hedgehog $H_n$ starting at $0$ and ending at $v^n_p$.
Recall that for every $m \geq 1$, the boundary of the renormalization sector $S_n^m$ contains two internal rays $I^n_{-\qq^n_{[m]}}$ and $I^n_{-\check{\qq}^n_{[m]}}$.
Denote by
\[
\psi_{n,m} : S_n^m \to \D
\]
the conformal gluing map for $S_n^m$ projecting $f_n^{[m]}$ to $f_{n+m}$.

Firstly, the triangulation $\hat{\Delta}^{[-1]}_n$ consists of two triangles, which are the closure of the connected components of $\hat{Z}^{[-1]}_n$ with the internal rays $I^n_{-1}$ and $I^n_0$ removed.
In general, for $m \geq 1$, we define $\hat{\Delta}^{[m-1]}_n$ out of $\hat{\Delta}^{[-1]}_{n+m}$ by pulling back and spreading around as follows.

The lift of $\hat{\Delta}^{[-1]}_{n+m}$ under $\psi_{n,m}$ is a triangulation of $\hat{Z}_n^{[m-1]} \cap S_n^m$ consisting of two elements $A_n^m$ and $\check{A}_n^m$.
The two triangles are separated by $I^n_0$.
Without loss of generality, we will assume that $A_n^m$ is bounded by $I^n_0$ and $I^n_{-\qq^n_{[m]}}$, and $\check{A}_n^m$ is bounded by $I^n_0$ and $I^n_{-\check{\qq}_{[m]}}$.
For $p \in \Time^n$, denote by $A_n^m(-p)$ (resp. $\check{A}_n^m(-p)$) the closure of the lift of the interior of $A_n^m$ (resp. $\check{A}_n^m$) under $f_n^p$ that contains the fixed point $0$ on its boundary.

\begin{lemma}
\label{lem:triangulation-first-return}
    The first return map of the semigroup $\mathcal{F}_n$ back to $\hat{Z}_n^{[m-1]} \cap S_n^m$ is the pair of maps
    \[
        f_n^{\qq^n_{[m]}}: \check{A}_n^m (-\qq^n_{[m]}) \to \check{A}_n^m (0), \quad 
        f_n^{\check{\qq}^n_{[m]}}: A_n^m(-\check{\qq}^n_{[m]}) \to A_n^m(0).
    \]
    In particular, for any point $z$ in $\hat{Z}_n^{[m-1]}$, there exists some time $p \in \Time^n \cup \{0\}$ such that $f_n^p(z)$ is in $\hat{Z}_n^{[m-1]} \cap S_n^m$; the smallest of such $p$ satisfies either $p< \qq^n_{[m]}$ or $p < \check{\qq}^n_{[m]}$.
\end{lemma}

\begin{proof}
    The case when $m=1$ follows from the way the pseudo-Siegel disk $\hat{Z}^{[0]}_n$ is defined by principal dams and their pullbacks. 
    The case when $m \geq 2$ follows from induction in a similar way.
\end{proof}

This lemma implies that we have a well-defined triangulation $\hat{\Delta}^{[m-1]}_n$ of $\hat{Z}_n^{[m-1]}$ given by
\[
\hat{\Delta}^{[m-1]}_n = \left\{ A_n^m(-p) \right\}_{0 \leq p < \check{\qq}^n_{[m]}} \cup 
\left\{ \check{A}_n^m(-p) \right\}_{0 \leq p < \qq^n_{[m]}}.
\]
See Figure \ref{fig:triangulation} for an example when $m=1$.

\begin{figure}
        \centering
       \begin{tikzpicture}
    \node[anchor=south west,inner sep=0] (image) at (0,0) {\includegraphics[width=0.98\linewidth]{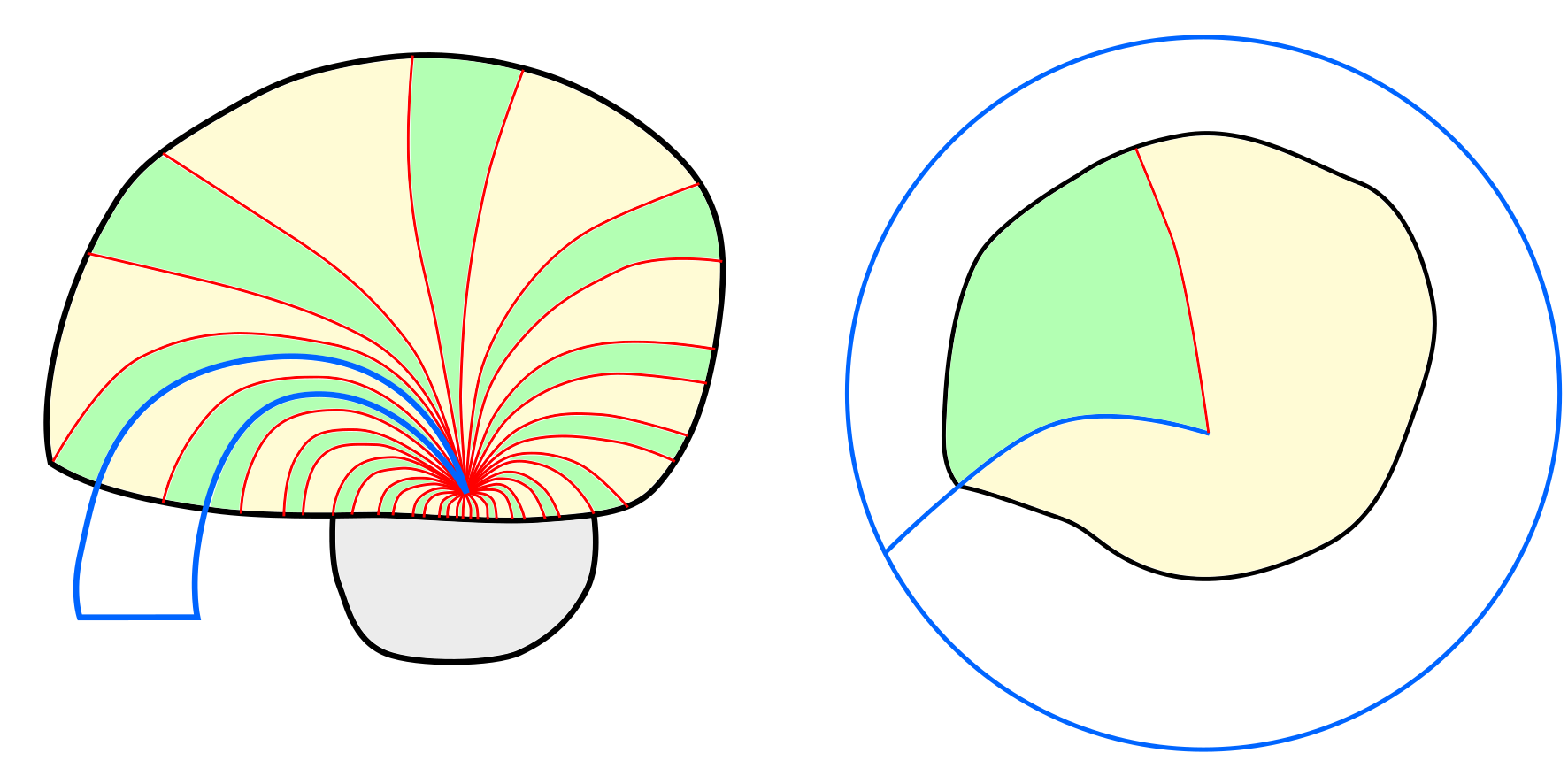}};
    \begin{scope}[
        x={(image.south east)},
        y={(image.north west)}
    ]
        \node [black] at (0.295,0.375) {\small $\bullet$};
        \draw[-latex] (0.32,0.965) .. controls (0.26,1.04) and (0.15,1) .. (0.095,0.865);
        \draw[-latex] (0.075,0.82) .. controls (-0.01,0.77) and (-0.03,0.55) .. (0.01,0.45);
        
        \node [black] at (0.035,0.4) {\small $\bullet$};
        \node [black] at (0.028,0.36) {\small $v_{n,-1}$};
        \node [black] at (0.105,0.355) {\small $\bullet$};
        \node [black] at (0.095,0.31) {\small $v_{n,0}$};
        \node [black] at (0.767,0.448) {$\bullet$};
        \node [black] at (0.77,0.41) {$0$};
        \node [black] at (0.45, 0.88) {\small $\hat{Z}_n^{[0]}$};
        \node [black] at (0.42, 0.24) {\small $\hat{Z}_n^{[-1]}$};
        \node [black] at (0.78, 0.2) {\small $\hat{Z}_{n+1}^{[-1]}$};
        \node [black] at (0.075,0.055) {\small $f_n$};
        \draw[-latex] (0.05,0.16) .. controls (0.055,0.09) and (0.105,0.09) .. (0.12,0.17);
        \node [black] at (0.69,0.64) {\small $f_n$};
        \draw[-latex] (0.69,0.475) .. controls (0.70,0.57) and (0.72,0.6) .. (0.755,0.62);
        \node [blue] at (0.35,0.03) {\small $\psi_n$};
        \draw[blue,-latex] (0.145,0.21) .. controls (0.25,0.03) and (0.45,0.04) .. (0.58,0.17);
        \node [blue] at (0.02,0.24) {\small $S_n$};
    \end{scope}
\end{tikzpicture}
    \caption{The triangulation $\hat{\Delta}_n^{[0]}$ of $\hat{Z}^{[0]}_n$ on the left is obtained by pulling back the triangulation $\hat{\Delta}_{n+1}^{[-1]}$ of $\hat{Z}^{[-1]}_{n+1}$ on the right and spreading it around.}
    \label{fig:triangulation}
\end{figure}
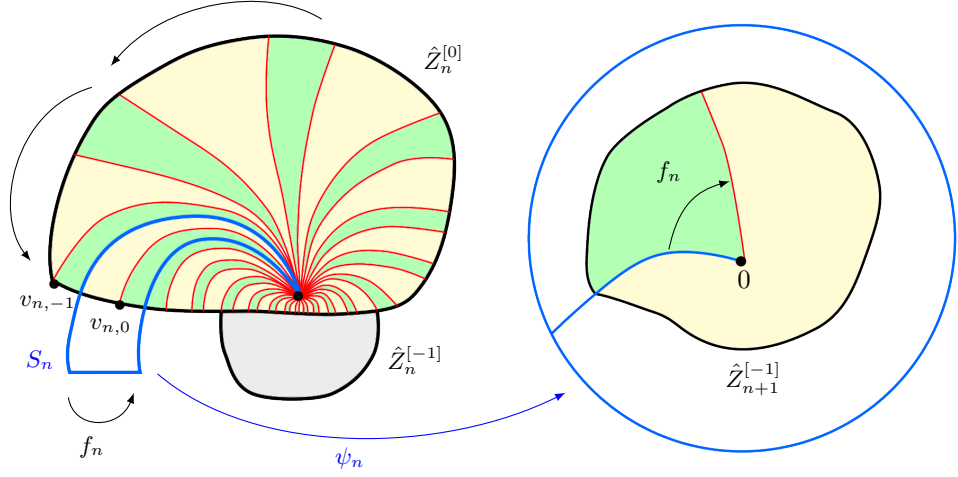

\begin{lemma}
\label{lem:triangulation-invariance-01}
    The triangulations $\hat{\Delta}_n^{[m]}$ of $\hat{Z}_n^{[m]}$ are invariant under sector renormalization, that is, for $m, n \geq 0$,
    \[
        \psi_{n,1}(\hat{\Delta}_n^{[m]} \cap S_n^1) = \hat{\Delta}_{n+1}^{[m-1]}.
    \]
\end{lemma}

\begin{proof}
    This is immediate from the construction.
\end{proof}

The triangulation $\hat{\Delta}_n^{[m]}$ also induces a tiling of the boundary of $\hat{Z}_n^{[m]}$ given by
\begin{equation}
\label{eqn:zeroth-tiling}
    \mathscr{T}_{n,0}^{[m]} := \left\{ A \cap \partial \hat{Z}_n^{[m]} \: : \: A \in \hat{\Delta}_n^{[m]} \right\}.
\end{equation}
This is precisely the tiling cut out by the critical points $\{ v^n_{-p} \}_{0 \leq p < \qq^n_{[m+1]} + \check{\qq}^n_{[m+1]} }$.

\begin{theorem}[Diffeo-tiling]
\label{thm:diffeo-tiling}
    Consider $\fbold = \seq{f_n}_{n \geq 0} \in \overline{\towsec}$ and the corresponding pseudo-Siegel pinched disks $\hat{Z}_n^{[m]} = \hat{Z}_n^{[m]}(\fbold)$.
    For $n \geq 0$ and $m \geq -1$, there exists nest of tilings 
    \[
    \mathscr{T}_n^{[m]} = \left\{ \mathscr{T}_{n,k}^{[m]}  \right\}_{ k \geq 0 } 
    \]
    of the boundary of $\hat{Z}_n^{[m]}$ with the following properties.
    \begin{enumerate}
        \item For $n \geq 0$ and $m \geq -1$, $\mathscr{T}_{n,0}^{[m]}$ is equal to the tiling introduced above in \textnormal{(\ref{eqn:zeroth-tiling})}.
        \item For every NP level $n$, the top level $n$ dam $d_n = \partial \hat{Z}_n^{[-1]} \backslash \partial \hat{Z}_n^{[0]}$ admits a nest of tilings $\mathscr{T}(d_n) = \left\{ \mathscr{T}_k (d_n) \right\}_{k \geq 0}$ such that for $k \geq 0$,
        \[
            \mathscr{T}_{n,k+1}^{[-1]} = \mathscr{T}_k (d_n) \cup \left\{ A \in \mathscr{T}_{n,k}^{[0]}  \: : \: A \subset \partial \hat{Z}_n^{[-1]} \right\}.
        \]
        Moreover, $\mathscr{T}_0 (d_n)$ has $10$ tiles, and for every $k \geq 0$, every tile in $\mathscr{T}_k (d_n)$ is partitioned into exactly $10$ tiles in $\mathscr{T}_{k+1} (d_n)$.
        \item For $n \geq 0$, $m \geq 0$, and $k \geq 0$, the tiling $\mathscr{T}_{n,k}^{[m]}$ is obtained by pulling back $\mathscr{T}_{n+1,k}^{[m-1]}$ via the conformal gluing map $\psi_{n} : S_{n} \to \D$ and spreading it around by lifting under the semigroup $\mathcal{F}_n$ before first return to the sector $S_{n}$.
        \item For $n \geq 0$, $\mathscr{T}_n^{[-1]}$ has $\threshold$-uniformly bounded combinatorics and inner and outer geometry. 
        \item For every NP level $n$, 
        the symmetric difference $\mathscr{T}_{n,k+1}^{[-1]} \triangle \mathscr{T}_{n,k}^{[0]}$
        gives a nest of tilings of the boundary of the top regularized fjord $\hat{V}_n$ having $\threshold$-uniformly bounded inner geometry.
    \end{enumerate}
\end{theorem}

\begin{proof}
    Items (1), (2), and (3) describe the construction of the nests of tilings.
    The main point of the theorem is that $\mathscr{T}(d_n)$ can be constructed such that items (4) and (5) hold.
    When $f_0$ is a neutral quadratic polynomial with bounded-type rotation number, such nests of tilings were constructed in \cite{DLy22} via near-degenerate regime.
    For a general tower $\fbold$ in $\overline{\towsec}$, this follows from taking a limit.
\end{proof}

\subsection{Proof of Theorem \ref{thm:qc-thurston-equivalence}}

Consider two renormalization towers $\fbold = \seq{f_n}_{n\geq 0}$ and $\gbold = \seq{g_n}_{n\geq 0}$ in $\overline{\towsec}$ with the same combinatorics $\ttheta = \seq{ (\varepsilon_n, \abar_n) }_{n\geq 1}$.
For $n \geq 0$, let $\ttheta_n = \shift^n(\ttheta)$, let $\Kont^{n}$ be the continuant group associated to $\ttheta_n$, and let $\Time^n$ be the corresponding time semigroup.

For $n \in \N$, let $\{v^n_p(\fbold)\}_{p \in \Kont^{n}}$ be the bi-infinite critical orbit of $\fbold$ on its $n$\textsuperscript{th} Mother Hedgehog $H_n(\fbold)$. 
For each $p \in \Kont^{n}$, we denote by $I^n_p(\fbold)$ the internal ray of $H_n(\fbold)$ landing at $v^n_p(\fbold)$.
In the dynamical plane of $g_n$, we will also use the same notation $v^n_p(\gbold)$ and $I^n_p(\gbold)$.

We will apply the triangulations and tilings described in the previous subsection.
We will need one more lemma before we proceed.

\begin{lemma}
\label{lem:dividing-pair-extension}
    Let $h$ be a $K_1$-quasisymmetric map from $\partial \hat{Z}_0^{[-1]}(\fbold)$ to $\partial \hat{Z}_0^{[-1]}(\gbold)$ with 
    \[
    h \big( v^0_0(\fbold) \big) = v^0_0(\gbold)
    \qquad \text{ and } \qquad 
    h \big( v^0_{-1}(\fbold) \big) = v^0_{-1}(\gbold).
    \]
    Then, $h$ admits a $K_2$-quasiconformal extension to the whole plane $\C$ such that
    \begin{enumerate}
        \item $h(I^0_j(\fbold)) = I^0_j(\gbold)$ for $j \in \{-1,0\}$;
        \item $h \circ f_0 = g_0 \circ h$ on $I^0_{-1}(\fbold)$;
        \item $K_2$ depends only on $K_1$ and $\threshold$.
    \end{enumerate}
\end{lemma}

\begin{proof}
    This is a consequence of \cite{DLy26a}.
    The idea is as follows.
    When $f_0$ is not near-parabolic, then $I_0^0(\fbold) \cup I_{-1}^0(\fbold)$ splits the pseudo-Siegel disk $Z_0^{[-1]}(\fbold)$ into two uniform quasidisks; the same goes to $\gbold$.
    In the near-parabolic case, uniformity is attained in perturbed Fatou coordinates.
\end{proof}

Inductively, over increasing $m \geq -1$, we will construct for every $n \geq 0$ a $K$-quasiconformal map $\xi_{n,m}: \D \to \D$ having the following properties.
    \begin{enumerate}[label={(\textnormal{\alph*})}]
        \item $\xi_{n,m}$ sends $v^n_0(\fbold)$ to $v^n_0(\gbold)$.
        \item For $k \in \{-1,\ldots, m\}$, $\xi_{n,m}$ sends the nest of tilings $\mathscr{T}_n^{[k]}(\fbold)$ onto the nest of tilings $\mathscr{T}_n^{[k]}(\gbold)$.
        \item $\xi_{n,m}$ sends the triangulation $\hat{\Delta}_n^{[m]}(\fbold)$ onto the triangulation $\hat{\Delta}_n^{[m]}(\gbold)$.
        \item $\xi_{n,m}$ is a pseudo-conjugacy in the following sense. 
        \begin{enumerate}[label={(\textnormal{\roman*})}]
            \item $\xi_{n,m}$ is equivariant on the internal rays making up $\hat{\Delta}_n^{[m]}(\fbold)$. 
            That is, for any two distinct internal rays $I^n_{-p}(\fbold)$ and $I^n_{-q}(\fbold)$ on $\cup \hat{\Delta}_n^{[m]}(\fbold)$ with $p < q$, then
            \[
                \xi_{n,m} \circ f_n^{q-p}(z) = g_n^{q-p} \circ \xi_{n,m}(z) \qquad \text{ for all } z \in I^n_{-q}(\fbold).
            \]
            \item When $m \geq 0$, $\xi_{n,m}$ is equivariant up until the first return back to the sector $S_n^{m+1}(\fbold)$. 
            That is, for every $z \in \hat{Z}_n^{[m]}(\fbold) \backslash S_n^{m+1}(\fbold)$, if $q(z) \in \Time^n$ is the smallest time such that $f_n^{q(z)}(z)$ is on the sector $S_n^{m+1}(\fbold)$, then,
        \[
            \xi_{n,m} \circ f_n^p(z) = g_n^p \circ \xi_{n,m}(z) \qquad \text{ for } 0 < p \leq q(z).
        \]
        \end{enumerate}
        \item The dilatation $K$ is $\threshold$-uniform over $\overline{\towsec}$.
    \end{enumerate}

    The construction starts with the case when $m=-1$.
    By Theorem \ref{thm:diffeo-tiling}, since $\fbold$ and $\gbold$ are combinatorially equivalent, there exists a unique $K_1$-quasisymmetric map
    $\xi_{n,-1}: \partial \hat{Z}_n^{[-1]}(\fbold) \to \partial \hat{Z}_n^{[-1]}(\gbold)$
    that satisfies properties (a) and (b) above.
    By Lemma \ref{lem:dividing-pair-extension}, $\xi_{n,-1}$ extends to a global $K_2$-quasiconformal map that satisfies (c), (d)(i), and (e).
    This resolves the base case of the induction.
    
    Next, let us construct the map $\xi_{n,m}$ out of $\xi_{n+1,m-1}$ for all $m,n \geq 0$.
    Consider the renormalization change of variables $\psi_{n,\fbold}: S_{n}(\fbold) \to \D$ from the dynamical plane of $f_n$ to the dynamical plane of $f_{n+1}$.
    Similarly, consider $\psi_{n,\gbold}: S_{n}(\gbold) \to \D$ associated to $\gbold$.
    By spreading around, the lift $\psi_{n,\gbold}^{-1} \circ \xi_{n+1,m-1} \circ \psi_{n,\gbold} : S_{n}(\fbold) \to S_{n}(\gbold)$ extends to a quasiconformal map $\xi_{n,m}: \hat{Z}_n^{[0]}(\fbold) \to \hat{Z}_n^{[0]}(\gbold)$ that satisfies (a), (c), (d), and sends nest of tilings $\mathscr{T}_n^{[k]}(\fbold)$ onto the nest of tilings $\mathscr{T}_n^{[k]}(\gbold)$ for $0 \leq k \leq m$.
    By Theorem \ref{thm:diffeo-tiling}, the map $\xi_{n,m}$ extends to a quasiconformal map on $\D$ that sends the nest of tilings $\mathscr{T}_n^{[-1]}(\fbold)$ onto the nest of tilings $\mathscr{T}_n^{[-1]}(\gbold)$, thus (b) holds.
    The dilatation of $\xi_{n,m}$ on $\hat{Z}_n^{[0]}(\fbold)$ is the same as that of $\xi_{n+1,m-1}$. By Theorem \ref{thm:diffeo-tiling}, the dilatation of $\xi_{n,m}$ outside of $\hat{Z}_n^{[0]}(\fbold)$ can be arranged to be a fixed $\threshold$-uniform constant $K>1$.
    Hence, $\xi_{n,m}$ also satisfies (e).

    For every $n \geq 0$, as $m \to \infty$, the map $\xi_{n,m}$ converges in subsequence to a $K$-quasiconformal map $h_n: \D \to \D$. It has the property that $h_n$ sends $\hat{Z}^{[m]}_n(\fbold)$ onto $\hat{Z}^{[m]}_n(\gbold)$ for all $m \geq -1$ and it respects the corresponding dynamical triangulations. On the nested intersection, $h_n$ is a conjugacy between $f_n^{[m]}: H_n(\fbold) \to H_n(\fbold)$ and $g_n^{[m]}: H_n(\gbold) \to H_n(\gbold)$ for all $m \geq 0$.
    
    Recall from Theorem \ref{thm:eventually-brjuno} that $H_n(\fbold)$ has non-empty interior if and only if $\fbold$ is eventually Brjuno.
    It remains to show that in this case, the conjugacy $h_n$ can be modified such that it is conformal in the interior of $H_n(\fbold)$.
    There are two cases to consider.
    \vspace{0.1in}
    
    \noindent \underline{Case 1:} $\fbold$ and $\gbold$ are Brjuno. \\
    We will perform a soft surgery procedure. 
    Consider their corresponding Siegel disks by $Z_{f_0} = \text{int} H_0(\fbold)$ and $Z_{g_0} = \text{int} H_0(\gbold)$
    and denote their Riemann mappings by 
    \[
    \eta_{f_0}: (Z_{f_0}, 0) \to (\D,0) 
    \qquad \text{and} \qquad 
    \eta_{g_0}: (Z_{g_0}, 0) \to (\D,0)
    \]
    respectively. 
    For any $\varepsilon \in (0,1)$, consider the invariant subdisk 
    \[
    Z_{f_0}(\varepsilon) = \{ z \in Z_{f_0} \: : \: |\eta_{f_0}(z)| < \varepsilon \}.
    \]
    Define the quasiconformal map 
    \[
        \phi : \D \to \D, \quad z \mapsto \eta_{g_0} \circ h_0 \circ \eta_{f_0}^{-1}(z).
    \]
    Observe that since $h_0$ is a conjugacy, $\phi$ will have to send concentric circles to concentric circles.
    So for each $\varepsilon \in (0,1)$, we can modify $\phi$ to be conformal on a disk of radius $1-\varepsilon$ as follows:
    \[
        \phi_{\varepsilon}(z) = \begin{cases}
            \frac{\phi(\varepsilon)}{\varepsilon} z & \textnormal{ for } |z|<\varepsilon, \\
            \phi(z) & \textnormal{ for } \varepsilon \leq |z| < 1.
        \end{cases}
    \]
    For $\varepsilon \in (0,1)$, we then define 
    \[
        \tilde{h}_{0,\varepsilon} : \D \to \D, \quad 
        \tilde{h}_{0,\varepsilon} (z) :=
        \begin{cases}
            h_0(z) & \text{ if } z \in \D \backslash Z_{f_0}(\varepsilon), \\
            \eta_{g_0}^{-1} \circ \phi_{\varepsilon} \circ \eta_{f_0}(z) & \text{ if } z \in Z_{f_0}(\varepsilon).
        \end{cases}
    \]
    The new map $\tilde{h}_{0,\varepsilon}$ is conformal on $Z_{f_0}(\varepsilon)$.
    Elsewhere, it remains $K$-quasiconformal and satisfies all the desired properties that $h_0$ possesses.  
    As we take $\varepsilon \to 1$, the map $\tilde{h}_{0,\varepsilon}$ converges in subsequence to desired $K$-quasiconformal map $\tilde{h}_0$ that is a conformal conjugacy between the Siegel disks and is equal to $h_0$ outside of the Siegel disks.
    \vspace{0.1in}
    
    \noindent \underline{Case 2:} $\fbold$ and $\gbold$ are eventually Brjuno but not Brjuno. \\ 
    Pick a number $N \geq 1$ such that $f_N$ and $g_N$ are Brjuno.
    By the argument above, we can modify the map $h_N$ to a map $\tilde{h}_N$ that is now conformal on the Siegel disk $Z_{f_N}$.
    Consider the triangulation $\hat{\Delta}_N^{[-1]}(\fbold)$ and let us define a new triangulation $\tilde{\Delta}_N^{[-1]}(\gbold) = h_N( \hat{\Delta}_N^{[-1]}(\fbold) )$ of the pseudo-Siegel disk $\hat{Z}^{[-1]}_N(\gbold)$.
    Combinatorially, $\tilde{\Delta}_N^{[-1]}(\gbold)$ is analogous to $\hat{\Delta}_N^{[-1]}(\gbold)$, but they may differ slightly on part of the internal rays within the Siegel disk $Z_{g_N}$.
    By lifting $\tilde{\Delta}_N^{[-1]}(\gbold)$ and spreading around, we have a triangulation $\tilde{\Delta}_0^{[N-1]}(\gbold)$ and a new renormalization sector $\tilde{S}_0^N(\gbold)$ such that they are combinatorially similar to $\hat{\Delta}_0^{[N-1]}(\gbold)$ and $S_N^0(\gbold)$ respectively and only differ slightly on part of the internal rays within the interior of $H_0(\gbold)$.
    The new sector also comes with a conformal gluing map $\tilde{\psi}_{0,N,\gbold}: \tilde{S}_0^N(\gbold) \to \D$ that projects $g_0^{[N]}$ onto $g_N$.
    Then, by lifting $\tilde{h}_N$ via the gluing maps $\psi_{0,N,\fbold}: S_0^N(\fbold) \to \D$ and $\tilde{\psi}_{0,N,\gbold}: S_0^N(\gbold) \to \D$, we obtain a new quasiconformal map $\tilde{h}_0$ that is now a conformal conjugacy on the interior and coincides with $h_0$ outside of the interior of $H_0(\fbold)$. 
\vspace{0.1in}

This wraps up the proof of Theorem \ref{thm:qc-thurston-equivalence}.


\section{Transcendental Dynamics}
\label{sec:transcendental-dynamics}
This section marks the beginning of our discussion on the renorm-attractor $\attr$ and the transcendental dynamics of neutral cascades. As it was mentioned in~\S\ref{sss:TD.outline}, in here we will transfer various geometric bounds for towers $\towsec$ from Section~\ref{sec:pseudo-Siegel} to neutral cascades. Refer to Figure~\ref{fig:transfer-to-cascade} for the main mechanism  of this transfer.

In \S\ref{ss:neutral-cascade}, we use the precompactness arising from pseudo-Siegel theory to construct \emph{neutral cascades} as the transcendental rescaled limits of the quadratic dynamical systems generated by $f_\theta$ for $\theta \in \Irrat$.
Then, in \S\ref{ss:cascade-tower}, we show that every bi-infinite tower $\fboldbar = \seq{ f_n }_{n \in \Z}$ in $\attr$ can be fit to a single dynamical plane of a neutral cascade $\Fbold = (\Fbold^P)_{P \in \Tbold}$; this makes the renormalization change of variables linear.
In \S\ref{ss:wZbounds for halfplanes}--\ref{ss:MotherHedg:halfplane}, we show that pseudo-Siegel disks and Mother Hedgehog of towers transfer to pseudo-Siegel (pinched) half planes $\hat{\Zbold}^{[m]}(\Fbold)$, $m\in\Z$ and the Mother Hedgehog $\Hbold(\Fbold)$ of neutral cascades $\Fbold$ satisfying the same key properties as discussed in Section \ref{sec:pseudo-Siegel}.

In \S\ref{ss:trans-external-coordinates}, we construct \emph{external cascades} $\Fext = \Fext_{\textnormal{ext}}$ from neutral cascades $\Fbold$ by uniformizing the complement of the Mother Hedgehog, and we establish Real Bounds (Proposition \ref{prop:R-apb-transcendental}) for the critical points $\{\crit_{-P}\}_{P \in \Tbold}$ of $\Fext$ (paralleling Theorem \ref{thm:R-APB}).
Lastly, we end in \S\ref{ss:near-parabolic-levels} with an estimate of parabolic fjords $\fjord$ (see Figure~\ref{fig:angular-control-fjords}) and the associated beta periodic points $\beta_{\fjord}$ in external coordinates.

The key notations introduced in this section are available in \S\ref{sss:cascade}.

\subsection{Neutral cascades}
\label{ss:neutral-cascade}

Another consequence of pseudo-Siegel bounds is that 
the set of pre-renormalizations of neutral quadratic maps, namely 
    \[
    \left\{ f^{[n]}_\theta \: : \: n \geq 0, \theta \in \Irrat \right\},
    \]
is pre-compact up to affine conjugacy. 
Below, we will make this observation more precise.

The following definition comes from McMullen.

\begin{definition}
\label{defn:sigma-proper}
    A holomorphic map $F: A \to B$ is said to be \emph{$\sigma$-proper} if there exist exhaustions $A_n$, $B_n$ of non-empty open subsets $A$, $B$ of $\C$ respectively such that for all $n$, $F : A_n \to B_n$ is a proper map; equivalently, every connected component of the preimage of a compact subset of $B$ under $F$ is compact.
\end{definition}

Following \S\ref{ss:pseudo-siegel}, for every irrational $\theta \in \Irrat$, we denote by $v_{\theta,0}$ the critical value of $f_{\theta}$ and by $\{v_{\theta,k}\}_{k \in \Z}$ the unique bi-infinite critical orbit of $f_{\theta}$ contained in the Mother Hedgehog $H_{\theta}$ of $f_{\theta}$.

In this subsection, we will fix
\begin{itemize}
    \item an increasing sequence of positive integers $\{n_k\}_{ k\geq 0}$, and 
    \item a sequence of irrationals $\{\theta_k\}_{k \geq 0}$ in $\Irrat$.
\end{itemize}

Recall the combinatorial threshold $\threshold \in \N$ for pseudo-Siegel disks.
Recall the homeomorphism $\mathfrak{X}: \Sigma^{\N} \to \Irrat$ and the projection map $\textnormal{proj}: \TheBiCpt \to \TheCpt$ from Section \ref{sec:combinatorics}.
By the compactness of $\TheCpt$, we have:

\begin{lemma}
    \label{lem:combinatorial-convergence}
    There exists an element $\tttheta = \seq{ (\varepsilon_n, \abar_n) }_{n \in \Z}$ of $\TheBiCpt$ such that, after passing to a subsequence over $k$, the following holds.
    Write $\mathfrak{X}^{-1}(\theta_k) = \seq{(\varepsilon_{k,j}, \abar_{k,j})}_{j \geq 1}$.
    Then, for all $m \in \N$, for $k \geq m$ and $j \in \{-m,\ldots, m\}$,
    \begin{itemize}
        \item $\varepsilon_{k,n_k+j} = \varepsilon_j$, 
        \item $\abar_{k,n_k+j} = \abar_j$ if $\abar_j < \infty$,
        \item $\abar_{k,n_k+j} \geq \threshold + m$ if $\abar_j = \infty$.
    \end{itemize}
    In particular, for $m \in \Z$, the sequences $\mathfrak{X}^{-1}(\gauss^{n_k+m}(\theta_k))$ converge to $\textnormal{proj}(\shift^m(\tttheta))$ in $\TheCpt$ as $k \to \infty$.
\end{lemma}

From now on, we will pass to subsequence over $k$ such that the lemma above holds.
We will fix the limiting bi-infinite sequence $\tttheta = \seq{ (\varepsilon_n, \abar_n) }_{n \in \Z} \in \TheBiCpt$ described above.

For every $k \in \Z$, let $\{q^k_{[l]}\}_{l\geq 0}$ be the return times associated to $\theta_k$, and let 
\[
    u_k(z) := \varepsilon_{0} \cdot\frac{z-v_{\theta_k,0}}{v_{\theta_k,-q^k_{[n_k-1]}}-v_{\theta_k,0}}
\]
be the unique affine map sending the critical value $v_{\theta_k,0}$ of $f_{\theta_k}$ to $0$ and the pre-critical point $v_{\theta_k, -q^k_{[n_k-1]}}$ to $\varepsilon_0$.
For every $k \geq 1$ and $l \in \Z$, denote
\begin{equation}
    \label{eqn:hash-notation}
        f_k^{\#} := u_k \circ f_{\theta_k} \circ u_k^{-1} \qquad
        \text{ and } \qquad
        v^{\#}_{k,l} := u_k(v_{\theta_k, l}).
\end{equation}

Let $\{Q_{[m]}\}_{m \in \Z}$ be the generating set for the continuant group $\Kbold_{\tttheta}$.
For every element $P = \sum_{j=1}^s t_j Q_{[m_j]}$ of $\Kbold_{\tttheta}$, we denote
\[
    P_k^{\#} := \sum_{j=1}^s t_j q^k_{[n_k+m_j]},
\]
which is well-defined for sufficiently high $k \in \N$ (depending on $P$).
If $P$ is in the time semigroup $\Tbold_{\tttheta}$, then the number $P_k^{\#}$ is also positive for sufficiently high $k$.

\begin{lemma}
\label{lem:scaling-critical-value}
    Up to subsequence, for every $ m \in \Z $, the point $v^{\#}_{k, Q_{[m]}}$ converges to a point $C_{Q_{[m]}}$.
    There exist universal constants $K >1$, $\lambda_1$, and $\lambda_2$ with $1<\lambda_1<\lambda_2$ such that for every $m \in \Z$,
    \[
        K^{-1} \lambda_1^m \leq | C_{Q_{[m]}}| \leq K \lambda_2^m.
    \]
\end{lemma}

\begin{proof}
    By Corollary \ref{cor:universal-scaling-law}, we have 
    \[
    K^{-1} \lambda_1^m \leq \left| v^{\#}_{k,q^k_{[n_k+m]}} \right| \leq K \lambda_2^m
    \]
    for some universal constants $K \in (0,1)$ and $1 < \lambda_1 < \lambda_2$.
    Therefore, up to subsequence, as $k \to \infty$, the point $v^{\#}_{k,q^k_{[n_k+m]}}$ converges to $C_{Q_{[m]}}$ and this limit point satisfies the same inequality.
\end{proof}

The following theorem describes the limiting transcendental dynamics arising from the limit of the dynamical system generated by $f_k^{\#}$.

\begin{theorem}
\label{thm:transcendental-extension}
    Up to taking subsequence in $k$, the following properties hold.
    \begin{enumerate}
        \item There exists a collection of holomorphic $\sigma$-proper maps $\Fbold = (\Fbold^P)_{P \in \Tbold_{\tttheta}}$ such that
        \begin{enumerate}[label=\textnormal{(\alph*)}]
            \item each $\Fbold^P$ is a $\sigma$-proper map from a non-empty open subset $\Dom(\Fbold^P)$ of $\C$ onto the whole plane $\C$;
            \item for any two distinct elements $P$ and $Q$ of $\Tbold_{\tttheta}$ with $P > Q$, then 
            \[
                \Dom(\Fbold^P) = \Fbold^{-(P-Q)} \big( \Dom(\Fbold^Q) \big) \subset \Dom(\Fbold^Q);
            \]
            \item $\Fbold$ is a commutative semigroup parametrized by $\Tbold_{\tttheta}$:
            \begin{equation}
            \label{eq:thm:transcendental-extension}
                 \Fbold^P \circ \Fbold^Q(z) = \Fbold^{P+Q}(z) \qquad \text{ for all } P, Q \in \Tbold_{\tttheta} \text{ and } z \in \Dom(\Fbold^{P+Q});
            \end{equation}
            \item for every $P \in \Tbold_{\tttheta}$, 
            \[
                \lim_{k \to \infty}
                (f_k^{\#})^{P_k^{\#}}
                = \Fbold^P
            \]
            uniformly on compact subsets of $\Dom(\Fbold^{P})$.
        \end{enumerate}
        \item There exists a collection of distinct points $\{C_P \}_{P \in \Kbold_{\tttheta}}$ in $\C$ such that
        \begin{enumerate}[label=\textnormal{(\alph*)}]
            \item for any $P \in \Kbold_{\tttheta}$ and $Q \in \Tbold_{\tttheta}$, we have 
            \[
            \Fbold^Q(C_P) = C_{P+Q},
            \]
            and the local degree of $\Fbold^Q$ at $C_P$ is two if $P+Q \geq 0$ and one otherwise;
            \item the bi-infinite orbit $\{C_P\}_{P \in \Kbold_{\tttheta}}$ is a subset of $\cap_{P \in \Tbold_{\tttheta}} \Dom(\Fbold^P)$;
            \item for every $P \in \Kbold_{\tttheta}$, 
            \[
                \lim_{k \to \infty} v^{\#}_{k, P_k^{\#}} = C_{P}.
            \]
        \end{enumerate}
    \end{enumerate}
\end{theorem}

From our normalization, we automatically have that
\[
    C_0 = 0 \qquad \text{ and } C_{Q_{[-1]}} = \varepsilon_0.
\]

\begin{proof}
    Below, we will use a diagonal argument to establish convergence up to subsequence countably many times.
    Throughout the proof, for $P \in \Kbold_{\tttheta}$, we will denote by $m_P$ the largest integer such that $|P|< Q_{[m_P]}$.
    
    For sufficiently high $k$, we will denote by
    \[
        H_k^{\#} := u_k(H_{\theta_{n_k}}) 
        \qquad \text{and} \qquad
        \hat{Z}_k^{\#}(\threshold) := u_k( \hat{Z}_{\theta_{n_k}}(M))
    \]
    the Mother Hedgehog of the rescaled map $f_k^{\#}$ and top pseudo-Siegel disk with combinatorial threshold $M$ respectively.
    \vspace{0.1in}

    \noindent\underline{Claim 1:} The collection of distinct points $\{C_P\}_{P \in \Kbold_{\tttheta}}$ satisfying (2)(c) exists.

    \begin{proof}
        Given any $M \geq \threshold$ and $m \in \Z$, for sufficiently high $k \in \N$ depending on $m$, we have a well-defined subarc $\gamma_{k,m,M}$ of $\partial \hat{Z}_k^{\#}(M)$ with endpoints $v^{\#}_{k,Q_{[m]}}$ and $v^{\#}_{k,-Q_{[m]}}$ that contains the critical value $v^{\#}_{k,0} = 0$.
        The arc $\gamma_{k,m,M}$ is a quasiarc with dilatation depending only on $M$ and not $k$.
        Hence, by Claim 1, $\gamma_{k,m,M}$ has Euclidean diameter bounded from both above and below by constants that depend on $M$ and $m$ but not on $k$.
        As $k \to \infty$, $\gamma_{k,m,M}$ converges in subsequence to a bounded non-degenerate quasiarc $\gamma_{\infty,m,M}$.
        
        Let us first fix $P \in \Kbold_{\tttheta}$. 
        By Proposition \ref{prop:proper-discontinuity}, $P$ must be between $-Q_{[m_P]}$ and $Q_{[m_P]}$ with respect to the space order $\lhd$.
        There exists some sufficiently high integer $M_P \geq \threshold$ (depending only on $P$) such that for sufficiently high $k \in \N$, the point $v^{\#}_{k,P_k^{\#}}$ is contained on the boundary of the rescaled pseudo-Siegel disk $\hat{Z}_k^{\#}(M_P)$.
        The point $v^{\#}_{k,P_k^{\#}}$ is contained in the quasiarc $\gamma_{k,m_P, M_P}$.
        Hence, $v^{\#}_{k,P_k^{\#}}$ is uniformly bounded and as $k \to \infty$, it converges in subsequence to a point $C_P$ in $\gamma_{\infty,m_P,M_P}$.

        Let $P, R \in \Kbold_{\tttheta}$ where $P \neq R$.
        It remains to show that $C_P \neq C_R$.
        Let $m = \max\{m_P,m_R\}$ and $M = \max\{M_P, M_R\}$.
        Then, the quasiarc $\gamma_{k,m,M}$ contains both $v^{\#}_{k,P_k^{\#}}$ and $v^{\#}_{k,R_k^{\#}}$.
        By Real Bounds (Theorem \ref{thm:R-APB}) and Corollary \ref{cor:qc-external-coords}, there exists a constant $\kappa=\kappa(m,M) > 1$ such that
        \[
            \kappa^{-1} \diam(\gamma_{k,m,M}) \leq |v^{\#}_{k,P_k^{\#}} - v^{\#}_{k,R_k^{\#}}| \leq \kappa \,\diam(\gamma_{k,m,M}).
        \]
        As $k \to \infty$, we obtain 
        \[
            \kappa^{-1} \diam(\gamma_{\infty,m,M}) \leq |C_P - C_R| \leq \kappa \,\diam(\gamma_{\infty,m,M}),
        \]
        and we are done.
    \end{proof}

    Next, we construct for every $P \in \Tbold_{\tttheta}$ and every $k \in \N$ the nest of domains 
    \[
        V^P_{k,1} \subset V^P_{k,2} \subset \ldots \subset V^P_{k,k}
    \]
    as follows.
    For $k \in \N$, let $\mathfrak{R}_k$ be the unbounded geodesic ray with respect to the hyperbolic metric of $\C\backslash H_k^{\#}$ that lands at the critical point $v^{\#}_{k,-1}$ of $f_k^{\#}$; such a ray is disjoint from $\hat{Z}_k^{\#}(\threshold)$.
    For $P \in \Tbold_{\tttheta}$, $k \in \N$, and $s \in \{1,\ldots,k\}$, we set
    \[
        V^P_{k,s} := \C \backslash ( \mathfrak{R}_k \cup \mathfrak{U}^P_{k,s} \cup \mathfrak{W}_{k,s}^P )
    \]
    where $\mathfrak{U}_{k,s}$ and $\mathfrak{W}_{k,s}^P$ are defined as follows.
    \begin{itemize}
        \item The internal rays of $H_k^{\#}$ landing at $v^{\#}_{k,Q_{[m_P-s]}}$ and $v^{\#}_{k,-Q_{[m_P-s]}}$ split $\hat{Z}_k^{\#}(\threshold + s)$ into two connected components.
        The closure of the connected component that does not contain the critical value $0$ is the set $\mathfrak{U}^P_{k,s}$.
        \item For $n \in \N$ with $\abar_{k,n+1} \geq \threshold+s$, let us define a \emph{wedge} of level $n$ of $\hat{Z}_k^{\#}(\threshold+s)$ to be a closed topological sector whose summit is a dam of level $n$ of $\hat{Z}_k^{\#}(\threshold+s)$ (cf. \S\ref{ss:pseudo-siegel}) and whose sides are the internal rays of $H_k^{\#}$ landing at the endpoints of the dam. 
        There are $q_{[n]}$ such dams, leading to $q_{[n]}$ wedges of level $n$.
        Then, the set $\mathfrak{W}_{k,s}^P$ is the union of all wedges of $\hat{Z}_k^{\#}(\threshold+s)$ of level $< n_k + m_P$ that are not contained in $\mathfrak{U}_{k,s}$.
    \end{itemize}
    Let $\textnormal{CV}^P_{k,s}$ denote the set of critical values of $(f_k^{\#})^{P_k^{\#}}$ in $V^P_{k,s}$.
    
    \vspace{0.1in}
    
    \noindent\underline{Claim 2:} The cardinality of $\textnormal{CV}^P_{k,s}$ is independent of $k$ and tends to $\infty$ as $s \to \infty$.

    \begin{proof}
        This property is straightforward by design.
    \end{proof}

    \vspace{0.1in}

    \noindent\underline{Claim 3:} There exists an annulus of modulus independent of $k$ that separates $\textnormal{CV}^P_{k,s}$ from $V^P_{k,s}$.

    \begin{proof}
        The rescaled pseudo-Siegel disk $\hat{Z}_k^{\#}(\threshold + s)$ is a quasidisk with dilatation independent of $k$.
        The claim follows from the property that the set of critical values of $(f_k^{\#})^{P_k^{\#}}$ along the boundary of $\hat{Z}_k^{\#}(\threshold + s)$ (in particular, not in the interior of wedges) are evenly spaced (cf. Theorem \ref{thm:R-APB}) independent of $k$.
    \end{proof}

    \vspace{0.1in}

    \noindent\underline{Claim 4:} There exists a sequence of positive real numbers $\{r_s\}_{s \in \N}$ increasing to $\infty$ such that for $P \in \Tbold_{\tttheta}$, for sufficiently high $k$, and for $s \in \{1,\ldots,k\}$, the conformal radius of $V^P_{k,s}$ at $0$ is bounded below by $r_s$.

    \begin{proof}
        The non-trivial part is to show that the lower bound $r_s$ tends to $\infty$ as $s \to \infty$. 
        Roughly, this amounts to saying that on a parabolic fjord, most of the critical values are close to the $\alpha$-fixed point.
        Such a property follows from \cite{DLy26a}.
    \end{proof}

    Pick $P \in \Tbold_{\tttheta}$ and $R \in \Kbold_{\tttheta}$ with $R<0$.
    There exists an integer $\mathbf{s}=\mathbf{s}(R) \geq 1$ such that for all $k \geq \mathbf{s}$, $V^P_{k,\mathbf{s}}$ contains the point $v^{\#}_{k,(P+R)_k^{\#}}$.
    For all $s \geq \mathbf{s}$, let $U^{P,R}_{k,s}$ be a disk neighborhood of $v^{\#}_{k,R_k^{\#}}$ such that 
        \begin{equation}
        \label{eqn:finite-branched-cov-map}
            (f_k^{\#})^{P_k^{\#}} : 
            \Big( U^{P,R}_{k,s}, v^{\#}_{k,R_k^{\#}} \Big)
            \to 
            \Big( V^P_{k,s}, v^{\#}_{k, (P+R)_k^{\#}} \Big)
        \end{equation}
    is a branched covering map of some degree $d_{P,R,s}$ independent of $k$ that grows to infinity as $s \to \infty$.
    \vspace{0.1in}

    \noindent\underline{Claim 5:} For sufficiently high $k$, the conformal radius of $ U^{P,R}_{k,s}$ at $v^{\#}_{k,R_k^{\#}}$ is bounded below by a positive constant independent of $k$.

    \begin{proof}
        This follows from local Koebe distortion control near $v^{\#}_{k,R_k^{\#}}$.
    \end{proof}

    The claims above guarantee that as $k \to \infty$, the map (\ref{eqn:finite-branched-cov-map}) converges in the Carath\'eodory topology to a degree $d_{P,R,s}$ branched covering map
        \[
            \Fbold^{P,R} : (U^{P,R}_s, C_R) \to (V^P_s, C_{P+R})
        \]
    The union of $V^{P}_s$ across all $s$ is the whole plane $\C$.
    Let $U^{P,R}$ be the union across all $s$ of the topological disk $U^{P,R}_s$.
    We have a $\sigma$-proper map
    \[
        \Fbold^{P,R} : (U^{P,R}, C_R) \to (V^P, C_{P+R}).
    \]
    
    \vspace{0.1in}
    
    \noindent\underline{Claim 6:} For $R_1, R_2 \in \Tbold_{\tttheta}$, if $U^{P,R_1}$ intersects $U^{P,R_2}$, then $U^{P,R_1} = U^{P,R_2}$ and $\Fbold^{P,R_1} = \Fbold^{P,R_2}$.

    \begin{proof}
        Suppose $U^{P,R_1}$ intersects $U^{P,R_2}$. 
        Then, for sufficiently high $s$ and $k$, $U^{P,R_1}_{k,s}$ intersects $U^{P,R_2}_{k,s}$, but since both of them are the preimage of the same domain $V^P_{k,s}$ under the same iterate of $f_k^{\#}$, it must be that $U^{P,R_1}_{k,s}$ is equal to $U^{P,R_2}_{k,s}$ and consequently $\Fbold^{P,R_1}$ is equal to $\Fbold^{P,R_2}$.
    \end{proof}
    
    The claim above allows us to take the union
        \[
            D(\Fbold^P) := \bigcup_{R <0} U^{P,R}.
        \]
    The union of $\Fbold^{P,R}$ across all $R$'s is a $\sigma$-proper map
        \[
            \Fbold^{P} : D(\Fbold^P) \to \C.
        \]
    To ensure that property (1)(b) is satisfied, we enlarge $D(\Fbold^P)$ as follows:
        \[
            \Dom(\Fbold^P) := \sum_{\substack{P_1 + \ldots + P_s = P, \\ P_j\textnormal{'s are generators of } \Tbold_{\tttheta}}} 
            \left( \Fbold^{P_1}|_{D(\Fbold^{P_1})} \right)^{-1} \circ 
            \ldots \circ
            \left( \Fbold^{P_s}|_{D(\Fbold^{P_s})} \right)^{-1}
            (\C).
        \]
    The commutative semigroup structure of $\Fbold = \big(\Fbold^P: \Dom(\Fbold^P) \to \C \big)_{P \in \Tbold_{\tttheta}}$ follows from the fact that they are limits of iterates of a single map $f_k^{\#}$ with the expected relations coming from the combinatorics.
    By now, all the desired properties have been established.
\end{proof}

\begin{remark}
    The open set $\Dom(\Fbold^P)$ is the natural maximal domain of definition of $\Fbold^P$.
    Indeed, we will see later in Theorem \ref{thm:no-wandering-domains} that for all time $P$, $\Dom(\Fbold^P)$ is a dense subset of the plane.
\end{remark}

\begin{remark}
    If $\tttheta$ is irrational, then $\Dom(\Fbold^P)$ is connected for all $P$.
    Else, $\Dom(\Fbold^P)$ will be disconnected for some time $P$.
\end{remark}

The semigroup $\Fbold$ in the theorem above is called a (normalized) \emph{neutral cascade}.
It comes with an element $\tttheta = \seq{ (\varepsilon_n, \abar_n) }_{n\in\Z}$ of $\TheBiCpt$, called the \emph{combinatorics of $\Fbold$}, such that $\Fbold$ is parametrized by the time semigroup of $\tttheta$.
Often, we will denote 
\[
    \Kbold = \Kbold_{\tttheta}, \qquad 
    \Tbold = \Tbold_{\tttheta}, \qquad
    \Fbold^{[m]} := \Fbold^{Q_{[m]}} 
\]
for simplicity.
The sequence $(\theta_k, n_k, u_k)$ that leads to Lemma \ref{lem:combinatorial-convergence}, Lemma \ref{lem:scaling-critical-value}, and Theorem \ref{thm:transcendental-extension} will be called a \emph{generating sequence} for $\Fbold$.
When such a sequence is fixed, we will often denote $q^k_{[l]} = q_{[l]}(\theta_k)$ and use the hash notation $f_k^{\#}$ and $v^{\#}_{k,l}$ as introduced in (\ref{eqn:hash-notation}).

\begin{lemma}
    \label{lem:crit-points}
    Let $\Fbold = (\Fbold^P)_{P \in \Tbold}$ be a neutral cascade.
    For every $P \in \Tbold$, 
    \begin{enumerate}
        \item a point $w \in \C$ is a critical point of $\Fbold^P$ if and only if there exist $R_1 \in \Tbold$ and $R_2 \in \Tbold \cup \{0\}$ with $R_1 + R_2\leq P$ such that $\Fbold^{R_2}(w) = C_{-R_1}$;
        \item a point $w \in \C$ is a critical value of $\Fbold^P$ if and only if $w = C_{Q}$ where $0 \leq Q < P$.
    \end{enumerate}
\end{lemma}

\begin{proof}
    Item (2) follows from item (1), so below, we just need to prove (1).
    Let us use the notation from the proof of Theorem \ref{thm:transcendental-extension}.
    Suppose first that $w$ is contained in $D(\Fbold^P)$.
    If $w$ is equal to $C_{-R_1}$ for any $R_1 \in \Tbold$ with $R_1< P$, then we can set $R_2=0$ and we are done. 
    Suppose otherwise.
    
    There exist $R \in \Tbold$ and $\mathbf{s} \in \N$ such that $w$ is contained in $U_{\mathbf{s}}^{P,R}$.
    We will consider the finite collection of times
    \[
        \mathscr{T}_{P,R,\mathbf{s}} = 
        \left\{ 
        S_1-S_2 \in \Tbold \: : \: S_1, S_2 \in \Tbold, \: S_1>S_2, \: \text{and } C_{S_1} \textnormal{ and } C_{S_2} \text{ are in } U_{\mathbf{s}}^{P,R} 
        \right\}. 
    \]
    For all sufficiently high $k \in \N$, there exists a critical point $w_k$ of $(f_k^{\#})^{P_k^{\#}}$ in $U_{k,\mathbf{s}}^{P,R}$ such that $w_k \to w$ as $k \to \infty$.
    There exists $M \in \Tbold \cup \{0\}$ with $M < P$ such that $\Fbold^P(w) = C_M$ and 
    $(f_k^{\#})^{P_k^{\#}}(w_k) = v^{\#}_{k,M_k^{\#}}$ 
    for all sufficiently high $k$.
    The point $w_k$ is contained in, and not the root of, a unique hoglet of $f_k^{\#}$. 
    The generation of such a hoglet is of the form $N_k^{\#}$ for some element $N$ of $\mathscr{T}_{P,R,\mathbf{s}}$ independent of $k$, and it must satisfy $M+N<P$.
    The image of $w_k$ under $(f_k^{\#})^{N_k^{\#}}$ has to be the critical point $v^{\#}_{k,-(P-M-N)_k^{\#}}$. 
    As we take $k \to \infty$, we obtain that $\Fbold^{N}(w) = C_{-(P-M-N)}$.
    Then, we are done if we set $R_2 = N$ and $R_1 = P-M-N$.

    Lastly, if $w$ is not contained in $\Dom(\Fbold^P) \backslash D(\Fbold^P)$, then $P$ can be decomposed into a sum $P_1+\ldots + P_m$ such that $w$ is contained in 
    \[
        \left(\Fbold^{P_1}|_{D(\Fbold^{P_1})}\right)^{-1}  \circ \ldots \circ  \left( \Fbold^{P_{m-1}}|_{D(\Fbold^{P_{m-1}})} \right)^{-1} 
        \left( D(\Fbold^{P_m}) \right).
    \]
    Let $y_0 := w$ and for $j \in \{1,\ldots, m\}$, denote $y_j := \Fbold^{P_1+\ldots+P_j}(y_0)$.
    Then, there exists a unique $j_* \in \{1\ldots,m\}$ such that $y_{j_*-1}$ is a critical point of $\Fbold^{P_{j_*}}$.
    Apply the argument in the second paragraph to $\Fbold^{P_{j_*}}: y_{j_*-1} \mapsto y_{j_*}$, and we are done.
\end{proof}

Mimicking the Fatou-Julia theory in classical holomorphic dynamics, we introduce global dynamical sets of neutral cascades as follows.

\begin{definition}\label{dfn:escaping_set}
    Consider a neutral cascade $\Fbold = (\Fbold^P)_{P \in \Tbold}$. For every $P \in \Tbold$, define the time $P$ \emph{escaping set} of $\Fbold$ to be
\[
    \Ibold^{\leq P} := \C \,\backslash\, \textnormal{Dom} \left(\Fbold^P\right).
\]
    The union is called the \emph{finite-time escaping set} of $\Fbold$:
\[
    \Ibold^{<\infty} := \bigcup_{P \in\ \Tbold} \Ibold^{\leq P}.
\]
    We define the \emph{Fatou set} $\mathfrak{F}(\Fbold)$ of $\Fbold$ to be the set of points in $\C \backslash \Ibold^{<\infty}$ near which $\left\{ \Fbold^P\right\}_{P \in \Tbold}$ is a normal family. 
    We define the \emph{Julia set} to be the complement of the Fatou set: $\mathfrak{J}(\Fbold) = \C \backslash \mathfrak{F}(\Fbold)$; it contains $\Ibold^{<\infty}$.
\end{definition}

By Theorem \ref{thm:transcendental-extension} (1)(b), the escaping sets are nested: for $P,Q \in \Tbold$,
\[
    P < Q \quad \xRightarrow{\qquad} \quad \Ibold^{\leq P} \subset \Ibold^{\leq Q}.
\]

In order to study the fine dynamical properties of neutral cascades, we will need to take into account sector renormalization.

\subsection{Cascades for bi-infinite towers}\label{ss:Casc:Towers}
\label{ss:cascade-tower}

Let $\attr$ be the space of all bi-infinite sector renormalization orbits 
\[
    \attr := \left\{ \seq{\ldots, f_{-2},f_{-1},f_0; f_1,f_2,\ldots} \: : \:
    \seq{f_{n+m}}_{n \geq 0} \in \overline{\towsec} \text{ for all } m \in \Z
    \right\}.
\]
We equip $\attr$ with the topology where 
\begin{itemize}
    \item the renormalization operator $\Rsec: \attr \to \attr$, which simply acts as the shift map, is a self-homeomorphism, and
    \item the projection map
    \[
    \textnormal{Proj}: \attr \to \overline{\towsec}, \quad \seq{\ldots, f_{-2}, f_{-1}, f_0; f_1, f_2 \ldots} \mapsto \seq{f_n}_{n\geq 0}.
    \]
    is continuous.
\end{itemize}

Every element $\fboldbar = \seq{f_n}_{n \in \Z}$ of $\attr$ comes with some combinatorics 
\[
\tttheta\big(\fboldbar\big)=\seq{ (\varepsilon_n, \abar_n) }_{n \in \Z} \in \TheBiCpt
\]
such that the following diagrams commute.
\[
        \begin{tikzcd}[column sep=large, row sep=large]
	\attr & \attr && \attr & \attr \\
	\overline{\towsec} & \overline{\towsec} && \TheBiCpt & \TheBiCpt 
	\arrow["\textnormal{Proj}"', from=1-1, to=2-1]
	\arrow["\textnormal{Proj}", from=1-2, to=2-2]
	\arrow["\Rsec", from=1-1, to=1-2]
	\arrow["\Rsec"', from=2-1, to=2-2]
	\arrow["\tttheta"', from=1-4, to=2-4]
	\arrow["\tttheta", from=1-5, to=2-5]
	\arrow["\Rsec", from=1-4, to=1-5]
	\arrow["\shift"', from=2-4, to=2-5]
        \end{tikzcd}
\]

For the rest of the section, we will fix a bi-infinite tower 
\[
\fboldbar = \seq{f_n}_{n \in \Z} \in \attr
\]
and denote 
\[
\tttheta = \tttheta(\fboldbar) = \seq{ (\varepsilon_n, \abar_n)}_{n\in\Z} \in \TheBiCpt.
\]
Let
$\Kbold = \Kbold_{\tttheta}$ be the continuant group of $\tttheta$, 
and $\Tbold = \Tbold_{\tttheta}$ be the time semigroup of $\tttheta$.
Let $\{Q_{[n]}\}_{n\in\Z}$ be the return times of $\tttheta$ generating $\Kbold$.
We will also denote
\[
    \check{Q}_{[n]} := Q_{[n]} - \varepsilon_n \varepsilon_{n+1} Q_{[n-1]} \qquad \text{ for } n \in \Z.
\]
For every $n \in \Z$, we denote 
\[
    \fbold_n = \seq{f_{k+n}}_{k \geq 0} \in \overline{\towsec}
\]
We will use the notation from \S\ref{sss:notation-for-towsec}, namely $\ttheta_n = \ttheta(\fbold_n)$, $\Kont^{n}$, $\{ \qq^n_{[m]} \}_{m \geq 0}$, $\Time^n$, $\mathcal{F}_n = (f_n^p)_{p \in \Time^n}$, $\{ \hat{Z}_n^m \}_{m \geq -1}$, $H_n$, $\{ v^n_p \}_{p \in \Time^{n}}$, $P_n$, $\{S_n^j\}_{j \geq 1}$, and $\psi_n: S_n^1 \to \D \backslash \gamma_{n+1}$.
Here, $\gamma_{n+1}$ is the slit associated to the gluing map $\psi_n$.
We will also denote
\[
    \check{\qq}^n_{[m]} := \qq^n_{[m]} - \varepsilon_{n+m} \varepsilon_{n+m+1} \qq^n_{[m-1]}.
\]

Recall from Definition \ref{def:sector} the notion of sectors. 
We say that a sector $S$ is an \emph{infinite sector} if its vertex is at $\infty$.

\begin{proposition}
\label{prop:trans-sectors}
    Given a bi-infinite tower $\fboldbar = \seq{ f_n }_{n \in \Z}$ as above, there exist a neutral cascade $\Fbold = (\Fbold^P)_{P \in \Tbold}$ with combinatorics $\tttheta$, a nest of sectors 
    \[
    \ldots \supset \Sbold^{-2} \supset \Sbold^{-1} \supset \Sbold^0 \supset \Sbold^1 \supset \Sbold^2 \supset \ldots,
    \]
    all containing $0$, together with a bi-infinite sequence of conformal gluing maps 
    \[
    \pphi_n: (\Sbold^n,0) \to (\D, v_0(f_n))
    \]
    with slit $\gamma_n$ having the following properties. 
    \begin{enumerate}
        \item The conformal radius of $\Sbold^m$ about $0$ satisfies
    \[
        \textnormal{rad}(\Sbold^m,0) \asymp |\Fbold^{[m]}(0)|,
    \]
        and the union $\cup_n \Sbold^n$ is the whole complex plane $\C$.
        \item Let $(\theta_k, n_k, u_k)$ be a generating sequence for $\Fbold$. 
        For all $m \in \Z$, the rescaled sector $u_k(S^{n_k+m}(f_{\theta_k}))$ together with the marked point at the origin converges in the Carath\'eodory topology to $(\Sbold^m, 0)$ as $k \to \infty$.
        \item For $n \in \Z$, $\psi_n \circ \pphi_n = \pphi_{n+1}$ on $\Sbold^{n+1}$.
        \item For $n \in \Z$, $\Sbold^n$ is an infinite sector and its sides are rays of the form $\boldsymbol{A}^n \subset \Dom(\Fbold^{[n-1]})$ and $\Fbold^{[n-1]}(\boldsymbol{A}^n)$, and the gluing map $\pphi_n$ identifies points of the form $z$ and $\Fbold^{[n-1]}(z)$.
        \item The nested intersection $\cap_{n} \overline{\Sbold^n}$ is a uniformly quasiconformal infinite ray $I_0$ emanating from $0$.
        \item For $n \in \Z$ and $m \geq 1$, the first return map $\big( f_n^{\qq^n_{[m]}}, f_n^{\check{\qq}^n_{[m]}} \big)$ of the semigroup $\mathcal{F}_n$ back to $S_n^m$ lifts under $\pphi_n$ to the pair of maps $(\Fbold^{Q_{[n+m]}}, \Fbold^{\check{Q}_{[n+m]}})$ acting on $\Sbold^{n+m}$.
    \end{enumerate}
\end{proposition}

See Figure \ref{fig:transfer-to-cascade}.

\begin{figure}
        \centering
       \begin{tikzpicture}
    \node[anchor=south west,inner sep=0] (image) at (0,0) {\includegraphics[width=1\linewidth]{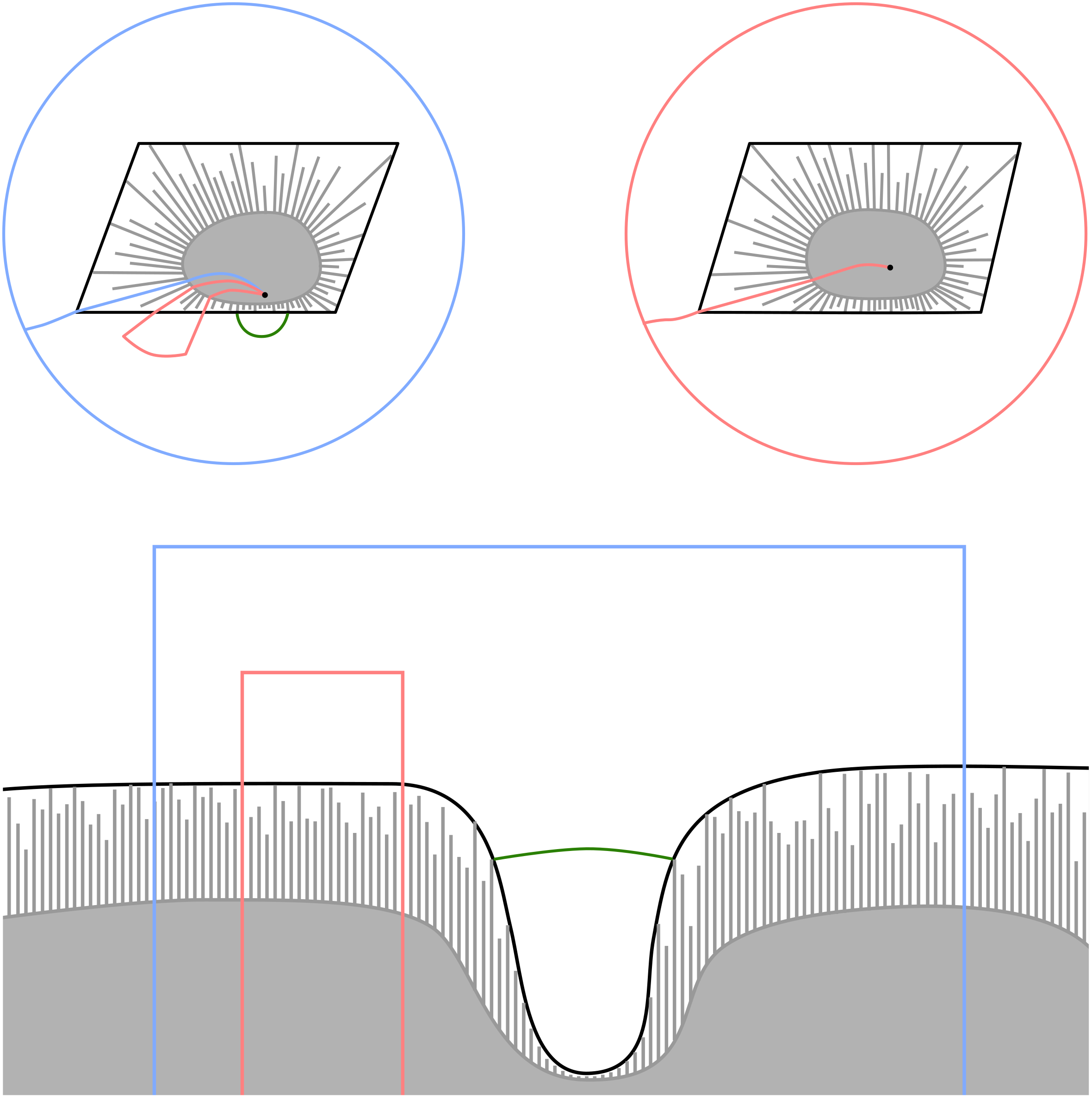}};
    \begin{scope}[
        x={(image.south east)},
        y={(image.north west)}
    ]
        \node [black] at (0.162,0.715) {\tiny $\bullet$};
        \node [black] at (0.725,0.715) {\tiny $\bullet$};
        \draw [black,-latex] (0.32,0.67) .. controls (0.3,0.6) and (0.42,0.62) .. (0.36,0.70);
        \node [black] at (0.37,0.605) {\small $f_n$};
        \draw [black,-latex] (0.9,0.67) .. controls (0.88,0.6) and (1,0.62) .. (0.94,0.70);
        \node [black] at (0.95,0.605) {\small $f_{n+1}$};
    
        \draw[red!60!white,-latex] (0.44,0.75) -- (0.56,0.75);
        \draw[red!60!white,-latex] (0.39,0.375) -- (0.65,0.6);
        \draw[blue!75!white,-latex] (0.175,0.45) .. controls (0.1,0.45) and (0.1,0.5) .. (0.1,0.6); 
        \node [blue!80!white] at (0.08,0.5) {$\pphi_{n}$};
        \node [red!80!white] at (0.54,0.55) {$\pphi_{n+1}$};
        \node [red!80!white] at (0.5,0.775) {$\psi_{n}$};
        \node [red!80!white] at (0.335,0.365) {\small $\Sbold^{n+1}$};
        \node [blue!70!white] at (0.83,0.46) { $\Sbold^{n}$};
        \node [black] at (0.276,0.286) {\footnotesize $\bullet$};
        \node [black] at (0.277,0.31) {\small $0$};

        \draw [black, -latex] (0.23,0.13) .. controls (0.26,0.11) and (0.34,0.11) .. (0.37,0.13);
        \draw [black, -latex] (0.16,0.07) .. controls (0.25,0) and (0.4,-0.01) .. (0.51,-0.01) .. controls (0.61,-0.01) and (0.77,0) .. (0.86,0.07);
        \node [black] at (0.3,0.095) {\small $\Fbold^{[n]}$};
        \node [black] at (0.51,-0.03) {\small $\Fbold^{[n-1]}$};

        \node[white!25!black] at (0.25,0.775) {\small $H_n$};
        \node[white!25!black] at (0.81,0.78) {\small $H_{n+1}$};
        \node[black] at (0.325,0.89) {\small $\hat{Z}_n^{[0]}$};
        \node[black] at (0.875,0.89) {\small $\hat{Z}_{n+1}^{[-1]}$};
        \node[green!50!black] at (0.25,0.67) {\small $\hat{Z}_n^{[-1]}$};
        \node [black] at (0.04,0.315) {$\hat{\Zbold}^{[n]}$};
        \node [green!50!black] at (0.55,0.26) {$\hat{\Zbold}^{[n-1]}$};
        \node [white!25!black] at (0.75,0.1) {$\Hbold$};
    \end{scope}
\end{tikzpicture}
    \caption{The gluing maps $\phi_n$ bridge the dynamical plane of a neutral cascade $\Fbold$ with the dynamical planes of a bi-infinite tower of sector renormalizations $\seq{f_n}_{n\in\Z}$.}
    \label{fig:transfer-to-cascade}
\end{figure}
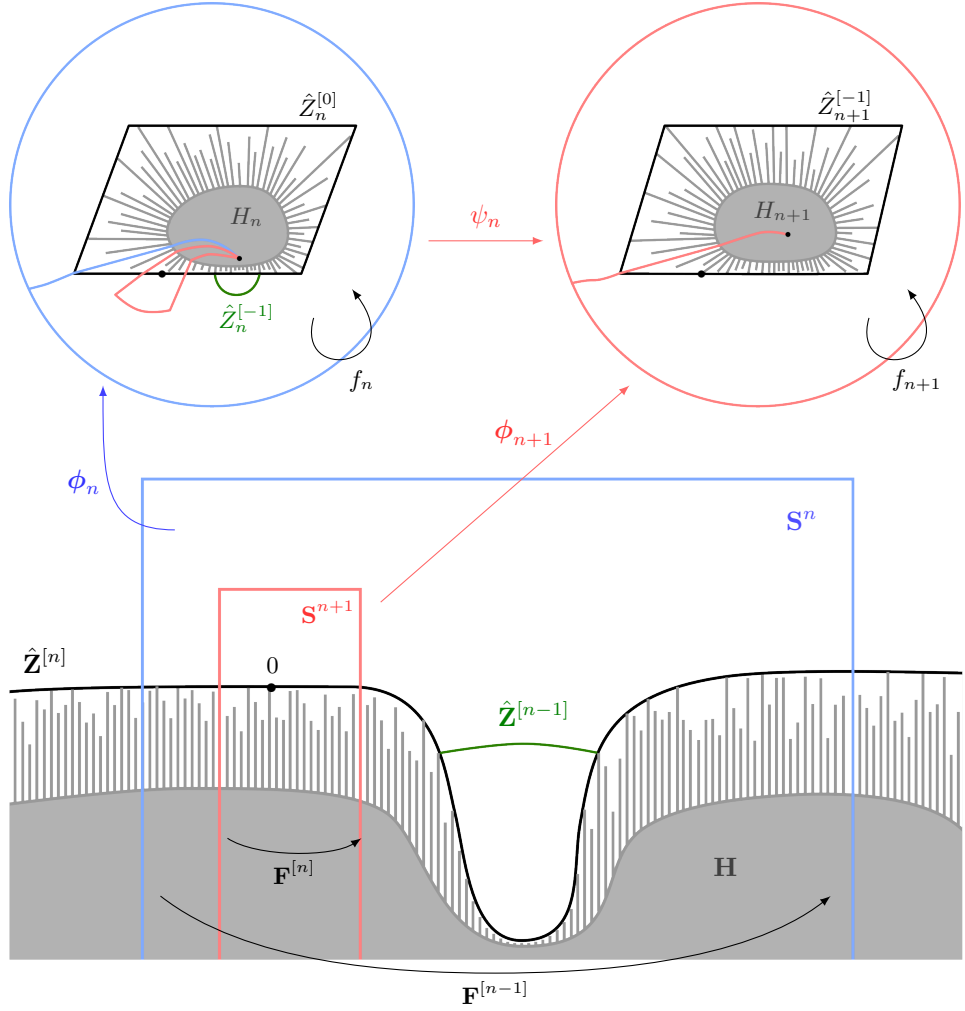

\begin{proof}
    There exist a sequence of irrationals $\theta_k \in \Irrat$ and an increasing sequence of integers $n_k \to \infty$ such that for all $m \in \Z$ and all sufficiently high $k$,
    \[
    g_{k,m} := \Rsec^{n_k+m}(f_{\theta_k}) \to f_m 
    \quad \text{ as } k \to \infty.
    \]
    For every $k \geq 1$, let $u_k$ be the affine map from Theorem \ref{thm:transcendental-extension} (2)(d). 
    According to Theorem \ref{thm:transcendental-extension}, after passing to a subsequence, $(\theta_k, n_k, u_k)$ is a generating subsequence for a neutral cascade $\Fbold = (\Fbold^P)_{P \in \Tbold}$ with combinatorics $\tttheta$.
    In particular,
    for every $m \in \Z$, the polynomial map
    \[
        \Fbold_{k}^{[m]} := (f_k^{\#})^{[n_k+m]}
    \]
    converges to the $\sigma$-proper map $\Fbold^{[m]}$ as $k \to \infty$.
    Let $\{C_P\}_{P \in \Kbold}$ denote the critical orbit of $\Fbold$ from Theorem \ref{thm:transcendental-extension}.
    For every $m \in \Z$ and every sufficiently high $k \geq 1$, we will also denote the rescaling of the renormalization sector $S^{n_k+m}(f_{\theta_k})$ by
    \[
        \Sbold_k^m := u_k(S^{n_k+m}(f_{\theta_k})).
    \]

    \noindent \underline{Claim 1:} 
    For every $m \in \Z$, after passing to subsequence, the pointed domains $(\Sbold_k^m,0)$ converge as $k \to \infty$ to a pointed domain $(\Sbold^m,0)$ in the Carath\'eodory topology.
    Moreover, there is a universal constant $K>1$ such that the conformal radius of $\Sbold^m$ about $0$ satisfies
    \[
        K^{-1} |C_{Q_{[m]}}| \leq \text{rad}(\Sbold^m,0) \leq K |C_{Q_{[m]}}|.
    \]
    
    \begin{proof}
        For integers $k, l \geq 0$, denote $q^k_{[l]} = q_{[l]}(\theta_k)$. 
        By Lemma \ref{lem:size-of-sectors}, there exist universal constants $1< \lambda_1 <\lambda_2$ and $K > 1$ such that the conformal radius $\textnormal{rad}(\Sbold_k^m,0)$ of $\Sbold_k^m$ about $0$ satisfies
    \[
        K^{-1} 
        \left| v_{k, q^k_{[n_k+m]}}^{\#} \right|
        \leq \textnormal{rad}(\Sbold_k^m,0) \leq 
        K \left| v_{k, q^k_{[n_k+m]}}^{\#} \right|.
    \]
        By Corollary \ref{cor:universal-scaling-law}, this implies that the conformal radius of $(\Sbold_k^m,0)$ depends only on $m$.
        This implies the desired claim as $k \to \infty$.
    \end{proof}
    
    Together with Lemma \ref{lem:scaling-critical-value}, Claim 1 implies that the conformal radius of $(\Sbold^m,0)$ increases exponentially to infinity as $m \to -\infty$ . This implies property (1).
    
    Next, for all $m \in \Z$ and all sufficiently high $k \geq 1$, denote by $\varphi_{k,m}$ the conformal gluing map for the sector $S^{n_k+m}(f_{\theta_k})$ with some slit $\gamma_{k,n_k+m}$ which projects the return map $f_{\theta_k}^{[n_k+m]}$ (mod $f_{\theta_k}^{[n_k+m-1]}$) to $g_{k,m}$. 
    These come from Theorem \ref{thm:sectorial-bounds} (3).
    As $k \to \infty$, the slit $\gamma_{k,n_k+m}$ converges to the slit $\gamma_{m}$.
    By Claim 1, the conformal maps 
    \[
        \pphi_{k,m} := \varphi_{k,m} \circ u_k^{-1}: (\Sbold_k^m,0) \to \left(\D \backslash \gamma_{k,n_k+m}, v_0(g_{k,m}) \right)
    \]
    converge as $k\to \infty$ to a conformal map 
    \[
        \pphi_m: \left( \Sbold^m,0 \right) \to \left( \D \backslash \gamma_m, v_{m,0} \right) 
    \]
    in the Carath\'eodory topology.
    Hence, (2) holds.
    The maps $\pphi_m$, $m \in \Z$ satisfy property (3) because the conformal gluing map $\varphi_{k,m+1} \circ \varphi_{k,m}^{-1}$ for the first renormalization sector of $g_{k,m}$ converges to the conformal gluing map $\psi_m$ for $f_m$.
    \vspace{0.1in}
    
    \noindent \underline{Claim 2:} For $m \in \Z$, the domain $\Sbold^m$ satisfies property (4).

    \begin{proof}
        For $l \geq 1$, the $l$\textsuperscript{th} renormalization sector $S_m^l$ of $f_m$ is contained in $\D \backslash \gamma_{m}$ and its boundary touches the slit $\gamma_{m}$ precisely at its vertex at the origin.
        The gluing map $\psi_{m}$ for $S_m^1$ sends $S_m^L$ to $\D\backslash \gamma_{m+1}$ if $l=1$ and to $S_{m+1}^{l-1}$ if $l\geq 2$.
        Let $A_m^l$ and $\check{A}_m^l = f_{m}^{[l-1]}(A_m^l)$ be the sides of $S_m^l$, and let $B_m^l$ be the summit.
        Observe that 
        \[
        \pphi_{m-1}(\Sbold^m) = S_{m-1}^1
        \]
        and $\Sbold_m$ is bounded by arcs 
        \[
        \boldsymbol{A}^{m} := \pphi_{m-1}^{-1}(A_{m-1}^1), \quad 
        \check{\boldsymbol{A}}^{m} := \pphi_{m-1}^{-1}(\check{A}_{m-1}^1), \quad \text{and} \quad 
        \boldsymbol{B}^{m} := \pphi_{m-1}^{-1}(B_{m-1}^1).
        \]
        For all $l \geq 1$, we have 
        \[
        \boldsymbol{A}^m = \pphi_{m-l}^{-1}(A_{m-l}^l), \quad 
        \check{\boldsymbol{A}}^m = \pphi_{m-l}^{-1}(\check{A}_{m-l}^l),
        \quad \text{and} \quad 
        \boldsymbol{B}^m = \pphi_{m-l}^{-1}(B_{m-l}^l).
        \]

        For all $m \in \Z$ and $l \geq 1$, since the arcs $A_m^l$ and $\check{A}_m^l$ together land at $0$ on the boundary of $\D \backslash \gamma_{m}$, then the arcs $\boldsymbol{A}^{m+l}$ and $\check{\boldsymbol{A}}^{m+l}$ all accumulate at a unique point $x_m$ on the Carath\'eodory boundary of $\Sbold^{m}$.
        This point $x_m$ represents a compact set $X_m \subset \RS$ at which $\boldsymbol{A}^{m+l}$ accumulates for all $l\geq 1$.
        As we take $m \to -\infty$, we deduce that there is a compact set $X \subset \RS$ such that $X_m = X$ for all $m$.
        Since $X$ has to be disjoint from $\cup_m \Sbold^m$, then by item (1), we must have $X = \{\infty\}$.
        Hence, $\Sbold^m$ is indeed an infinite sector with sides $\boldsymbol{A}^m$ and $\check{\boldsymbol{A}}^m$.
        
        It remains to show that $\check{\boldsymbol{A}}^{m} = \Fbold^{[m-1]}(\boldsymbol{A}^m)$ for all $m \in \Z$.
        For every $k$, the sides of the sector $\Sbold_k^m$ are of the form $\boldsymbol{A}_k^m$ and $\Fbold_{k}^{[m-1]}(\boldsymbol{A}_k^m)$.
        As $k \to \infty$, the arcs $\boldsymbol{A}_k^m$, $\Fbold_{k}^{[m-1]}(\boldsymbol{A}_k^m)$, and the summit $\boldsymbol{B}_k^m$ of $\Sbold_k^m$ converge in the Hausdorff metric to $\boldsymbol{A}^m$, $\check{\boldsymbol{A}}^m$, and $\boldsymbol{B}^m$ respectively.
        We have 
        \[
            \Fbold_{k}^{[m-1]} \circ \pphi_{k,m-1}^{-1}(z) = \pphi_{k,m-1}^{-1} \circ g_{k,m-1}(z) \qquad \text{ for } z \in \pphi_{k,m-1}(\boldsymbol{A}_k^m).
        \]
        As $k \to \infty$, the equation above implies that
        \[
            \Fbold^{[m-1]}  \circ \pphi_{m-1}^{-1}(z) = \pphi_{m-1}^{-1} \circ f_{m-1}(z) \qquad \text{ for } z \in \boldsymbol{A}^m.
        \]
        Hence, $\check{\boldsymbol{A}}^m = \Fbold^{[m-1]}(\boldsymbol{A}^m)$.
    \end{proof}

    For every $m$, the nested intersection $\cap_n S_m^n$ is equal to the internal ray of $H_m$, an arc joining the fixed point $0$ and the critical value $v_{m,0}$ of $f_m$.
    As we lift under the map $\pphi_m$, we deduce that the nested intersection $\cap_n \Sbold^n$ is an infinite ray $I_0$ emanating from $0$.
    Observe that $I_0$ is equal to the Hausdorff limit as $k \to \infty$ of the nested intersection $\cap_{m \geq 1} \Sbold_k^m$, which is equal to the uniformly quasiconformal internal ray of the Mother Hedgehog of the quadratic polynomial $f_k^{\#}$ landing at its critical value at $0$.
    Hence, $I_0$ is a uniformly quasiconformal ray, which proves property (5).

    For all $m \in \Z$ and all sufficiently high $k$, 
    the map $g_{k,m}$ lifts under $\pphi_{k,m}$ to the map
    $\Fbold_k^{[m]}$ (mod $\Fbold_k^{[m-1]}$) acting on $\pphi_{k,m}^{-1}(\text{Dom}(g_{k,m})) \subset \Sbold_k^m$.
    As we take $k \to \infty$, we conclude that the map $f_m$ lifts under $\pphi_m$ to the map $\Fbold^{[m]}$ acting on $\pphi_m^{-1}(\text{Dom}(g_{m})) \subset \Sbold^m$.
    This implies property (6), and we are done.
\end{proof}

For the rest of this section, we will fix the neutral cascade $\Fbold = (\Fbold^P)_{P \in \Tbold}$, the infinite sectors $\Sbold^n$, and the conformal gluing maps $\pphi_n$ described in the proposition above.
We will denote by $\{Q_{[n]}\}_{n\in \Z}$ the standard generators of the continuant group $\Kbold$ and by $\{C_P\}_{P \in \Tbold}$ the bi-infinite critical orbit described in Theorem \ref{thm:transcendental-extension}.
The ray $I_0$ described in property (5) will be called a \emph{valuable internal ray} of $\Fbold$.

\subsection{Pseudo-Siegel half-planes}
\label{ss:wZbounds for halfplanes}

\begin{proposition}
\label{prop:trans-pseudo-siegel}
    There exists a sequence 
    \[
    \{\hat{\Zbold}^{[m]} = \hat{\Zbold}^{[m]}(\threshold)\}_{m\in\Z}
    \]
    of closed subsets of the plane with the following properties.
    \begin{enumerate}
        \item For every $m \in \Z$, 
        \begin{itemize}
            \item if $\abar_{m+1}< \threshold$, then $\hat{\Zbold}^{[m]} = \hat{\Zbold}^{[m-1]}$;
            \item if $\abar_{m+1} \geq\threshold$, then $\hat{\Zbold}^{[m]} \subsetneq \hat{\Zbold}^{[m-1]}$.
        \end{itemize}
        \item For every $m,n \in \Z$ with $m \geq -1$, the image of $\hat{\Zbold}^{[m+n]} \cap \Sbold^n$ under $\pphi_n$ is the level $m$ pseudo-Siegel pinched disk $\hat{Z}_n^{[m]}$ of $f_n$ with combinatorial threshold $\threshold$ with the slit $\gamma_n$ removed.
        \item Every connected component of $\hat{\Zbold}^{[m]}$ is unbounded and the complement $\C \backslash \hat{\Zbold}^{[m]}$ is connected, simply-connected, and unbounded.
        \item The union $\hat{\Zbold} = \cup_m \hat{\Zbold}^{[m]}$ is $K$-quasiconformally equivalent to the closed upper half plane for some $\threshold$-uniform constant $K>1$.
        \item Let $(\theta_k, n_k, u_k)$ be a generating sequence for $\Fbold$.
        For all $m \in \Z$, the rescaled pseudo-Siegel disk $u_k(\hat{Z}^{[n_k+m]}(f_{\theta_k}))$ of $f_{\theta_k}$ converges to $\hat{\Zbold}^{[m]} \cup \{\infty\}$ as $k \to \infty$ in the Hausdorff topology on compact connected subsets of $\RS$.
        \item For every $m \in \Z$, $\hat{\Zbold}^{[m-1]}$ is contained in $\Dom(\Fbold^{[m]})$ and $\Fbold^{[m]}$ is injective on $\hat{\Zbold}^{[m-1]}$.
    \end{enumerate}
\end{proposition}

See Figure \ref{fig:transfer-to-cascade} for an illustration of $\hat{\Zbold}$.

\begin{definition}
    The set $\hat{\Zbold}^{[m]} = \hat{\Zbold}^{[m]}(\threshold)$ above is called the \emph{level $m$ pseudo-Siegel pinched half-plane} of $\Fbold$ with \emph{combinatorial threshold} $\threshold \in \N$.
    The union $\hat{\Zbold}$ is called the \emph{top pseudo-Siegel half-plane} of $\Fbold$.
\end{definition}

\begin{proof}  
    The construction of $\hat{\Zbold}^{[m]}$, $m \in \Z$ is as follows.
    For every $n \in \Z$ with $n \leq m+1$, we have
    \[
    \psi_{n-1}(\hat{Z}_{n-1}^{[m-n+1]} \cap S_{n-1}^1) = \hat{Z}_{n}^{[m-n]} \backslash \gamma_n.
    \]
    This implies
    \[
        \pphi^{-1}_n \big( \hat{Z}_n^{[m-n]} \backslash \gamma_n \big) = 
        \pphi^{-1}_{n-1} \big( \hat{Z}_n^{[m-n+1]} \backslash \gamma_{n-1} \big) \cap \Sbold^n,
    \]
    and as a consequence, 
    \[
    \pphi_{m+1}^{-1} \big( \hat{Z}_{m+1}^{[-1]} \backslash \gamma_{m+1} \big) \, \subset \, 
    \pphi_{m}^{-1} \big( \hat{Z}_{m}^{[0]} \backslash \gamma_{m} \big) \, \subset \, 
    \pphi_{m-1}^{-1} \big( \hat{Z}_{m-1}^{[1]} \backslash \gamma_{m-1} \big) \, \subset \, \ldots.
    \]
    We set $\hat{\Zbold}^{[m]}$ to be the union of these sets:
    \[
        \hat{\Zbold}^{[m]} := \bigcup_{n \leq m+1} \pphi^{-1}_n \big( \hat{Z}_n^{[m-n]} \backslash \gamma_n \big).
    \]
    Properties (1) and (2) are now immediate.

    Let us prove property (3).
    Recall that each $\pphi_n$ sends the vertex of $\Sbold^n$ at infinity to the fixed point $0$ of $f_n$.
    Every connected component of $\hat{\Zbold}^{[m]}$ is unbounded because whenever $n \leq m+1$, every connected component of $\hat{Z}_n^{[m-n]} \backslash \gamma_n$ contains $0$ on its boundary.
    Moreover, observe that $\C \backslash \hat{\Zbold}^{[m]}$ is equal to
    \[
        \C \backslash \hat{\Zbold}^{[m]} = \bigcup_{n \leq m+1}  \pphi^{-1}_n \left( \D \backslash ( \hat{Z}_n^{[m-n]} \cup \gamma_n) \right)
    \]
    and each of the domains $\D \backslash ( \hat{Z}_n^{[m-n]} \cup \gamma_n)$ is an open topological disk.
    It remains to show that $\C \backslash \hat{\Zbold}$ is unbounded.
    To do this, we will construct an unbounded topological ray $\mathfrak{R}$ in $\C \backslash \hat{\Zbold}^{[m]}$ that is a concatenation of bounded arcs $\mathfrak{R}_0$, $\mathfrak{R}_{-1}$, $\mathfrak{R}_{-2}$, $\ldots$.
    For $n \leq 1$, pick a point $w_n$ on $\partial \D$ disjoint from $\gamma_n$ and let $w'_{n-1} = \psi_{n-1}^{-1}(w_n)$.
    For $n \leq 0$, let $\mathfrak{R}_n$ to be a proper arc in $\Sbold^{n} \backslash \big( \Sbold^{n+1} \cup \hat{\Zbold} \big)$ such that $\pphi_n(\mathfrak{R}_n)$ starts from $w'_{n}$ and ends at $w_n$.
    By Proposition \ref{prop:trans-sectors} (1), $\pphi_n^{-1}(w_n) \to \infty$ as $n \to - \infty$, so then the concatenation $\mathfrak{R}$ is indeed an unbounded ray avoiding $\hat{\Zbold}$.

    Let $(\theta_k, n_k, u_k)$ be a generating sequence for $\Fbold$.
    For $m \in \Z$, $j \geq -1$, and sufficiently high $k$, denote
    \[
        A^{[j]}_{k,m} = u_k \left( \partial \hat{Z}^{[j]}(f_{\theta_k}) \cap S^{n_k+m}(f_{\theta_k}) \right).
    \]
    Observe that $A^{[j]}_{k,m} = A^{[-1]}_{k,m}$ whenever $j \leq n_k + m -1$. 
    There exists an $\threshold$-uniform $K > 1$ such that $A^{[-1]}_{k,m}$ is a $K$-quasiarc.
    Let us use the notation $g_{k,m} = \Rsec^{[n_k+m]}f_{\theta_k}$ and $\pphi_{k,m}$ from the proof of Proposition \ref{prop:trans-sectors}.
    Then, $\pphi_{k,m}^{-1}(A^{[n_k + m + M]}_{k,m})$ is equal to the boundary of the pseudo-Siegel pinched disk $\hat{Z}^{[M]}(g_{k,m})$ of $g_{k,m}$.
    As $k \to \infty$, $\hat{Z}^{[M]}(g_{k,m})$ converges to the level $M$ pseudo-Siegel pinched disk $\hat{Z}^{[M]}_m$ of $f_m$.
    Hence, as $k \to \infty$, the quasiarc $A^{[-1]}_{k,m}$ indeed converges to the $K$-quasiarc $A_m = \partial \hat{\Zbold} \cap \Sbold^m$.

    By Lemma \ref{lem:scaling-critical-value} and Proposition \ref{prop:trans-sectors} (1), the distance between $\{0\}$ and the endpoints of $A_m$ grows uniformly exponentially as $m \to -\infty$.
    Therefore, the union $\cup_m A_m$ is $K$-quasiconformally equivalent to a straight line and it is equal to the boundary of $\hat{\Zbold}$ in $\C$.
    This implies property (4).

    Let $(\theta_k, n_k, u_k)$ be a generating sequence for $\Fbold$.
    For every $j, m \in \Z$ and every sufficiently high $k$, the arc $A^{[n_k+m]}_{k,j}$ converges to $\partial \hat{\Zbold}^{[m]} \cap \Sbold^{m-j}$ as $k \to \infty$.
    As we take $j \to \infty$, we observe that as $k \to \infty$, the image under $u_k$ of $\partial \hat{Z}^{[n_k+m]}(f_{\theta_k})$ converges to $\partial \hat{\Zbold}^{[m]} \cup \{\infty\}$ in the Hausdorff metric.
    This implies property (5).

    Lastly, property (6) follows from property (2), Proposition \ref{prop:trans-sectors} (6), and the fact that for all $n \in \Z$ and $m \geq -1$, the $m$\textsuperscript{th} pre-renormalization $f_n^{[m]}$ of $f_n$ is well-defined and injective on $\hat{Z}_n^{[m-1]}$.
\end{proof}

\subsection{The Mother Hedgehog}
\label{ss:MotherHedg:halfplane}

\begin{proposition}
\label{prop:trans-mother}
    There exists a unique closed subset $\Hbold = \Hbold(\Fbold)$ of $\C$ called the \emph{Mother Hedgehog} of $\Fbold$ with the following properties.
    \begin{enumerate}
        \item For every $n\in\Z$, the image of $\Hbold \cap \Sbold^n$ under $\pphi_n$ is the Mother Hedgehog $H_n$ of $f_n$ minus the slit $\gamma_n$.
        \item For any sufficiently high $\threshold \in \N$, $\Hbold$ is equal to the intersection of $\hat{\Zbold}^{[m]}(\threshold)$ across all $m \in \Z$. 
        \item Every connected component of $\Hbold$ is unbounded and the complement $\C \backslash \Hbold$ is connected, simply connected, and unbounded.
        \item Let $(\theta_k, n_k, u_k)$ be a generating sequence for $\Fbold$. 
        As $k \to \infty$, the rescaled Mother Hedgehog $u_k (H(f_{\theta_k}))$ converges to $\Hbold \cup \{\infty\}$ in the Hausdorff topology on compact connected subsets of $\RS$.
        \item $\Hbold$ is disjoint from $\Ibold^{<\infty}$ and for all $P \in \mathbf{T}$, $\Fbold^{P} : \Hbold \to \Hbold$ is a homeomorphism. 
        \item $\Hbold \cup \{\infty\}$ is a bouquet of arcs centered at $\infty$. More precisely, $\Hbold$ is a disjoint union of rays
        \[
            \Hbold = \bigcup_{\pphi \in \mathcal{Q}} \hair_\phi,
        \]
        where each $\hair_\phi$, called an \emph{internal ray} of $\Hbold$, is mapped by $\Fbold^P$ for any $P \in \Tbold$ onto some other internal ray $\hair_{\phi'}$.
        \item $0$ is a non-dividing point on the boundary of $\Hbold$ and it is accessible from the complement of $\Hbold$. Moreover, the ray $I_0$ from Proposition \ref{prop:trans-sectors} (5) is an internal ray.
        \item The boundary of $\Hbold$ is the postcritical set:
        \[
            \partial \Hbold = \overline{ \{ C_P \}_{P \in \Tbold} }.
        \]
    \end{enumerate}
\end{proposition}

Again, refer to Figure \ref{fig:transfer-to-cascade} for an illustration of $\Hbold$.

\begin{proof}
    For every $n\in \Z$, we have 
    \[
        \pphi_{n-1}^{-1}(H_{n-1}) \cap \Sbold^n = \pphi_n^{-1}(H_n \backslash \gamma_n).
    \]
    Let us set 
    \[
    \Hbold = \bigcup_n \pphi_n^{-1}(H_n \backslash \gamma_n).
    \]
    Property (1) is automatically satisfied.
    Property (2) follows from Proposition \ref{prop:trans-pseudo-siegel} (2) and the fact that $H_n$ is equal to $\cap_{m \geq -1} \hat{Z}^{[m]}_n$.
    Properties (3) and (4) follow from items (3) and (5) of Proposition \ref{prop:trans-pseudo-siegel}.

    By property (2) and Proposition \ref{prop:trans-pseudo-siegel} (6), we know that for all $Q \in \Tbold$, $\Fbold^Q$ is well-defined and injective on $\Hbold$.
    To prove property (5), it is sufficient to show that we always have $\Fbold^Q(\Hbold) = \Hbold$ where $Q$ is a generator of $\Tbold$, that is, 
    $Q = Q_{[m+1]} - j Q_{[m]}$ for some $m \in \Z$ and $j \geq 0$ where $0 \leq j \leq 1$ if $\abar_{m+1} < \infty$.
    Let us pick any point $z$ in $\Hbold$.
    Let $n \in \Z$ be sufficiently large negative number such that $n < m$ and that the infinite sector $\Sbold^{n+1}$ contains $z$.
    Consider the conformal gluing map $\pphi_n : \Sbold^n \to \D$; it sends $z$ to a point $\zeta$ in the Mother Hedgehog $H_n$ of $f_n$.
    Let $\zeta'$ be the unique point in $H_n$ such that $f_n^{\qq^n_{[m+1-n]}-j\qq^n_{[m-n]}}(\zeta') = \zeta$.
    Near $\zeta'$, $\pphi_n$ conjugates $\Fbold^Q$ and $f_n^{\qq^n_{[m+1-n]}-j\qq^n_{[m-n]}}$.
    Therefore, if we set $z' = \pphi^{-1}(\zeta')$, then $\Fbold^Q(z') = z$.
    This implies that $\Fbold^Q(\Hbold) = \Hbold$.
    
    Lastly, the remaining properties (6)--(8) follow from analogous properties for each $H_n$ as described in Theorem \ref{thm:mother-hedgehog} (2)--(4) and Theorem \ref{thm:renorm-limits} (6).
\end{proof}

Proposition \ref{prop:trans-mother} and Theorem \ref{thm:zero-area} automatically imply:

\begin{proposition}
\label{prop:zero-area-trans}
    The boundary of $\Hbold$ has zero Lebesgue measure.
\end{proposition}

This property will be used later in Section \ref{sec:hairiness-NILF}.

\begin{definition}
    We define the \emph{Siegel set} $\Zbold = \Zbold(\Fbold)$ of $\Fbold$ to be the interior of the Mother Hedgehog $\Hbold(\Fbold)$. 
\end{definition}

Consider the combinatorics $\tttheta \in \TheBiCpt$ of $\Fbold$, and consider the rotation number map $\bar{\mu} : \TheCpt \to \T$ from Proposition \ref{prop:rotation-number-map}.

\begin{definition}
\label{defn:brjuno}
    We say that $\Fbold$ is (\emph{eventually}) \emph{Brjuno} if for every $N \in \Z$, the corresponding combinatorics $\tttheta$ of $\Fbold$ has the property that $\bar{\mu}(\shift^N\tttheta)$ is (eventually) Brjuno.
\end{definition}

Proposition \ref{prop:trans-mother} and Theorem \ref{thm:eventually-brjuno} imply:

\begin{proposition}[The Brjuno condition]
\label{prop:eventually-brjuno}
\leavevmode
    \begin{enumerate}
        \item The Siegel set $\Zbold$ is non-empty if and only if $\Fbold$ is eventually Brjuno.
        \item Every connected component of $\Zbold$ is an unbounded topological disk with a single access to infinity.
        \item If $\Fbold$ is eventually Brjuno, then for all sufficiently high $N \in \Z$, the image of $\Zbold \cap \Sbold^N$ under $\boldsymbol{\phi}_N$ is the Siegel disk of $f_N$.
        \item If $\Fbold$ is Brjuno, then $\Zbold$ is connected and $\Fbold: \Zbold \to \Zbold$ is analytically conjugate to a cascade of translations with combinatorics $\tttheta$ acting on the lower half-plane (cf. Proposition \ref{thm:topological-cascade}).
        \item If $\Fbold$ is eventually Brjuno but not Brjuno, then
        \begin{itemize}
            \item $\Zbold$ has infinitely many connected components;
            \item if $N \in \Z$ denotes the smallest index such that $\bar{\mu}(\shift^N\tttheta)$ is an irrational, then every connected component of $\Zbold$ is invariant under $\Fbold^{Q_{[n]}}$ if and only $n \geq N$.
        \end{itemize}
    \end{enumerate}
\end{proposition}

In the language of transcendental dynamics, every connected component of $\Zbold$ is a simply parabolic Baker domain.

\subsection{External coordinates}
\label{ss:trans-external-coordinates}

Consider the unique Riemann mapping
\[
    \Psi: \C \backslash \Hbold \to \UHP
\]
fixing $0$, $\varepsilon_0$, and $\infty$. For $P \in \Tbold$, define 
\[
    \Fext^{P} := \Psi \circ \Fbold^{P} \circ \Psi^{-1} : \Psi( \Fbold^{-P}(\C \backslash \Hbold) ) \to \UHP,
\]
This defines the semigroup 
\[
\Fext := (\Fext^P)_{P \in \Tbold}
\]
called the \emph{external cascade} associated to $\Fbold$.
Every connected component of the domain of $\Fext^P$ is called a \emph{lake} of \emph{generation} $P$.
We also denote the time $P$ escaping set in external coordinates by
\[
    \Iext^{\leq P} := \Psi \left( \Ibold^{\leq P} \right),
\]
which is a subset of the complement of the domain of $\Fext^P$.

\begin{lemma}
    Let $\lake$ be a lake of some generation $P \in \Tbold$.
    Then,
    \begin{enumerate}
        \item $\lake$ is a topological disk;
        \item $\Fext^P : \lake \to \C$ is a conformal isomorphism;
        \item the boundary of $\lake$ intersects $\Iext^{\leq P} \cup \{\infty\}$. 
    \end{enumerate}
\end{lemma}

\begin{proof}
    By Lemma \ref{lem:crit-points}, the $\sigma$-proper map $\Fbold^P$ admits no critical value outside of $\Hbold$. 
    Hence, $\Fbold^P$ restricts to an infinite degree unbranched covering map of $\Dom(\Fbold^P) \backslash \Fbold^{-P}(\Hbold)$ onto $\C \backslash \Hbold$.
    This implies the lemma.
\end{proof}

Further topological and geometric properties of lakes will be deduced in the next section.

Proposition \ref{prop:trans-mother} (7) implies that $\Psi$ extends continuously to the bi-infinite critical orbit $\{C_P\}_{P\in \Kbold}$ on the boundary of $\Hbold$. 
Hence, the points
\[
    \crit_P := \Psi(C_P),\qquad P \in \Kbold
\]
are well-defined on the real line. We always have $\crit_0 = 0$ and $\crit_{-Q_{[0]}} = \varepsilon_0$. 
Denote
\[
    \mathtt{Crit} := \{ \crit_{P} \}_{P \in \Kbold}.
\]

\begin{lemma}
\label{lem:convergence-external}
    Let $(\theta_k,n_k, u_k)$ be a generating sequence for $\Fbold$.
    For every $k$, consider the Riemann mapping 
    \[
        \Psi_k: \RS \backslash u_k(H(f_{\theta_k})) \to \UHP
    \]
    fixing $0$ and $\varepsilon_0$ and sending the critical point $u_k(v_{\theta_k,-1})$ of $f_{\theta_k}$ to $\infty$.
    Then, as $k \to \infty$, $\Psi_k$ converges in the Caratheodory topology to $\Psi$.
    In particular, the set of critical points $\{u_k(v_{\theta_k,-p})\}_{2 \leq p \leq q^k_{[n_k+m]}}$ converges to the set $\{\crit_{-P}\}_{0<P\leq Q_{[m]}}$.
\end{lemma}

\begin{proof}
    The first part is an immediate consequence of Proposition \ref{prop:trans-mother} (4). The second is a consequence of Theorem \ref{thm:transcendental-extension} (2)(d).
\end{proof}

Let us also consider the slow generators $\{Q_n\}_{n\in\Z}$ of the continuant group $\Kbold$ as described in \S\ref{ss:natural-extension}.
It comes with a sequence of numbers $\{a_n\}_{n\in\Z}$ in $\overline{\N}$ such that 
\[
    Q_n = a_n Q_{n-1} + Q_{n-2} \qquad \textnormal{whenever } a_{n+1}<\infty.
\]
As a consequence of the lemma above, we have:

\begin{corollary}
\label{cor:extension-critical}
    For every $P \in \Tbold$, $\Fext^P$ extends to a homeomorphism
    \[
    \Fext^{P} : \mathtt{Crit} \to \mathtt{Crit}, \qquad \crit_Q \mapsto \crit_{P+Q}
    \]
    that preserves the orientation along the real line.
    For every $n \in \Z$,
\begin{enumerate}[label = \textnormal{(\roman*)}]
    \item $\crit_{-Q_n}$ is always between $\crit_0$ and $\crit_{Q_{n-1}}$;
    \item For $n \in \Z$ with $a_{n+1} = \infty$, $\crit_{-kQ_n}$ is always between $\crit_0$ and $\crit_{Q_{n-1}}$ for all $k \geq 1$.
\end{enumerate}
\end{corollary}

\begin{definition}
    For $Q \in \Tbold$, we define the tiling $\tiling_Q$ of $\R$ of generation $Q$ to be the collection of the closure of the connected components of the interior of $\R \backslash \{\crit_{-P}\}_{0 < P \leq Q}$.
\end{definition}

Whenever $P>Q$, then clearly the tiling $\tiling_P$ is a refinement of $\tiling_Q$. 
For every $n \in \Z$, the tiling of generation $Q_n$ can be written explicitly:
\[
    \tiling_{Q_n} = \{[\crit_{-P+Q_{n-1}-Q_n}, \crit_{-P}]\}_{0<P\leq Q_{n-1}} \cup \{\crit_{-P}, \crit_{-P+Q_{n-1}}\}_{0<P\leq Q_{n}-Q_{n-1}}.
\]
Observe that every tile in $\tiling_{Q_{n+2}}$ is the union of at least two tiles in $\tiling_{Q_{n}}$. 
The next proposition will be extremely important in later sections.

\begin{proposition}[Real Bounds]
\label{prop:R-apb-transcendental}
    For every $P \in \Tbold$, $\Fext^P$ uniquely extends to an orientation-preserving homeomorphism of the real line. 
    Let $I$ be a tile in $\tiling_{Q_n}$ for some $n \in \Z$.
    \begin{enumerate}
        \item If $J \in \tiling_{Q_n}$ is adjacent to $I$, then $|I| \asymp |J|$.
        \item If $a_{n+1} < \infty$, then $I$ contains only a finite number of tiles $J_1,\ldots, J_k$ in $\tiling_{Q_{n+1}}$, written in consecutive order, and they satisfy
        \[
            |J_j| \asymp \frac{|I|}{\min\{j,k+1-j\}^2} \qquad \text{ for all } j \in \{1,\ldots,k\}.
        \]
        \item If $a_{n+1} = \infty$, then
        \begin{enumerate}[label=\textnormal{(\alph*)}]
            \item the interior of $I$ contains a simple parabolic fixed point $\beta_I$ of $\Fext^{Q_n}$;
            \item every connected component of $I \backslash \{\beta_I\}$ can be decomposed to an infinite number of tiles $J_1,J_2,J_3,\ldots$ in $\tiling_{Q_{n+1}}$, written in consecutive order, such that $J_j$'s accumulate towards $\beta$ and
        \[
            |J_j| \asymp \frac{|I|}{j^2} \qquad \text{ for all } j \geq 1.
        \]
        \end{enumerate}
    \end{enumerate}
\end{proposition}

\begin{proof}
    By Lemma \ref{lem:convergence-external}, items (1) and (2) are reduced to Real Bounds for $f_\theta$ (Theorem \ref{thm:R-APB}).
    Theorem \ref{thm:R-APB} (2) implies that in the limiting parabolic case, the interior of $I$ admits a point $\beta_I$ satisfying item (3)(b).
    
    Items (1), (2), and (3)(b) imply that $\Fext^{P}: \mathtt{Crit} \to \mathtt{Crit}$ extends to an orientation-preserving homeomorphism of $\R$. 
    We can then perform Schwarz reflection and extend each map $\Fext^P$ to the lower half plane, forming a real holomorphic map.
    
    Now, let us go back to $I \in \tiling_{Q_n}$ again. 
    The point $\beta_I$ has to be a fixed point of $\Fext$. 
    The nature of $\beta_I$ as a fixed point of a real map on an interval has to be parabolic because it does not persist under perturbation.
    Moreover, $\beta$ is simple because it is the limit of collision of a pair of real-symmetric repelling fixed points (cf. Theorem \ref{thm:beta-fixed-points}).
\end{proof}

This proposition, together with Corollary \ref{cor:extension-critical}, implies that $(\Fext^P : \R \to \R)_{P \in \Tbold}$ is a (non-invertible) topological cascade with combinatorics $\tttheta$ in the sense of Theorem \ref{thm:topological-cascade}.
In later sections, we will apply Real Bounds to carefully justify uniform complex bounds for cascades, which are an analog of Complex Bounds for analytic critical circle maps.

\subsection{Near-parabolic levels}
\label{ss:near-parabolic-levels}

Let us explicitly describe the construction of the pseudo-Siegel pinched half-planes $\hat{\Zbold}^{[n]}$ from the Mother Hedgehog $\Hbold$.

For $n \in \Z$, denote
\[
    \mathtt{x}_n = \crit_{-Q_{[n]}} \qquad \mathtt{y}_n = \crit_{-\check{Q}_{[n]}}.
\]
If $n$ is an NP level ($\bar{a}_{n+1} \geq \threshold$), we are also interested in the pre-critical points 
\[
    \mathtt{x}'_n = \crit_{-Q_{[n+1]} + \threshold' Q_{[n]}},\qquad
\mathtt{y}'_n = \crit_{-\check{Q}_{[n]} - \threshold' Q_{[n]}}.
\]
We have
\[
    \mathtt{x}_n < 0 < \mathtt{x}'_n < \mathtt{y}'_n < \mathtt{y}_n \qquad \text{if } \varepsilon_{n+1} = -1, 
\]
and
\[
    \mathtt{y}_n < \mathtt{y}'_n < \mathtt{x}'_n < 0 < \mathtt{x}_n  \qquad \text{if } \varepsilon_{n+1} = +1. 
\]

For a near-parabolic level $n$, we define the level $n$ \emph{dam} $\mathtt{d}_0^{[n]}$ to be the unique hyperbolic geodesic of $\UHP$ with endpoints $\mathtt{x}'_n$ and $\mathtt{y}'_n$.
For $P \in \Tbold$ with $P \leq Q_{[n]}$, we also define $\mathtt{d}_{-P}^{[n]}$ to be the unique connected component of $\Fext^{-P}(\mathtt{d}_0^{[n]})$ that has both endpoints on $\R$.
Denote by $\fjord_{-P}^{[n]}$ the closed Jordan disk bounded by the arc $\mathtt{d}_{-P}^{[n]}$ and the interval $\Fext|_{\R}^{-P}([\mathtt{x}'_n, \mathtt{y}'_n])$. 
We denote the union by
    \[
        \Fjord^{[n]} := \bigcup_{0 \leq P < Q_{[n]}} \fjord_{-P}^{[n]}.
    \]
If $n$ is a bounded level, we will simply set $\Fjord^{[n]}$ to be the empty set.
The set $\fjord_{-P}^{[n]}$ will be called a \emph{parabolic fjord} of level $n$.

\begin{lemma}
\label{lem:fjord-pseudo-siegel-connection}
    For every $m \in \Z$, we have
    \[
        \hat{\Zbold}^{[m]} = \Hbold \cup \bigcup_{n > m} \Psi^{-1}\big( \Fjord^{[n]} \big).
    \]
\end{lemma}

\begin{proof}
    This follows from the construction of fjords and pseudo-Siegel disks of $f_\theta$, $\theta \in \Irrat$, together with the convergence described in Theorem \ref{thm:transcendental-extension}, Proposition \ref{prop:trans-pseudo-siegel} (5), and Proposition \ref{prop:trans-mother} (4).
\end{proof}

\begin{lemma}[Almost invariance]
\label{lem:fjord-almost-invariance}
    There exists a universal constant $K>0$ such that for every $n \in \Z$ and $P \in \Kbold \cup \{0\}$ with $P < Q_{[n]}$, every point in the dam $\mathtt{d}^{[n]}_{-P}$ is of hyperbolic distance at most $K$ away from the hyperbolic geodesic with the same endpoints as $\mathtt{d}^{[n]}_{-P}$.
\end{lemma}

\begin{proof}
    Let $\mathtt{D}^{[n]}_0$ be the closed round disk bounded by the union of $\mathtt{d}^{[n]}_0$ and its reflection in $\mathbb{R}$.
    Consider the domain $U_0^{[n]} = \UHP \cup (0,\mathtt{y}_n) \cup -\UHP$; it contains no critical values of $\Fext^{Q_{[n]}}$ and, by Real Bounds, $\text{mod}(U_0^{[n]} \backslash \mathtt{D}^{[n]}_0) \geq K$ for some universal constant $K >0$.
    Pick any $P \in \Tbold$ with $P < Q_{[n]}$ and let $U^{[n]}_{-P}$ and $\mathtt{D}^{[n]}_{-P}$ be the univalent lift of $U_0^{[n]}$ and $\mathtt{D}^{[n]}_0$ under $\Fext^{P}$ containing $\mathtt{d}^{[n]}_{-P}$ respectively.
    Then, the lemma follows from the property that the inverse branch $\Fext^{-P}: \mathtt{D}^{[n]}_0 \to \mathtt{D}^{[n]}_{-P}$ has uniformly bounded distortion.
\end{proof}

\begin{lemma}
\label{lem:fjords-bounds}
    There exists a universal constant $K >1$ such that for every $m \in \Z$, the domain $\UHP \backslash \bigcup_{n > m} \Fjord^{[n]}$ is $K$-quasiconformally equivalent in $\C$ to $\UHP$.
\end{lemma}
\begin{proof}
    This follows from Lemmas \ref{lem:qc-fjord-external} and \ref{lem:convergence-external}.
\end{proof}

Let us denote the union of all fjords by
\begin{equation}
    \label{eq:dfn:Fjord}
\Fjord := \bigcup_{n \in \Z} \Fjord^{[n]}.
\end{equation}
At the top level, we have:

\begin{corollary}[Uniform quasiconformality of $\Psi$]
\label{cor:Psi-qc}
    There exists an $\threshold$-uniform constant $K>1$ such that the map $\Psi: \C \backslash \hat{\Zbold} \to \UHP \backslash \Fjord$ extends to a $K$-quasiconformal map of the whole plane.
\end{corollary}

\begin{proof}
    This is because, according to Proposition \ref{prop:trans-pseudo-siegel} (4) and Lemma \ref{lem:fjords-bounds}, the boundaries of $\hat{\Zbold}$ and $\UHP \backslash \Fjord$ are $\threshold$-uniformly quasiconformally equivalent to a straight line.
\end{proof}

Recall from \S\ref{ss:natural-extension} the conversion from $Q_{[m]}$'s and $\abar_m$'s to $Q_n$'s and $a_n$'s.

\begin{definition}[Depths vs. levels]
\label{def:parabolic-fjord}
    If $n \in \Z$ is such that $Q_{n} = Q_{[m]}$ for some $m \in \Z$, we will denote $\Fjord^n = \Fjord^{[m]}$ and call every connected component $\fjord^n_{-P} = \fjord^{[m]}_{-P}$ of $\Fjord^n$ a \emph{parabolic fjord} of \emph{depth} $n$ and \emph{level} $m$. 
    We call
\[
    \fjord_n := \fjord^{[m]}_0.
\]
    the \emph{principal parabolic fjord} of \emph{depth} $n$ and level $m$.
    
    If $\fjord_n$ is non-empty, we say that the integer $n \in \Z$ is a \emph{near-parabolic} (\emph{NP}) \emph{depth}.
    If $a_{n+1} = \infty$, then we say that $n$ is a \emph{parabolic depth}.
\end{definition}

Near-parabolicity is governed by the combinatorial threshold $\threshold$. 
Observe that if $n$ is an NP depth, we must have $a_{n+1} \geq \threshold -1$. 
Conversely, if $a_{n+1} \geq \threshold$, then $n$ is an NP depth.

\begin{lemma}[Angular control of fjords]
\label{lem:fjord-angle-control}
    There exists an $\threshold$-uniform constant $\nnu = \nnu(\threshold) \in \left( 0,\frac{\pi}{2} \right)$ such that $\nnu(\threshold) \to 0$ as $\threshold \to \infty$ and the following holds for every fjord $\fjord$.
    Let $n$ be the depth of $\fjord$.
    Let $x_- < y_- < y_+ < x_+$ be real numbers such that $[x_-,x_+]$ be the unique tile in $\tiling_{Q_{n-1}}$ that contains $\parR \fjord$ and that both $[x_-,y_-]$ and $[y_+,x_+]$ are tiles in $\tiling_{Q_n}$. 
    Then,
    \[
        \fjord \subset \{z \in \C \: : \: 0 \leq \arg(z-y_-) \leq \nnu, \: \pi-\nnu \leq \arg(z-y_+) \leq \pi \}. 
    \]
\end{lemma}

See Figure \ref{fig:angular-control-fjords}.
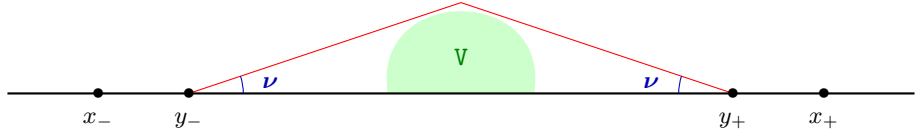
\begin{figure}
    \centering
    \begin{tikzpicture}[scale=1.2]
    \filldraw[color=green!20!white] (-0.8,0) .. controls (-1,1.2) and (1,1.2) .. (0.8,0) -- (-0.8,0);
    \draw[black,thick] (-5,0)--(5,0);
    \draw[red] (-3,0)--(0,1)--(3,0);
    
    \node at (-4,0) {\small $\bullet$};
    \node at (-3,0) {\small $\bullet$};
    \node at (3,0) {\small $\bullet$};
    \node at (4,0) {\small $\bullet$};
    
    \node at (-4,-0.3) {\small $x_-$};
    \node at (-3,-0.3) {\small $y_-$};
    \node at (3,-0.3) {\small $y_+$};
    \node at (4,-0.3) {\small $x_+$};
    \node[green!50!black] at (0,0.4) {$\fjord$};

    \draw[blue!70!black] (-2.4,0) arc (0:18.5:0.6);
    \draw[blue!70!black] (2.4,0) arc (180:161.5:0.6);
    \node[blue!70!black] at (2.1,0.12) {\small $\nnu$};
    \node[blue!70!black] at (-2.1,0.12) {\small $\nnu$};
\end{tikzpicture}
    \caption{The uniform angular control of fjords as described in Lemma \ref{lem:fjord-angle-control}.}
    \label{fig:angular-control-fjords}
\end{figure}

\begin{proof}
    Real Bounds imply the following estimates on $\parR \fjord$:
    \[
        |\parR \fjord| \asymp \frac{y_+-y_-}{\threshold'} \qquad \text{and} \qquad
        \dist(\parR \fjord, \{y_-,y_+\}) \asymp \frac{y_+-y_-}{2}\left( 1-\frac{1}{\threshold'}\right).
    \]
    Then, the desired lemma follows from the estimates above together with Lemma \ref{lem:fjord-almost-invariance}.
\end{proof}

\begin{lemma}[Beta periodic points]
\label{lem:repelling-periodic-point}
    Let $n \in \Z$ be an NP depth and $\fjord$ be a fjord of depth $n$. 
    Then, $\fjord$ contains a periodic point $\beta_\fjord$ of $\Fext$ with period $Q_n$ satisfying the following properties.
    Let $I_{\fjord}$ be the unique tile in $\tiling_{Q_{n-1}}$ that contains $\parR \fjord$.
    \begin{enumerate}
        \item If $n$ is a parabolic depth, then $\beta_\fjord$ is equal to the point $\beta_{I_{\fjord}}$ from Proposition \ref{prop:R-apb-transcendental}.
        \item Otherwise, $\beta_\fjord$ is a repelling periodic point located in the interior of $\fjord$ with 
            \[
                \textnormal{Im} \beta_\fjord \asymp \frac{|I_{\fjord}|}{a_{n+1}}.
            \]
    \end{enumerate}
    Moreover, $\{\beta_{\fjord} \: : \: \fjord \text{ is a fjord of depth }n\}$ is a single cycle of periodic points of $\Fext$.
\end{lemma}

\begin{proof}
    The whole lemma follows from standard parabolic bifurcation theory (compare with \cite[Lemma 2.4]{dFdM00}).
\end{proof}

For the principal fjord $\fjord_n$, we will denote the corresponding beta periodic point in $\fjord_n$ by 
\[
\beta_n := \beta_{\fjord_n}.
\]


\section{Hoglets, pseudo-bubbles, and lakes}
\label{sec:bounds}

As mentioned in~\S\ref{ss:outline}, we are interested in the dynamics of hoglets and surrounding pseudo-bubbles. 
We emphasize that the QC geometry of hoglets on all levels remembers anti-renormalization history (cf. Remark~\ref{rem:anti.renorm.history}).

Let us fix a neutral cascade $\Fbold = (\Fbold^P)_{P \in \Tbold}$ with some fixed combinatorics $\tttheta = \seq{ (\varepsilon_n, \abar_n) }_{n \in \Z}$. 
Notation introduced in the previous section, e.g. $\{C_P\}_{P \in \Kbold}$, $\{Q_{[n]}\}_{n \in \Z}$, $\{Q_n\}_{n \in \Z}$, $\{\Sbold^m\}_{m \in \Z}$, $\{\hat{\Zbold}^{[m]}\}_{m \in \Z}$, $\Hbold$, $\Psi$, $\Fjord$, etc will be used throughout the next three sections.

We will study in detail the dynamics of $\Fbold$ away from the Mother Hedgehog $\Hbold$, or equivalently, the structure of the corresponding external cascade $\Fext = (\Fext^P)_{P \in \Tbold}$ via the Riemann mapping $\Psi: \C \backslash \Hbold \to \UHP$. 
In this section, we will describe the properties of \emph{hoglets} (strict preimages of $\Hbold$); particularly, we transfer Theorem \ref{thm:bubble-bounds} to Pseudo-Bubble Bounds (Proposition \ref{prop:bubble-bounds}) for $\Fbold$ and $\Fext$ in \S\ref{ss:bubbles-and-pseudo-bubbles}.
These are then used in \S\ref{ss:alpha-points} to obtain the existence and some uniform estimates of the alpha-point associated to each hoglet, i.e. the ``preimage`` of the neutral fixed point.
In \S\ref{ss:disjointness-of-pseudo-bubbles}, we apply Pseudo-Bubble Bounds and the angular estimate of fjords (Lemma \ref{lem:fjord-angle-control}) to prove disjointness of primary pseudo-bubbles assuming that the combinatorial threshold $\threshold$ is sufficiently large.

In \S\ref{ss:primary-lakes}, we formulate some basic properties of \emph{primary lakes} (preimages of $\UHP$ under $\Fext$ touching the real line).
In \S\ref{ss:bubble-chains}, we organize hoglets into principal periodic chains for each depth $n$, and describe how these they uniformly shadow alpha-points of primary hoglets and, when $n$ is an NP depth, interact with principal fjords and parabolic basins.
Finally, in \S\ref{ss:lake}, we combine all of the above to obtain uniform geometric bounds on primary lakes (Proposition \ref{prop:lake-bounds-0}); particularly, it implies boundedness of all lakes (Corollary \ref{cor:bounded-lakes}).

The key notations introduced in this section are available in \S\ref{sss:cascade}.

\subsection{Pseudo-Bubble Bounds}
\label{ss:bubbles-and-pseudo-bubbles}

Consider the top pseudo-Siegel half-plane $\hat{\Zbold}$ of $\Fbold$ with a fixed sufficiently high combinatorial threshold $\threshold$.

\begin{definition}
\label{dfn:prim:bubble}
    For $P \in \Tbold$, denote by $\hat{\Bbold}_{P}$ the closure of the lift of the interior of $\hat{\Zbold}$ under $\Fbold^P$ that is disjoint from $\Hbold$ but still contains the critical point $C_{-P}$ on its boundary. 
    Let $\Bbold_{P}$ be the subset of $\hat{\Bbold}_{P}$ that is the closure of the preimage of $\Hbold$ under $\Fbold^P$. 
    We will call $\Bbold_P$ the \emph{primary hoglet} of $\Fbold$ of generation $P$ and $\hat{\Bbold}_P$ the \emph{primary pseudo-bubble} of $\Fbold$ of generation $P$. 
    The critical point $C_{-P}$ is called the \emph{root} of $\Bbold_P$ and $\hat{\Bbold}_P$.
    
    In general, a \emph{hoglet} (resp. \emph{pseudo-bubble}) of generation $P \in \Tbold$ is a connected component of $\Fbold^{-(P-Q)}(\Bbold_Q)$ (resp. $\Fbold^{-(P-Q)}(\hat{\Bbold}_Q)$) for some $Q \in \Tbold$ with $Q<P$. 
    The \emph{root} is the corresponding lift of $C_{-Q}$ under $\Fbold^{P-Q}$ contained in the hoglet. 
    We say that a hoglet (resp. pseudo-bubble) is \emph{secondary} if it is not primary and its root is contained in a primary hoglet.

    In later applications, we will be primarily working with hoglets and pseudo-bubbles in the external coordinates.
    We define a \emph{hoglet} (resp. \emph{pseudo-bubble}) of $\Fext$ of some generation $P \in \Tbold$ to be the image under $\Psi$ of a \emph{hoglet} (resp. \emph{pseudo-bubble}) of $\Fbold$ of generation $P$. 
\end{definition}

\begin{proposition}[Pseudo-Bubble Bounds]
\label{prop:bubble-bounds}
    There exist universal constants $K >1$ and $\ttau \in \left( 0,\frac{\pi}{2} \right)$ and an $\threshold$-uniform constant $K' >1$ such that the following holds for every hoglet $\Bbold$ of $\Fbold$.
    Let $\hat{\Bbold}$ be the corresponding pseudo-bubble and let $\Bubb = \Psi(\Bbold)$ and $\hat{\Bubb} = \Psi(\hat{\Bbold})$ be the corresponding objects in external coordinates.
    \begin{enumerate}
        \item Both $\hat{\Bbold}$ and $\hat{\Bubb}$ are closed $K'$-quasidisks in $\C$.
        \item Suppose $\Bbold=\Bbold_P$ is a primary hoglet of some generation $P \in \Tbold$.
        \begin{enumerate}[label=\textnormal{(\alph*)}]
            \item Size control: Let $n \in \Z$ be such that $Q_n < P \leq Q_{n+1}$. We have
            \[
                K^{-1} |C_{-P} - C_{Q_n-P}| \leq \diam(\Bbold) \leq \diam(\hat{\Bbold}) \leq K|C_{-P} - C_{Q_n-P}|,
            \]
            and 
            \[
                K^{-1} |\crit_{-P} - \crit_{Q_n-P}| \leq \diam(\Bubb) \leq \diam(\hat{\Bubb}) \leq K|\crit_{-P} - \crit_{Q_n-P}|,
            \]
            \item Angular control: We have 
        \[
            \hat{\Bubb}_P \backslash \{\crit_{-P}\} \subset \left\{ z \in \UHP \: : \: \left| \arg(z-\crit_{-P}) - \frac{\pi}{2} \right| < \ttau \right\};
        \]
        \end{enumerate}
        \item If $\Bbold$ is non-primary, then $\hat{\Bubb}$ has diameter at most $K$ with respect to the hyperbolic metric of $\UHP$.
    \end{enumerate} 
\end{proposition}

\begin{proof}
    Following Theorem \ref{thm:transcendental-extension}, let $(\theta_k,n_k,u_k)$ be a generating sequence for $\Fbold$.
    For $k \geq 1$ and $j \in \Z$, we will again use the notation
    \[
        f^{\#}_k = u_k \circ f_{\theta_k} \circ u_k^{-1}
        \quad \text{ and } \quad
        v^{\#}_{k,j} = u_k(v_{\theta_k,j}).
    \]
    For any $k \geq 1$ and $m \geq -1$, we will also denote the rescaled Mother Hedgehog and pseudo-Siegel disks by
    \[
        H^{\#}_k = u_k(H_{\theta_k}),
        \quad \text{ and } \quad
        \hat{Z}^{\#}_k = u_k(\hat{Z}_{\theta_k}).
    \]
    According to Propositions \ref{prop:trans-pseudo-siegel} and \ref{prop:trans-mother}, as $k \to \infty$,
    \begin{equation}
    \label{eqn:z-and-h-convergence}
        \hat{Z}_k^{\#} \to \hat{\Zbold} \cup \{\infty\}
        \qquad \text{ and } \qquad 
        H_k^{\#} \to \Hbold \cup \{\infty\}
    \end{equation}
    in the Carath\'eodory topology of compact subsets of $\RS$.
    For $k \geq 1$ and $j \geq 1$, let $B_{\theta_k}^j$ and $\hat{B}_{\theta_k}^j$ be the primary hoglet of $f_{\theta_k}$ of generation $j$; denote the corresponding rescalings by
    \[
    B_{k,j}^{\#} = u_k(B_{\theta_k}^j)
    \quad \text{and} \quad
    \hat{B}_{k,j}^{\#} = u_k(\hat{B}_{\theta_k}^j).
    \]
    Again, we will also use the notation $q^k_l = q_l(\theta_k)$.
    
    Pick any time $P \in \Tbold$ and let $n \in \Z$ be such that $Q_n < P \leq Q_{n+1}$.
    For $k \geq 1$, let $p_k \in \N$ be equal to the number $X_{P,k}$ from Theorem \ref{thm:transcendental-extension}.
    It has the property that $(f_k^{\#})^{p_k}$ converges to $\Fbold^P$ and the point $v_{k,-p_k}^{\#}$ converges to the critical point $C_{-P}$ as $k \to \infty$.
    Let $\{ m_k \}_{k\in \N}$ be a sequence of natural numbers such that $q^k_{m_k} < p_k \leq q^k_{m_k+1}$.
    According to Theorem \ref{thm:bubble-bounds}, there exist a universal constant $K>1$ and an $\threshold$-uniform constant $K'>1$ such that for every $k \geq 1$, we have the following properties.
    \begin{enumerate}
        \item[(i)] $\hat{B}_{k,p_k}^{\#}$ is a closed $K'$-quasidisk touching $\hat{Z}_k^{\#}$ exactly at the point $v_{k,-P_k}^{\#}$.
        \item[(ii)] We have
        \[
        K^{-1} \big| v^{\#}_{k, -p_k} - v^{\#}_{k, q^k_{m_k}- p_k} \big| \leq \diam(B_{k, p_k}^{\#}) \leq \diam(\hat{B}_{k, p_k}^{\#}) \leq 
        K \big| v^{\#}_{k, -p_k} - v^{\#}_{k, q^k_{m_k} - p_k} \big|.
        \]
        \item[(iii)] With respect to the hyperbolic metric of $\C \backslash H^{\#}_k$, every point in $\hat{B}_{k,P_k}^{\#}$ except for its root has distance at most $K$ away from the unique geodesic with endpoints $v^{\#}_{k,-P_k}$ and $\infty$.
    \end{enumerate}
    As a consequence of (i), we also have the following additional property.
    \begin{enumerate}
        \item[(iv)] The map $(f^{\#}_k)^{p_k}: \hat{B}_{k,p_k}^{\#} \to \hat{Z}^{\#}_k$ extends to a global $K'$-quasiconformal map $\hat{F}_k^{p_k}: \RS \to \RS$ that is conformal in the interior of $\hat{B}_{k, p_k}^{\#}$.
    \end{enumerate}  
    
    Let $x_{k}^+$ and $x_{k}^-$ be the unique pair of points on the boundary of $\hat{B}_{k,p_k}^{\#}$ that are mapped by $(f^{\#}_k)^{p_k}$ to the points $v^{\#}_{k,q^k_{m_k+1}}$ and $v^{\#}_{k,-q^k_{m_k+1}}$ respectively.
    By Koebe distortion control of $(f^{\#}_k)^{p_k}$ near the critical point $v^{\#}_{k,-p_k}$, the distance between any two points in the triplet $\{x_{k}^+, v^{\#}_{k,-p_k}, x_{k}^-\}$ is uniformly comparable to $| v^{\#}_{k,-p_k} - v^{\#}_{k,q^k_{m_k}-p_k} |$.
    Therefore, as $k \to \infty$, the pre-critical points $x_{k}^+$ and $x_{k}^-$ converge to distinct points $x^+$ and $x^-$ and the distance between any two points in the triple $\{x^+, C_{-P}, x^-\}$ is uniformly comparable to $|C_{-P} - C_{Q_{n}-P} |$.
    With this normalization, it is clear that the sequence of quasiconformal maps $\hat{F}_k^{p_k}$ described in (iv) is pre-compact. 
    Hence, up to subsequence, $\hat{F}_k^{p_k}$ converges uniformly on compact subsets to a $K$-quasiconformal map $\hat{F}_P : \RS \to \RS$ that sends $x^+$, $C_{-P}$, $x^-$ to $C_{Q_{n+1}}$, $C_0=0$, $C_{-Q_{n+1}}$ respectively.

    Let $\Bbold_P$ and $\hat{\Bbold}_P$ be the preimage under $\hat{F}_P$ of $\Hbold \cup \{\infty\}$ and $\hat{\Zbold} \cup \{\infty\}$ respectively.
    By \ref{eqn:z-and-h-convergence}, $\Bbold_P$ and $\hat{\Bbold}_P$ are the Hausdorff limits of $B^{\#}_{k,p_k}$ and $\hat{B}^{\#}_{k,p_k}$ as $k \to \infty$.
    Properties (i), (ii), (iii) imply that
    \begin{enumerate}
        \item[(i)'] $\hat{\Bbold}_P$ is a closed $K'$-quasidisk touching $\hat{\Zbold}$ exactly at the point $C_{-P}$;
        \item[(ii)'] we have
        \[
            K^{-1} \big| C_{-P} - C_{Q_n-P} \big| \leq \diam(\Bbold_P) \leq \diam(\hat{\Bbold}_P) \leq K \big| C_{-P} - C_{Q_n-P}|;
        \]
        \item[(iii)'] with respect to the hyperbolic metric of $\C \backslash \Hbold$, every point in $\hat{\Bbold} \backslash \{C_{-P}\}$ is of hyperbolic distance at most $K$ away from the unique geodesic with endpoints $C_{-P}$ and $\infty$.
    \end{enumerate}
    Property (i)' in particular implies that the map $\hat{F}_P$ on the interior of $\hat{\Bbold}_P$ is equal to $\Fbold^P$ and so $\hat{\Bbold}_P$ is indeed the primary pseudo-bubble of $\Fbold$ of generation $P$.
    Since Theorem \ref{thm:bubble-bounds} also describes the uniform quasiconformality and size control of the pseudo-bubbles $\hat{B}^{\#}_{k,p_k}$ in external coordinates, the uniform quasiconformality and the size control are passed down to the limit $\hat{\Bubb}_P = \Psi(\hat{\Bbold}_P)$ in external coordinates by virtue of Lemma \ref{lem:convergence-external}.

    The three paragraphs above give us the proof of the proposition for primary pseudo-bubbles.
    Now, let us consider a secondary hoglet $\Bbold$ and let $\hat{\Bbold}$ be the corresponding pseudo-bubble.
    Its root can be written as $C_{-P,-R}$ where $P \in \Tbold$ is the first time the image $\Fbold^P(\Bbold)$ is a primary hoglet and $R \in \Tbold$ is such that $\Fbold^{P}(C_{-P,-R}) = C_{-R}$.
    By Theorem \ref{thm:transcendental-extension}, there exist sequences of positive integers $p_k= X_{P,k}$ and $r_k=X_{R,k}$ such that $v^{\#}_{k,-p_k} \to C_{-P}$ and $v^{\#}_{k,-r_k} \to C_{-R}$ as $k \to \infty$.
    We then locate the critical point $v^{\#}_{k,-p_k,-r_k}$ of $f_k^{\#}$ located on $B_{k,p_k}^{\#}$ where $(f_k^{\#})^{p_k}(v^{\#}_{k,-p_k,-r_k}) = v^{\#}_{k,-r_k}$.
    Denote by $B^{\#}_k$ the secondary hoglet of $f_k^{\#}$ rooted at $v^{\#}_{k,-p_k,-r_k}$.
    Similar to the previous paragraphs, we have that as $k \to \infty$, the critical point $v^{\#}_{k,-p_k,-r_k}$ converges to $C_{-P,-R}$, the secondary hoglet $B^{\#}_k$ of $f_k^{\#}$ rooted at $v^{\#}_{k,-p_k,-r_k}$ converges to $\Bbold$, and the corresponding pseudo-bubble $\hat{B}^{\#}_k$ converges to $\hat{\Bbold}$.
    The uniform quasiconformality and hyperbolic diameter bound for the corresponding pseudo-bubble $\hat{\Bbold}$ as well as $\hat{\Bubb} = \Psi(\hat{\Bbold})$ follow from Theorem \ref{thm:bubble-bounds} (1) and (3).

    Lastly, if $\hat{\Bbold}$ is a non-primary non-secondary pseudo-bubble of $\Fbold$, then there exists some time $P \in \Tbold$ such that $\Fbold^P(\hat{\Bbold})$ is a secondary pseudo-bubble. 
    Properties (1) and (3) for $\Fbold^P(\hat{\Bbold})$ imply that $\Fbold^P: \hat{\Bbold} \to \Fbold^P(\hat{\Bbold})$ is a homeomorphism with uniformly bounded distortion.
    This implies properties (1) and (3) for $\hat{\Bbold}$.
    The same holds for $\hat{\Bubb} = \Psi(\hat{\Bbold})$.
\end{proof}

\subsection{Alpha-points}
\label{ss:alpha-points}

\begin{figure}
        \centering
       \begin{tikzpicture}
    \node[anchor=south west,inner sep=0] (image) at (0,0) {\includegraphics[width=1\linewidth]{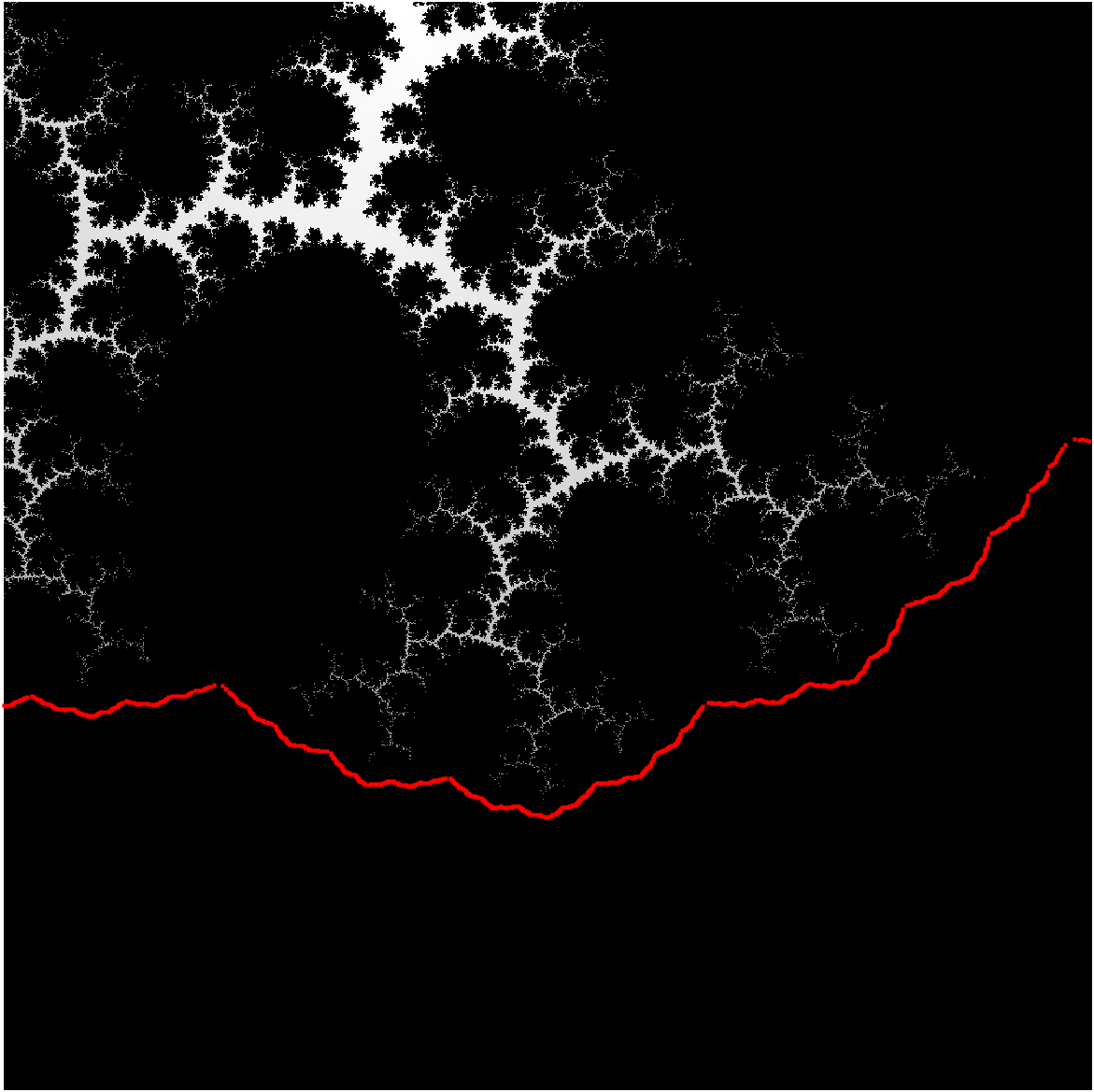}};
    \begin{scope}[
        x={(image.south east)},
        y={(image.north west)}
    ]
        \node [white] at (0.5,0.25) {$\bullet$};
        \node [white] at (0.5,0.22) {$C_0$};
        
        \node [white] at (0.645,0.355) {$\bullet$};
        \node [white] at (0.67,0.325) {$C_{-Q_0}$};
        \node [white] at (0.2,0.37) {$\bullet$};
        \node [white] at (0.195,0.34) {$C_{-Q_{-1}}$};
        \node [white] at (0.41,0.285) {$\bullet$};
        \node [white] at (0.405,0.255) {$C_{-Q_1}$};
        \node [green] at (0.32,0.775) {\large$\bullet$};
        \node [green] at (0.305,0.745) {$\alpha_{Q_{-1}}$};
        \node [green] at (0.53,0.555) {\large$\bullet$};
        \node [green] at (0.555,0.525) {$\alpha_{Q_0}$};
        \node [green] at (0.445,0.41) {\large$\bullet$};
        \node [green] at (0.435,0.38) {$\alpha_{Q_1}$};
        
        \draw[white,-latex] (0.64,0.3) .. controls (0.62,0.2) and (0.56,0.18) .. (0.53,0.2);
        \node [white] at (0.62,0.17) {$\Fbold^{Q_0}$};
        
        \draw[white,-latex] (0.41,0.23) .. controls (0.425,0.16) and (0.46,0.2) .. (0.48,0.21);
        \node [white] at (0.39,0.19) {$\Fbold^{Q_1}$};
        
        \draw[white,-latex] (0.21,0.3) .. controls (0.24,0.15) and (0.4,0.12) .. (0.48,0.18);
        \node [white] at (0.3,0.14) {$\Fbold^{Q_{-1}}$};
    \end{scope}
\end{tikzpicture}
    \caption{An approximately realistic picture of the dynamical plane of the neutral casacade $\Fbold$ with golden mean combinatorics $\langle \ldots,(+,2),(+,2), (+,2),\ldots\rangle$.
    The boundary of the Mother Hedgehog is colored red.
    The primary hoglets sit above it and each of them has an alpha-point marked in green.}
    \label{fig:bubbles-and-alphas}
\end{figure}
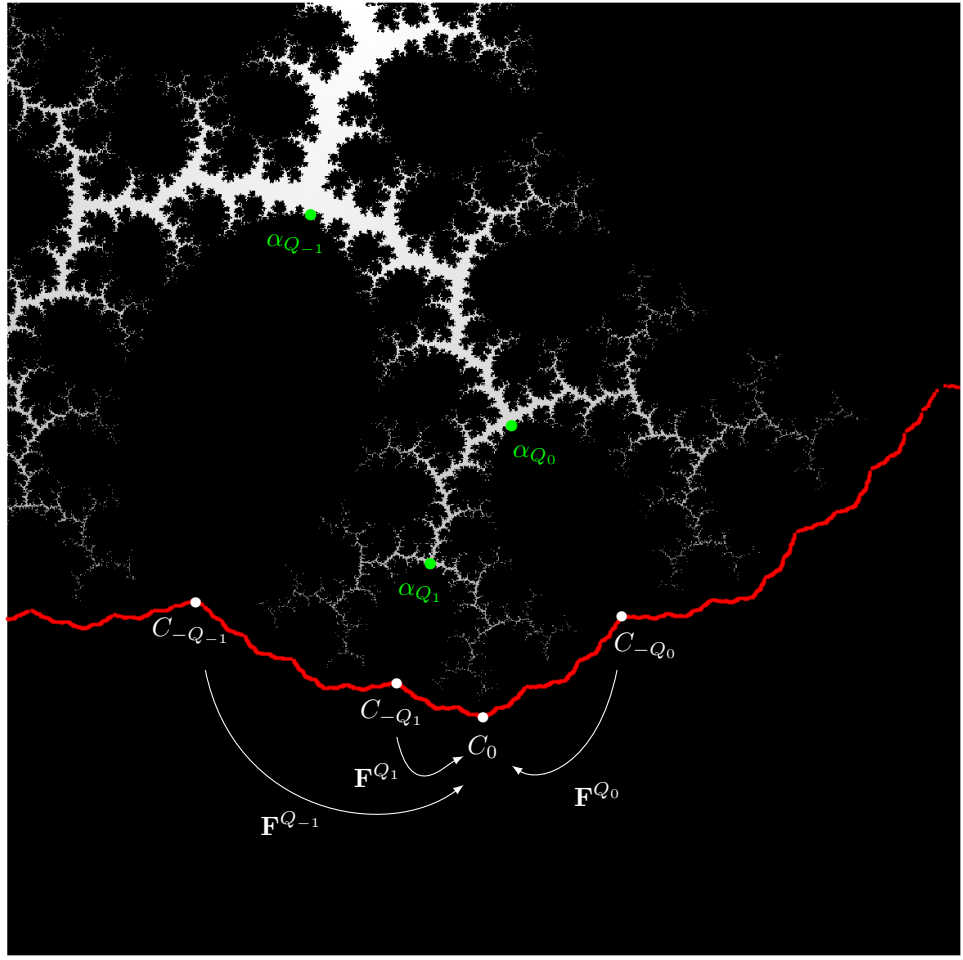

\begin{corollary}[Alpha-points]
\label{cor:alpha-points}
    Let $\Bbold$ be a hoglet of $\Fbold$ of some generation $P \in \Tbold$.
    There exists a unique point $\aalpha \in \C$, called the alpha-point of $\Bbold$, such that 
        \[
            \Ibold^{\leq P} \cap \partial \Bbold = \Ibold^{\leq P} \cap \partial \hat{\Bbold} = \{ \aalpha\}.
        \]
    Moreover, the map $\Fbold^P: \hat{\Bbold} \backslash \{\aalpha\} \to \hat{\Zbold}$ extends to an $\threshold$-uniformly quasiconformal map of the whole sphere $\RS$ sending $\aalpha$ to $\infty$.
\end{corollary}

See Figure \ref{fig:bubbles-and-alphas} for an illustration.
In \cite{DLy23}, alpha-points are constructed using boundedness of limbs and a hyperbolic contraction argument.
In the proof below, we take a shortcut and show the existence of alpha-points using the boundedness and the quasiconformality of pseudo-bubbles.

\begin{proof}
    By Proposition \ref{prop:bubble-bounds} (1), the conformal map $\Fbold^P: \textnormal{int}(\hat{\Bbold}) \to \textnormal{int}(\hat{\Zbold})$ extends to an $\threshold$-uniformly quasiconformal map $\hat{F}_{\hat{\Bbold}} : \RS \to \RS$.
    Note that the extension is unique along the boundary of $\hat{\Bbold}$.
    The desired alpha-point $\aalpha$ is equal to $(\hat{F}_{\hat{\Bbold}})^{-1}(\infty)$.
\end{proof}

In external coordinates, the primary hoglet and pseudo-bubble of $\Fext$ of generation $P \in \Tbold$ will be denoted by
\[
    \Bubb_P := \Psi(\Bbold_P) \qquad \text{ and } \qquad \hat{\Bubb}_P := \Psi(\hat{\Bbold}_P).
\]
The root is the point $\crit_{-P}$.
The alpha-point of $\Bbold_P$ will be denoted by $\aalpha_P$ and in external coordinates, it will be denoted by $\alpha_P = \Psi(\aalpha_P)$.

The following technical lemma will be useful later.

\begin{lemma}
\label{lem:estimate-lift-sector}
    There exists a universal constant $\delta>0$ such that the following holds.
    Let $m, n \in \Z$ be such that $Q_{[n-1]} < Q_m \leq Q_{[n]}$ and let
    $\Sbold_{Q_m}^{n}$ be the unique lift of the sector $\Sbold^{n}$ under $\Fbold^{Q_m}$ that contains the critical point $C_{-Q_m}$.
    Then,
    \[
        \dist \left( \Psi \left( \partial \Sbold^n_{Q_m} \backslash \Hbold \right), \crit_{-Q_m} \right) \asymp \diam(\Bubb_{Q_m}).
    \]
\end{lemma}

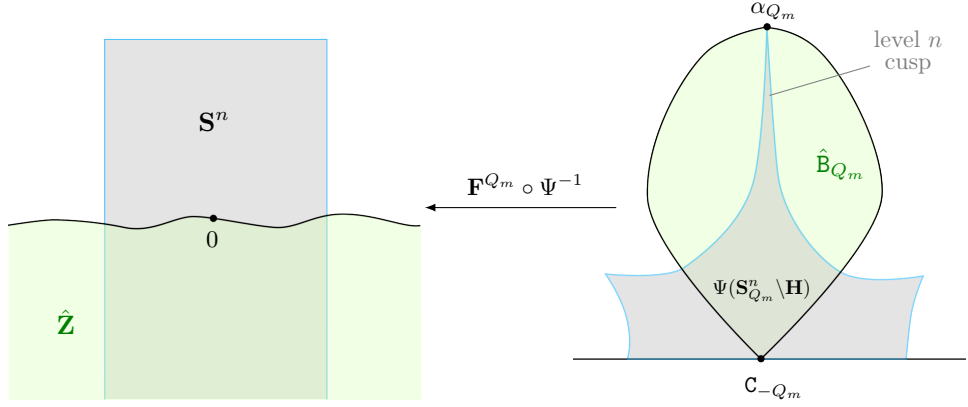
\begin{figure}
        \centering
       \begin{tikzpicture}
    \node[anchor=south west,inner sep=0] (image) at (0,0) {\includegraphics[width=1\linewidth]{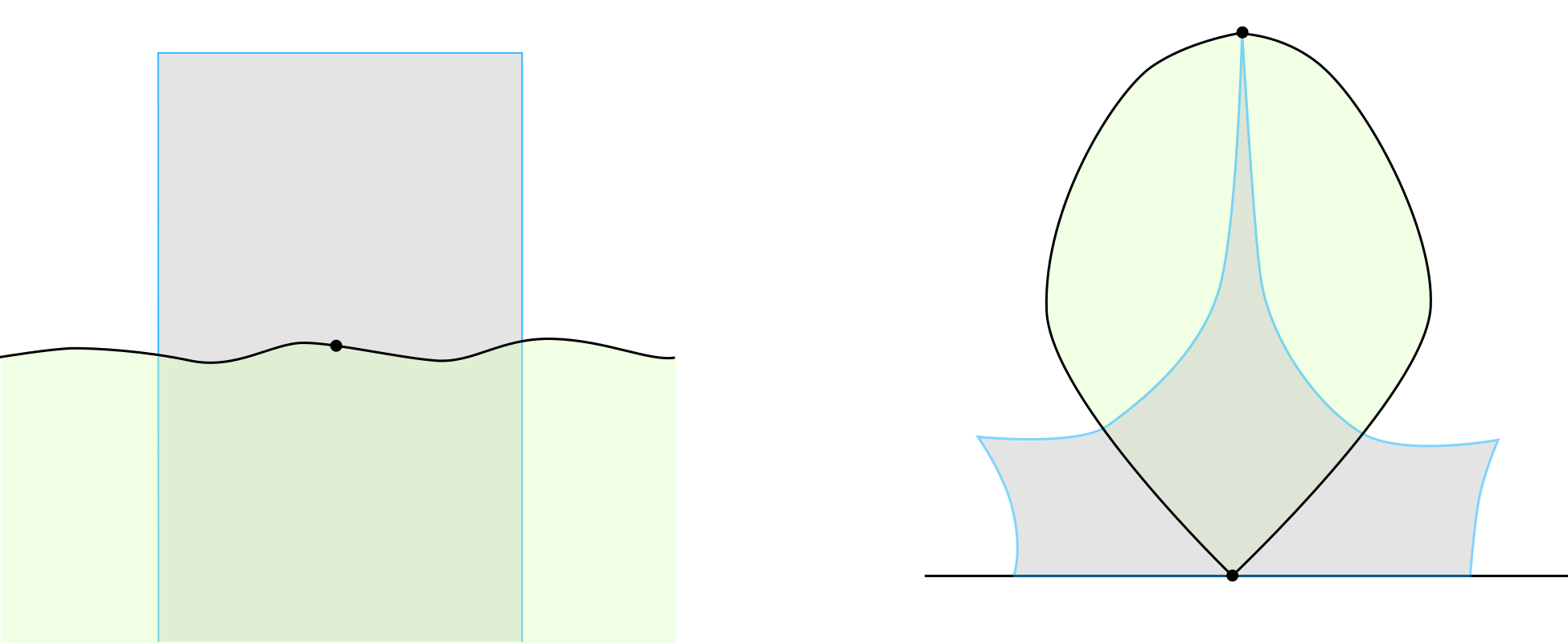}};
    \begin{scope}[
        x={(image.south east)},
        y={(image.north west)}
    ]
        \node [green!50!black] at (0.06,0.2) {$\hat{\Zbold}$};
        \node [green!50!black] at (0.87,0.6) {$\hat{\Bubb}_{Q_m}$};
        \node [black] at (0.215,0.71) {$\Sbold^{n}$};
        \node [black] at (0.8,0.03) {\scalebox{0.9}{$\crit_{-Q_{m}}$}};
        \node [black] at (0.8,0.99) {\scalebox{0.9}{$\alpha_{Q_{m}}$}};
        \node [black] at (0.213,0.41) {\small $0$};
        \node [black] at (0.54,0.55) {\scalebox{0.9}{$\Fbold^{Q_m}\circ \Psi^{-1}$}};
        \draw[-latex] (0.635,0.49) -- (0.435,0.49);
        \node [black] at (0.787,0.28) {\scalebox{0.75}{$\Psi (\Sbold^n_{Q_m} \backslash \Hbold )$}};
        \node [gray] at (0.94,0.92) {\scalebox{0.9}{level $n$}}; 
        \node [gray] at (0.94,0.85) {\scalebox{0.9}{cusp}};
        \draw [gray] (0.9,0.85) -- (0.795,0.775);
    \end{scope}
\end{tikzpicture}
    \caption{In external coordinates, the lift of the infinite sector $\Sbold^n$ under $\Fbold^{Q_{m}}$, where $Q_{[n-1]} < Q_m \leq Q_{[n]}$, has a cusp reaching all the way to $\alpha_{Q_m}$. Its distance from the critical value $\crit_{-Q_m}$ is described in Lemma~\ref{lem:estimate-lift-sector}. The QC geometry of this level $n$ cusp is described in Remark~\ref{rem:anti.renorm.history}.}
    \label{fig:inf-sec-and-lifts}
\end{figure}

See Figure \ref{fig:inf-sec-and-lifts} for illustration.

\begin{proof}
    The only critical value of $\Fbold^{Q_m}$ inside of $\Sbold^n$ is $0$.
    Therefore, the map $\Fbold^{Q_m} : \Sbold^{n}_{Q_m} \to \Sbold^{n}$ is a double covering map branched at the unique critical point $C_{-Q_{m}}$.
    Recall from Proposition \ref{prop:trans-sectors} (1) that the conformal radius of $\Sbold^{n}$ about $0$ is $\asymp |C_{Q_{[n]}}|$.
    By Real Bounds, there is a universal constant $N \in \N$ such that $\Sbold^{n}$ contains the round disk $E = \D(0,|C_{Q_{[n+N]}}|)$ and $\modu(\Sbold^n \backslash E) \asymp 1$.
    We will keep in mind that by virtue of Pseudo-Bubble Bounds and Real Bounds,
    \begin{equation}
    \label{eqn:bubble-size-estimate-1}
        \diam(\Bubb_{Q_m}) \asymp |\crit_{Q_{[n+N]}-Q_m} - \crit_{-Q_m}| \asymp |\crit_{-Q_{[n]}-Q_m} - \crit_{-Q_m}|.
    \end{equation}
    
    Let $E' \subset \Sbold_{Q_m}^{n}$ be the lift of $E$ under $\Fbold^{Q_m}$.
    By Koebe distortion theorem, $\Fbold^{Q_m}: E' \to E$ is a composition of a univalent map with uniformly bounded distortion and the power map $z \mapsto z^2$.
    This implies that 
    \[
        \dist \left( \partial \Sbold^{n}_{Q_m}, C_{-Q_m} \right) \geq \dist\left( \partial E', C_{-Q_m} \right) \succ |C_{Q_{[n+N]} - Q_m} - C_{-Q_m}|.
    \]
    Observe that $X = \partial \Sbold^n_{Q_m} \backslash \Hbold$ is disjoint from $\hat{\Zbold}$.
    Therefore, by Corollary \ref{cor:Psi-qc} and Real Bounds,
    \[
        \dist ( \Psi(X), \crit_{-Q_m} ) \succ 
        |\crit_{Q_{[n+N]}-Q_m} - \crit_{-Q_m}|.
    \]
    On the other hand, we also know from the construction of $\Sbold^n$ that $C_{-Q_{[n]}-Q_m}$ is a real endpoint of $X$.
    Therefore,
    \[
        \dist ( \Psi(X), \crit_{-Q_m} ) \leq |\crit_{-Q_{[n]}-Q_m}-\crit_{-Q_m}|. 
    \]
    These two estimates, together with (\ref{eqn:bubble-size-estimate-1}), imply the lemma.
\end{proof}

\begin{corollary}
\label{cor:location-of-alpha}
    For every $P \in \Tbold$, we have
    \[
        |\alpha_P - \crit_{-P}| \asymp \diam(\Bubb_P) \quad \text{ and } \quad
        \left| \arg(\alpha_P-\crit_{-P}) - \frac{\pi}{2} \right| < \ttau.
    \]
\end{corollary}

\begin{proof}
    Since both $\crit_{-P}$ and $\alpha_P$ are contained in $\Bubb_P$, the estimate $|\alpha_P - \crit_{-P}| \leq \diam(\Bubb_P)$ is trivial and the angular estimate follows directly from Pseudo-Bubble Bounds. So below, we just need to prove 
    \begin{equation}
    \label{eqn:cor-location-of-alpha}
    |\alpha_P - \crit_{-P}| \succ \diam(\Bubb_P).
    \end{equation}
    When $P = Q_m$ for some $m \in \Z$, Lemma \ref{lem:estimate-lift-sector} gives us (\ref{eqn:cor-location-of-alpha}) because $\alpha_{Q_m}$ is contained in $\Psi(\partial \Sbold^n_{Q_m} \backslash \Hbold )$.
    It remains to consider the case when $Q_m < P < Q_{m+1}$ for some $m \in \Z$.
    
    We claim that $\Fext^{Q_{m+1}-Q_m}$ sends $\hat{\Bubb}_{Q_{m+1}}$ onto $\hat{\Bubb}_{Q_m}$ with uniformly bounded distortion; this implies the desired estimate (\ref{eqn:cor-location-of-alpha}) for $P$ from the same estimate for $Q_m$.
    Let $I \subset \R$ be the connected component of $\R \backslash \{\crit_{R}\}_{0\leq R < Q_{m+1}-Q_m}$ containing $\crit_{-Q_m}$, and let $I' = (\Fext^{Q_{m+1}-Q_m}|_{\R})^{-1}(I)$.
    Consider the domain $W = \UHP \cup I \cup -\UHP$.
    Then, the homeomorphism $\Fext^{Q_{m+1}-Q_m}: I' \to I$ extends to a real-symmetric univalent map $\Fext^{Q_{m+1}-Q_m}: W' \to W$ for some domain $W'$ containing $I'$.
    By Real Bounds, the endpoints of $I$ are of distance $\asymp |\crit_{-Q_m}|$ away from $\crit_{-Q_m}$, hence by Pseudo-Bubble Bounds, $\modu(W \backslash \hat{\Bubb}_{Q_m}) \asymp 1$.
    Therefore, by Koebe, $\Fext^{Q_{m+1}-Q_m}: \hat{\Bubb}_{Q_{m+1}} \to \hat{\Bubb}_{Q_{m}}$ indeed has uniformly bounded distortion. 
\end{proof}

\begin{remark}[Anti-renormalization history of hoglets near alpha-points]
\label{rem:anti.renorm.history} 
For $P \in \Tbold$, Corollary \ref{cor:alpha-points} implies that the local QC geometry of $\Bbold_{P}$ near $\alpha_P$ is equivalent to the local QC geometry of $\Hbold$ near $\infty$. The latter is governed by the anti-renormalization history of $\Fbold$. Consequently, as shown in Figure \ref{fig:inf-sec-and-lifts}, the local QC geometry of the level $n$ cusp at $\alpha_{Q_{m}}$ is likewise governed by the anti-renormalization history of $\Fbold$. This also translates to the corresponding statement concerning hoglets for sectorial towers and is one of the reasons why we do not expect a hybrid conjugacy for forward towers, as stated in Remark~\ref{rem:intro.forw.towers}.
\end{remark}

\subsection{Disjointness of primary pseudo-bubbles}
\label{ss:disjointness-of-pseudo-bubbles}

It is obvious that distinct hoglets are disjoint.
In the next three lemmas, we will deduce the disjointness of pseudo-bubbles by playing with the combinatorial threshold $\threshold$.
Recall the $\threshold$-uniform constant $\nnu(\threshold) \in (0,\frac{\pi}{2})$ from Lemma \ref{lem:fjord-angle-control} that determine the angular control of fjords; we have that as $\threshold \to \infty$, then $\nnu(\threshold) \to \infty$.
Recall also the universal constants $K>1$ and $\ttau \in (0,\frac{\pi}{2})$ from Pseudo-Bubble Bounds.
From now on, we will assume that $\threshold$ is sufficiently high enough such that
\begin{equation}
\label{eqn:tau-and-nu}
    \ttau + \nnu(\threshold) < \frac{\pi}{2}.
\end{equation}

\begin{lemma}
    \label{lem:main-bubble-disjoint-fjord}
    For every $n\in\Z$, both $\hat{\Bubb}_{Q_n}$ and $\hat{\Bubb}_{Q_n-Q_{n-1}}$ are disjoint from $\Fjord$.
\end{lemma}

\begin{proof}
    According to Lemma \ref{lem:fjord-angle-control}, every fjord is contained in
    \[
        \{ 0 \leq \arg (z-\crit_{-Q_n}) \leq \nnu(\threshold)\} \cup \{\pi - \nnu(\threshold) \leq \arg (z-\crit_{-Q_n})  \leq \pi \}.
    \]
    Similar estimates hold when the critical point $\crit_{-Q_n}$ is replaced with $\crit_{-Q_n+Q_{n-1}}$.
    From here, it is clear that (\ref{eqn:tau-and-nu}) implies the claim.
\end{proof}

\begin{lemma}
\label{lem:bubble-disjoint-fjord}
    There exists a universal constant $\delta_1>0$ such that for every parabolic fjord $\fjord$, and for every $P \in \Tbold$,
    if $\parR \fjord$ lies in the connected component of $\R \backslash \{ 0 \}$ that does not contain $\crit_{-P}$, then
    $\fjord$ is disjoint from $\nbh_{\delta_1}(\hat{\Bubb}_P)$.
\end{lemma}

\begin{proof}
    Let $n = n(P) \in \Z$ be such that $Q_n < P \leq Q_{n+1}$.
    By Pseudo-Bubble Bounds, the pseudo-bubble $\hat{\Bubb}_P$ is contained in the set
    \[
        \Xi_P = \{ \crit_{-P}\} \cup \left\{ z \in \UHP \: : \: \left| \arg(z-\crit_{-P}) - \frac{\pi}{2} \right| \leq \ttau, |z-\crit_{-P}| \leq K |\crit_{Q_n-P}-\crit_{-P} | \right\}.
    \]
    By Real Bounds, there exists some small universal constant $\delta_1>0$ such that $\nbh_{\delta_1}(\Xi)$ is disjoint from $0$.
    We know from Lemma \ref{lem:fjord-angle-control} that every fjord is contained in
    \[
        \{ 0 \leq \arg z \leq \nnu(\threshold)\} \cup \{\pi - \nnu(\threshold) \leq \arg z \leq \pi \}.
    \]
    Then, (\ref{eqn:tau-and-nu}) implies that for every fjord $\fjord$, if $\parR \fjord$ lies in the connected component of $\R \backslash \{ 0 \}$ that does not contain $\crit_{-P}$, then $\nbh_{\delta_1}(\Xi_P)$ is disjoint from $\fjord$.
\end{proof}

\begin{lemma}
\label{lem:primary-bubble-disjoint}
    There exists a universal constant $\delta_2>0$ such that for any two $P, R \in \Tbold$ with $P>R$, the pseudo-bubble $\hat{\Bubb}_R$ is disjoint from $\nbh_{\delta_2}(\hat{\Bubb}_P)$.
\end{lemma}

\begin{proof}
    Consider the constants $n = n(P) \in \Z$, $\delta_1>0$ and the set $\Xi_P$ in the proof of the previous lemma.
    By Real Bounds, there exists a uniform constant $N \geq 1$ such that
    $\nbh_{\delta_1}(\Xi) \cap \UHP$ is contained in the unique isosceles triangle in $\overline{\UHP}$ having two of its three vertices at $\crit_{Q_{n-N}-P}$ and $\crit_{Q_{n-N-1}-P}$ with interior angles $\ttau$.
    Then, Real Bounds guarantees that if $\crit_{-P}$ avoids the interval $[\crit_{Q_{n-N}-P}, \crit_{Q_{n-N-1}-P}]$, for example if $P-R < Q_{n-N}$, then $\hat{\Bubb}_P$ avoids $\nbh_{\delta_1}(\Xi)$.

    It remains to consider the case when $P-R \geq Q_{n-N}$.
    By Real Bounds and Pseudo-Bubble Bounds, there exists a universal constant $\delta' \in (0,\delta_1)$ (depending on $N$) such that $\nbh_{\delta'}(\hat{\Bubb}_{P-R})$ is disjoint from the set of critical values of $\Fext^R$.
    Then, $\nbh_{\delta'}(\hat{\Bubb}_{P-R})$ univalently lifts to an open disk neighborhood $U$ of $\hat{\Bubb}_P$ under $\Fext^R$.
    The set $U$ contains $\nbh_{\delta_2}(\hat{\Bubb}_{P})$ for some universal constant $\delta_2>0$. 
    By design, $U$ is also disjoint from $\hat{\Bubb}_R$ because otherwise, its image $\nbh_{\delta'}(\hat{\Bubb}_{P-R})$ would have to intersect parabolic fjords on the other side of $0$, contradicting Lemma \ref{lem:bubble-disjoint-fjord}.
\end{proof}

\subsection{Primary lakes}
\label{ss:primary-lakes}

\begin{definition}
    We say that a lake $\lake$ of $\Fext$ is a \emph{primary lake} if $\parR \lake$ contains an open interval.
\end{definition}

\begin{lemma}
\label{lem:primary-lake}
    For every $P \in \Tbold$ and every tile $I$ of $\tiling_P$, there exists a unique primary lake $\lake_P(I)$ of generation $P$ such that $I$ is a connected component of $\partial_{\R} \lake_P(I)$.
\end{lemma}

\begin{proof}
    The map $\Fext^P: I \to \Fext^P(I)$ is a homeomorphism that contains no critical points of $\Fext^P$ except at the two endpoints of $I$.
    So if we consider the Schwarz reflection, the domain $\UHP \cup \textnormal{int}(\Fext^P(I)) \cup -\UHP$ lifts univalently under $\Fext^P$ to a real-symmetric domain $\lake'_P(I)$ with real slice $\textnormal{int}(I)$.
    Then, $\lake_P(I)$ is equal to $\lake'_P(I) \cap \UHP$.
    The lake $\lake_P(I)$ is uniquely determined by $P$ and $I$ because any two distinct lakes of the same generation must be disjoint.
\end{proof}

The following lemma is elementary.

\begin{lemma}
\label{lem:primary-lakes-and-adjacent-bubbles}
    Consider $P \in \Tbold$ and $I \in \tiling_P$, and let $\crit_{-R}$ be an endpoint of $I$ for some $R \in \Tbold$ with $R\leq P$. 
    \begin{enumerate}
        \item Suppose $R=P$. Then, the Carath\'eodory boundary $\lake_P(I)$ intersects the Carath\'eodory boundary of $\C \backslash \Bubb_P$ along one of the two intervals of $\partial^c(\C \backslash \Bubb_P)$ with endpoints $\crit_{-P}$ and $\alpha_P$.
        \item Suppose $R<P$. 
        Let $S \in \Tbold$ be such that 
        $\crit_{-S}$ is the critical point of $\Fext^{P-R}$ closest to $0$ such that the interval $(0,\crit_{-S})$ is disjoint from $\Fext^R(I)$.
        Let $\crit_{-R,-S}$ be the unique point on $\partial \Bubb_{-R}$ that is mapped under $\Fext^R$ to $\crit_{-S}$.
        Then, the interval of $\partial^c(\C \backslash \Bubb_P)$ that has endpoints $\crit_{-P}$, $\crit_{-P,-S}$ and is disjoint from $\alpha_P$ is a connected component of the intersection between the Carath\'eodory boundaries of $\lake_P(I)$ and $\C \backslash \Bubb_P$.
    \end{enumerate}
\end{lemma}

\begin{proof}
    Case (1) follows from the fact that the connected component of $\R \backslash \{0\}$ that is disjoint from $\Fext^P(I)$ lifts under $\Fext^P : \lake_P(I) \to \UHP$ to one of the two intervals of $\partial^c(\C \backslash \Bubb_P)$ with endpoints $\crit_{-P}$ and $\alpha_P$.

    In Case (2), let $J$ be the unique tile of $\tiling_{P-R}$ that contains $\Fext^R(I)$.
    By Lemma \ref{lem:primary-lake}, the image of $\lake_P(I)$ under $\Fext^R$ is equal to $\lake_{P-R}(J)$.
    The point $\crit_{-S}$ is precisely the endpoint of $J$ that is closer to $0$ than the other endpoint of $\Fext^R(I)$.
    Then, the claim follows from the fact that the subinterval $[0,\crit_{-S}] \subset J$ lifts under $\Fext^R$ to the interval of $\partial^c(\C \backslash \Bubb_P)$ that has endpoints $\crit_{-P}$, $\crit_{-P,-S}$ and is disjoint from $\alpha_P$.
\end{proof}

\begin{lemma}[Uniqueness of primary lakes]
\label{lem:unique-primary-lake}
    For every $P \in \Tbold$ and $I_1, I_2 \in \tiling_P$, $\lake_P(I_1) = \lake_P(I_2)$ if and only if $I_1 = I_2$.
\end{lemma}

\begin{proof}
    Let us pick any two distinct intervals $I_1,I_2 \in \tiling_P$ and suppose for a contradiction that $\lake_P(I_1)$ and $\lake_P(I_2)$ are equal to the same lake $\lake$.
    Let us label $I_1 =[\crit_{-R_1^-},\crit_{-R_1^+}]$ and $I_2 =[\crit_{-R_2^-},\crit_{-R_2^+}]$ for some times $R_1^\pm$, $R_2^\pm \in \Tbold_{\leq P}$.
    Assume without loss of generality that $R_1^+ < R_2^-$ and that $\{\crit_{-R_1^+}, \crit_{-R_2^-}\}$ is contained in the interval $[\crit_{-R_1^-}, \crit_{-R_2^+}]$.
    
    Let $J$ be the unique tile in $\tiling_{R_1^+}$ that contains $I_1$.
    Then, $\lake$ is contained in $\lake_{R_1^+}(J)$ and so the real boundary of $\lake_{R_1^+}(J)$ also contains the interval $I_2$.
    From Lemma \ref{lem:primary-lakes-and-adjacent-bubbles}, the boundary of $\lake_{R_1^+}(J)$ contains an ``interval`` $A$ of $\partial^c \Bubb_{R_1^+}$ with endpoints $\crit_{-R_1^+}$ and $\alpha_{R_1^+}$.
    Then, the preimage of $\Fext^{R_1^+}(I_2)$ under $\Fext^{R_1^+}: \lake_{R_1^+}(J) \to \UHP$ contains both $I_2$ and a subinterval of $A$; this contradicts the fact that $\Fext^{R_1^+}: \lake_{R_1^+}(J) \to \UHP$ is a conformal isomorphism.
\end{proof}

\begin{remark}
    The lemma above implies that $\parRess \lake_P(I) = I$. 
    We would like to emphasize the possibility that $\parR \lake(P,I)$ contains some points outside of $I$.
    We will see later how this arises from parabolic phenomena.
\end{remark}

\subsection{Pseudo-bubble chains}
\label{ss:bubble-chains}

To understand in detail the structure of the boundary of lakes, we need to first understand the geometry of pseudo-bubble chains.

\begin{definition}
    A \emph{hoglet chain} of $\Fext$ is a sequence of hoglets $\Gamma_0, \Gamma_1, \Gamma_2, \ldots, $ such that each $\Gamma_{j+1}$ has generation greater than $\Gamma_{j}$ and the root of $\Gamma_{j+1}$ is contained in $\Gamma_{j}$.
    A hoglet chain $\{ \Gamma_j \}_{j \geq 0}$ is called \emph{periodic} with some period $P \in \Tbold$ if $\Fext^P(\Gamma_{j+1}) = \Gamma_j$ for all $j \geq 0$.
\end{definition}

For $n \in \Z$, consider the $Q_n$-periodic hoglet chain
\[
    \Gamma_{n,0}, \quad \Gamma_{n,1}, \quad \Gamma_{n,2}, \quad \Gamma_{n,3}, \quad \ldots
\]
where $\Gamma_{n,0} = \Bubb_{Q_n}$ is the principal hoglet of generation $Q_n$.
Each $\Gamma_{n,j}$ is a hoglet of generation $(j+1)Q_n$ and it is contained in the pseudo-bubble of the same generation which we will denote by $\hat{\Gamma}_{n,j}$.

For every NP depth $n$, we consider the principal parabolic fjord $\fjord_n$ of depth $n$ and the corresponding $\Fbold^{Q_n}$-periodic point $\beta_n$ in $\fjord_n$ described in Lemma \ref{lem:repelling-periodic-point}.
The fjord $\fjord_n$ is almost invariant under $\Fext^{Q_n}$ and it is attached to the tile $[\crit_{-Q_{n-1}},\crit_{-Q_n}] \in \tiling_{Q_n}$ containing $0$. 

\begin{proposition}[Bounds for $\hat{\Gamma}_{n,j}$]
\label{prop:bubble-chain-01}
    There exist universal constants $K>1$ and $\mathbf{m} \in \N$ such that the following properties hold whenever the combinatorial threshold $\threshold$ is sufficiently large.
    \begin{enumerate}
        \item For $n \in \Z$ and $1\leq j \leq a_{n+1}$, the pseudo-bubble $\hat{\Gamma}_{n,j}$ is contained in the hyperbolic disk of radius $K$ centered at the alpha-point $\alpha_{jQ_n+Q_{n-1}}$.
        \item Let $\threshold'$ be the constant from \S\ref{ss:near-parabolic-levels}. 
        If $n \in Z$ is an NP depth, then 
        \begin{enumerate}[label = \textnormal{(\alph*)}]
            \item when $\threshold' + \mathbf{m} \leq j \leq a_{n+1} -\threshold' - \mathbf{m}$, the pseudo-bubble $\hat{\Gamma}_{n,j}$ is contained in the fjord $\fjord_n$;
            \item when $0 \leq j \leq \threshold' - \mathbf{m}$ or $j \geq a_{n+1} -\threshold' + \mathbf{m}$, the pseudo-bubble $\hat{\Gamma}_{n,j}$ is disjoint from the fjord $\fjord_n$.
        \end{enumerate}
    \end{enumerate}
    When depth $n$ is parabolic, property \textnormal{(1)} holds for all $j \geq 1$, property \textnormal{(2)(a)} holds for $j \geq \threshold' + \mathbf{m}$, and property \textnormal{(2)(b)} holds for $0 \leq j \leq \threshold'-\mathbf{m}$.
\end{proposition}

\begin{figure}
        \centering
       \begin{tikzpicture}
    \node[anchor=south west,inner sep=0] (image) at (0,0) {\includegraphics[width=1\linewidth]{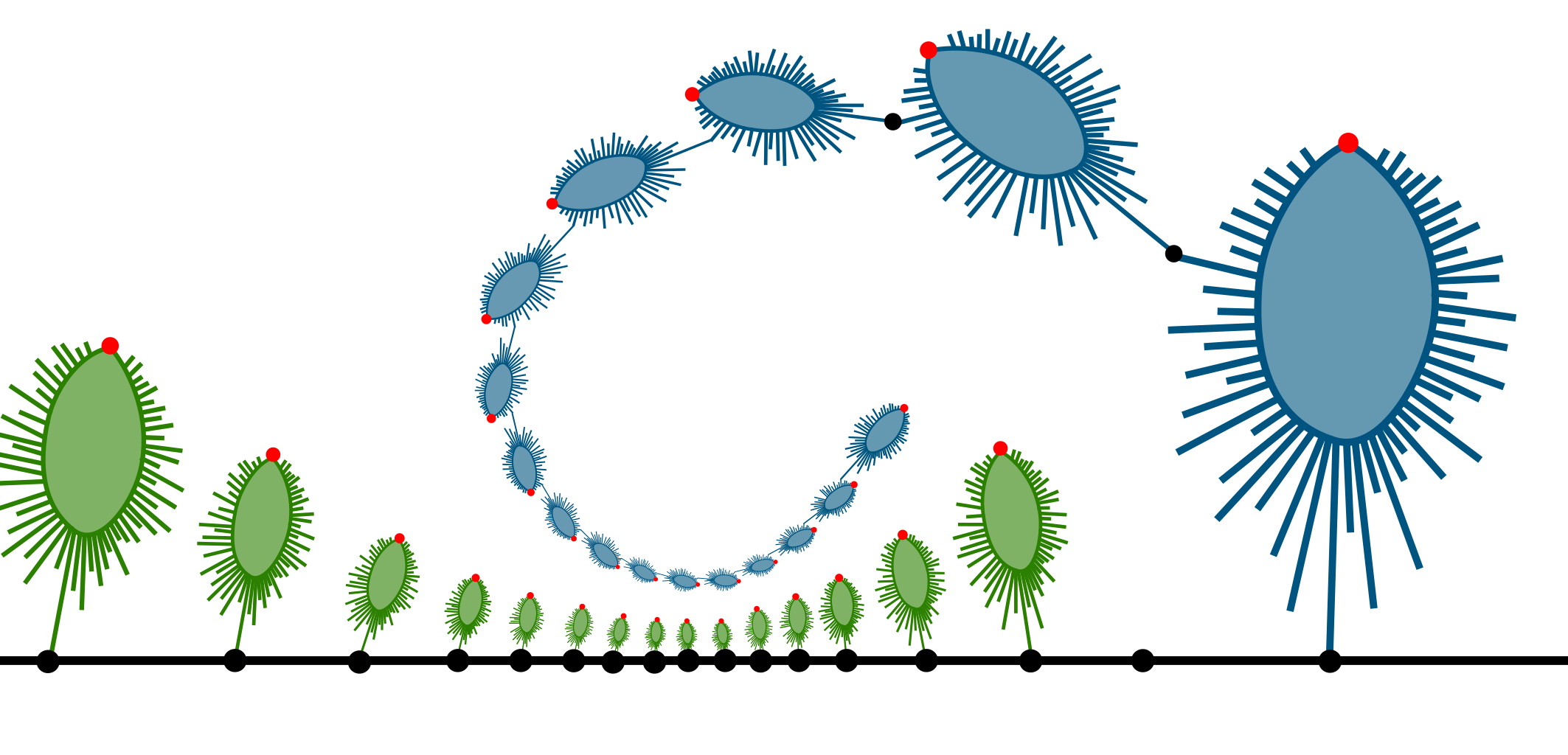}};
    \begin{scope}[
        x={(image.south east)},
        y={(image.north west)}
    ]
        \node [blue!50!black] at (0.95,0.79) {\scalebox{0.9}{$\Gamma_{n,0}$}};
        \node [blue!50!black] at (0.74,0.91) {\scalebox{0.9}{$\Gamma_{n,1}$}};
        \node [blue!50!black] at (0.5,0.97) {\scalebox{0.9}{$\Gamma_{n,2}$}};
        \node [blue!50!black] at (0.38,0.86) {\scalebox{0.9}{$\Gamma_{n,3}$}};
        \node [blue!50!black] at (0.56,0.52) {\scalebox{0.85}{$\Gamma_{n,a_{n+1}}$}};
        
        \node [black] at (0.05,0.05) {$\crit_{-Q_{n-1}-Q_n}$};
        \node [black] at (0.65,0.05) {$\crit_{-Q_{n+1}}$};
        \node [black] at (0.73,0.05) {$0$};
        \node [black] at (0.865,0.05) {$\crit_{-Q_{n}}$};

        \node [black] at (0.25,0.7) {$\Fext^{Q_n}$};
        \draw[-latex] (0.29,0.62) .. controls (0.27,0.66) and (0.28,0.75) .. (0.34,0.78);
        \draw[-latex] (0.25,0.34) .. controls (0.23,0.48) and (0.2,0.48) .. (0.18,0.44);

        \node [black] at (0.72,0.625) {\scalebox{0.9}{$\mathtt{r}_{n,1}$}};
        \node [black] at (0.56,0.78) {\scalebox{0.9}{$\mathtt{r}_{n,2}$}};
    \end{scope}
\end{tikzpicture}
    \caption{The hoglet chain $\Gamma_{n,j}$ follows the alpha-points (in red) of primary hoglets of generation $jQ_n + Q_{n-1}$ for $j = 1,\ldots,a_{n+1}$}
    \label{fig:bubble-chain-01}
\end{figure}
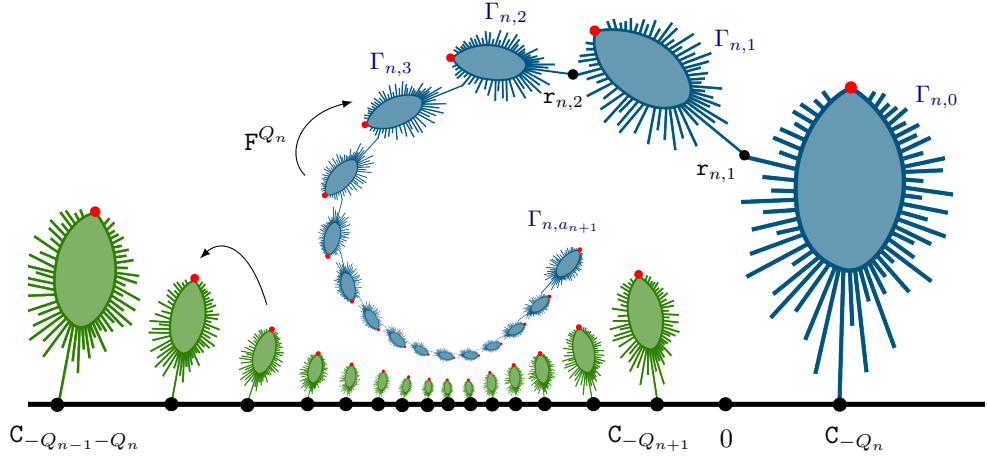

See Figure \ref{fig:bubble-chain-01} for an illustration.

\begin{proof}
    In the proof, we will fix $n \in \Z$ and assume that it is non-parabolic. 
    The limiting parabolic case is analogous.
    For $j \in \{0,1\ldots, a_{n+1}\}$, we will denote by $\mathtt{r}_{n,j}$ the root of the hoglet $\Gamma_{n,j}$.

    Let $m \in \Z$ be the largest integer such that $Q_{[m-1]} < Q_{n} \leq Q_{[m]}$.
    In the dynamical plane of $\Fbold$, the boundary of the sector $\Sbold^m$ contains the critical point $C_{-Q_n}$.
    Let $\Sbold^m_{Q_n}$ be the unique lift of $\Sbold^m$ under $\Fbold^{Q_n}$ that contains $C_{-Q_n}$.
    Then, the boundary of $\Psi(\Sbold^m_{Q_n} \backslash \Hbold)$ contains the root $\mathtt{r}_{n,1}$.
    By Lemma \ref{lem:estimate-lift-sector}, we have that 
    \[
    |\mathtt{r}_{n,1} - \crit_{-Q_n}| \succ |\crit_{-Q_n}|.
    \]
    Together with Real Bounds, Pseudo-Bubble Bounds, and Corollary \ref{cor:location-of-alpha}, we have that 
    \[
        \dist_{\UHP}(\mathtt{r}_{n,1},\alpha_{Q_n+Q_{n-1}}) = O(1).
    \]
    This estimate, together with the hyperbolic diameter bound on $\hat{\Gamma}_{n,1}$ (Proposition \ref{prop:bubble-bounds} (3)), implies property (1) for $j=1$.

    Let us prove (1) for $j \geq 2$.
    For $j \in \{1,\ldots,a_{n+1}\}$, let $\lake_{n,j}$ be the unique lake of generation $j Q_n$ that touches $\R$ along the interval $A_{n,j} = [\crit_{-Q_{n-1}-(j-1)Q_n},\crit_{-Q_n}] \in \tiling_{jQ_n}$,
    which exists by virtue of Lemma \ref{lem:primary-lake}.
    Notice that $A_{n,j} \cup \Fext^{Q_n}(A_{n,j})$ is a subset of $A_{n,j-1}$ for $2 \leq j \leq a_{n+1}$.
    Therefore, for $2 \leq j \leq a_{n+1}$, 
    \begin{itemize}
        \item $\lake_{n,j}$ is contained in $\lake_{n,j-1}$, and 
        \item $\Fext^{Q_n}: \lake_{n,1} \to \UHP$ sends $\lake_{n,j}$ onto $\lake_{n,j-1}$.
    \end{itemize}
    Moreover, notice that the lake $\lake_{n,1}$ contains the primary pseudo-bubble $\hat{\Bubb}_{Q_n+Q_{n-1}}$ as well as the secondary pseudo-bubble $\hat{\Gamma}_{n,1}$.
    Inductively, we conclude that for $j \in \{1,\ldots,a_{n+1}\}$, 
    \begin{itemize}
        \item $\overline{\lake_{n,j}}$ contains the primary pseudo-bubbles $\hat{\Bubb}_{k Q_n+ Q_{n-1}}$ for $j\leq k \leq a_{n+1}$, and
        \item $\lake_{n,j}$ contains the pseudo-bubble chain $\{ \hat{\Gamma}_{n,k}\}_{j \leq k \leq a_{n+1}}$. 
    \end{itemize}
    Observe that for $j \in \{2,\ldots,a_{n+1}\}$, the univalent map $\Fext^{(j-1) Q_n}: \lake_{n,j-1} \to \UHP$ sends $\hat{\Gamma}_{n,j}$ onto $\hat{\Gamma}_{n,1}$ and $\alpha_{jQ_n+Q_{n-1}}$ to $\alpha_{Q_n+Q_{n-1}}$. Therefore, by Schwarz Lemma, property (1) for $j=1$ implies (1) for $2\leq j \leq a_{n+1}$. 

    Suppose that $n$ is an NP depth.
    Recall that the principal parabolic fjord $\fjord_n$ is enclosed by the semicircular dam $\mathtt{d}_n$ with endpoints $\crit_{-Q_{n+1}+m_1 Q_n}$ and $\crit_{-Q_{n-1}-m_2 Q_n}$ for some integers $m_1 \in \{ \threshold'-1,\threshold'\}$ and $m_2 \in \{ \threshold', \threshold'+1\}$.
    According to Property (1) and Corollary \ref{cor:location-of-alpha}, there exists uniform constants $K''>0$ and $\ttau'' \in \left( 0,\frac{\pi}{2} \right)$ such that for $j \in \{1,\ldots, a_{n+1}\}$, the pseudo-bubble $\hat{\Gamma}_{n,j}$ is contained in the set
    \[
        \Xi_j = \left\{ w + \crit_{-Q_{n-1}-jQ_n} \: : \: \left| \arg(w) - \frac{\pi}{2} \right| < \ttau'', |w| \leq K'' |I_{n,j}| \right\}.
    \]
    By Real Bounds, there exists a uniform constant $\mathbf{m}_{\textnormal{out}} \in \N$ such that for $j \in \{1,\ldots, \threshold'-\mathbf{m}_{\textnormal{out}}\} \cup \{ a_{n+1}-\threshold'+\mathbf{m}_{\textnormal{out}},\ldots, a_{n+1}\}$, the set $\Xi_j$ is disjoint from $\mathtt{d}_n$.
    Real Bounds also imply the existence of another uniform constant $\mathbf{m}_{\textnormal{in}} \in \N$ such that whenever $\threshold' + \mathbf{m}_{\textnormal{in}} \leq j \leq a_{n+1}-\threshold'-\mathbf{m}_{\textnormal{in}}$, the set $\Xi_j$ avoids $\mathtt{d}_n$.
    Then, the constant $\mathbf{m} = \max\{ \mathbf{m}_{\textnormal{out}}, \mathbf{m}_{\textnormal{in}} \}$ yields property (2).
\end{proof}

In the parabolic case, the pseudo-bubbles $\hat{\Gamma}_{n,j}$ form a parabolic basin.

\begin{lemma}
\label{lem:parabolic-basin}
    Let $n$ be a parabolic depth and let $\beta_n$ be the parabolic periodic point from Lemma \ref{lem:repelling-periodic-point}.
    \begin{enumerate}
        \item The infinite pseudo-bubble chain $\{\hat{\Gamma}_{n,j}\}_{j \geq 0}$ lands at $\beta_n$, that is,
        \[
            \bigcap_{m \geq 0} \overline{ \bigcup_{j \geq m} \hat{\Gamma}_{n,j} } = \{\beta_n \}.
        \]
        \item The open disk $\Dext_n$ enclosed by $[\beta_n, \crit_{-Q_n}] \cup \{\Gamma_{n,j}\}_{j \geq 0}$ is the immediate attracting basin of $\beta_n$ as a fixed point of $\Fext^{Q_n}$. 
    \end{enumerate}
\end{lemma}

See Figure \ref{fig:bubble-chain-parabolic}.

\begin{proof}
    By Real Bounds, the critical points $\crit_{-Q_{n-1}-jQ_n}$ converge to $\beta_n$ as $j \to \infty$. 
    According to Pseudo-Bubble Bounds and Corollary \ref{cor:location-of-alpha}, we have
    \begin{align*}
        |\text{Re} (\alpha_{Q_{n-1} + j Q_n} - \crit_{-Q_{n-1} - jQ_n})|
        &= O( |\crit_{-Q_{n-1} - jQ_n} - \crit_{-Q_{n-1} - (j-1) Q_n}|), \\
        \text{Im} (\alpha_{Q_{n-1}+ j Q_n}) &\asymp |\crit_{-Q_{n-1} - jQ_n} - \crit_{-Q_{n-1} - (j-1) Q_n}|.
    \end{align*}
    Therefore, the alpha-points $\alpha_{Q_{n-1}+jQ_n}$ also converge to $\beta_n$. 
    Together with Proposition \ref{prop:bubble-chain-01} (1), this observation implies the landing of $\hat{\Gamma}_{n,j}$'s at $\beta_n$. 
    
    Observe that the map $\Fext^{Q_n}$ acts as a conformal automorphism of the domain $\Dext_n$. 
    Indeed, the interval $[\beta_n,\crit_{-Q_n}]$ is mapped onto the subinterval $[\beta_n,0]$, the subinterval of $\partial^c \Dext_n$ with endpoints $\mathtt{r}_{n,1}$ and $\crit_{-Q_n}$ is mapped onto the real interval $[0,\crit_{-Q_n}]$, and for $j \geq 1$, the subinterval of $\partial^c \Dext_n$ with endpoints $\mathtt{r}_{n,{j+1}}$ and $\mathtt{r}_{n,{j}}$ is mapped onto the subinterval of $\partial^c \Dext_n$ with endpoints $\mathtt{r}_{n,{j}}$ and $\mathtt{r}_{n,{j-1}}$ (where $\mathtt{r}_{n,0} = \crit_{-Q_n})$.
    Therefore, by Denjoy-Wolff, every point in $\Dext_n$ is attracted to the parabolic fixed point $\beta_n$. The disk $\Dext_n$ is the whole immediate attracting basin of $\Fext^{Q_n}$ as a map onto $\UHP$ because every point in $\UHP \cap \partial \Dext_n$ is eventually mapped to the hoglet $\Bubb_{Q_n}$ on which $\Fext^{Q_n}$ is not well-defined. 
\end{proof}

When $n$ is a parabolic depth, we will also consider the $Q_n$-periodic hoglet chain
\[
 \ldots, \quad \Gamma_{n,\iinfty-3}, \quad \Gamma_{n,\iinfty-2}, \quad \Gamma_{n,\iinfty-1}, \quad \Gamma_{n,\iinfty}
\]
that is constructed as follows.
For every $j \geq 1$ and $k \geq 0$, the hoglet $\Gamma_{n,\iinfty-k}$ is the image of $\Gamma_{n,j}$ under the inverse branch of $\Fext^{Q_{n+1}-(j+k)Q_n-Q_{n-1}}$ that sends $\crit_{-Q_{n-1}-j Q_n}$ to $\crit_{-Q_{n+1}-kQ_n}$. 
See Figure \ref{fig:bubble-chain-parabolic}.
Each $\Gamma_{n,\iinfty-k}$ is rooted at a point $\mathtt{r}_{n,\iinfty-k}$ in $\Gamma_{n,\iinfty-k-1}$.
We will denote the corresponding pseudo-bubble by $\hat{\Gamma}_{n,\iinfty-k}$.

\begin{figure}
        \centering
       \begin{tikzpicture}
    \node[anchor=south west,inner sep=0] (image) at (0,0) {\includegraphics[width=0.9\linewidth]{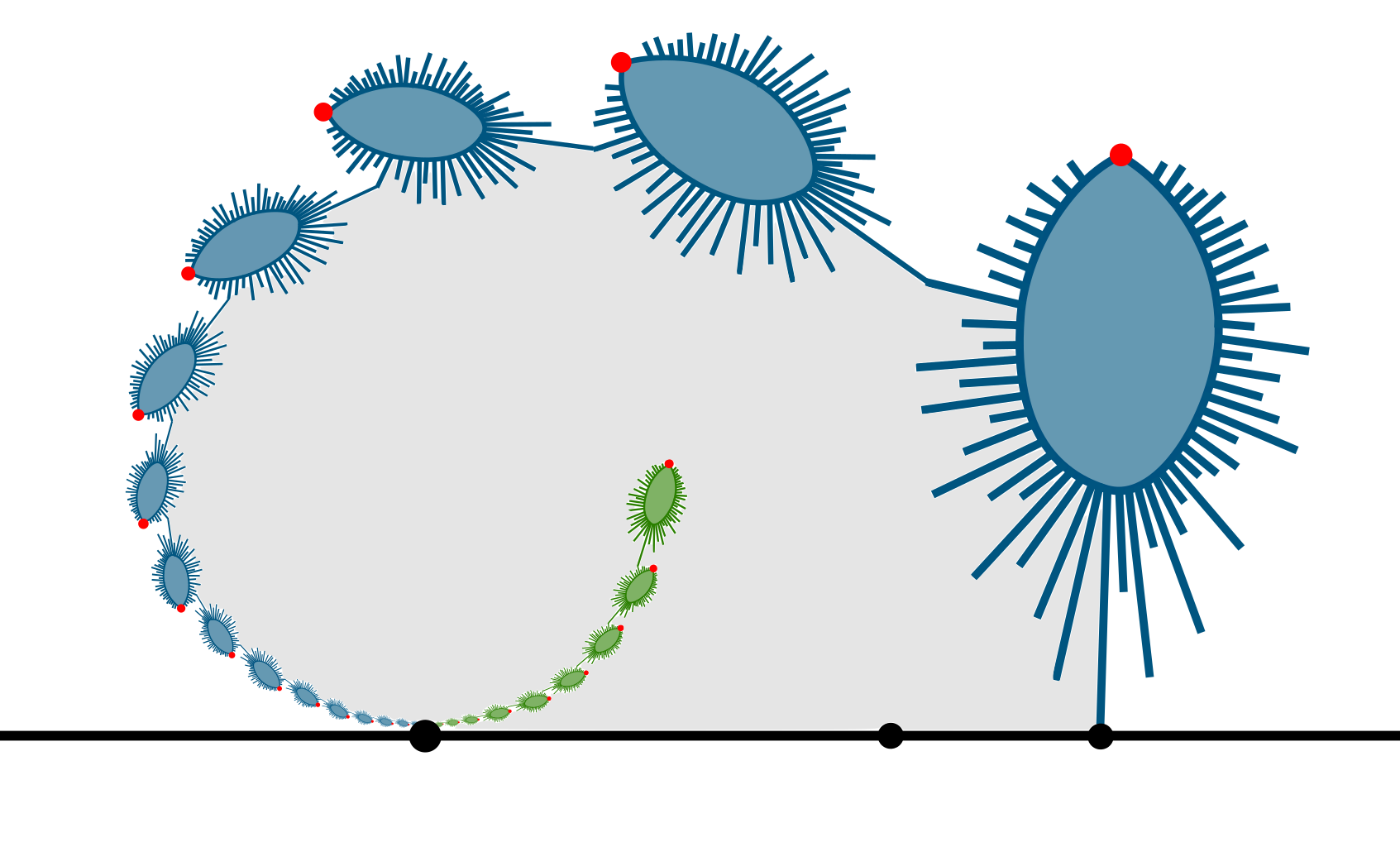}};
    \begin{scope}[
        x={(image.south east)},
        y={(image.north west)}
    ]
        \node [black] at (0.3,0.45) {\scalebox{1.1}{$\Dext_n$}};
        \node [blue!50!black] at (0.9,0.81) {$\Gamma_{n,0}$};
        \node [blue!50!black] at (0.63,0.93) {$\Gamma_{n,1}$};
        \node [blue!50!black] at (0.31,0.97) {$\Gamma_{n,2}$};
        \node [blue!50!black] at (0.16,0.82) {$\Gamma_{n,3}$};
        \node [green!50!black] at (0.535,0.41) {$\Gamma_{n,\iinfty}$};
        \node [green!50!black] at (0.515,0.255) {\scalebox{0.85}{$\Gamma_{n,\iinfty-1}$}};
        
        \node [black] at (0.64,0.06) {$0$};
        \node [black] at (0.81,0.06) {$\crit_{-Q_{n}}$};
        \node [black] at (0.3,0.06) {$\beta_n$};

        \node [black] at (0.015,0.65) {$\Fext^{Q_n}$};
        \draw[-latex] (0.08,0.58) .. controls (0.05,0.6) and (0.06,0.73) .. (0.12,0.72);
        \node [black] at (0.38,0.343) {\scalebox{0.9}{$\Fext^{Q_n}$}};
        \draw[-latex] (0.41,0.28) .. controls (0.4,0.32) and (0.38,0.3) .. (0.39,0.23);
    \end{scope}
\end{tikzpicture}
    \caption{The hoglets $\Gamma_{n,\iinfty-j}$, $j \geq 0$ inside the basin $\Dext_n$ of the parabolic periodic point $\beta_n$.}
    \label{fig:bubble-chain-parabolic}
\end{figure}
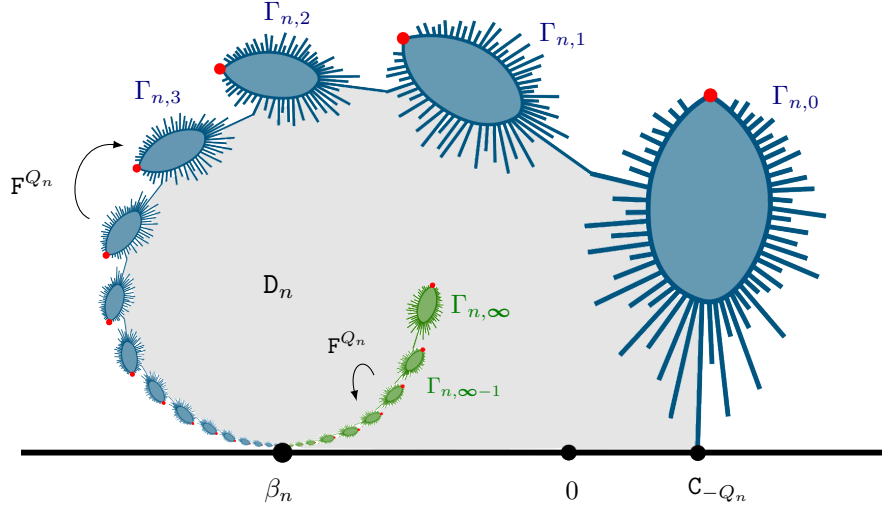

\begin{proposition}[Bounds for $\hat{\Gamma}_{n,\iinfty-k}$]
\label{prop:bubble-chain-enriched}
    Let $n$ be a parabolic depth. 
    There exist universal constants $K>1$ and $\mathbf{m} \in \N$ such that the following properties whenever the combinatorial threshold $\threshold$ is sufficiently large.
    \begin{enumerate}
        \item For every $k \geq 0$, the pseudo-bubble $\hat{\Gamma}_{n,\iinfty-k}$ is contained in the hyperbolic disk of radius $K$ centered at the alpha-point $\alpha_{Q_{n+1}-kQ_n}$.
        \item When $k \geq \threshold' + \mathbf{m}$, the pseudo-bubble $\hat{\Gamma}_{n,\iinfty-k}$ is contained in the fjord $\fjord_n$.
        When $k \leq \threshold' - \mathbf{m}$, the pseudo-bubble $\hat{\Gamma}_{n,\iinfty-k}$ is disjoint from $\fjord_n$.
        \item The pseudo-bubbles $\{ \hat{\Gamma}_{n,\iinfty-k} \}_{k \geq 0}$ accumulate at $\beta_n$ and are contained in the parabolic basin $\Dext_n$ of $\beta_n$.
    \end{enumerate}
\end{proposition}

\begin{proof}
    Properties (1) and (2) are analogous to Proposition \ref{prop:bubble-chain-01}.
    Below, we will show how (3) follows from (1).
    According to Pseudo-Bubble Bounds and Corollary \ref{cor:location-of-alpha}, we have
    \begin{align*}
        |\text{Re} (\alpha_{Q_{n+1}-kQ_n} - \crit_{-Q_{n+1}+kQ_n})|
        &= O( |\crit_{-Q_{n+1}+kQ_n} - \crit_{-Q_{n+1}+(k+1) Q_n}|), \\
        \text{Im} (\alpha_{Q_{n+1}-kQ_n}) &\asymp |\crit_{-Q_{n+1}+kQ_n} - \crit_{-Q_{n+1}+(k+1) Q_n}|.
    \end{align*}
    Since $\crit_{-Q_{n+1}+k Q_n}$ is a monotone sequence converging to $\beta_n$ as $k \to \infty$, then $\alpha_{Q_{n+1}-kQ_n}$ also converges to $\beta_n$.
    Together with property (1), the pseudo-bubble $\hat{\Gamma}_{n,\iinfty-k} = \Fext^{kQ_n}(\hat{\Gamma}_{n,\iinfty})$ also converges to $\beta_n$ as $k \to \infty$.
    Then, (3) follows.
\end{proof}

For $n \in \Z$, consider the principal (enriched) hoglet chain
\begin{align}
\label{eqn:hoglet-chain-old}
    \Bchain_n := 
    \begin{cases}
        \displaystyle \bigcup_{j=1}^{a_{n+1}} \Gamma_{n,j} & \text{ if } a_{n+1}<\infty \\
        \displaystyle \bigcup_{i=1}^{\infty} \Gamma_{n,i} \cup \bigcup_{j=0}^{\infty} \Gamma_{n,\iinfty-j} & \text{ if } a_{n+1} = \infty.
    \end{cases}
\end{align}
Let us denote by $\hat{\Bchain}_n$ the corresponding pseudo-bubble chain.

\begin{lemma}[Bounds for $\hat{\Bchain}_n$]
\label{lem:estimate-enriched-pseudo-bubble-chain}
    There exists an $\threshold$-uniform constant $K'>0$ such that for all $n \in \Z$,
    \[
        \hat{\Bchain}_n \subset \D_{\UHP}(\alpha_{Q_n}, K') \cup \fjord_n.
    \]
\end{lemma}

Here, the principal fjord $\fjord_n$ is taken to be the empty set if depth $n$ is not NP.

\begin{proof}
    For convenience, we will assume that $n$ is a non-parabolic depth.
    By Real Bounds, Pseudo-Bubble Bounds, and Corollary \ref{cor:location-of-alpha}, there exists a universal constant $K_0>0$ such that
    \begin{equation}
    \label{eqn:alpha-points-relation}
        \{ \alpha_{Q_{n-1}+Q_n}, \alpha_{Q_{n+1}}\} \subset \D_{\UHP}(\alpha_{Q_n}, K_0).
    \end{equation}
    Consider the constants $K$ and $\mathbf{m}$ from Proposition \ref{prop:bubble-chain-01}.
    Then, we have that for every $s \geq 1$ with $s \leq \lceil a_{n+1}/2\rceil$, we have
    \begin{equation}
        \bigcup_{j=1}^{s} \hat{\Gamma}_{n,j} \cup 
        \bigcup_{j= a_{n+1}-s}^{a_{n+1}} \hat{\Gamma}_{n,j}
        \subset \D_{\UHP} \left(\alpha_{Q_n}, s K + K_0 \right).
        \tag{$\clubsuit_s$}
    \end{equation}
    We will set 
    \[
        K' := K_0 + K \cdot \max\left\{ \Nbold , \threshold'+\mathbf{m}-1 \right\}, \qquad \text{ where } \Nbold  = \left\lceil \frac{\threshold}{2} \right\rceil. 
    \]

    Firstly, suppose that depth $n$ is not NP, i.e. $a_{n+1} < \threshold$.
    Then by ($\clubsuit_{\lceil a_{n+1}/2\rceil}$), we have that 
    \[
        \hat{\Bchain}_n \subset \D_{\UHP}\left( \alpha_{Q_n}, \left\lceil \frac{\threshold}{2} \right\rceil K + K_0 \right).
    \]
    Secondly, suppose instead that depth $n$ is NP.
    We know from Proposition \ref{prop:bubble-chain-01} that the union of $\hat{\Gamma}_{n,j}$ across $j \in \{ \threshold' + \mathbf{m},\ldots, a_{n+1}-\threshold'-\mathbf{m} \}$ is contained in $\fjord_n$.
    By ($\clubsuit_{\threshold'+\mathbf{m}-1}$), the remainder of the pseudo-bubble chain is contained in $\D_{\UHP}\left(\alpha_{Q_n}, (\threshold'+\mathbf{m}-1) K + K_0 \right)$.
    Lastly, the limiting parabolic case is analogous to the general NP case by virtue of Proposition \ref{prop:bubble-chain-enriched}.
\end{proof}

\subsection{Geometric bounds for primary lakes}
\label{ss:lake}

We are now ready to establish uniform bounds on the size of particular primary lakes.
According to Lemmas \ref{lem:primary-lake} and \ref{lem:unique-primary-lake}, for every $P \in \Tbold$, there exists a unique primary lake of $\Fext$ of generation $P$ that contains the critical value $0$ on its boundary.
We will denote such a lake by $\lake_P$.

\begin{proposition}[Uniform bounds for $\lake_{Q_n+Q_{n+1}}$]
\label{prop:lake-bounds-0}
    There exists a universal constant $K>0$ such that for every $n \in \Z$, 
    \[
        \parR \lake_{Q_n + Q_{n-1}} = [\crit_{-Q_n-Q_{n-1}},\crit_{-Q_n}]
    \]
    and $\parH \lake_{Q_n +Q_{n-1}}$ is the union of the following non-empty sets:
    \begin{enumerate}
        \item a subset of $\partial \Bubb_{Q_n}$;
        \item a subset of $\partial \Bubb_{Q_n + Q_{n-1}}$;
        \item a subset of the hyperbolic disk $\D_{\UHP}(\alpha_{Q_n+Q_{n-1}}, K)$.
    \end{enumerate}
    Moreover,
    \[
    \partial \lake_{Q_n+Q_{n-1}} \cap \Iext^{\leq Q_n+Q_{n-1}} = \{\alpha_{Q_n+Q_{n-1}}\}.
    \]
\end{proposition}

In particular, this proposition states that $\lake_{Q_n+Q_{n-1}}$ is bounded in $\C$.

\begin{corollary}
\label{cor:bounded-lakes}
    Every lake of $\Fext$ is a bounded subset of $\C$.
\end{corollary}

\begin{proof}
    Consider any lake $\lake$ of some generation $P \in \Tbold$. There exists some time $P' \in \Tbold$ with $P' < P$ such that $\lake'=\Fext^{P'}(\lake)$ is a primary lake. 
    Then, pick $P'' \in \Tbold$ such that $P'' < P-P'$ and that the boundary of $\lake'' = \Fext^{P''}(\lake')$ contains $0$. 
    Pick $n \in \Z$ such that $Q_n + Q_{n-1} < P-P'-P''$. 
    Then, the lake $\lake''$ is contained in $\lake_{Q_n+Q_{n-1}}$, 
    hence $\lake''$ is bounded. 
    By $\sigma$-properness, the lift $\lake$ is also bounded.
\end{proof}

Before going into details, let us first sketch the main idea of the proof of the proposition.

The key step is illustrated in Figure \ref{fig:lake-boundary}.
For every $n$, the primary lake $\lake_{Q_{n-1} + Q_{n}}$ is the image of $\lake_{Q_{n+1}+Q_{n+2}}$ under $\Fext^{Q_{n+2}-Q_n}$ followed by $\Fext^{Q_{n+1}-Q_{n-1}}$.
Pulling back $\lake_{Q_{n-1} + Q_{n}}$ by the map $\Fext^{Q_{n+1}-Q_{n-1}}$ creates a new primary lake whose upper-half boundary goes along a new chain of hoglets, namely $\Bchain_n$.
We prove that under the other map $\Fext^{Q_{n+2}-Q_{n+1}}$, $\Bchain_n$ lifts to a hoglet chain along the boundary of $\lake_{Q_{n+1}+Q_{n+2}}$ that is within a definite hyperbolic disk about the alpha-point $\alpha_{Q_{n+2}+Q_{n+1}}$.
Then, we use the uniform hyperbolic expansion of $\Fext^{Q_{n+2}-Q_n} \circ \Fext^{Q_{n+1}-Q_{n-1}}$ at $\alpha_{Q_{n+2}+Q_{n+1}}$ and iterate this procedure to prove the proposition.

\begin{proof}[Proof of Proposition \ref{prop:lake-bounds-0}]
    For every $n \in \Z$, 
    the unique tile of $\tiling_{Q_n+Q_{n-1}}$ that contains $0$ is the interval 
    \[
    J_n :=[\crit_{-Q_n-Q_{n-1}}, \crit_{-Q_n}].
    \]
    By Lemma \ref{lem:unique-primary-lake}, $J_n$ is the essential real boundary of $\lake_{Q_n+Q_{n-1}}$. 
    By Lemma \ref{lem:primary-lakes-and-adjacent-bubbles}, under the map $\Fext^{Q_n+Q_{n-1}}: \lake_{Q_n+Q_{n-1}}\to \UHP$, the real interval $[\crit_{Q_{n-1}},\crit_{Q_{n-1}-Q_{n-2}}]$ lifts to a connected subset of the Carath\'eodory boundary of the hoglet $\Bubb_{Q_n}$ 
    and the connected component of $\R \backslash \{0\}$ that avoids $\crit_{Q_{n-1}}$ lifts to a connected subset of the Carath\'eodory boundary of the hoglet $\Bubb_{Q_n+Q_{n-1}}$ that starts from the root and ends at the alpha-point.
    This implies (1) and (2).
    
    Denote
    \[
        X_n := \partial \lake_{Q_n+Q_{n-1}} \backslash 
        \left( \Bubb_{Q_n+Q_{n-1}} \cup J_n \cup \Bubb_{Q_{n}}
        \right).
    \]
    Below, we will show that $X_n$ contains $\alpha_{Q_n+Q_{n-1}}$ and every other point is in a hyperbolic disk centered at $\alpha_{Q_n+Q_{n-1}}$ with radius bounded above by a uniform constant. 
    This will prove the rest of the proposition.
    To lighten the notation, we will just prove it for $n=0$.
    \vspace{0.05in}

    \noindent \underline{Claim 1:} For every even negative integer $m$, we can write 
    \[
    X_{m+2} \subset \Bchain_m^2 \cup X_m^2
    \]
    where $X_m^2$ is a subset of $X_{m+2}$ that is mapped under $\Fext^{(Q_{m+2}-Q_m)+(Q_{m+1}-Q_{m-1})}$ onto $X_{m}$, and $\Bchain_m^2$ is the lift of the hoglet chain $\Bchain_m$ under $\Fext^{Q_{m+2}-Q_{m}}$ that is attached to the principal hoglet $\Bubb_{Q_{m+2}}$.
    \vspace{0.05in}

    The proof is illustrated in Figure \ref{fig:lake-boundary}.

\begin{figure}
    \centering
    \begin{tikzpicture}
    \node[anchor=south west,inner sep=0] (image) at (0,0) {\includegraphics[width=1\linewidth]{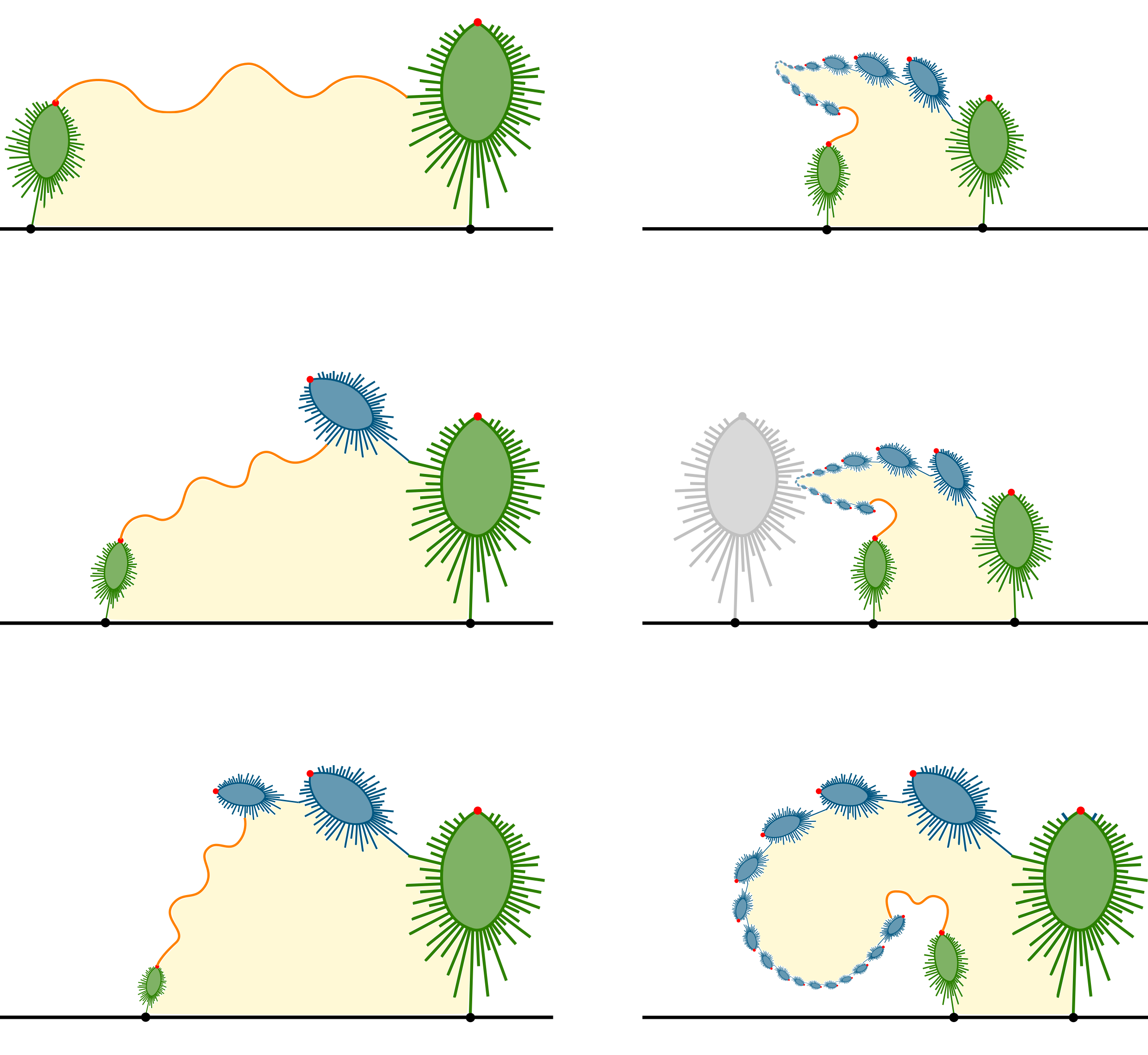}};
    \begin{scope}[
        x={(image.south east)},
        y={(image.north west)}
    ]
        \node [yellow!30!black] at (0.22,0.84) {$\lake_{Q_m+Q_{m-1}}$};
        \node [yellow!30!black] at (0.28,0.48) {$\Omega_{m,1}$};
        \node [yellow!30!black] at (0.275,0.1) {$\Omega_{m,2}$};
        \node [yellow!30!black] at (0.76,0.185) {\scalebox{0.8}{$\Omega_{m,a_{m+1}}$}};
        \node [yellow!30!black] at (0.79,0.805) {\scalebox{0.7}{$\lake_{Q_{m+2}+Q_{m+1}}$}};
        
        \node [green!50!black] at (0.5,0.92) {\scalebox{0.9}{ $\Gamma_{m,0}$}};
        \node [blue!50!black] at (0.36,0.64) {\scalebox{0.9}{ $\Gamma_{m,1}$}};
        \node [blue!50!black] at (0.165,0.23) {\scalebox{0.9}{$\Gamma_{m,2}$}};
        \node [blue!50!black] at (0.63,0.21) {\scalebox{0.9}{ $\Bchain_{m}$}};
        \node [blue!50!black] at (0.65,0.92) {\scalebox{0.9}{ $\Bchain_m^2$}};

        \node [black] at (0.058,0.755) {\scalebox{0.9}{$\crit_{-Q_{m}-Q_{m-1}}$}};
        \node [black] at (0.42,0.755) {\scalebox{0.9}{$\crit_{-Q_{m}}$}};
        \node [black] at (0.71,0.755) {\scalebox{0.9}{$\crit_{-Q_{m+2}-Q_{m+1}}$}};
        \node [black] at (0.87,0.755) {\scalebox{0.9}{$\crit_{-Q_{m+2}}$}};

        \node [black] at (0.11,0.38) {\scalebox{0.9}{$\crit_{-2Q_{m}-Q_{m-1}}$}}; 
        \node [black] at (0.42,0.38) {\scalebox{0.9}{$\crit_{-Q_{m}}$}};
        \node [black] at (0.63,0.382) {\scalebox{0.85}{$\crit_{-Q_{m+1}}$}}; 
        \node [black] at (0.77,0.382) {\scalebox{0.85}{$\crit_{-Q_m-2Q_{m+1}}$}};
        \node [black] at (0.925,0.382) {\scalebox{0.85}{$\crit_{-Q_m-Q_{m+1}}$}};

        \node [black] at (0.14,0.01) {\scalebox{0.9}{$\crit_{-3Q_{m}-Q_{m-1}}$}};
        \node [black] at (0.41,0.01) {\scalebox{0.9}{$\crit_{-Q_{m}}$}};
        \node [black] at (0.815,0.01) {\scalebox{0.9}{$\crit_{-Q_{m+1}-Q_m}$}};
        \node [black] at (0.94,0.01) {\scalebox{0.9}{$\crit_{-Q_{m}}$}};

        \node [black] at (0.265,0.7) {\scalebox{0.9}{$\Fext^{Q_m}$}};
        \draw[-latex] (0.23,0.66) -- (0.23,0.74);
        \node [black] at (0.265,0.325) {\scalebox{0.9}{$\Fext^{Q_m}$}};
        \draw[-latex] (0.23,0.285) -- (0.23,0.365);
        \node [black] at (0.55,0.16) {\scalebox{0.85}{$\Fext^{(a_{m+1}-2)Q_m}$}};
        \draw[-latex] (0.61,0.13) -- (0.48,0.13);
        \node [black] at (0.845,0.665) {\scalebox{0.9}{ $\Fext^{(a_{m+2}-1)Q_{m+1}}$}};
        \draw[-latex] (0.76,0.72) -- (0.76,0.60);
        \node [black] at (0.805,0.325) {\scalebox{0.9}{$\Fext^{Q_{m+1}}$}};
        \draw[-latex] (0.76,0.355) -- (0.76,0.285);
    \end{scope}
\end{tikzpicture}
    \caption{The creation of the hoglet chain $\Bchain_m^2$ on the boundary of the primary lake $\lake_{Q_{m+2}+Q_{m+1}}$}
    \label{fig:lake-boundary}
\end{figure}
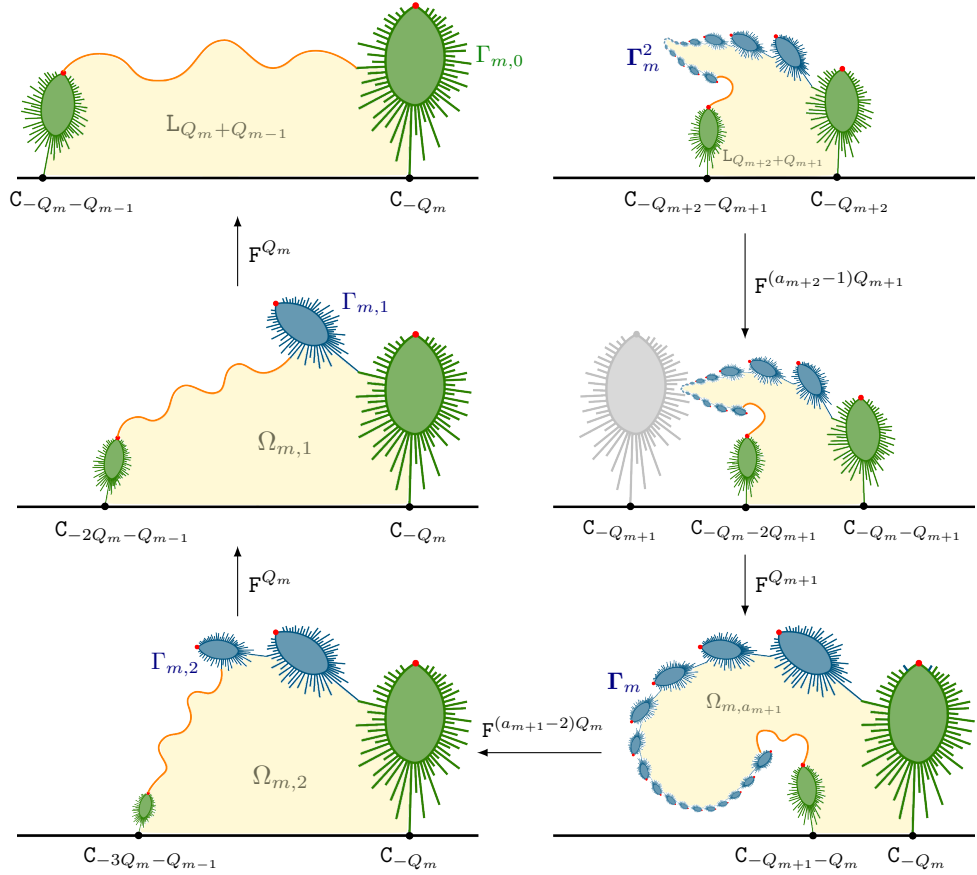

    \begin{proof}
        For convenience, we will first assume that $m$ is a non-parabolic depth. 
        For every $j \in \{0,1,\ldots, a_{m+1}\}$, denote
    \[
        \Omega_{m,j} := \Fext^{(a_{m+1}-j)Q_m + (Q_{m+2}-Q_m)}(\lake_{Q_{m+2}+Q_{m+1}}),
    \]
    which is a lake of generation $(j+1)Q_m + Q_{m-1}$.
    We have the following sequence of conformal isomorphisms of lakes:
    \begin{center}
    \begin{tikzcd}
        \lake_{Q_{m+2}+Q_{m+1}} \arrow[r,  "\Fext^{Q_{m+2}-Q_m}"] &[3.5em] \Omega_{m,a_{m+1}} \arrow[r, "\Fext^{Q_{m}}"] &[1.4em] \ldots \arrow[r, "\Fext^{Q_{m}}"] &[1.4em] \Omega_{m,1} \arrow[r, "\Fext^{Q_m}"] &[1.4em] \Omega_{m,0}             
    \end{tikzcd}
    \end{center}
    
    Observe that 
    \[
        \Fext^{(Q_{m+2}-Q_m)+(Q_{m+1}-Q_{m-1})}(J_{m+2}) \subset J_m
    \]
    and
    \[
        \Fext^{(Q_{m+2}-Q_m)+(Q_{m+1}-Q_{m-1})-Q_m}(J_{m+2}) \subset J_m.
    \]
    By Lemma \ref{lem:unique-primary-lake}, the first inclusion implies $\Omega_{m,0}$ is equal to $\lake_{Q_m+Q_{m-1}}$, and the second inclusion implies that $\Omega_{m,1}$ is contained in $\Omega_{m,0}$.
    It can be checked from elementary calculation that for $j \in \{ 1, 2,\ldots, a_{m+1}\}$,
    \[
        J_{m,j}:= [\crit_{-(j+1)Q_m-Q_{m-1}}, \crit_{-Q_m}] = \parRess \Omega_{m,j}
    \]
    and if $j \leq a_{m+1}-1$, then $0$ remains the only critical value of $\Fext^{Q_m}$ contained in it.

    Let $g_m: \Omega_{m,0} \to \Omega_{m,1}$ be the inverse of $\Fext^{Q_m} : \Omega_{m,1} \to \Omega_{m,0}$.
    Consider the hoglet chain $\{\Gamma_{m,j} \}_{0 \leq j \leq a_{m+1}}$ from the previous section.
    Under $g_m$,
    \begin{itemize}
        \item for $j \in \{ 0, \ldots, a_{m+1}-1\}$, the subinterval $[\crit_{-(j+1)Q_m-Q_{m-1}}, 0]$ of $J_{m,j}$ is mapped onto the interval $J_{m,j+1}$,
        \item the remaining subinterval $[0, \crit_{-Q_m}]$ of $J_m$ is mapped to the connected subset of the Carath\'eodory boundary of $\Gamma_{m,0}$ between $\crit_{-Q_m}$ and the root $\mathtt{r}_{m,1}$ of $\Gamma_{m,1}$, and 
        \item for $j \in \{ 1, \ldots, a_{m+1}-1\}$, the hoglet $\Gamma_{m,j}$ is mapped onto $\Gamma_{m,j+1}$.
    \end{itemize}

    The essential real boundary of $\Omega_{m,a_{m+1}}$ is equal to
    \[
        J'_{m} := [\crit_{-Q_{m+1}-Q_m}, \crit_{-Q_m}] = \parRess \Omega_{m,a_{m+1}}
    \]
    and it does not contain $0$.
    Let 
    \[
        Y_m = \parH \Omega_{m,a_{m+1}} \backslash (\Bubb_{Q_{m+1}+Q_m} \cup \Bubb_{Q_m}). 
    \]
    Then, we can write
    \[
        Y_{m} \subset \Bchain_m \cup X_{m}^1 \qquad \text{ where } \qquad X_{m}^1 := g_m^{a_{m+1}}(X_m).
    \]
    
    If depth $m$ is parabolic, we also consider the domains 
    \[
    \Omega_{m,\iinfty-j} = \Fext^{Q_{m+2}+(j-1)Q_m}(\lake_{Q_{m+2}+Q_{m+1}}) \quad \text{ for } \quad j \geq 0.
    \]
    In this case, we also have $J'_{m} = \parRess \Omega_{m,\iinfty}$. 
    Denote by $Y_m$ the part of $\parH \Omega_{m, \iinfty}$ that avoids the primary hoglets $\Bubb_{Q_{m+1}+Q_m}$ and $\Bubb_{Q_m}$, and by $X_m^1$ the lift of $X_m$ under the map $\Fext^{Q_{m+1}-Q_{m-1}}: \Omega_{m,\iinfty} \to \Omega_{m,0}$. 
    Bear in mind that, as described in \S\ref{ss:bubble-chains}, $\Bchain_m$ is now the enriched hoglet chain $\{\Gamma_{m,j}\}_{j \geq 1} \cup \{\Gamma_{m,\iinfty-j}\}_{j\geq 0}$.
    Then, $Y_m$ is still contained in $\Bchain_m \cup X_m^1$.
    
    Lastly, consider the map 
    \[
    \Fext^{Q_{m+2}-Q_m} : \lake_{Q_{m+2}+Q_{m+1}} \to \Omega_{m,a_{m+1}}.
    \]
    Observe that the interval $J'_m$ does not contain any critical value of $\Fext^{Q_{m+2}-Q_m}$. 
    Hence, under this map, $J'_m$ lifts to the interval $J_{m+2}=\parRess \lake_{Q_{m+2}+Q_{m+1}}$, the set $Y_{m}$ lifts to $X_{m+2}$, and in particular, $X_m^1$ lifts to $X_m^2$ and $\Bchain_m$ lifts to $\Bchain_m^2$.
    \end{proof}

    \noindent \underline{Claim 2:} There exists a universal constant $K>0$ such that for every even negative integer $m$,  $\Bchain_m^2$ is contained in the hyperbolic disk $\D_{\UHP}(\alpha_{Q_{m+2}+Q_{m+1}}, K)$.

    \begin{proof}
        By Real Bounds, Pseudo-Bubble Bounds, and Corollary \ref{cor:location-of-alpha}, there is a universal constant $K_0>0$ such that the alpha-points $\alpha_{Q_m}$, $\alpha_{Q_{m+1}+Q_m}$, $\alpha_{Q_{m+1}}$, and $\alpha_{2Q_{m+1}+Q_m}$ are within hyperbolic distance $K_0$ from each other.
        According to Lemma \ref{lem:estimate-enriched-pseudo-bubble-chain}, there also exists a universal constant $K_1>0$ such that 
        \[
            \Bchain_m \subset \D_\UHP(\alpha_{Q_m}, K_1) \cup \fjord_m.
        \]
        Thus,
        \[
            \Bchain_m \subset \D_\UHP(\alpha_{Q_m+Q_{m+1}}, K_0+K_1) \cup \fjord_m.
        \]
        By Schwarz Lemma, the lift of $\D_\UHP(\alpha_{Q_{m+1}+Q_m}, K_0+K_1)$ under the map $\Fext^{Q_{m+2}-Q_m}: \lake_{Q_{m+2}+Q_{m+1}} \to \Omega_{m,a_{m+1}}$
        is contained in $\D_{\UHP}(\alpha_{Q_{m+2}+Q_{m+1}}, K_0+ K_1)$.
        If $\fjord_m$ is empty, then the desired claim follows.
        So let us suppose now that $m$ is an NP depth.
        
        An appropriate inverse branch of $\Fext^{Q_{m+1}}$ will send the fjord $\fjord_m$ into a subset $\tilde{\fjord}_m$ of $\hat{\Bubb}_{Q_{m+1}}$. 
        Let $m' \in \Z$ be such that $Q_{m+1} \leq Q_{[m']} < Q_{m+3}$.
        Since the renormalization sector $\Sbold^{m'}$ avoids the set $\Psi^{-1} (\fjord_m)$, then
        $\Psi^{-1}(\tilde{\fjord}^m)$ is outside of the lift $\Sbold_{Q_{m+1}}^{m'}$ of $\Sbold^{m'}$ under $\Fbold^{Q_{m+1}}$.
        By Lemma \ref{lem:estimate-lift-sector} and Pseudo-Bubble Bounds, $\Psi(\hat{\Bbold}_{Q_{m+1}}\backslash \Sbold^{m'}_{Q_{m+1}})$ must be contained in $\D_{\UHP}(\alpha_{Q_{m+1}}, K_2)$ for some universal constant $K_2>0$. 
        Hence, $\tilde{\fjord}_m \subset \D_{\UHP}(\alpha_{2Q_{m+1}+Q_m}, K_0+ K_2)$.
        By applying Schwarz Lemma to $\Fext^{Q_{m+2}-Q_{m+1}-Q_m}$, we deduce that the part of $\Bchain_m^2$ that gets mapped to $\fjord_m$ is also contained in $\D_{\UHP}(\alpha_{Q_{m+2}+Q_{m+1}}, K_0 + K_2)$.
    \end{proof}

    For any two distinct even integers $k,m$ with $m<k\leq 0$, let us denote by $\Bchain_{m}^{k-m}$ and $X_{m}^{k-m}$ the images of $\Bchain_m^2$ and $X_{m}^{2}$ respectively under the inverse branch of $\Fext^{Q_k+Q_{k-1}-Q_{m+2}-Q_{m+1}}: \lake_{Q_k+Q_{k-1}} \to \lake_{Q_{m+2}+Q_{m+1}}$. As we apply Claim 1 iteratively, we have
    \[
        X_{k} \, \subset \, \Bchain_{k-2}^2 \cup \Bchain_{k-4}^4 \cup \Bchain_{k-6}^6 \cup \ldots \cup \Bchain_{m+2}^{k-m-2} \cup \Bchain_{m}^{k-m} \cup X_{m}^{k-m}.
    \]

    \noindent \underline{Claim 3:} There exists some universal constants $K>0$ and $\lambda \in (0,1)$ such that for any two distinct even integers $k,m$ with $m<k\leq 0$, the hoglet chain $\Bchain_m^{k-m}$ is contained in $\D_{\UHP}(\alpha_{Q_k+Q_{k-1}}, K \lambda^{k-m})$.

    \begin{proof}
        By Real Bounds, Pseudo-Bubble Bounds, and Corollary \ref{cor:location-of-alpha}, for all $m$, the hyperbolic distance between $\alpha_{Q_{m+2}+Q_{m+1}}$ and the primary hoglet $\Bubb_{Q_{m+2}+Q_{m+1}-Q_m-Q_{m-1}}$, in particular the boundary of the primary lake $\lake_{Q_{m+2}+Q_{m+1}-Q_m-Q_{m-1}}$ containing $\alpha_{Q_{m+2}+Q_{m+1}}$, is uniformly bounded above. 
        Let $K>0$ be the constant from Claim 2. Then, the inverse branch of 
        \[
        \Fext^{Q_{m+2}+Q_{m+1}-Q_m-Q_{m-1}}: \lake_{Q_{m+2}+Q_{m+1}-Q_m-Q_{m-1}} \to \UHP
        \]
        contracts the hyperbolic metric of $\UHP$ on the hyperbolic disk $\D_\UHP(\alpha_{Q_m+Q_{m-1}},K)$ at some uniform rate $\lambda^2 \in (0,1)$. Then, the claim follows from induction.
    \end{proof}

    To complete the proof of the proposition, we will apply Claim 3 to $k=0$. For $j \geq 1$, the hoglet chain $\Bchain_{-2j}^{2j}$ will be of hyperbolic distance $K \lambda^{2j}$ away from $\alpha_{Q_0+Q_{-1}}$, and the limiting set $\bigcap_{s \geq 1} \overline{\bigcup_{ j \geq s } \Bchain_{-2j}^{2j}}$ has to be the singleton $\{\alpha_{Q_0+Q_{-1}}\}$.
\end{proof}

The proof of Proposition \ref{prop:lake-bounds-0} also gives us a uniform estimate of the primary lake $\Fext^{Q_{n-1}}(\lake_{Q_n+Q_{n-1}})$ of generation $Q_n$ for all $n$.

\begin{proposition}
\label{prop:lake-bounds-1}
    There exists a universal constant $K > 0$ such that for every $n \in \Z$,
    the lake $\Fext^{Q_{n-1}}(\lake_{Q_n+Q_{n-1}})$ satisfies the following properties. 
    We have
    \begin{align*}
        \parR \Fext^{Q_{n-1}}(\lake_{Q_n+Q_{n-1}}) &= [\crit_{-Q_n}, \crit_{-Q_n+Q_{n-1}}]
    \end{align*}
    and $\parH \Fext^{Q_{n-1}}(\lake_{Q_n+Q_{n-1}})$ is the union of the following non-empty sets:
    \begin{enumerate}
        \item a subset of $\partial \Bubb_{Q_n}$;
        \item a subset of $\partial \Bubb_{Q_n-Q_{n-1}}$;
        \item a subset of the hyperbolic disk $\D_{\UHP} (\alpha_{Q_n}, K)$.
    \end{enumerate}
    Moreover,
    \[
        \partial \Fext^{Q_{n-1}}(\lake_{Q_n+Q_{n-1}}) \cap \Iext^{\leq Q_n} = \{\alpha_{Q_n}\}.
    \]
\end{proposition}
For other primary lakes, the proof of Proposition \ref{prop:lake-bounds-0} gives us the following.

\begin{proposition}
\label{prop:lake-bounds-2}
    There exists a universal constant $\kappa > 0$ such that for every primary lake $\lake$ of the form $\Fext^P(\lake_{Q_{n+1}+Q_n})$ for some $0 \leq P < Q_{n+1}+ Q_n - Q_{n-1}-Q_{n-2}$, there exists a proper arc $\gamma = \gamma(\lake)$ in $\lake$ such that if we write $\parRess \lake = [\crit_{-R_1}, \crit_{-R_2}]$, then
    \begin{enumerate}
        \item $\gamma$ starts from a point on $\Bubb_{R_1}$ and ends at $\Bubb_{R_2}$, and
        \item $\gamma$ is contained in the horizontal line $\{\textnormal{Im}(z) = \kappa |\crit_{-R_1} - \crit_{-R_2}| \}$.
    \end{enumerate}
\end{proposition}


\section{Dynamical hairiness and the absence of invariant line fields}
\label{sec:hairiness-NILF}

Consider a neutral cascade $\Fbold = (\Fbold^P)_{P \in \Tbold}$ corresponding to a bi-infinite tower $\seq{f_n}_{n\in\Z} \in \attr$, and let $\tttheta=\seq{ (\varepsilon_n, \abar_n) }_{n \in \Z}$ be the associated combinatorics.
We will be using the standard notation set in the previous sections; see \S\ref{sss:notation-for-combinatorics}--\ref{sss:cascade}.
In this section, we will explore the global structure of the dynamical plane of $\Fbold$ and prove the following three theorems.

\begin{theorem}[Dynamical Hairiness]
\label{thm:no-wandering-domains}
    The grand orbit of the Mother Hedgehog of $\Fbold$ is dense in $\C$. Moreover,
    \begin{itemize}
        \item if $\Fbold$ is eventually Brjuno, then the Fatou set of $\Fbold$ is equal to the grand orbit of the interior of the Mother Hedgehog $\Hbold$;
        \item otherwise, the Julia set of $\Fbold$ is equal to $\C$.
    \end{itemize}
\end{theorem}

In particular, this theorem establishes the absence of wandering domains and shows that $\Dom(\Fbold^P)$ is dense for all $P \in \Tbold$.
It is an analog of the Dynamical Hairiness phenomenon for infinitely renormalizable ql maps \cite{McM96, Hin00}.

For any $P,Q \in \Tbold$, we write $P \precsim Q$ if there exists a positive integer $m$ such that $P \leq m Q$.
For any $P \in \Tbold$, the \emph{Archimedean class} of $P$ is the sub-semigroup of $\Tbold$ consisting of elements $Q$ such that $P \precsim Q \precsim P$.
The number of Archimedean classes of $\Tbold$ is one plus the number of parabolic depths.
In particular, if $\tttheta$ is irrational, then there is only one Archimedean class.

For any $n \in \Z$, denote
\[
    \Tbold_{\Arch,n} := \{ P \in \Tbold \: : \: P \precsim Q_n \}.
\]
Below is a variant of Theorem \ref{thm:no-wandering-domains}.

\begin{theorem}[Hairiness and parabolic basins]
\label{thm:parabolic-fatou-set}
    Suppose $N \in \Z$ is a parabolic depth. 
    The Fatou set of $(\Fbold^P)_{P \in \Tbold_{\Arch,N}}$ is the grand orbit of the lift of the parabolic basin of attraction of $f_N$. 
    It is the union of iterated preimages of $\Hbold(\Fbold)$ under $(\Fbold^P)_{P \in \Tbold_{\Arch,N}}$ together with the open set of points which correspond to the full attracting basin of the parabolic periodic point $\beta_N$ of $\Fext^{Q_N}$ in external coordinates.
\end{theorem}

Together with the two theorems above, we will also establish:

\begin{theorem}[NILF]
\label{thm:NILF}
    The neutral cascade $\Fbold$ admits no invariant line field supported on its Julia set.
\end{theorem}

A central conjecture in rational dynamics is the absence of invariant line fields, which implies the density of hyperbolicity conjecture. 
In our framework, an invariant line field of a cascade $\Fbold$ is defined to be a measurable Beltrami differential $\nu(z) \frac{d \bar{z}}{dz}$ on $\C$ such that $\left(\Fbold^P\right)^*\nu=\nu$ almost everywhere for all $P \in \Tbold$, $|\nu|=1$ on a positive measure set, and $\nu=0$ elsewhere. 
Theorem \ref{thm:NILF} essentially states that there is no non-trivial deformation of $\Fbold$ from the Julia set.

The leading mechanism behind these theorems is the property that every point outside of hoglets admits infinitely many iterates with definite expansion relative to the scale (Remark \ref{rem:expanding.iterates}).
This is achieved by iterating Propositions \ref{prop:escaping-fjord} and \ref{prop:uniform-expansion}.
In Proposition \ref{prop:escaping-fjord}, we prove that every point that does not escape a particular parabolic fjord $\fjord$ is equal to either the corresponding periodic point $\beta_\fjord$ or lies in the parabolic basin of attraction of $\beta_\fjord$; the proof uses the hoglet chain and lake analysis from Section \ref{sec:bounds}.
In Proposition \ref{prop:uniform-expansion}, we prove that at any point $z \in \UHP$ away from fjords and primary hoglets, there is some time $P_z$ such that $\Fext^{P_z}$ is uniformly expanding at $z$ with respect to the hyperbolic metric.

This expansion mechanism, together with definite non-linearity (depicted in Figure \ref{fig:non-linearity}), implies the main Lemma \ref{lem:non-linearity-2} which asserts that almost every point $z \in \UHP$ that is not contained in any hoglet picks up definite non-linearity from a nearby hoglet at arbitrarily small scales.
In \S\ref{ss:proof-3-theorems}, we apply this lemma and prove the three theorems \ref{thm:no-wandering-domains}, \ref{thm:parabolic-fatou-set}, and \ref{thm:NILF} simultaneously.

\subsection{Escaping fjords}
\label{ss:parabolic-levels}

Recall that in external coordinates, every fjord $\fjord \subset \Fjord$ contains a unique fixed point $\beta_{\fjord}$ of $\Fext^{Q_n}$ where $n$ is the depth of $\fjord$.
For the principal fjord $\fjord_n$, the corresponding fixed point is denoted by $\beta_n$.
In this subsection, we shall prove the following proposition.

\begin{proposition}
\label{prop:escaping-fjord}
    Let $z \in \UHP \cap \Fjord$ and suppose $n \in \Z$ is the largest NP depth such that $z$ is contained in a depth $n$ fjord.
    Let $\fjord$ be the depth $n$ fjord containing $z$.
    If $\Fext^P(z) \in \Fjord$ for all $P \in \Tbold_{\Arch,n}$, then either
    \begin{enumerate}
        \item $a_{n+1} < \infty$ and $z$ is equal to the repelling periodic point $\beta_\fjord$, or
        \item $a_{n+1} = \infty$, $z$ is in the immediate parabolic basin of $\beta_\fjord$, and there exists some $P \in \Tbold_{\Arch,n+1} \backslash \Tbold_{\Arch,n}$ such that $\Fext^{P}(z)$ is not in $\Fjord$.
    \end{enumerate}
\end{proposition}

Part (1) follows easily from Denjoy-Wolff.

\begin{proof}[Proof of Proposition \ref{prop:escaping-fjord} \textnormal{(1)}]
    After replacing $z$ with $\Fext^P(z)$ for some time $P \in \Tbold$ with $0 \leq P < Q_n$, we assume that $z$ is contained in the principal parabolic fjord $\fjord_n$.
    In particular, $\fjord_n$ is not contained in a larger fjord and so $\Fext^{jQ_n}(z)$ is in $\fjord_n$ for all $j \geq 0$.
    When $a_{n+1}<\infty$, the periodic point $\beta_n$ in $\fjord_n$ is contained in the principal lake $\lake_{Q_n}$.
    By Denjoy-Wolff, it is the unique attracting fixed point of the inverse of the map $\Fext^{Q_n}: \lake_{Q_n} \to \UHP$.
    Therefore, we must have $z = \beta_n$.
\end{proof}

For the proof of (2), we will make use of the $Q_n$-periodic hoglets discussed previously in \S\ref{ss:bubble-chains}. 
Suppose depth $n$ is parabolic. 
According to Lemma \ref{lem:parabolic-basin}, 
the hoglet chain $\Gamma_{n,j}$, $j \geq 0$ lands at $\beta_n$ and defines the immediate basin of attraction $\Dext_n$ of $\beta_n$ of the map $\Fext^{Q_n}$. 
Inside of $\Dext_n$, there also exists a bi-infinite hoglet chain $\Gamma_{n,\iinfty+j}$, $j \in \Z$ with the property that 
\[
    \Fext^{Q_n}(\Gamma_{n,\iinfty+j}) = \Gamma_{n,\iinfty+j-1}
    \qquad \text{ for all } j \in \Z
\]
and
\[
    \Fext^{Q_{n+1}-Q_{n-1}}(\Gamma_{n,\iinfty+j}) = \Gamma_{n,j}
    \qquad \text{ for all } j \geq 0.
\]
On both ends, this hoglet chain lands at $\beta_n$, hence it bounds a sub-domain $\Dext'_n$ of $\Dext_n$ where
\[
    \Fext^{Q_{n+1}-Q_{n-1}}: \Dext'_n \to \Dext_n
\]
is a conformal isomorphism and $\partial \Dext'_n$ touches $\partial \Dext_n$ exactly at $\beta_n$. 
See Figure \ref{fig:double-enrichment}.

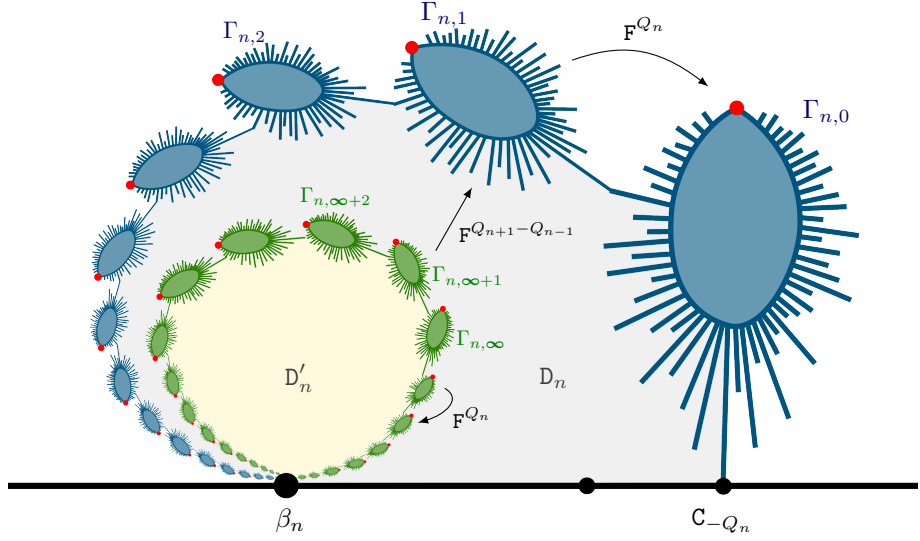
\begin{figure}
        \centering
       \begin{tikzpicture}
    \node[anchor=south west,inner sep=0] (image) at (0,0) {\includegraphics[width=0.95\linewidth]{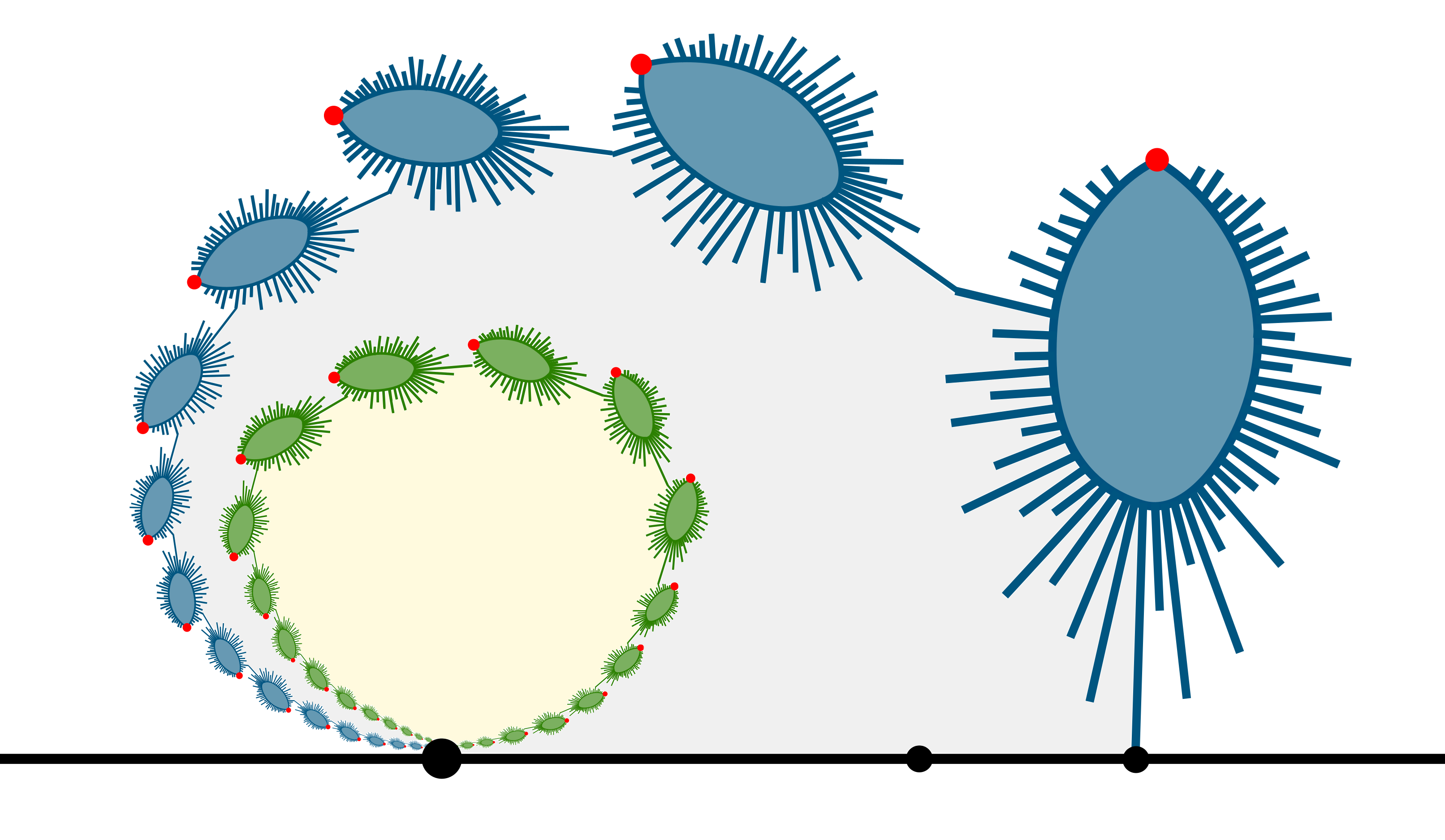}};
    \begin{scope}[
        x={(image.south east)},
        y={(image.north west)}
    ]
        \node [black] at (0.31,0.03) {$\beta_n$};
        \node [black] at (0.785,0.03) {$\crit_{-Q_n}$};
        
        \node [gray!50!black] at (0.32,0.3) {$\Dext'_n$};
        \node [gray!50!black] at (0.6,0.3) {$\Dext_n$};
        
        \draw[-latex] (0.62,0.90) .. controls (0.67,0.93) and (0.71,0.93) .. (0.77,0.85);
        \node [black] at (0.7,0.96) {\small $\Fext^{Q_n}$};
        \draw[-latex] (0.47,0.54) -- (0.51,0.66);
        \node [black] at (0.56,0.575) {\scalebox{0.85}{$\Fext^{Q_{n+1}-Q_{n-1}}$}};
        \draw[-latex] (0.48,0.28) .. controls (0.5,0.26) and (0.48,0.23) .. (0.45,0.215);
        \node [black] at (0.51,0.225) {\scalebox{0.85}{$\Fext^{Q_n}$}};
        
        \node [blue!50!black] at (0.9,0.8) {$\Gamma_{n,0}$};
        \node [blue!50!black] at (0.48,0.985) {$\Gamma_{n,1}$};
        \node [blue!50!black] at (0.26,0.95) {\scalebox{0.9}{$\Gamma_{n,2}$}};
        \node [green!50!black] at (0.52,0.37) {\scalebox{0.8}{$\Gamma_{n,\iinfty}$}};
        \node [green!50!black] at (0.505,0.49) {\scalebox{0.8}{$\Gamma_{n,\iinfty+1}$}};
        \node [green!50!black] at (0.36,0.64) {\scalebox{0.8}{$\Gamma_{n,\iinfty+2}$}};
    \end{scope}
    \end{tikzpicture}
    \caption{The bi-infinite $\Fext^{Q_n}$-periodic hoglet chain $\Gamma_{n,\iinfty+j}$, $j \in \Z$ carves the domain $\Dext'_n$ inside of the parabolic basin $\Dext_n$ of $\beta_n$.}
    \label{fig:double-enrichment}
\end{figure}

\begin{lemma}
    \label{lem:Wj's}
    There exist a universal constant $\Nbold  \in \N$ such that for any parabolic depth $n$, there exist infinite collections of disjoint arcs $\gamma_0$, $\gamma_1$, $\gamma_2$, $\ldots$ and disjoint domains $\Xext_0$, $\Xext_1$, $\Xext_2$, $\ldots$ with the following properties.
    \begin{enumerate}[label=\textnormal{(\alph*)}]
        \item Each $\gamma_j$ is a proper arc in $\UHP \backslash \cup_{k\geq 0} \Gamma_{n,k}$ that starts at $\crit_{-Q_{n-1}-j\Nbold Q_n}$ and ends at the root of the hoglet $\Gamma_{n,1+j\Nbold }$.
        \item Each $\Xext_j$ is the unique bounded open domain in $\UHP$ enclosed by the arcs $\gamma_j$, $\gamma_{j+1}$, and the hoglets $\Gamma_{n,k}$, $1+j \Nbold  \leq k \leq (j+1)\Nbold $.
        \item For $j \geq 0$, $\Fext^{\Nbold  Q_n}$ maps $\gamma_{j+1}$ onto $\gamma_j$ and $\overline{\Xext_{j+1}}$ onto $\overline{\Xext_{j}}$.
        \item $\Xext_j$'s accumulate at $\beta_n$, that is,
        \[
            \bigcap_{k \geq 1} \overline{\bigcup_{j \geq k} \Xext_j} = \{\beta_n\}.
        \]
        \item If a point $z$ is such that $\Fext^{k Q_n}(z)$ is in $\overline{\cup_{j \geq 1} \Xext_j}$ for all $k \geq 0$, then $z = \beta_n$.
        \item $\overline{\Xext_0}$ is disjoint from the principal fjord $\fjord_n$.
    \end{enumerate}
\end{lemma}

    See Figure \ref{fig:W-X-enrichment}.

\begin{proof}
    For $j \geq 1$, denote 
        \[
        I_{n,j} := [\crit_{-Q_{n-1}-(j-1)Q_n},\crit_{-Q_{n-1}-jQ_n}].
        \]
    We have
        \[
            \ldots \: \xrightarrow[]{\: \Fext^{Q_n} \:} \:  I_{n,4}
            \: \xrightarrow[]{\: \Fext^{Q_n} \:} \:  I_{n,3}
            \: \xrightarrow[]{\: \Fext^{Q_n} \:} \:  I_{n,2}
            \: \xrightarrow[]{\: \Fext^{Q_n} \:} \: I_{n,1}.
        \]
    According to Real Bounds, we have
        \begin{align}
        \label{eqn:Inj-estimate}
            |I_{n,j}| \asymp \frac{|\crit_{-Q_{n-1}}-\beta_n|}{j^2}.
        \end{align}
        
    By Proposition \ref{prop:bubble-chain-01} (1) and Corollary \ref{cor:location-of-alpha},
    there exists a uniform constant $\tau\in (0,\pi)$ such that the secondary hoglet $\Gamma_{n,1}$ is contained in the half-Poincar\'e domain $\poincare_{\tau}(I_{n,1})$.
    Then, by repeatedly pulling back under $\Fbold^{Q_n}$, Schwarz Lemma guarantees that $\Gamma_{n,j}$ is contained in $\poincare_{\tau}(I_{n,j})$ for all $j \geq 2$.
    By (\ref{eqn:Inj-estimate}), we know that the entirety of the boundary of $\Gamma_{n,1}$ is accessible from $I_{n,1}$ on the domain $\poincare_{\tau}(I_{n,1})$ with the remaining hoglets $\Gamma_{n,j}$, $j \geq 2$ removed.
    Therefore, there exists a proper arc $\gamma_0$ in $\poincare_{\tau}(I_{n,1})$ that starts from $\crit_{-Q_{n-1}}$ and ends at the root of $\Gamma_{n,1}$.

    By (\ref{eqn:Inj-estimate}), there is a universal constant $\Nbold  \in \N$ such that $\poincare_{\tau}(I_{n,1+\Nbold })$ is disjoint from $\poincare_{\tau}(I_{n,1})$.
    Let $\gamma_1$ be the lift of $\gamma_0$ under $\Fbold^{\Nbold Q_n}$ that starts from $\crit_{-Q_{n-1}-\Nbold Q_n}$ and the root of the hoglet $\Gamma_{n,1+\Nbold }$.
    Since, by Schwarz Lemma, $\gamma_1$ is contained in $\poincare_{\tau}(I_{n,1+\Nbold })$, it is disjoint from $\gamma_0$.
    This implies that the topological disk $\Xext_0 \subset \UHP$ bounded by $\gamma_0$, $\gamma_1$, and the hoglets $\Gamma_{n,j}$, $1\leq j \leq \Nbold $ is well-defined.

    By Proposition \ref{prop:bubble-chain-01} (2), we can assume that the combinatorial threshold $\threshold$ is sufficiently large and so $\fjord_n$ is sufficiently small enough to avoid the boundary of $\Xext_{n,0}$.
    By repeatedly lifting under $\Fbold^{\Nbold  Q_n}$, 
    we obtain the arcs $\gamma_j$ and domains $\Xext_j$ for all $j \geq 1$ satisfying all the desired properties.
\end{proof}

\begin{figure}
        \centering
       \begin{tikzpicture}
    \node[anchor=south west,inner sep=0] (image) at (0,0) {\includegraphics[width=0.98\linewidth]{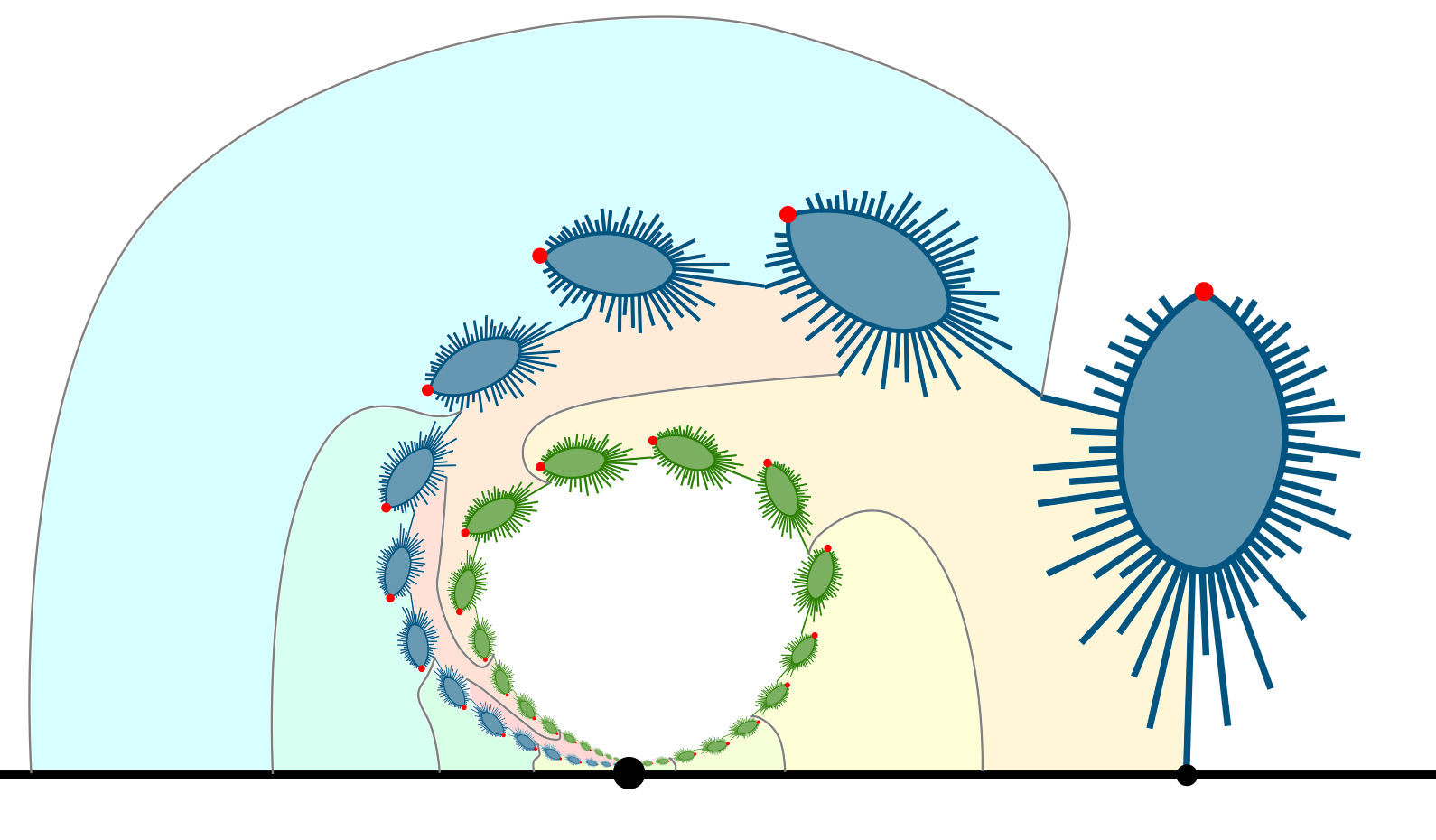}};
    \begin{scope}[
        x={(image.south east)},
        y={(image.north west)}
    ]
        \node [blue!50!black] at (0.115,0.3) {$\Xext_0$};        
        \node [blue!50!black] at (0.245,0.17) {$\Xext_1$};
        \node [yellow!30!black] at (0.615,0.2) {$\Yext_{-1}$};
        \node [yellow!30!black] at (0.695,0.33) {$\Yext_0$};
        \node [yellow!30!black] at (0.51,0.585) {$\Yext_1$};
        
        \node [black] at (0.44,0.025) {$\beta_n$};
        \node [black] at (0.835,0.03) {$\crit_{-Q_n}$};
        \node [black] at (0.03,0.03) {$\crit_{-Q_{n-1}}$};
        
        \draw[-latex] (0.68,0.48) .. controls (0.75,0.9) and (0.75,0.98) .. (0.6,0.89);
        \node [black] at (0.81,0.9) {\scalebox{0.9}{$\Fbold^{Q_{n+1}-Q_{n-1}}$}};
        \draw[-latex] (0.73,0.12) .. controls (0.69,0.04) and (0.63,0.04) .. (0.59,0.12);
        \node [black] at (0.65,0.02) {\scalebox{0.9}{$\Fbold^{\Nbold Q_n}$}};
        \draw[-latex] (0.23,0.12) .. controls (0.20,0.04) and (0.14,0.04) .. (0.11,0.12);
        \node [black] at (0.17,0.02) {\scalebox{0.9}{$\Fbold^{\Nbold Q_n}$}};
    \end{scope}
    \end{tikzpicture}
    \caption{Cartoon picture of the fundamental domains $\Xext_0, \Xext_1,\ldots$ and $\ldots, \Yext_{-1}, \Yext_0, \Yext_1,\ldots$ if $\Nbold =3$.}
    \label{fig:W-X-enrichment}
\end{figure}

Now, let us apply this lemma to prove the second part of Proposition \ref{prop:escaping-fjord}. 

\begin{proof}[Proof of Proposition \ref{prop:escaping-fjord} \textnormal{(2)}]
    Suppose $a_{n+1} = \infty$.
    Consider a point $z \in \UHP \cap \Fjord$.
    Similar to the proof of (1), we will assume that $\Fext^{jQ_n}(z)$ is in $\fjord_n$ for all $j \geq 0$.
    
    Consider the domains $\Xext_j$'s from Lemma \ref{lem:Wj's}.
    We have that
    \[
        \fjord_n \cap \UHP \subset \Dext_n \cup \bigcup_{j=1}^\infty \overline{\Xext_j} \cup \bigcup_{j=1}^\infty \Gamma_{n,j}.
    \]
    Since the primary hoglet $\Gamma_{n,0}$ is disjoint from $\fjord_n$ (Lemma \ref{lem:main-bubble-disjoint-fjord}), $z$ cannot lie in $\Gamma_{n,j}$ for any $j \geq 1$.
    By property (c) in the claim above, $z$ also cannot lie in $\overline{\Xext_j}$ for any $j \geq 1$.
    Hence, $z$ is contained in the immediate parabolic basin $\Dext_n$.

    Lastly, it remains to show that $z$ escapes $\Fjord$ after some time in $\Tbold_{\Arch,n+1} \backslash \Tbold_{\Arch,n}$.
    Consider the bi-infinite sequence of domains $\Yext_j$, $j \in \Z$ inside $\Dext_n$ with the property that both
    \[
        \Fext^{\Nbold  Q_n}: \Yext_j \to \Yext_{j-1} \quad \text{ and } \quad \Fext^{Q_{n+1}-Q_{n-1}}: \Yext_0 \to \Xext_0
    \]
    are conformal isomorphisms. See Figure \ref{fig:W-X-enrichment}.
    Observe that 
    \[
        \Dext_n = \Dext_n' \cup \bigcup_{j \in \Z} \Gamma_{\iinfty+j} \cup \bigcup_{j \in \Z} \Yext_j.
    \]
    Set $z = z_0$. 
    There are two cases to consider. 
    \begin{enumerate}[leftmargin=0.6in]
        \item[\underline{Case 1:}] $z_0$ is in $\Gamma_{n,\iinfty+j} \cup \Yext_j$ for some $j \in \Z$. 
        In this case, the map $\Fext^{Q_{n+1}-Q_{n-1}+jQ_n}$ would send $z$ to a point in either the primary hoglet $\Gamma_{n,0}$ or the set $\overline{\Xext_0}$, both of which are outside of $\fjord_n$. 
        \item[\underline{Case 2:}] $z_0$ is in $\Dext_n'$. In this case, $z_1 := \Fext^{Q_{n+1}-Q_{n-1}}(z_0)$ is in $\Dext_n$.
    \end{enumerate}
    If Case 2 occurs, we can rerun the argument for the point $z_1$.
    Now, suppose Case 2 occurs all the time, so we inductively obtain an infinite orbit $z_0, z_1,z_2,\ldots$ of $\Fext^{Q_{n+1}-Q_{n-1}}$ in $\Dext_n$. 
    Applying Denjoy-Wolff Theorem to the map $\Fext^{Q_{n+1}-Q_{n-1}}: \Dext'_n \to \Dext_n$, we conclude that $z_0$ has to be equal to the unique fixed point of $\Fext^{Q_{n+1}-Q_{n-1}}$ in $\overline{\Dext'_n}$, which is $\beta_n$.
    But this would violate the assumption that $z$ is in $\UHP$.
\end{proof}

\subsection{Uniform expansion}

In this subsection, we shall prove the following.

\begin{proposition}
\label{prop:uniform-expansion}
    Let $z$ be a point in $\UHP$ that avoids fjords $\Fjord$ and primary hoglets $\Bubb_P$ for all $P \in \Tbold$. 
    There exists some time $P_z\in \Tbold$ such that
    \begin{enumerate}
        \item $z$ is contained in a primary lake $\lake_z$ of generation $P_z$,
        \item the critical point $\crit_{-P_z}$ is an endpoint of $\partial_{\R} \lake_z$, and
        \item $\dist_{\UHP}(z, \alpha_{P_z}) = O_{\threshold}(1)$.
    \end{enumerate}
    Moreover, if $N$ is a parabolic depth and $z$ is not in the full parabolic basin of $\beta_N$, then $P_z$ is contained in $\Tbold_{\Arch,N}$.
\end{proposition}

In particular, this proposition tells us that with respect to the hyperbolic metric, the map 
\begin{equation}
    \label{eq:defin.expans}
\Fext^{P_z}: \lake_z \to \UHP
\end{equation}
is $\threshold$-uniformly expanding at $z$.

\begin{remark}[Infinitely many iterates with definite expansion]
    \label{rem:expanding.iterates} Consider any point $z\in \UHP$ outside of hoglets. Propositions~\ref{prop:escaping-fjord} (escaping fjords) and~\ref{prop:uniform-expansion} imply that there are infinitely many times $T_1<T_2<T_3<\dots$ in $\Tbold$ such that the associated images $z_n\coloneq \Fext^{T_n}(z)$ are well-defined and satisfy Proposition~\ref{prop:uniform-expansion}: the map $\Fext^{P_{z_n}}: \lake_{z_n} \to \UHP$ as in~\eqref{eq:defin.expans} is $\threshold$-uniformly expanding, where $P_{z_n}\le T_{n+1}-T_{n}$. If $z$ is in $\Iext^{\leq S}$, then $T_n\leq S$ for all $n \geq 1$.
\end{remark}

Consider the renormalization tiling $\tiling_{Q_n}$, $n \in \Z$ which is given by
\[
    \tiling_{Q_n} = 
    \{[\crit_{-P}, \crit_{-P+Q_{n-1}-Q_n} ]\}_{0<P\leq Q_{n-1}} 
    \cup 
    \{[\crit_{-P+Q_{n-1}},\crit_{-P}]\}_{0<P\leq Q_{n}-Q_{n-1}}.
\]
The core idea in the proof of Proposition~\ref{prop:uniform-expansion} is to find a tile in $\tiling_{Q_n}$ for some $n$ such that its diameter is comparable to both its distance from $z$ and the value $\textnormal{Im}(z)$.

Let us go into the technical details.
Consider lakes
\[
    \mathtt{E}'_n(P) := \lake_{Q_n+Q_{n-1}-P}[\crit_{-P}, \crit_{-P+Q_{n-1}-Q_n}] \qquad \text{ for } 0 < P \leq Q_{n-1},
\]
and
\[
    \mathtt{E}''_n(P) := \lake_{2Q_n-Q_{n-1}-P}[\crit_{-P-Q_{n-1}}, \crit_{-P}] \qquad \text{ for } 0 < P \leq Q_n - Q_{n-1}.
\]
These have the property that
\[
    \Fext^R(\mathtt{E}'_n(P+R)) = \mathtt{E}'_n(P) \qquad \text{ for } P,R \in \Tbold \text{ with } P+R \leq Q_{n-1},
\]
\[
    \Fext^R(\mathtt{E}''_n(P+R)) = \mathtt{E}''_n(P) \qquad \text{ for } P,R \in \Tbold \text{ with } P+R \leq Q_{n} - Q_{n-1}.
\]
Let us denote
\[
    \mathtt{E}_n := \bigcup_{0 < P \leq Q_{n-1}} \mathtt{E}'_n(P) \cup \bigcup_{0 < P \leq Q_{n}-Q_{n-1}} \mathtt{E}''_n(P)
\]
For every $n \in \Z$, $\mathtt{E}_n$ is contained in $\mathtt{E}_{n-1}$.

Let $\kappa >0$ be the universal constant from Proposition \ref{prop:lake-bounds-2}.
For every connected component of $\mathtt{E}_n$ of the form $\mathtt{E}'_n(P)$, there exists a unique connected component $\Xi'_n(P)$ of the intersection
\[
    \mathtt{E}'_n(P) \cap \left\{ 0 < \textnormal{Im}(z) < \kappa |\parRess \mathtt{E}'_n(P)| \right\}
\]
such that $\parR \Xi'_n(P) = \parRess \mathtt{E}'_n(P)$.
Similarly, we define $\Xi''_n(P) \subset \mathtt{E}''_n(P)$.
Denote
\[
    \Xi_n := \bigcup_{0 < P \leq Q_{n-1}} \Xi'_n(P) \cup \bigcup_{0 < P \leq Q_{n}-Q_{n-1}} \Xi''_n(P).
\]
We still have the property that $\Xi_n$ is contained in $\Xi_{n-1}$ for all $n \in \Z$.

Consider the universal angle $\ttau \in \left( 0, \frac{\pi}{2} \right)$ from Pseudo-Bubble Bounds.

\begin{lemma}
    \label{lem:estimate-Yn}
    For every $n \in \Z$ and every connected component $X$ of $\Xi_n$, $X$ is contained in the trapezium with real base $\parR X$, interior angles $\pi - \ttau$ at the two vertices on the base, and height $\kappa |\parR X|$.
\end{lemma}

\begin{proof}
    This immediately follows from the definition of $\Xi_n$ and the angular control of hoglets.
\end{proof}

\begin{lemma}
    We have
    \[
        \UHP = \bigcup_{n \in \Z} \Xi_n.
    \]
\end{lemma}

\begin{proof}
    For $n \in \N$, let $\square_n$ be the maximal open trapezium with base $[\crit_{-Q_{n-1}},\crit_{-Q_n}]$, interior angle $\ttau$ radian at both $\crit_{-Q_{n-1}}$ and $\crit_{-Q_n}$, and height $\leq \kappa |\crit_{-Q_{n-1}} - \crit_{-Q_n}|$.
    The domain $\Xi_n$ contains $\square_n$.
    As $n \to -\infty$, $|\crit_{-Q_n}|$ grows to $\infty$ uniformly exponentially.
    Hence, $\square_n$ grows to the whole plane as $n \to -\infty$.
\end{proof}

\begin{proof}[Proof of Proposition \ref{prop:uniform-expansion}]
    Let $z$ be a point in $\UHP$ that avoids fjords and primary hoglets.
    Let $n = n(z) \in \Z$ be the maximal integer such that $z$ is contained in $\Xi_n$.
    Let $X_z$ and $\mathtt{L}_z$ be the connected component of $\Xi_n$ and $\mathtt{E}_n$ containing $z$ respectively.
    Let $P_z \in \Tbold$ be the generation of $\lake_z$.
    We have
    \begin{align}
    \label{eqn:Pz-bound}
    Q_n \leq P_z < \max\{ Q_n + Q_{n-1}, 2 Q_n- Q_{n-1}\}.
    \end{align}
    When depth $n$ is NP, we denote by $\fjord_z$ be the depth $n$ parabolic fjord contained in $\mathtt{L}_z$; otherwise, $\fjord_z$ is set to be empty.

    We will assume for convenience that $a_{n+1}<\infty$; the limiting parabolic case will follow from an analogous argument.
    Let $\mathtt{x}_0, \mathtt{x}_1, \ldots, \mathtt{x}_k$ be the set of critical points of $\Fext^{Q_{n+1}}$ that are contained on $\parR X_z$, labeled in order from left to right. 
    The number $k$ is equal to either $a_{k+1}$ or $a_{k+1}+1$.
    The interval $J_z := [\mathtt{x}_0 , \mathtt{x}_k]$ is the real boundary of $X_z$.
    
    For $j \in \{1,\ldots, k\}$, let $Y_j$ be the unique connected component of $\Xi_{n+1}$ with real boundary $[\mathtt{x}_{j-1}, \mathtt{x}_{j}]$.
    By Real Bounds, we have
    \[
        |x_{j}-x_{j-1}| \asymp \frac{|J_z|}{\min\{j,k+1-j\}^2}.
    \]
    Then by Lemma \ref{lem:estimate-Yn}, there exists a universal constant $\mathbf{m}' \in \N$ such that $Y_j$ is contained in $\fjord$ for $\threshold' + \mathbf{m}' \leq j \leq a_{n+1}-\threshold' - \mathbf{m}'$.
    Altogether, since $z$ is not contained in $\fjord \cup Y_1 \cup \ldots \cup Y_k$, then there exists an $\threshold$-uniform constant $\kappa' \in (0, \kappa)$ such that
    \begin{align}
        \label{eqn:imag-estimate}
        \kappa'|J_z| < \textnormal{Im}(z) < \kappa|J_z|.
    \end{align}
    By Lemma \ref{lem:estimate-Yn}, we also have that
    \begin{align}
        \label{eqn:real-estimate}
        \dist(\textnormal{Re}(z), J_z) \leq \cot \ttau \cdot |J_z|.
    \end{align}
    Observe that the critical point $\crit_{-P_z}$ is an endpoint of $J$ (so either $\mathtt{x}_0$ or $\mathtt{x}_k$).
    The diameter of the primary hoglet $\Bubb_{P_z}$ is uniformly comparable to $|J_z|$.
    Therefore, by (\ref{eqn:imag-estimate}), (\ref{eqn:real-estimate}), and Corollary \ref{cor:location-of-alpha}, we have:
    \[
        \dist_{\UHP}(z, \alpha_{P_z}) = O_{\threshold}(1).
    \]
    
    It remains to prove the final claim in the proposition.
    Let $N$ be a parabolic depth.
    Observe that both $\mathtt{E}'_{N+1}(Q_N)$ and $\mathtt{E}''_{N+1}(Q_{N+1}-Q_N)$ are contained in the immediate parabolic basin $\Dext_{N}$ of $\beta_N$.
    Therefore, the union of primary lakes $\mathtt{E}_{N+1}$ is contained in the full parabolic basin of $\beta_N$.
    This would mean that if $z$ is not in the full parabolic basin of the parabolic periodic point $\beta_N$ of $\Fext$, then the integer $n=n(z)$ in the paragraphs above is bounded above by $N$.
    In this case, by (\ref{eqn:Pz-bound}), the time $P_z$ is contained in $\Tbold_{\Arch,n}$.
\end{proof}

\subsection{Uniform non-linearity everywhere}
The next two lemmas can be seen as an analog of \cite[Section 6]{dFdM00} from classical theory of critical circle maps.

\begin{lemma}
\label{lem:non-linearity-1}
    For every $P \in \Tbold$, there exist pointed disks $(\Uext_P, \mathtt{u}_P)$ and $(\Wext_P, \mathtt{w}_P)$ in $\UHP$ and a double branched covering map
    \[
        \mathtt{h}_P : ( \Uext_P, \mathtt{u}_P) \to (\Wext_P, \mathtt{w}_P)
    \]
    with the following properties. 
    Let $n \in \Z$ be such that $Q_n < P \leq Q_{n+1}$.
    \begin{enumerate}
        \item We have $\Fbold^{P} \circ \Psi^{-1} \circ \mathtt{h}_P \equiv \Fbold^{P+Q_n} \circ  \Psi^{-1}$.
        \item $\mathtt{u}_P$ is the critical point of $\mathtt{h}_P$ and both $\mathtt{u}_P$ and $\mathtt{w}_P$ are on the boundary of the primary hoglet $\Bubb_P$.
        \item $(\Uext_P, \mathtt{u}_P)$ and $(\Wext_P, \mathtt{w}_P)$ have uniformly bounded shape.
        \item The disks $\Uext_P$ and $\Wext_P$ have definite hyperbolic diameters:
        \[
        \diam_{\UHP}(\Uext_P) \asymp \diam_{\UHP}(\Wext_P) \asymp 1.
        \]
        \item We have
        \[
            \dist_{\UHP}(\mathtt{u}_P, \alpha_P) = O(1) 
            \quad \text{ and } \quad
            \dist_{\UHP}(\mathtt{w}_P, \alpha_P) = O(1).
        \]
    \end{enumerate}
\end{lemma}

This lemma states that primary hoglets always capture uniform non-linearity of the holomorphic dynamical system coming from $\Fbold$.

\begin{figure}
        \centering
       \begin{tikzpicture}
    \node[anchor=south west,inner sep=0] (image) at (0,0) {\includegraphics[width=\linewidth]{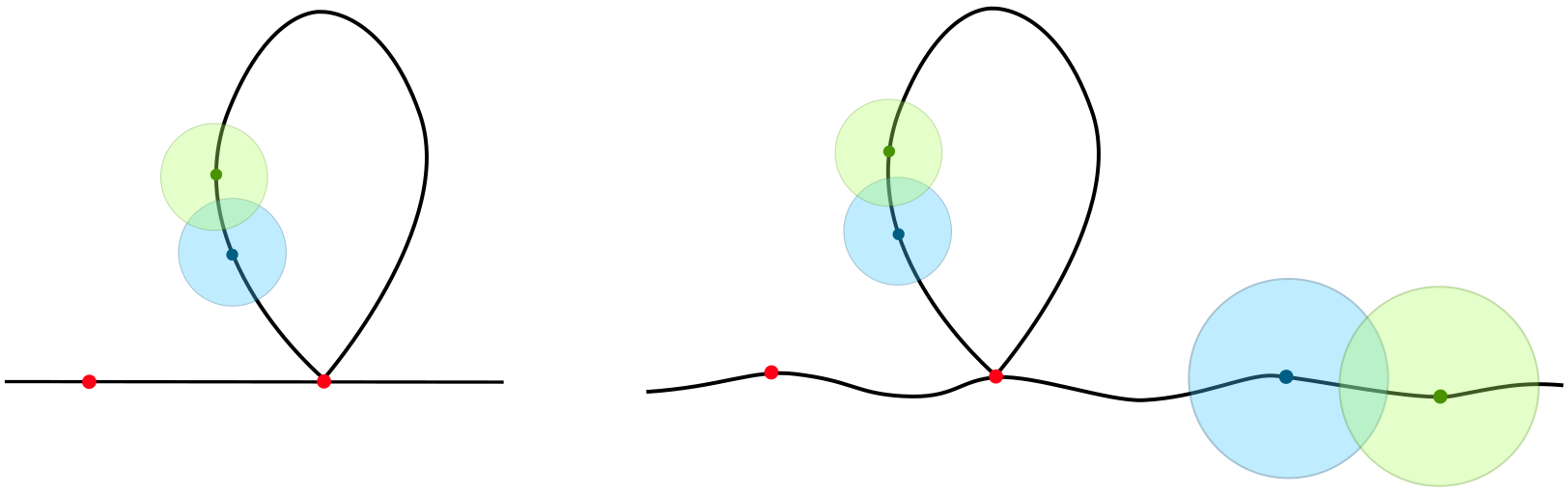}};
    \begin{scope}[
        x={(image.south east)},
        y={(image.north west)}
    ]
        \node [black] at (0.56,0.04) {\small $\hat{\Zbold}$};
        \node [black] at (0.66,0.58) {\scalebox{0.85}{$\hat{\Bbold}_{Q_n}$}};
        \node [red!90!black] at (0.493,0.15) {\small $0$};
        \node [red!90!black] at (0.65,0.135) {\small $C_{-Q_n}$};
        \node [cyan!70!black] at (0.86,0.47) {\small $W'_n$};
        \node [green!80!black] at (0.95,0.47) {\small $U'_n$};
        \draw[-latex] (0.92,0.08) .. controls (0.9,-0.09) and (0.83,-0.09) .. (0.81,0.08);
        \node [black] at (0.87,-0.12) {\small $\Fbold^{Q_n}$};
        \draw[-latex] (0.625,0.75) .. controls (0.7,0.8) and (0.75,0.75) .. (0.81,0.55);
        \node [black] at (0.75,0.81) {\small $\Fbold^{Q_n}$};
        \node [cyan!70!black] at (0.56,0.37) {\small $W''_n$};
        \node [green!80!black] at (0.55,0.85) {\small $U''_n$};
        \draw[-latex] (0.52,0.675) .. controls (0.49,0.655) and (0.49,0.515) .. (0.525,0.495);
        \node [black] at (0.475,0.56) {\small $g_n$};
        
        \node [black] at (0.35,0.87) {\small $\Psi^{-1} \circ \Fext^{P-Q_n}$};
        \draw[-latex] (0.22,0.75) .. controls (0.26,0.8) and (0.42,0.8) .. (0.46,0.75);
        \node [green!70!black] at (0.12,0.8) {\small $\Uext_P$};
        \node [cyan!80!black] at (0.13,0.34) {\small $\Wext_P$};
        \node [black] at (0.05,0.54) {\small $\mathtt{h}_P$};
        \draw[-latex] (0.095,0.67) .. controls (0.07,0.65) and (0.07,0.52) .. (0.105,0.49);
        \node [black] at (0.21,0.97) {\footnotesize $\bullet$};
        \node [black] at (0.21,1.025) {\small $\alpha_P$};
        \node [black] at (0.23,0.6) {\small $\hat{\Bubb}_P$};
        \node [red!90!black] at (0.07,0.13) {\small $\crit_{-P+Q_n}$};
        \node [red!90!black] at (0.21,0.13) {\small $\crit_{-P}$};
    \end{scope}
\end{tikzpicture}
    \caption{The construction of the branched covering map $\mathtt{h}_P: \Uext_P \to \Wext_P$ capturing definite non-linearity near the primary hoglet of generation $P$ as described in Lemma \ref{lem:non-linearity-1}.}
    \label{fig:non-linearity}
\end{figure}
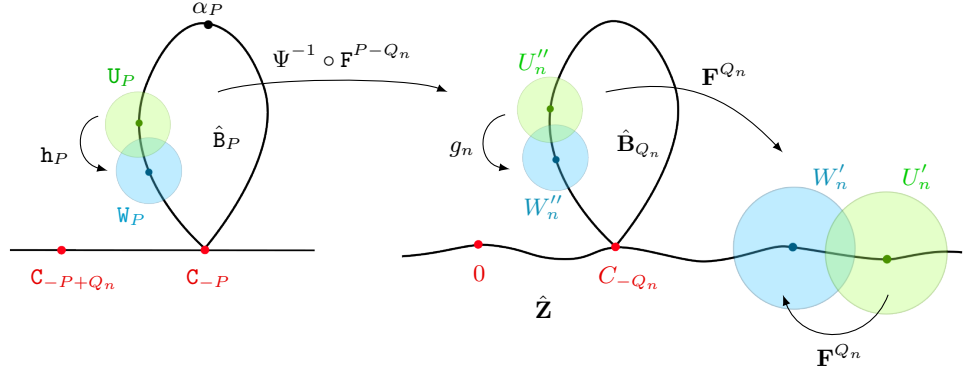

\begin{proof}
    We will first construct $h_{Q_n}$ for $P = Q_n$, and then later $h_P$ for arbitrary $P$. 
    The proof is illustrated in Figure \ref{fig:non-linearity}.
    
    Let us pick an open quasidisk neighborhood $W'_n$ of $C_{Q_{n-1}}$ such that 
    \begin{enumerate}[label=\textnormal{(\roman*)}]
        \item there is a uniform quasiconformal map of $\C$ sending $C_{Q_{n-1}}$ to $0$ and $\partial \hat{\Zbold} \cup \partial W'_n$ onto the union of $\R$ and the unit circle $\{|z|=1\}$;
        \item the distance between $W'_n$ and all the critical values of $\Fbold^{Q_n}$ with the exception of $C_{Q_{n-1}}$ is uniformly comparable to $|C_{Q_{n-1}}|$.
    \end{enumerate} 
    Property (i) is possible by virtue of pseudo-Siegel bounds (Proposition \ref{prop:trans-pseudo-siegel} (4)), and property (ii) is possible because Real Bounds uniformly control the spacing of critical values of $\Fext^{Q_n}$ away from fjords. 
    Denote by $U'_n$ the lift of $W'_n$ under $\Fbold^{Q_n}$ that contains $C_{-Q_n+Q_{n-1}}$. 
    Then, 
    \[
    \Fbold^{Q_n}: (U'_n, C_{-Q_n+Q_{n-1}}) \to (W'_n, C_{Q_{n-1}})
    \]
    is a double covering map branched at $C_{-Q_n+Q_{n-1}}$. 
    By shrinking $(W'_n,C_{Q_{n-1}})$ and $(U'_n,C_{-Q_n+Q_{n-1}})$ by a little bit, we can further assume that they have uniformly bounded shape and that there exists an annulus of modulus uniformly bounded from below that separates $W'_n \cup U'_n$ from the connected component of $\partial \hat{\Zbold} \backslash \{0\}$ that does not contain $C_{Q_{n-1}}$.
    
    The inverse branch of $\Fbold^{Q_n}$ sending $C_{Q_{n-1}}$ to a point on the hoglet $\Bbold_{Q_n}$ sends $(W'_n,C_{Q_{n-1}})$ and $(U'_n,C_{-Q_n+Q_{n-1}})$ onto pointed disks $(W''_n, w''_n)$ and $(U''_n, u''_n)$. 
    The map $\Fbold^{Q_n}$ lifts to a holomorphic double branched covering 
    \[
    g_n: (U''_n,u''_n) \to (W''_n,w''_n).
    \]
    Then, set
    \[
    (\Uext_{Q_n}, \mathtt{u}_{Q_n}) := 
    \Psi(U''_n,u''_n), \qquad 
    (\Wext_{Q_n}, \mathtt{w}_{Q_n}) := \Psi(W''_n,w''_n),
    \]
    and 
    \[
        \mathtt{h}_{Q_n} := \Psi \circ g_n \circ \Psi^{-1}: (\Uext_{Q_n}, \mathtt{u}_{Q_n}) \to (\Wext_{Q_n}, \mathtt{w}_{Q_n}).
    \]
    From the way we select $U'_n$ and $W'_n$, Koebe distortion theorem guarantees that $\Uext_{Q_n}$ and $\Wext_{Q_n}$ satisfy all the desired properties; note that property (5) follows from the estimate on the location of $\alpha_{Q_n}$ (Corollary \ref{cor:location-of-alpha}).

    Let us now consider the general case. 
    In external coordinates, let $\lake$ be the unique primary lake of generation $P-Q_n$ that contains $\crit_{-P}$ on its boundary.
    Define 
    \[
        \mathtt{h}_P: (\Uext_P, \mathtt{u}_P) \to (\Wext_P, \mathtt{w}_P)
    \]
    to be the lift of the map $\mathtt{h}_{Q_n}$ under $\Fext^{P-Q_n}: \lake \to \UHP$ such that the marked points $\mathtt{u}_P$ and $\mathtt{w}_P$ are on the boundary of the hoglet $\Bubb_{P}$.
    Then, (1) and (2) follow immediately, and (3) follows from applying Koebe distortion to the map $\Fext^{P-Q_n}: \lake \to \UHP$. 
    Moreover, Schwarz Lemma immediately gives us property (5) as well as an upper bound of the hyperbolic diameter of $\Uext_P$ and $\Wext_P$.
    
    It remains to deduce the lower bound of the hyperbolic diameter of $\Uext_P$ and $\Wext_P$. 
    To do so, we will consider the extension of $\lake$ by reflection along the real axis. 
    Let $J \subset \R$ be the interior of the unique tile in $\tiling_{P-Q_n}$ that contains $\crit_{-P}$, and let $J' = \Fext^{P-Q_n}(J)$.
    One of the endpoints of $J$ is $\crit_{-P+Q_n}$ and one of the endpoints of $J'$ is $0$.
    Let $\lake^{\textnormal{new}}$ be the union of $\lake$, its reflection in the real axis, and the interval $J$.
    Then, $\Fext^{P-Q_n}: \lake \to \UHP$ extends to a conformal isomorphism 
    \[
        \Fext^{P-Q_n} : \lake^{\textnormal{new}} \to \C \backslash (\R \backslash J').
    \]
    By virtue of Pseudo-Bubble Bounds and property (4) for $Q_n$, each of the sets $\Uext_{Q_n}$, $\Wext_{Q_n}$, and $\hat{\Bubb}_{Q_n}$ and their union have diameter $\asymp 1$ with respect to the hyperbolic metric of $\C \backslash (\R \backslash J')$.
    Therefore, the map $\Fext^{P-Q_n}: \Uext_P \cup \Wext_P \cup \hat{\Bubb}_P \to \Uext_{Q_n} \cup \Wext_{Q_n} \cup \hat{\Bubb}_{Q_n}$ has uniformly bounded distortion.
    In particular, this implies that the diameters of $\Uext_P$ and $\Wext_P$ are uniformly comparable to $\hat{\Bubb}_{P}$.
    This implies property (4).
\end{proof}

The next lemma states that almost every point in the Julia set admits uniform non-linearity at arbitrarily small scales.
It is essentially a combination of Lemma \ref{lem:non-linearity-1} and Propositions \ref{prop:escaping-fjord} and \ref{prop:uniform-expansion}, and it will serve as the key towards proving Theorems \ref{thm:no-wandering-domains}, \ref{thm:parabolic-fatou-set}, and \ref{thm:NILF}.

\begin{lemma}
\label{lem:non-linearity-2}
    Let $z \in \UHP$ be a point that is not eventually mapped into any of the beta periodic points $\beta_n$ nor any of the primary hoglets $\Bubb_P$ under $\Fext$. 
    There exist an infinite sequence $\{r_k\}_{k \geq 1}$ of positive real numbers decreasing to $0$, sequences of pointed disks $\{(\Uext_k, \mathtt{u}_k)\}_{k \geq 1}$ and $\{(\Wext_k, \mathtt{w}_k)\}_{k \geq 1}$ in $\UHP$, and a sequence of double branched covering maps 
    \[
        \big\{ \mathtt{h}_k : (\Uext_k, \mathtt{u}_k) \to (\Wext_k, \mathtt{w}_k) \big\}_{k\geq 1}
    \]
    such that for every $k \geq 1$,
    \begin{enumerate}
        \item we have $\Fbold^{S_k} \circ \Psi^{-1} \circ \mathtt{h}_k = \Fbold^{S_k + S'_k} \circ \Psi^{-1}$ for some $S_k, S'_k \in \Tbold$;
        \item $\mathtt{u}_k$ is the critical point of $\mathtt{h}_k$ and both $\mathtt{u}_k$ and $\mathtt{w}_k$ are on the boundary of a common hoglet of $\Fext$;
        \item both $(\Uext_k, \mathtt{u}_k)$ and $(\Wext_k, \mathtt{w}_k)$ have uniformly bounded shape;
        \item $\diam(\Uext_k) \asymp \diam(\Wext_k) \asymp r_k$, 
        $|z - \mathtt{u}_k| \asymp r_k$, 
        and $|z - \mathtt{w}_k| \asymp r_k$.
    \end{enumerate}
    Additionally, if $N$ is a parabolic depth and $z$ is not in the full parabolic basin of $\beta_N$, then the times $S_k$ and $S'_k$ in \textnormal{(1)} can be picked to be in $\Tbold_{\Arch,N}$.
\end{lemma}

\begin{proof}
    Let 
    \[
    \Tbold_z := \{ P \in \Tbold \: : \: z \textnormal{ is in a lake of generation }P\}.
    \]
    Denote $z_0 = z$ and for $P \in \Tbold_z$, we will denote $z_P = \Fext^P(z)$.

    If $z$ is contained in $\Fjord$, then according to Proposition \ref{prop:escaping-fjord}, there exists some time $R_0 \in \Tbold_z$ such that $z_{R_0}$ is in $\UHP \backslash \Fjord$.
    Otherwise, we will just set $R_0 = 0$.
    By Proposition \ref{prop:uniform-expansion}, there exists some other time $P_0 \in \Tbold_z$ such that $z_{R_0}$ is contained in a primary lake of generation $P_0$ and $\dist_{\UHP}(z_{R_0}, \alpha_{P_0}) = O_{\threshold}(1)$.
    As we repeat this argument inductively, 
    we obtain sequences of times $P_0, P_1, P_2,\ldots$ and $R_0, R_1, R_2,\ldots$ in $\Tbold_z$ such that for all $j \geq 0$, we have 
    \begin{enumerate}
        \item[(i)] $R_j < R_j + P_j \leq R_{j+1}$, and
        \item[(ii)] $\dist_{\UHP}(z_{R_j}, \alpha_{P_j}) = O_{\threshold}(1)$.
    \end{enumerate}
    
    For every $j \geq 0$, associated to $P_j$ is the double branched covering map 
    \[
        \mathtt{h}_{P_j} : (\Uext_{P_j}, \mathtt{u}_{P_j}) \to (\Wext_{P_j}, \mathtt{w}_{P_j})
    \]
    described in Lemma \ref{lem:non-linearity-1}. 
    By (ii), there exists a $\threshold$-uniform constant $K>0$ such that
    \begin{equation}
    \label{eqn:Usn-Wsn}
        \Uext_{P_j} \cup \Wext_{P_j} \subset \D_{\UHP}({z_{R_j}}, K).
    \end{equation}

    By (i), each $z_{R_j}$ is contained in a unique lake $\lake(j)$ of generation $R_{j+1}-R_j$.
    Since $R_{j+1}-R_j \geq P_j$, then $\alpha_{P_j}$ is outside of $\lake(j)$ and consequently (ii) implies
    \[
        \dist_{\UHP}(z_{R_j}, \partial \lake(j) ) = O_{\threshold}(1).
    \]
    By \cite[Theorem 2.25]{McM94}, this implies that the inverse $\mathtt{g}_{j+1}$ of the map $\Fext^{R_{j+1} - R_j}: \lake(j) \to \UHP$ is $\threshold$-uniformly contracting near $z_{R_{j+1}}$. 
    More precisely, there exists an $\threshold$-uniform constant $\lambda \in (0,1)$ such that for all $j \geq 0$, 
    \begin{equation}
    \label{eqn:absolute-expansion}
        \| \mathtt{g}_{j+1}'(w) \| \leq \lambda
        \quad \text{ for all }
        w \in \D_{\UHP}(z_{R_{j+1}},K).
    \end{equation}
    For convenience, we will also denote by $\mathtt{g}_{0}$ the inverse branch of $\Fext^{R_0}$ sending $z_{R_0}$ back to $z_0$.
    
    For $k \geq 0$, denote by $(\Uext_k, \mathtt{u}_k)$ and $(\Wext_k, \mathtt{w}_k)$ the image under the univalent map $\mathtt{g}_0 \circ \mathtt{g}_1 \circ \ldots \circ \mathtt{g}_k$ of the pointed disks $(\Uext_{R_k}, \mathtt{u}_{R_k})$ and $(\Wext_{R_k}, \mathtt{w}_{R_k})$ respectively.
    Then, by (\ref{eqn:Usn-Wsn}) and (\ref{eqn:absolute-expansion}),
    \[
        \Uext_k \cup \Wext_k \subset \D_{\UHP}( z, \lambda^k K )
        \qquad \textnormal{ for all } k \geq 0.
    \]
    Let $\mathtt{h}_k: \Uext_k \to \Wext_k$ be the double branched covering map that is the lift of $\mathtt{h}_{R_k}$ under $\mathtt{g}_0 \circ \ldots \circ \mathtt{g}_k$.
    Then, properties (1) and (2) are immediately satisfied from the construction, and properties (3) and (4) follow from the fact that for every $k$, the composition $\mathtt{g}_0 \circ \ldots \circ \mathtt{g}_k$ is a univalent map with uniformly bounded distortion on $\Uext_{R_k} \cup \Wext_{R_k}$.

    Lastly, suppose that $N \in \Z$ is a parabolic depth and $z$ is not in the full parabolic basin of $\beta_N$.
    Recall that each map $\mathtt{h}_k$ satisfies $\Fbold^{P_k + R_k} \circ \Psi^{-1} \circ \mathtt{h}_k = \Fbold^{P_k + R_k + Q_{n_k}} \circ \Psi^{-1}$ where $n_k \in \Z$ is such that $Q_{n_k} < P_k \leq Q_{n_k+1}$.
    According to Propositions \ref{prop:uniform-expansion} and \ref{prop:escaping-fjord}, the times $P_k$ and $R_k$ can be picked to be in $\Tbold_{\Arch,N}$.
    Hence, $Q_{n_k}$ is also contained in $\Tbold_{\Arch,N}$.
\end{proof}

\subsection{Proof of Theorems \ref{thm:no-wandering-domains}, \ref{thm:parabolic-fatou-set}, and \ref{thm:NILF}}
\label{ss:proof-3-theorems}

According to Lemma \ref{lem:non-linearity-2}, almost every point outside of the grand orbit of the Mother Hedgehog $\Hbold$ is an accumulation point of hoglets of $\Fbold$. 
Hence, the grand orbit of $\Hbold$ is dense in the plane and every connected component of the Fatou set is eventually mapped into $\Hbold$.
By Proposition \ref{prop:eventually-brjuno}, $\Hbold$ has interior if and only if $\Fbold$ is eventually Brjuno. 
This implies Theorem \ref{thm:no-wandering-domains}.

Theorem \ref{thm:parabolic-fatou-set} works in a similar way. Let $N \in \Z$ be a parabolic depth. 
Lemma \ref{lem:non-linearity-2} states that in external coordinates, hoglets of generation in $\Tbold_{\Arch,N}$ are dense in the complement of the full basin of attraction of the parabolic periodic point $\beta_N$.
Back on the $\Fbold$-plane, since $\Hbold$ is contained in the basin of attraction of $\Fbold^{Q_N}$ at $\infty$, then the full basin of attraction is dense in $\C$.
This establishes Theorem \ref{thm:parabolic-fatou-set}.

Next, let us prove Theorem \ref{thm:NILF}. 
We will adapt McMullen's argument for rational maps that robust non-linearity prevents any invariant parallel line field. 
Fine details will be spared and the interested reader may refer to \cite[Proposition 3.2]{She03}.
    
Suppose for a contradiction that $\Fbold$ admits an invariant line field $\nu$ supported on the Julia set. 
According to Proposition \ref{prop:zero-area-trans}, the boundary of the Mother Hedgehog $\Hbold$ of $\Fbold$ has zero area. 
In particular, the support of $\nu$ must be disjoint from the grand orbit of the boundary of $\Hbold$.
As such, we can lift $\nu$ to an invariant line field $\nu_{\textnormal{ext}}$ of $\Fext$ in external coordinates supported on a subset of $\UHP$ that avoids all hoglets.
    
Pick a point $z$ on the support of $\nu_{\textnormal{ext}}$ at which $\nu_{\textnormal{ext}}$ is almost continuous (cf. \cite[Corollary 2.15]{McM94}). 
This means that $\nu_{\textnormal{ext}}$ is nearly parallel on small disks centered at $z$. 
Consider the maps $\mathtt{h}_k: (\Uext_k, \mathtt{u}_k) \to (\Wext_k, \mathtt{w}_k)$ of scale $r_k$ from Lemma \ref{lem:non-linearity-2}. 
Consider the affine map $A_k(w) = r_k w + z$. 
Then, as $k \to \infty$, $A_k^{-1} \circ \mathtt{h}_k \circ A_k$ converges in the Carath\'eodory topology to a double covering map $\mathtt{h}_\infty: (\Uext, \mathtt{u}) \to (\Wext,\mathtt{w})$ branched at a point $\mathtt{u}$. 
By almost continuity at $z$, the rescaling of $\nu_{\textnormal{ext}}$ by $A_k$ also converges in measure to an almost parallel line field $\nu_\infty$ on the plane. 
By \cite[Theorem 5.14]{McM94}, $\nu_\infty$ is invariant under $\mathtt{h}_\infty$. But this would imply that the derivative of $\mathtt{h}_\infty$ is constant, which is impossible.


\section{Hoglet butterfly bounds}
\label{sec:butterfly}

In this section, we justify a key construction of the paper, hoglet butterflies~\eqref{eq:dfn.HB_n}, illustrated in Figure~\ref{fig:pseudo-butterfly}, and establish various geometric bounds for them. These bounds will yield QC Thurston equivalence on butterflies in~\S\ref{ss:QC.Th.eq.batt}. 

Following the last two sections, we will fix a neutral cascade $\Fbold = (\Fbold^P)_{P \in \Tbold}$ and consider the corresponding external cascade $\Fext = (\Fext^P)_{P \in \Tbold}$. In \S\ref{ss:butterfly}, we apply the uniform geometric control of pseudo-bubbles and lakes from Section \ref{sec:bounds} to construct hoglet butterflies, i.e. pairs of maps of the form
\begin{equation}
    \label{eq:dfn.HB_n}
    \butterfly_n = \left\{ \Fext^{Q_n+Q_{n-1}}: \Uext'_n \to \Wext_n, \; \Fext^{Q_n}: \Uext''_n \to \Wext_n \right\},
    \qquad
    n \in \Z,
\end{equation}
where $\Uext'_n$ and $\Uext''_n$ are disjoint open domains near $0$ that are well-contained in $\Wext_n$.
See Figure \ref{fig:pseudo-butterfly} for an illustration. Unlike the classical butterfly construction for critical circle maps, which uses the Commuting Pair return times $Q_{n+1},Q_n$, we use the Sectorial return times $Q_n+Q_{n-1}, Q_n$; see Remark \ref{rem:sector.return.time} below.

The boundaries of the butterfly wings $\Uext'_n$ and $\Uext''_n$ may be non-locally connected.
In \S\ref{ss:regularity-of-butterflies}, we overcome this issue by thickening the hoglets adjacent to $\Uext'_n$ and $\Uext''_n$ to pseudo-bubbles and apply the disjointness results in \S\ref{ss:disjointness-of-pseudo-bubbles} to prove that the new boundary is uniformly quasiconformally controlled (Proposition \ref{prop:uniform-qc-bounded}).

At a parabolic depth $n$, the non-escaping set of $\butterfly_n$ will contain a parabolic basin.
To address this issue, in \S\ref{ss:enriched-butterfly}, we introduce the enrichment of hoglet butterflies $\butterfly_n^m$ obtained from $\butterfly_n$ by adding deeper level maps.
Finally, in \S\ref{ss:non-esc-set-butterfly}, we show that the non-escaping set of the fully enriched hoglet butterflies $\butterfly_n^\infty$ will always be nowhere dense.

\subsection{The construction of hoglet butterflies}

\label{ss:butterfly}

Recall that for any closed bounded interval $I \subset \R$, $\poincare_{\tau} (I)$ denotes the half-Poincar\'e domain of $I$ with angle $\tau>0$.
Recall that for any $P \in \Tbold$, $\lake_{P}$ denotes the unique primary lake of generation $P$ that contains $0$ on its boundary.
For $n \in \Z$, denote
\[
    \lake'_n := \lake_{Q_n+Q_{n-1}} \qquad \text{and} \qquad
    \lake''_n := \Fext^{Q_{n-1}}(\lake_{Q_n+Q_{n-1}}),
\]
which are disjoint primary lakes of generation $Q_n+Q_{n-1}$ and $Q_n$ respectively.

For $n \in \Z$ and $\tau \in (0 ,\pi)$, denote
\[
    \Wext_{n,\tau} := \poincare_{\tau}[\crit_{-Q_{n-1}}, \crit_{Q_{n-1}-2Q_{n}}]
\]
and let $\Uext'_{n,\tau}$ and $\Uext''_{n,\tau}$ be the unique domains in $\lake'_n$ and $\lake''_n$ respectively such that the following is a commutative diagram of conformal isomorphisms.
\[
        \begin{tikzcd}[column sep=large, row sep=large]
	& \Wext_{n,\tau} & \\
	\Uext'_{n,\tau} & & \Uext''_{n,\tau} 
	\arrow["\Fext^{Q_{n-1}}"', from=2-1, to=2-3]
	\arrow["\Fext^{Q_n+Q_{n-1}}", from=2-1, to=1-2]
	\arrow["\Fext^{Q_n}"', from=2-3, to=1-2]
        \end{tikzcd}
\]
Denote 
\[
    \Pi_{n,\tau} := \parH \Wext_{n,\tau},
\]
and let $\Pi'_{n,\tau}$ (resp. $\Pi''_{n,\tau}$) be the part of the boundary of $\Uext'_{n,\tau}$ (resp. $\Uext''_{n,\tau}$) that gets mapped to $\Pi_{n,\tau}$.

\begin{lemma}[Structure of $\Uext'_{n,\tau}$ and $\Uext''_{n,\tau}$]
    Let $n \in \Z$ and $\tau \in (0 ,\pi)$.
    The boundaries of $\Uext'_{n,\tau}$ and $\Uext''_{n,\tau}$ are combinatorially tame:
        \begin{align*}
            \parR \Uext'_{n,\tau} &= [\crit_{-Q_{n}-Q_{n-1}}, \crit_{-Q_{n}}],
            & \qquad
            \parH \Uext'_{n,\tau} &\subset \Pi'_{n,\tau} \cup 
            \Bubb_{Q_n} \cup \Bubb_{Q_n+Q_{n-1}}, \\
            \parR \Uext''_{n,\tau} &= [\crit_{-Q_{n}}, \crit_{-Q_{n}+Q_{n-1}}],
            & \qquad
            \parH \Uext''_{n,\tau} &\subset \Pi''_{n,\tau} \cup 
            \Bubb_{Q_n} \cup \Bubb_{Q_n-Q_{n-1}}.
        \end{align*}
    The arc $\Pi'_{n,\tau}$ connects $\Bubb_{Q_n+Q_{n-1}}$ and $\Bubb_{Q_n}$, and the arc $\Pi''_{n,\tau}$ connects $\Bubb_{Q_n}$ and $\Bubb_{Q_n-Q_{n-1}}$.
\end{lemma}

\begin{proof}
    The only critical values of $\Fext^{Q_n}$ that are located in $\partial_{\R} \Wext_{n,\tau}$ are $0$ and $\crit_{Q_{n-1}}$ and it splits $\parR \Wext_{n,\tau}$ into three subintervals. 
    The middle subinterval $[0,\crit_{Q_{n-1}}]$ is lifted under $\Fext^{Q_n}$ onto the interval $\parR \lake''_n = [\crit_{-Q_n}, \crit_{-Q_n+Q_{n-1}}]$. 
    The left and the right subintervals are lifted into the Carath\'eodory boundary of $\Bubb_{Q_n}$ and $\Bubb_{Q_n-Q_{n-1}}$ respectively.
    Therefore $\Pi''_{n,\tau}$ starts from $\Bubb_{Q_n}$ and ends at $\Bubb_{Q_n-Q_{n-1}}$.

    Observe that the interval $\parR \lake'_n$ does not contain any critical value of $\Fext^{Q_{n-1}}$.
    Therefore, $\parR \lake'_n$ lifts under $\Fext^{Q_{n-1}}$ to the real interval $\parR \Uext'_{n,\tau} = [\crit_{-Q_{n}-Q_{n-1}}, \crit_{-Q_{n}}]$, and $\parH \Uext''_{n,\tau}$ lifts to $\parH \Uext'_{n,\tau}$ as described in the lemma.
\end{proof}

\begin{lemma}
\label{lem:W-contains-bubbles}
    There exists a universal constant $\ttau' \in \left( 0, \frac{\pi}{2} \right]$ such that for all $n \in \Z$, the domain $\Wext_{n,\ttau'}$ contains every primary pseudo-bubble rooted at a point on the interval $[\crit_{-Q_{n-1}}, \crit_{Q_{n-1}-2Q_{n}}]$.
\end{lemma}

\begin{proof}
    Every critical point $\crit_{-R}$ located on the interval $[\crit_{-Q_{n-1}}, \crit_{Q_{n-1}-2Q_{n}}]$ has generation $R \geq Q_{n-1}$.
    By Real Bounds, there exists a universal constant $K>1$ such that for all such $R$ and for all $k \geq n-2$, we have 
    \[
    |\crit_{-R}-\crit_{-R+Q_k}| \leq K |\crit_{-Q_{n-1}} - \crit_{Q_{n-1}-2Q_{n}}|.
    \]
    By Pseudo-Bubble Bounds, the inequality above implies that there is a sufficiently small universal constant $\ttau' \in \left( 0,\frac{\pi}{2} \right]$ such that for all such $R$, the primary pseudo-bubble $\hat{\Bubb}_R$ is contained in the half-Poincar\'e domain $\Wext_{n,\ttau'}$.  
\end{proof}

From now on, we will fix the universal constant $\ttau'$ from the lemma above.
For $n \in \Z$ and $\tau \in (0 ,\pi)$, denote
    \[
        \hat{\Uext}_{n,\tau} := \overline{ \Uext''_{n,\tau} } \cup \overline{ \Uext'_{n,\tau} } \cup \hat{\Bubb}_{Q_n+Q_{n-1}} \cup \hat{\Bubb}_{Q_{n+1}} \cup \hat{\Bubb}_{Q_n-Q_{n-1}}.
    \]
We claim that when $\tau$ is sufficiently small, $\hat{\Uext}_{n,\tau}$ is contained in the closure of $\Wext_{n,\tau}$. 
See Figure \ref{fig:pseudo-butterfly}.
    
\begin{proposition}[Uniform moduli]
\label{prop:butterflies}
    There exist universal constants $K>0$ and $\ttau \in ( 0, \ttau']$ such that for every $n \in \Z$,
    \begin{enumerate}
        \item $\hat{\Uext}_{n,\ttau}$ is contained in $\overline{\Wext_{n, \ttau}}$, and
        \item the Euclidean distance between $\hat{\Uext}_{n, \ttau}$ and $\Pi_{n,\ttau}$ satisfies 
    \[
         K^{-1} |\crit_{-Q_n}| \leq \dist \left( \hat{\Uext}_{n, \ttau}, \Pi_{n,\ttau} \right) \leq K |\crit_{-Q_n}|.
    \]
    \end{enumerate}
\end{proposition}

\begin{figure}
        \centering
       \begin{tikzpicture}
    \node[anchor=south west,inner sep=0] (image) at (0,0) {\includegraphics[width=1\linewidth]{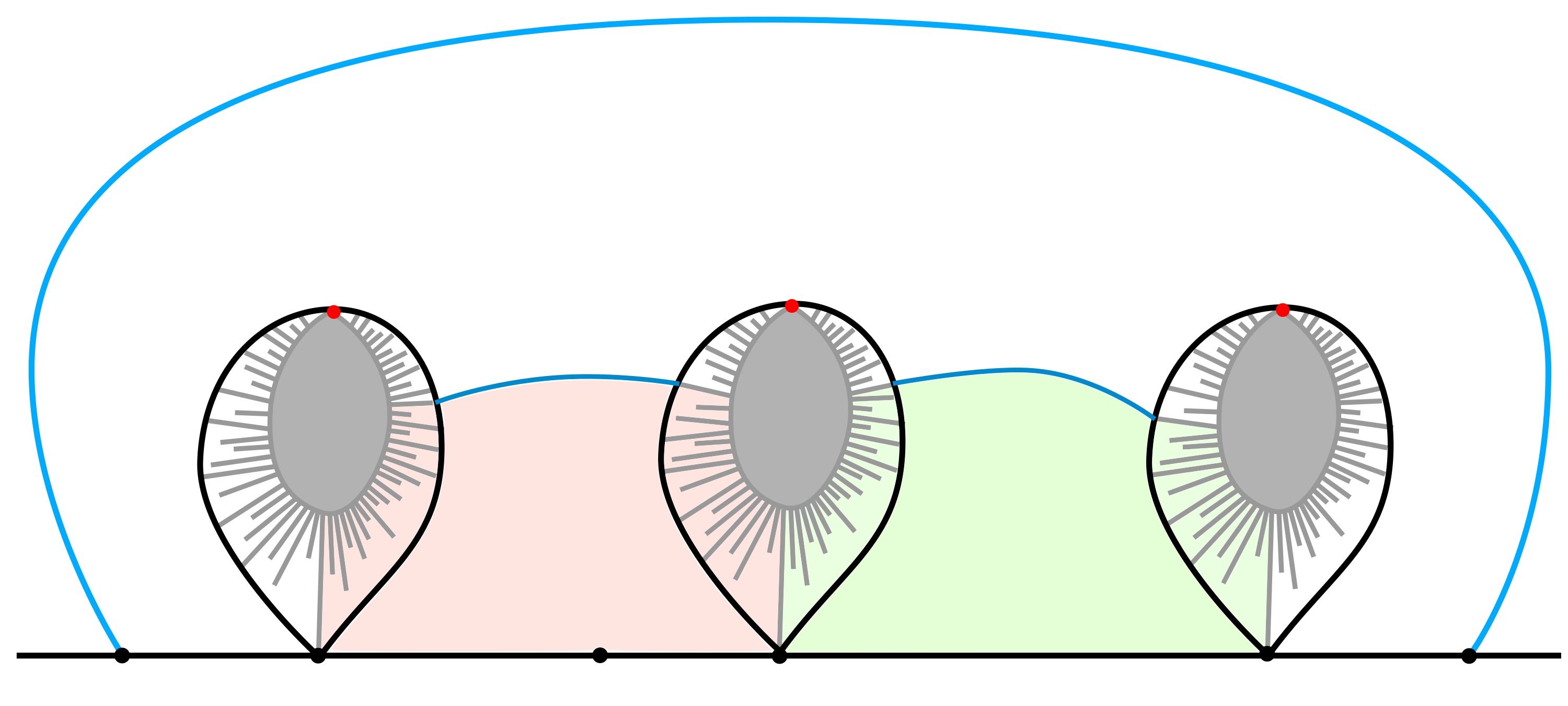}};
    \begin{scope}[
        x={(image.south east)},
        y={(image.north west)}
    ]
        \node [blue!50!black] at (0.5,0.75) {$\Wext_{n,\tau}$};
        \node [red!50!black] at (0.34,0.33) {$\Uext'_{n,\tau}$};
        \node [green!50!black] at (0.67,0.33) {$\Uext''_{n,\tau}$};
        
        \node [black] at (0.384,0.01) {$0$};
        \node [black] at (0.51,0.01) {$\crit_{-Q_n}$};
        \node [black] at (0.22,0.01) {$\crit_{-Q_n-Q_{n-1}}$};
        \node [black] at (0.79,0.01) {$\crit_{-Q_n+Q_{n-1}}$};
        
        \node [black] at (0.94,0.01) {$\crit_{-2Q_{n}+Q_{n-1}}$};
        \node [black] at (0.07,0.01) {$\crit_{-Q_{n-1}}$};

        \draw[orange, thick, -latex] (0.38,0.2) -- (0.62,0.2);
        \draw[orange, thick, -latex] (0.35,0.5) -- (0.4,0.9);
        \draw[orange, thick, -latex] (0.65,0.5) -- (0.6,0.9);
        \node[orange!80!black] at (0.57,0.15) {$\Fext^{Q_{n-1}}$};
        \node[orange!80!black] at (0.305,0.7) {$\Fext^{Q_n+Q_{n-1}}$};
        \node[orange!80!black] at (0.67,0.7) {$\Fext^{Q_{n}}$};

        \node[black] at (0.11,0.55) {\scalebox{0.8}{$\hat{\Bubb}_{Q_n+Q_{n-1}}$}};
        \node[black] at (0.55,0.585) {\scalebox{0.8}{$\hat{\Bubb}_{Q_n}$}};
        \node[black] at (0.91,0.55) {\scalebox{0.8}{$\hat{\Bubb}_{Q_n-Q_{n-1}}$}};
    \end{scope}
\end{tikzpicture}
    \caption{Hoglet butterflies, a key construction of the paper. When the exterior angle $\tau$ is sufficiently small  (Lemma~\ref{lem:W-contains-bubbles}), the half-Poincar\'e domain $\Wext_{n,\tau}$ well-contains the butterfly wings $\Uext'_{n,\tau}$, $\Uext''_{n,\tau}$ and the primary pseudo-bubbles $\hat{\Bubb}_{Q_n+Q_{n-1}}$, $\hat{\Bubb}_{Q_n}$, $\hat{\Bubb}_{Q_n-Q_{n-1}}$.}
    \label{fig:pseudo-butterfly}
\end{figure}
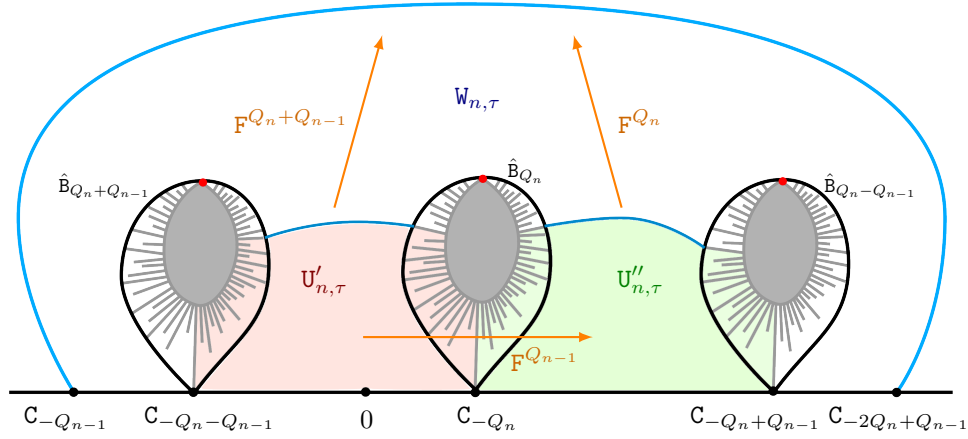

\begin{proof}
    Fix $n \in \Z$.
    For any $\tau \in (0 ,\pi)$, the closed set $\hat{\Uext}_{n,\tau}$ is contained in 
    \[
        \mathtt{A}_n := \lake'_n \cup \lake''_n \cup \hat{\Bubb}_{Q_n-Q_{n-1}} \cup \hat{\Bubb}_{Q_{n+1}} \cup \hat{\Bubb}_{Q_n-Q_{n-1}}.
    \]
    Our goal is then to find a universal constant $\ttau$ such that $\mathtt{A}_n$ is well-contained in $\Wext_{n,\ttau}$.
    
    By Propositions \ref{prop:lake-bounds-0} and \ref{prop:lake-bounds-1}, there exists a universal constant $K>0$ such that the upper-half boundaries of $\lake'_n$ and $\lake''_n$ are contained in
    \[
        \partial \Bubb_{Q_n-Q_{n-1}} \cup \partial \Bubb_{Q_{n+1}} \cup \partial \partial \Bubb_{Q_n-Q_{n-1}} \cup \D_{\UHP}(\alpha_{Q_n+Q_{n-1}}, K) \cup \D_{\UHP}(\alpha_{Q_n}, K).
    \]
    Then, we deduce from Pseudo-Bubble Bounds and the estimate on the location of $\alpha_{Q_n}$ (Corollary \ref{cor:location-of-alpha}) that there exists an $\threshold$-uniform constant $\tau_0 \in \left( 0, \frac{\pi}{2} \right]$ such that $\mathtt{A}_n$ is contained in the half-Poincar\'e neighborhood 
    \begin{align*}
    \label{eqn:inclusion-Uext}
        \mathtt{A}_{n} = \poincare_{\tau_0}[\crit_{-Q_{n-1}},\crit_{-Q_n+Q_{n-1}}].
    \end{align*}

    Let us set $\ttau := \min \{ \tau_0 , \ttau'\}$.
    By Real Bounds, the points $\crit_{-Q_{n-1}-Q_n}, \crit_{Q_{n-1}-Q_n}$ split the interval $[\crit_{-Q_{n-1}}, \crit_{Q_{n-1}-2Q_n}]$ into three sub-intervals of uniformly comparable length.
    Then, this observation, together with (\ref{eqn:inclusion-Uext}), implies the desired properties (1) and (2).
\end{proof}

From now on, we will fix the universal constant $\ttau$ from the proposition above.
We will also assume that the combinatorial threshold $\threshold$ is sufficiently high such that $\ttau$ is greater than the angle $\nnu = \nnu(\threshold)$ controlling the fjords (cf. Lemma \ref{lem:fjord-angle-control}) so that the circular arc $\Pi_n$ is disjoint from $\Fjord$.
We will write
\begin{gather*}
    \Wext_n = \Wext_{n,\ttau}, \quad 
    \Uext'_n = \Uext'_{n,\ttau}, \quad \Uext''_n = \Uext''_{n,\ttau}, \quad \hat{\Uext}_n = \hat{\Uext}_{n,\ttau}, \\
    \Pi_n = \Pi_{n,\ttau}, \quad  \Pi'_n = \Pi'_{n,\ttau}, \quad \Pi''_n = \Pi''_{n,\ttau}.
\end{gather*}

\begin{remark}[Sectorial vs. Commuting Pair return times]
\label{rem:sector.return.time} Unlike the butterfly constructions for critical circle maps in~\cite{dF99,dFdM00,Y99}, the butterflies defined in~\eqref{eq:dfn.HB_n} use the Sectorial return times $Q_n+Q_{n-1}, Q_n$ rather than the Commuting Pair return times $Q_{n+1},Q_n $. With additional care (namely, using the geometric bounds for hoglet chains described in \S\ref{ss:bubble-chains}), we can also construct and control butterflies relative to the return times $Q_{n+1}, Q_{n}$, as illustrated in Figure~\ref{fig:commuting-pair-butterfly}. The difference between Figures~\ref{fig:pseudo-butterfly} and~\ref{fig:commuting-pair-butterfly} is due to the following combinatorial fact:

\vspace{-0.08cm}
\[Q_n+Q_{n-1} \ \asymp\  Q_n\qquad\text{ but }\qquad Q_{n+1} \ \not\asymp \ Q_n.\]
\end{remark}

\begin{figure}
        \centering
       \begin{tikzpicture}
    \node[anchor=south west,inner sep=0] (image) at (0,0) {\includegraphics[width=1\linewidth]{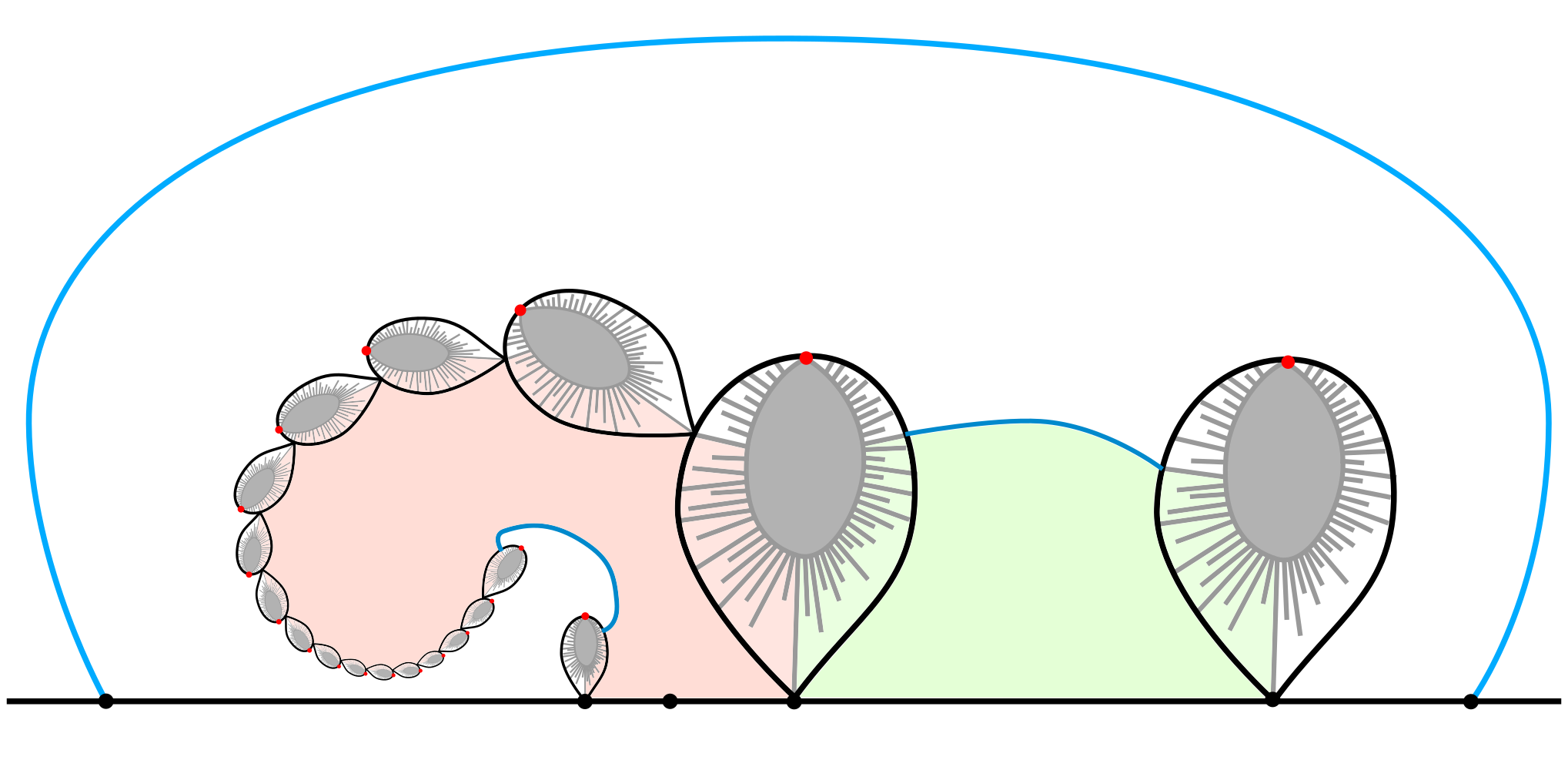}};
    \begin{scope}[
        x={(image.south east)},
        y={(image.north west)}
    ]
        \node [blue!50!black] at (0.5,0.775) {$\Wext_{n}$};
        \node [black] at (0.43,0.025) {$0$};
        \node [black] at (0.515,0.025) {$\crit_{-Q_n}$};
        \node [black] at (0.35,0.025) {$\crit_{-Q_{n+1}}$};
        \node [black] at (0.785,0.025) {$\crit_{-Q_n+Q_{n-1}}$};
        
        \node [black] at (0.94,0.025) {$\crit_{-2Q_{n}+Q_{n-1}}$};
        \node [black] at (0.07,0.025) {$\crit_{-Q_{n-1}}$};

        \draw[orange, thick, -latex] (0.3,0.61) -- (0.35,0.9);
        \draw[orange, thick, -latex] (0.65,0.5) -- (0.6,0.9);
        \node[orange!80!black] at (0.275,0.75) {$\Fext^{Q_{n+1}}$};
        \node[orange!80!black] at (0.67,0.7) {$\Fext^{Q_{n}}$};
    \end{scope}
\end{tikzpicture}
    \caption{Butterfly structure
    relative to the Commuting Pair return times $Q_{n+1}, Q_n$ as illustrated; see Remark~\ref{rem:sector.return.time}. Compare with \cite[Figure 6]{Y99} for critical circle maps. The left butterfly wing in this figure is contained in the the left wing in Figure \ref{fig:pseudo-butterfly}.}
    \label{fig:commuting-pair-butterfly}
\end{figure}
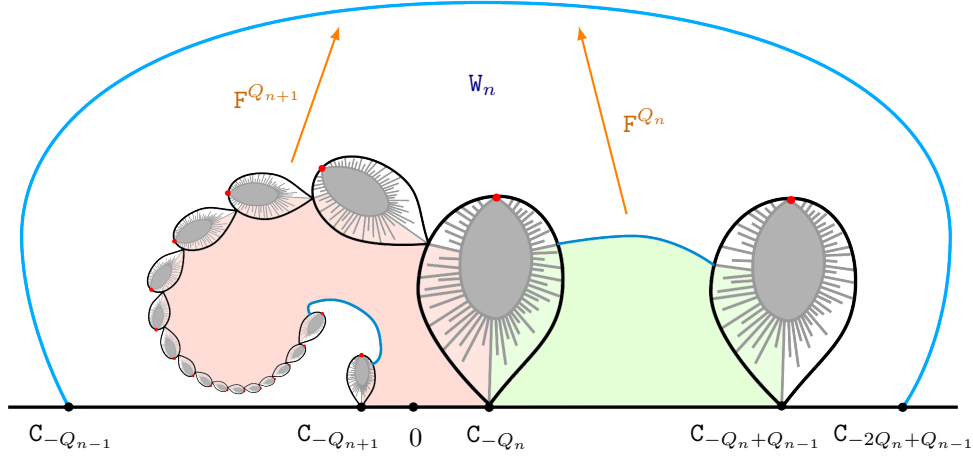

\subsection{The regularity of hoglet butterflies}
\label{ss:regularity-of-butterflies}

Recall that even though the boundary of the Mother Hedgehog $\Hbold$ may have complicated topology, the boundary of the pseudo-Siegel half-plane $\hat{\Zbold}$ is simply a uniform quasi-line.
In this subsection, we will justify a similar observation that while the boundaries of the butterfly wings $\Uext'_n$ and $\Uext''_n$ go along hoglets and may have complicated topology, the outer boundary of $\hat{\Uext}_n$ is quasiconformally tame.
This crucial property will allow us to perform quasiconformal interpolation in the next section.

\begin{proposition}
    \label{prop:uniform-qc-bounded}
    There exists an $\threshold$-uniform constant $K>0$ such that for all $n \in \Z$, the set $\Wext_n \backslash \hat{\Uext}_n$ is a $K$-quasidisk.
\end{proposition}

In the proof, we will need the following lemma.

\begin{lemma}
\label{lem:disjointness-arcs}
    For $n \in \Z$, consider the ordered list $\mathcal{L}_n$ of compact sets
    \[
        \hat{\Bubb}_{Q_n+Q_{n-1}}, \quad 
        \Pi'_n, \quad 
        \hat{\Bubb}_{Q_n}, \quad
        \Pi''_n, \quad 
        \hat{\Bubb}_{Q_n-Q_{n-1}}. 
    \]
    For sufficiently high $\threshold$ and for all $n \in \Z$, every element $A$ of $\mathcal{L}_n$ has an open neighborhood $D(A)$ with the following properties.
    \begin{enumerate}
        \item Every element of $\mathcal{L}_n$ is either a uniform quasiarc or an $\threshold$-uniform quasidisk.
        \item Every element of $\mathcal{L}_n$ has diameter $\asymp |\crit_{Q_n}|$.
        \item There exists a universal constant $\delta>0$ such that for all $A \in \mathcal{L}_n$,
        \[
        \nbh_{\delta}(A) \subset D(A).
        \]
        \item For $A \in \{ \Pi'_n, \Pi''_n \}$,
        \[
            D(A) \subset \UHP \quad \text{ and } \quad
            \diam_{\UHP}(D(A)) \asymp 1.
        \]
        \item For any two non-consecutive elements $A_1, A_2$ of $\mathcal{L}_n$, $D(A_1)$ and $D(A_2)$ are disjoint.
        \item For any two consecutive elements $A_1, A_2$ of $\mathcal{L}_n$,
        \begin{enumerate}[label=\textnormal{(\roman*)}]
            \item the intersection $A_1 \cap A_2$ is a singleton at which $A_1$ meets $A_2$ at a definite angle, and
            \item $D(A_1) \cap D(A_2)$ is a small disk neighborhood of $A_1 \cap A_2$.
        \end{enumerate}
    \end{enumerate}
\end{lemma}

\begin{proof}
    Let $-\hat{\UHP} := \C \backslash \left( \overline{\UHP} \backslash \Fjord \right)$. 
    It is uniformly quasiconformally equivalent to the lower half plane $-\UHP$.
    Throughout this proof, we fix $n \in \Z$.
    
    According to Lemma \ref{lem:main-bubble-disjoint-fjord}, the primary hoglets in the list $\mathcal{L}_n$ are disjoint from $-\hat{\UHP}$.
    For every pseudo-bubble $A$ in the list $\mathcal{L}_n$, 
    properties (1) and (2) follows from Pseudo-Bubble Bounds and Real Bounds.
    
    For $P \in \{Q_n+Q_{n-1},Q_n,Q_n-Q_{n-1}\}$, 
    we define $D_\delta(\hat{\Bubb}_P)$ to be the Jordan disk neighborhood of $\hat{\Bubb}_P$ such that its boundary is the union of 
        \begin{itemize}
            \item the subarc $\gamma_{P,\delta}$ of the boundary of $\nbh_{\delta} (\hat{\Bubb}_P)$ that is a proper arc in $\UHP \backslash \Fjord$ that surrounds $\hat{\Bubb}_P$, and
            \item the hyperbolic geodesic of $-\hat{\UHP}$ that connects the endpoints of $\gamma_{P,\delta}$.
        \end{itemize}
    Pseudo-Bubble Bounds tell us that $\hat{\Bubb}_P$ is a uniform quasidisk with diameter $\asymp |\crit_{Q_n}|$ meeting the real line at its root with definite angle.
    As such, there exists some universal constant $\lambda \in (0,1)$ independent of $\delta$ such that 
        \[
            \nbh_{\lambda \delta}(\hat{\Bubb}_P) \subset D_\delta(\hat{\Bubb}_P) \subset \nbh_{\lambda^{-1} \delta}(\hat{\Bubb}_P).
        \]

    \noindent \underline{Claim 1:} There exists a universal constant $\delta_1>0$ such that
    \[
        D_{\delta_1}(\hat{\Bubb}_{Q_{n+1}}), \qquad D_{\delta_1}(\hat{\Bubb}_{Q_{n}}), \qquad D_{\delta_1}(\hat{\Bubb}_{Q_n-Q_{n-1}})
    \]
    are pairwise disjoint.
    
    \begin{proof}
        This follows from Lemma \ref{lem:primary-bubble-disjoint} and the fact that $\hat{\Bubb}_{Q_{n+1}}$, $\hat{\Bubb}_{Q_{n}}$, and $\hat{\Bubb}_{Q_n-Q_{n-1}}$ have uniformly comparable diameter.
    \end{proof}
    
    Next, the neighborhoods of $\Pi_n'$ and $\Pi_n''$ are constructed as follows.
    
    For $k \in \{1,2,3\}$, consider the circular arcs $\gamma_{+,k}$ and $\gamma_{-,k}$ given by
        \[
        \gamma_{\pm,k} = \parH \poincare_{\ttau}
        [\crit_{-Q_{n-1} \mp kQ_{n+3}},\crit_{Q_{n-1}-2Q_{n} \pm k Q_{n+3}}].
        \]
    The rectangle in $\UHP$ bounded by $\gamma_{-,3}$ and $\gamma_{+,3}$ is a definite thickening of the arc $\Pi_n$ in $\UHP$ with respect to the Euclidean metric.
    According to Real Bounds, this rectangle is split by the curves $\gamma_{-,2}$, $\gamma_{-,1}$, $\Pi_n$, $\gamma_{+,1}$, and $\gamma_{+,2}$ into six sub-rectangles of uniformly comparable moduli.
    
    For $k \in \{1,2,3\}$, let us define the neighborhood $D_k(\Pi_n'')$ of $\Pi_n''$ such that the boundary is a union of the following arcs:
        \begin{itemize}
            \item $\gamma''_{-,k}$ = the lift of $\gamma_{-,k}$ under $\Fext^{Q_n} : \lake''_n \to \UHP$;
            \item $\gamma''_{+,k}$ = the lift of $\gamma_{+,k}$ under $\Fext^{Q_n} : \lake''_n \to \UHP$;
            \item the hyperbolic geodesic $\sigma''_{1,k}$ of the interior of $\hat{\Bubb}_{Q_{n}}$ connecting the starting points of $\gamma''_{-,k}$ and $\gamma''_{+,k}$;
            \item the hyperbolic geodesic $\sigma''_{2,k}$ of the interior of $\hat{\Bubb}_{Q_n+Q_{n-1}}$ that connects the endpoints of $\gamma''_{-,k}$ and $\gamma''_{+,k}$.
        \end{itemize}
    It is an elementary combinatorial exercise to check that the endpoints of $\gamma''_{\pm,k}$ are indeed on the boundary of $\hat{\Bubb}_{Q_n}$ and $\hat{\Bubb}_{Q_{n}+Q_{n-1}}$ and thus the construction above makes sense.
    We then take $D_k(\Pi''_n)$ to be the lift of $D_k(\Pi'_n)$ under $\Fext^{Q_{n-1}}$; it is enclosed by a concatenation of four analogous arcs.
    \vspace{0.1in}
    
    \noindent \underline{Claim 2:}
    The neighborhood $D_3(\Pi''_n)$ is contained in $\UHP$ and each of annuli 
    \[
    D_3(\Pi'_n) \backslash D_2(\Pi''_n), \quad D_2(\Pi''_n) \backslash D_1(\Pi''_n), \quad D_1(\Pi''_n) \backslash \Pi''_n
    \]
    has modulus $\asymp 1$.
    The same holds for $\Pi'_n$.

    \begin{proof}
        This follows immediately by design.
    \end{proof}

    \noindent \underline{Claim 3:}
    The set $\Pi''_n$ is a uniform quasiarc with Euclidean diameter $\asymp |\crit_{Q_n}|$, and for $k \in \{1,2\}$, the neighborhood $D_k(\Pi''_n)$ has Euclidean diameter $\asymp |\crit_{Q_n}|$ and hyperbolic diameter $\asymp 1$.
    The same holds for $\Pi'_n$.

    \begin{proof}
        Claim 2 implies that
        \begin{align}
        \label{eqn:hyp-diam-lower-bound}
            \diam_{\UHP}(D_2(\Pi''_n)) = O(1).
        \end{align}
        Hence, $\Fbold^{Q_n} \circ \Psi$ sends $D_2(\Pi''_n)$ onto its image with uniformly bounded distortion.
        By Lemma \ref{cor:Psi-qc}, $\Psi^{-1}(\Pi_n)$ is a uniform quasiarc because $\Pi_n$ is a uniform quasiarc in $\UHP$ that avoids all fjords in external coordinates.
        Therefore, $\Pi''_n$ is also uniform quasiarc.
    
        By Lemma \ref{lem:primary-bubble-disjoint}, the pseudo-bubbles $\hat{\Bubb}_{Q_n}$ and $\hat{\Bubb}_{Q_n-Q_{n-1}}$ are disjoint and have distance $\asymp |\crit_{Q_n}|$ from one another.
        Thus,
        \begin{align}
        \label{eqn:euc-diam-lower-bound}
            \diam( \Pi''_n ) \succ |\crit_{Q_n}|.
        \end{align}
        Altogether, (\ref{eqn:euc-diam-lower-bound}) and (\ref{eqn:hyp-diam-lower-bound}) imply the rest of Claim 3 for $\Pi''_n$.

        The argument for $\Pi'_n$ is analogous.
    \end{proof}

    The two claims above imply that for $k \in \{1,2\}$, $D_k(\Pi'_n)$ and $D_k(\Pi''_n)$ satisfy properties (1)--(4).
    \vspace{0.1in}

    \noindent \underline{Claim 4:} $D_2(\Pi''_n)$ is disjoint from both $D_2(\Pi''_n)$ and $\hat{\Bubb}_{Q_n+Q_{n+1}}$.

    \begin{proof}
        The neighborhood $D_2(\Pi'_n)$ is contained in the union of $\hat{\Bubb}_{Q_n+Q_{n-1}}$, $\hat{\Bubb}_{Q_n}$, and $\lake'_n$, whereas the neighborhood $D_2(\Pi''_n)$ is contained in the union of $\hat{\Bubb}_{Q_n}$, $\hat{\Bubb}_{Q_n-Q_{n-1}}$, and $\lake''_n$. 
        The lake $\lake'_n$ is disjoint from the lake $\lake''_n$ because it is contained in a different primary lake of generation $Q_n$, namely $\lake_{Q_n}$.
        So if $D_2(\Pi'_n)$ and $D_2(\Pi''_n)$ intersect, the intersection would be contained in $\hat{\Bubb}_{Q_n}$, but this is prevented by the geodesic construction.
        Hence, $D_2(\Pi'_n)$ and $D_2(\Pi''_n)$ are disjoint.

        The bubble $\hat{\Bubb}_{Q_n+Q_{n+1}}$ is contained in the lake $\lake_{Q_n}$ of generation $Q_n$ distinct from $\lake''_n$.
        Hence, $D_2(\Pi''_n)$ is disjoint from $\hat{\Bubb}_{Q_n+Q_{n+1}}$.
    \end{proof}

    Unlike $\Pi''_n$, the argument for $\Pi'_n$ requires a finer analysis.
    \vspace{0.1in}

    \noindent \underline{Claim 5:} For sufficiently high $\threshold$, $D_2(\Pi'_n)$ is disjoint from $\hat{\Bubb}_{Q_n-Q_{n-1}}$.
    
    \begin{proof}
        Recall again that $D_2(\Pi'_n)$ is contained in the union of $\hat{\Bubb}_{Q_n}$, $\hat{\Bubb}_{Q_n+Q_{n-1}}$, and $\lake'_n$.
        By Claim 1, it suffices to show that the lake $\lake''_n$ is disjoint from $\hat{\Bubb}_{Q_n-Q_{n-1}}$.
        
        The pseudo-bubble $\hat{\Bubb}_{Q_n-Q_{n-1}}$ is contained in two lakes of generation $Q_n-Q_{n-1}$; one of them is $\lake_{Q_n-Q_{n-1}}$ and it contains $\lake'_n$.
        The image of $\lake'_n$ under the conformal map $\Fext^{Q_n-Q_{n-1}}: \lake_{Q_n-Q_{n-1}} \to \UHP$ is a lake $\lake'''_n$ of generation $2Q_{n-1}$.
        Meanwhile, the image of $\hat{\Bubb}_{Q_n-Q_{n-1}} \cap \lake_{Q_n-Q_{n-1}}$ under $\Fext^{Q_n-Q_{n-1}}: \lake_{Q_n-Q_{n-1}} \to \UHP$ is precisely equal to the set $\mathcal{V}_n$ of fjords that are attached to the connected component of $\R \backslash \{0\}$ containing $\crit_{Q_{n-1}}$.
        Below, we will show that $\lake'''_n$ is disjoint from $\mathcal{V}_n$.
        
        By the analysis in the proof of Proposition \ref{prop:lake-bounds-0}, there exists a uniform constant $K'>0$ such that the upper-half boundary of $\lake'''_n$ is contained in
        \[
            \partial_{\UHP} \lake'''_n \subset \Bubb_{Q_{n-1}} \cup \Bubb_{2Q_{n-1}} \cup \D_{\UHP}(\alpha_{2Q_{n-1}},K').
        \]
        Together with Real Bounds and Corollary \ref{cor:location-of-alpha}, this implies that $\partial_{\UHP} \lake'''_n$ is contained in the half-Poincar\'e neighborhood $\Delta_{\delta'''}[\crit_{-2Q_{n-1}},\crit_{-Q_{n-1}}]$ for some universal constant $\delta'''>0$.
        We will assume that $\threshold$ is sufficiently high such that the angle $\nnu(\threshold)$ from Lemma \ref{lem:fjord-angle-control} is smaller than $\varepsilon>0$.
        Then, the set $\Delta_{\varepsilon}[\crit_{-2Q_{n-1}},\crit_{-Q_{n-1}}]$ and therefore the lake $\lake'''_n$ is disjoint from $\mathcal{V}_n$.
    \end{proof}

    All the five claims above imply that there exists $\delta_2 \in (0,\delta_1)$ such that the neighborhoods
    \[
        D(A) = \begin{cases}
            D_1(A) & \textnormal{ for } A \in \{ \Pi'_n, \Pi''_n \}, \\
            D_{\delta_2}(A) & \textnormal{ for } A \in \{ \hat{\Bubb}_{Q_n+Q_{n-1}}, \hat{\Bubb}_{Q_n}, \hat{\Bubb}_{Q_n-Q_{n-1}} \}
        \end{cases}
    \]
    satisfy property (5).
    Property (6)(i) follows by design.
    Property (6)(ii) can be arranged by local modification of these neighborhoods around the intersection points.
\end{proof}

\begin{proof}[Proof of Proposition \ref{prop:uniform-qc-bounded}]
    Fix $n \in \Z$ and denote by $X_n$ the boundary of $\Wext_n \backslash \hat{\Uext}_n$. 
    To prove that $X_n$ is a $K(N)$-quasicircle, it suffices to prove the following property: 
    there exist constants $\eta \in \N$ and $\kappa>0$ depending only on $N$ such that $X$ is $(\eta,\kappa)$-\emph{bounded turning}. 
    That is, there exists some constant $\varepsilon>0$ such that $X_n$ can be covered by $\eta$ arcs of diameter $\leq \varepsilon$ and for any two distinct points $x$ and $y$ on $X_n$ with distance $\leq \varepsilon$, the smallest diameter of an arc on $X_n$ joining $x$ and $y$ is at most $\kappa |x-y|$.
    The dilatation of $X_n$ depends only on $\eta$ and $\kappa$.

    Let us present $X_n$ as the concatenation of closed arcs 
    \begin{equation}
    \label{eqn:quasiarc-pieces}
        \Pi_n, \quad \Pi'_n, \quad \Pi''_n, \quad J'_n, \quad J''_n, \quad X'_n, \quad X^\dagger_n, \quad X''_n
    \end{equation}
    where 
    \begin{itemize}
        \item $J'_n = [\crit_{-Q_{n-1}},\crit_{-Q_n-Q_{n-1}}]$ and $J''_n=[\crit_{Q_{n-1}-Q_n},\crit_{Q_{n-1}-2 Q_{n}}]$;
        \item $X'_n$ is the subarc of $\partial \hat{\Bubb}_{Q_n+Q_{n-1}}$ that contains the alpha-point $\alpha_{Q_n-Q_{n-1}}$ and connects the root $\crit_{Q_n+Q_{n-1}}$ to the endpoint of $\Pi'_n$;
        \item $X''_n$ is the subarc of $\partial \hat{\Bubb}_{Q_{n}-Q_{n-1}}$ that contains the alpha-point $\alpha_{Q_{n}-Q_{n-1}}$ and connects the root $\crit_{Q_{n}-Q_{n-1}}$ to the endpoint of $\Pi''$;
        \item $X^\dagger_n$ is the subarc of $\partial \hat{\Bubb}_{Q_n}$ that contains the alpha-point $\alpha_{Q_{n}}$ and connects an endpoint of $\Pi'$ to an endpoint of $\Pi''$.
    \end{itemize}
    Again, see Figure \ref{fig:pseudo-butterfly} for illustration.
    We say that any two distinct arcs on the list (\ref{eqn:quasiarc-pieces}) is adjacent if they share a common endpoint.
    \vspace{0.1in}

    \noindent \underline{Claim 1:}
        Every arc $A$ on the list (\ref{eqn:quasiarc-pieces}) has Euclidean diameter $\asymp |\crit_{Q_n}|$.

    \begin{proof}
        This follows from Real Bounds and Lemma \ref{lem:disjointness-arcs}.
    \end{proof}
    
    \noindent \underline{Claim 2:}
        For every pair of adjacent arcs $A_1$ and $A_2$ on the list (\ref{eqn:quasiarc-pieces}), the union $A_1 \cup A_2$ is a uniform quasiarc.

    \begin{proof}
        Lemma \ref{lem:disjointness-arcs} (1) tells us that every arc on the list (\ref{eqn:quasiarc-pieces}) is a uniform quasiarc.
        To deal with adjacent pairs, we apply the property that adjacent arcs meet at a definite angle; this property follows from Lemma \ref{lem:disjointness-arcs} (6) as well as Pseudo-Bubble Bounds.
    \end{proof}
    
    \noindent \underline{Claim 3:} Every arc on the list (\ref{eqn:quasiarc-pieces}) is of distance $\succ|\crit_{Q_n}|$ away from every other arc in (\ref{eqn:quasiarc-pieces}) not adjacent to $A$.

    \begin{proof}
        This follows from Proposition \ref{prop:butterflies} and Lemma \ref{lem:disjointness-arcs}.
    \end{proof}

    These three claims clearly imply the bounded turning property for $X$.
\end{proof}

\setlength\arraycolsep{0pt}
    
\begin{figure}
    \centering
    \begin{tikzpicture}[scale=1.1]
        \filldraw[white,fill=blue!10!white] (-6,4) -- (-5.5,6) -- (0.5,6) -- (1,4) -- (-6,4);
        \filldraw[white,fill=red!10!white] (1,4) -- (1.5,6) -- (3.5,6) -- (4,4) -- (1,4);
        \draw[blue!75!black,line width=1pt] (-6,4) -- (1,4);
        \draw[red,line width=1pt] (1,4) -- (4,4);
        
        \draw[line width=0.5pt] (-6,4.1) -- (-6,3.9);
        \node [black, font=\bfseries] at (-6,3.6) {\scalebox{0.8}{$-P_{n,m}$}};
        \draw[line width=0.5pt] (4,4.1) -- (4,3.9);
        \node [black, font=\bfseries] at (4,3.6) {\scalebox{0.8}{$-P_{n,m+1}$}};
        
        \draw[line width=0.5pt] (-4.5,4.1) -- (-4.5,3.9);
        \node [black, font=\bfseries] at (-4.5,2.8) {\scalebox{0.75}{$-Q_n-a_{m+1}Q_m$}};
        \draw[gray,line width=0.5pt,-latex] (-4.5,3) -- (-4.5,3.8);
        \draw[line width=0.5pt] (-3.2,4.1) -- (-3.2,3.9);
        \node [black, font=\bfseries] at (-2.8,3.3) {\scalebox{0.75}{$- Q_n - (a_{m+1}-1)Q_m$}};
        \draw[gray,line width=0.5pt,-latex] (-3.2,3.5) -- (-3.2,3.8);

        \draw[line width=0.5pt] (-2.5,4.1) -- (-2.5,3.9);
        \draw[line width=0.5pt] (-2.05,4.1) -- (-2.05,3.9);
        \draw[line width=0.5pt] (-1.75,4.1) -- (-1.75,3.9);
        \draw[line width=0.5pt] (-1.45,4.1) -- (-1.45,3.9);
        \draw[line width=0.5pt] (-1,4.1) -- (-1,3.9);
        
        \draw[line width=0.5pt] (-0.3,4.1) -- (-0.3,3.9);
        \node [black, font=\bfseries] at (-0.3,3.43) {\scriptsize  $\begin{array}{rl}
            - & Q_n \\
            - & Q_m
        \end{array}$};
        \draw[line width=0.5pt] (1,4.1) -- (1,3.9);
        \node [black, font=\bfseries] at (1,3.6) {\scriptsize $-Q_n$};
        \draw[line width=0.5pt] (2.5,4.1) -- (2.5,3.9);
        \node [black, font=\bfseries] at (2.5,3.43) {\scriptsize  $\begin{array}{rl}
            - & Q_n \\
            + & Q_m
        \end{array}$};
        
        \draw[blue,line width=0.5pt,-latex] (-3.8,4.4) .. controls (-3.7,4.9) and (-3.05,4.9) .. (-2.95,4.4);
        \draw[blue,line width=0.5pt,-latex] (-2.75,4.4) .. controls (-2.65,4.8) and (-2.4,4.75) .. (-2.3,4.35);
        \draw[blue,line width=0.5pt,-latex] (-2.2,4.3) .. controls (-2.15,4.6) and (-2.00,4.55) .. (-1.95,4.25);
        \draw[blue,line width=0.5pt,-latex] (-1.85,4.25) .. controls (-1.8,4.5) and (-1.7,4.5) .. (-1.65,4.25);
        \draw[blue,line width=0.5pt,-latex] (-1.55,4.25) .. controls (-1.5,4.55) and (-1.35,4.6) .. (-1.3,4.3);
        \draw[blue,line width=0.5pt,-latex] (-1.2,4.35) .. controls (-1.1,4.75) and (-0.85,4.8) .. (-0.75,4.4);
        \draw[blue,line width=0.5pt,-latex] (-0.55,4.4) .. controls (-0.45,4.9) and (0.1,4.9) .. (0.2,4.4);
        \draw[blue,line width=0.5pt,-latex] (0.6,4.4) .. controls (0.7,4.9) and (1.4,4.9) .. (1.5,4.4);
        \node [blue, font=\bfseries] at (-0.9,5) {\footnotesize $\Fext^{Q_m}$};
        \node [blue!60!black, font=\bfseries] at (-4.5,5.25) {\small $\Yext_{n,m}$};
        \node [red, font=\bfseries] at (3,5.25) {\small $\Yext_{n,m+1}$};
\end{tikzpicture}

    \caption{The action of $\Fext^{Q_m}$ on the real line near $\Yext_{n,m}$ when $a_{m+1}=8$.}
    \label{fig:interval-orbit}
\end{figure}
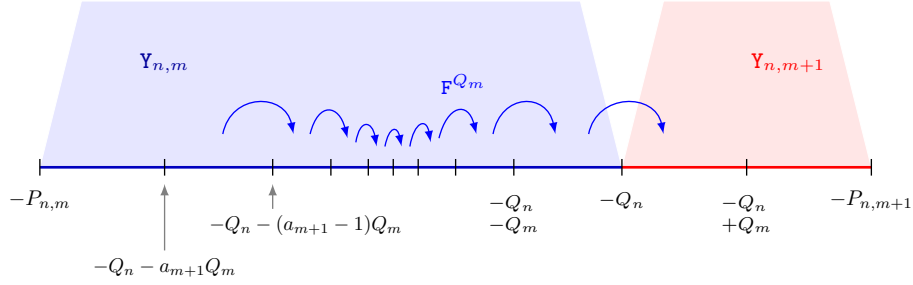

\subsection{Parabolic enrichment}
\label{ss:enriched-butterfly}

Let us fix $n \in \Z$. 
Our goal in this subsection is to carefully study the iterated preimages of $\Wext_n$ along the real line.

For $m \geq n$, denote
\[
    P_{n,m} := \begin{cases}
        Q_n - Q_{n-1} & \text{ if } m = n, \\
        Q_n - Q_{m-1} + Q_m & \text{ if } m \geq n+1.
    \end{cases}
\]
Denote by $\Yext_{n,m}$ the unique domain with essential real boundary
\[
\parRess \Yext_{n,m} = [\crit_{-Q_n}, \crit_{-P_{n,m}}]
\]
such that
\[
    \Fext^{Q_m}: \Yext_{n,m} \to \Wext_n
\]
is a conformal isomorphism.
It is well-defined because $\parRess \Yext_{n,m}$ is in $\tiling_{Q_m}$. 
The interval dynamics of the map $\Fext^{Q_m}$ on $\Yext_{n,m}$ is illustrated in Figure \ref{fig:interval-orbit}. 
For $m=n$, we have
\[
\Yext_{n,n} = \Uext''_n.
\]

For every integer $1\leq j \leq a_{m+1}$, denote by $\Yext^j_{n,m}$ the unique domain with essential real boundary
\[
\parRess \Yext^j_{n,m} = [\crit_{-Q_n-(j-1)Q_m},\crit_{-P_{n,m}}].
\]
such that 
\[
    \Fext^{jQ_m}: \Yext^j_{n,m} \to \Wext_n
\]
is a conformal isomorphism.
The following sequence of maps are conformal isomorphisms:
\[
    \Yext^{a_{m+1}}_{n,m} \; \xrightarrow[\;\Fext^{Q_m}\;]{} \; 
    \Yext^{a_{m+1}-1}_{n,m} \; \xrightarrow[\;\Fext^{Q_m}\;]{} \;
    \ldots \; \xrightarrow[\;\Fext^{Q_m}\;]{} \;
    \Yext^{1}_{n,m} = \Yext_{n,m} \; \xrightarrow[\;\Fext^{Q_m}\;]{} \; \Wext_n.
\]
When $a_{m+1}=\infty$, the sequence above is infinite and we can extend it further as follows. 
For $k \geq 0$, we will use the notation
\[
    (\iinfty-k)Q_m := Q_{m+1}-Q_{m-1}-kQ_m, 
\]
and let $\Yext^{\iinfty-k}_{n,m}$ be the domain with essential real boundary 
\[
    \parRess \Yext^{\iinfty-k}_{n,m} = [\crit_{-Q_n-(\iinfty-k-1)Q_m},\crit_{-P_{n,m}}]
\]
such that
\[
    \Fext^{(\iinfty-k)Q_m}: \Yext^{\iinfty-k}_{n,m} \to \Wext_n
\]
is a conformal isomorphism. 
See Figure \ref{fig:Y-domains}.
Then, for any $k \geq 0$ and $j \geq 1$, the following maps are conformal isomorphisms.
\[
    \Yext^{\iinfty}_{n,m} \xrightarrow[\;\Fext^{Q_m}\;]{} 
    \Yext^{\iinfty-1}_{n,m} \xrightarrow[\;\Fext^{Q_m}\;]{} 
    \Yext^{\iinfty-2}_{n,m} \xrightarrow[\;\Fext^{Q_m}\;]{} 
    \ldots \xrightarrow[\;\Fext^{Q_m}\;]{} 
    \Yext^{\iinfty-k}_{n,m} \xrightarrow[\;\Fext^{(\iinfty-k-j)Q_m}\;]{} 
    \Yext^{j}_{n,m}.
\]

When depth $m$ is parabolic, recall that the principal fjord $\fjord_m$ of level $m$ contains a unique parabolic fixed point $\beta_m$ of $\Fext^{Q_m}$.
In this case, we will denote
\[
    \beta_{n,m} := \Fext^{Q_m-P_{n,m}}(\beta_n).
\]
The following lemma summarizes the elementary properties of these $\Yext$ domains.

\begin{figure}
        \centering
    \begin{tikzpicture}
    \node[anchor=south west, inner sep=0] (image) at (0,0) {\includegraphics[width=0.98\linewidth]{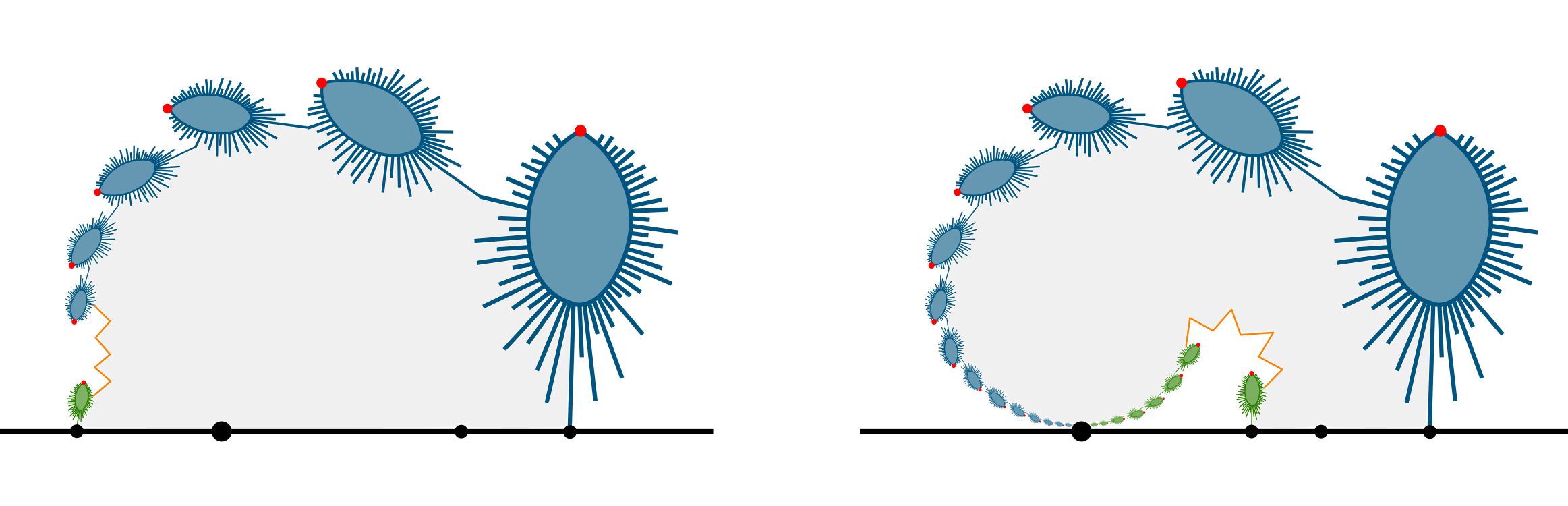}};
    \begin{scope}[
        x={(image.south east)},
        y={(image.north west)}
    ]
        \node [black, font=\bfseries] at (0.1,-0.15) {$\crit_{-Q_n-(j+1)Q_m}$};
        \draw[gray,line width=0.5pt,-latex] (0.05,-0.08) -- (0.05,0.075);
        
        \node [black, font=\bfseries] at (0.8,-0.15) {$\crit_{-Q_n-(\iinfty-k-1)Q_m}$};
        \draw[gray,line width=0.5pt,-latex] (0.797,-0.08) -- (0.797,0.075);
        
        \node [black, font=\bfseries] at (0.38,0.04) {$\crit_{-P_{n,m}}$};
        \node [black, font=\bfseries] at (0.925,0.04) {$\crit_{-P_{n,m}}$};
        \node [black, font=\bfseries] at (0.2,0.4) {$\Yext^j_{n,m}$};
        \node [black, font=\bfseries] at (0.72,0.5) {$\Yext^{\iinfty-k}_{n,m}$};
        \node [black, font=\bfseries] at (0.145,0.04) {$\beta_{n,m}$};
        \node [black, font=\bfseries] at (0.685,0.04) {$\beta_{n,m}$};
    \end{scope}
\end{tikzpicture}
        
        \caption{Domains $\Yext^j_{n,m}$ and $\Yext^{\iinfty-k}_{n,m}$}
        \label{fig:Y-domains}
\end{figure}

\begin{lemma}
\label{lem:non-esc-parab}
    For any $n \in \Z$ and $m \geq n$,
    \begin{enumerate}[label = \textnormal{(\arabic*)}]
        \item $\Yext_{n,m}$ is contained in $\Wext_n$;
        \item for every integer $1 \leq j < a_{m+1}$,
        \[
        \Yext^{j+1}_{n,m} \subset \Yext^j_{n,m};
        \]
        \item if $a_{m+1}=\infty$, then
        \begin{enumerate}[label=\textnormal{(\alph*)}]
            \item the non-escaping set of $\Fext^{Q_m}: \Yext_{n,m} \to \Wext_n$ is equal to 
            \[
            \Dext_{n,m} := \bigcap_{j\geq 1} \Yext^j_{n,m},
            \]
            and it is the immediate attracting basin of $\beta_{n,m}$;
            \item for every $k \geq 0$, we have 
            \[
            \Yext^{\iinfty-k}_{n,m} \subset \Yext^{\iinfty-k-1}_{n,m} \subset \Dext_{n,m};
            \]
        \end{enumerate}
        \item regardless of the value of $a_{m+1} \in \{1,2,\ldots,\iinfty\}$, we have the following commutative diagram.
\begin{center}
 \begin{tikzcd}
 & \Yext^{a_{m+1}}_{n,m} \arrow[dd, "\Fext^{Q_{m+1}-Q_{m-1}}"] \\
 \Yext_{n,m+1} \arrow[ur, "\Fext^{Q_{m-1}}"] \arrow[dr, "\Fext^{Q_{m+1}}"'] &
 \\
 & \Wext_n
 \end{tikzcd}
\end{center}
    \end{enumerate}
\end{lemma}

\begin{proof}
    We will prove this by induction over $m$. 
    For any $j \in \{1,2,3,4\}$ and $m \geq n$, we will denote by $(j)_m$ the property described in item $(j)$ above applied for depth $m$. 
    The induction begins with the observation that $(1)_n$ holds trivially. 
    Let us fix $m \geq n$ and assume $(1)_m$.
    
    First, observe that $(2)_m$ directly follows from $(1)_m$ and the fact that each $\Yext^{j+1}_{n,m}$ is the lift of $\Yext^j_{n,m}$ under the map $\Fext^{Q_m}:\Yext_{n,m} \to \Wext_n$.
    
    Suppose that $a_{m+1}=\infty$. 
    The interval $\parRess \Yext_{n,m}$ contains a unique parabolic fixed point of $\Fext^{Q_m}$ which is $\beta_{n,m}$. 
    The attracting basin $\Dext_{n,m}$ of $\beta_{n,m}$ is equal to $\Fext^{P_{n,m}-Q_m}(\Dext_n)$ where $\Dext_n$ (from \S\ref{ss:bubble-chains}) is equal to the immediate attracting basin of $\beta_n$.
    The domain $\Dext_{n,m}$ is enclosed by a $\Fext^{Q_m}$-periodic hoglet chain $\Gamma_{n,m,j} := \Fext^{P_{n,m}-Q_m}(\Gamma_{m,j})$, $j \geq 0$ starting from the primary hoglet $\Gamma_{n,m,0} = \Bubb_{P_{n,m}}$.
    By Lemma \ref{lem:W-contains-bubbles}, $\Gamma_{n,m,0}$ is contained in $\Wext_n$.
    As such, the lift $\Gamma_{n,m,0}$ is contained in $\Yext_{n,m}$, which, by $(1)_m$, is contained in $\Wext_n$. 
    Inductively, the entire hoglet chain is contained in $\Wext_n$. 
    Therefore, $\Dext_{n,m}$ is contained in $\Wext_{n,m}$ and hence in $\Yext^j_{n,m}$ for all $j \geq 1$. 
    Since the non-escaping set of $\Fext^{Q_m}: \Yext_{n,m} \to \Wext_n$ is connected, contains $\Dext_{n,m}$, and has to avoid the hoglets $\Gamma_{n,m,j}$, $j \geq 0$, then the non-escaping set is precisely equal to $\Dext_{n,m}$. 
    This establishes $(3)_m$(a).
    
    For every $k \geq 0$, $\Yext^{\iinfty-k}_{n,m}$ is contained in a unique lake $\lake^{\iinfty-k}_{n,m}$ of generation $(\iinfty-k)Q_m$. 
    Observe that $\parRess \lake^{\iinfty-k}_{n,m}$ is contained in $\parR \Dext_{n,m} = [\crit_{-P_{n,m}}, \beta_{n,m}]$ but $\lake^{\iinfty-k}_{n,m}$ must avoid $\parH \Dext_{n,m}$ because it consists of hoglets of generation less than $(\iinfty-k)Q_m$. 
    Thus, $\Yext^{\iinfty-k}_{n,m}$ and $\lake^{\iinfty-k}_{n,m}$ are contained in $\Dext_{n,m}$. 
    Lastly, the containment $\Yext^{\iinfty-k}_{n,m} \subset \Yext^{\iinfty-k-1}_{n,m}$ follows from $(2)_m$ and the fact that these domains are mapped by $\Fext^{(\iinfty-k-2)Q_m}$ onto $\Yext^2_{n,m} \subset \Yext_{n,m}$. This establishes $(3)_m$(b).

    The map $\Fext^{Q_{m-1}}$ sends the interval $J =\parRess \Yext_{n,m+1}$ onto the interval 
    \[
    J' = [\crit_{-Q_n+Q_{m-1}},\crit_{-Q_n+Q_{m-1}+Q_{m}-Q_{m+1}}],
    \]
    which is contained in the interval $\parRess \Yext^{a_{m+1}}_{n,m}$. Hence, the lift of $\Yext^{a_{m+1}}_{n,m}$ along the map $\Fext^{Q_{m-1}}: J \to J'$ is the unique domain $\mathtt{E}$ containing $J$ on its boundary such that the composition
    \[
        E \xrightarrow[\; \Fext^{Q_{m-1}} \;]{} \Yext^{a_{m+1}}_{n,m} \xrightarrow[\; \Fext^{Q_{m+1}-Q_{m-1}} \;]{} \Wext_n
    \]
    is a conformal isomorphism. 
    By the definition of $\Yext_{n,m+1}$, this domain $\mathtt{E}$ must be equal to $\Yext_{n,m+1}$. 
    This establishes $(4)_m$.

    So far, we know that $(1)_n$ holds trivially and that for $m \geq n$, we have
    \begin{align}
        \label{eqn:1-implies-234}
        (1)_m \enspace \Longrightarrow \enspace (2)_m, \, (3)_m, \, (4)_m.
    \end{align}
    By $(2)_n$ and $(3)_n$, $\Yext^{a_{n+1}}_{n,n}$ is contained in $\Wext_n$, and by $(4)_n$, $\Yext_{n,n+1}$ is the lift of $\Yext^{a_{n+1}}_{n,n}$ under the map $\Fext^{Q_{n-1}}: \Uext''_n \to \Wext_n$.
    Therefore, the domain $\Yext_{n,n+1}$ is contained in $\Wext_n$ and so $(1)_{n+1}$ holds.
    
    To finish the proof, pick $m \geq n+1$. 
    Assuming $(1)_m, (2)_m, (3)_m, (4)_m$, the nest of domains $\Yext^{a_{n+1}}_{n,m} \subset \Wext_n$ lifts under $\Fext^{Q_{m-1}}$ to the nest of domains $\Yext_{n,m+1} \subset \Yext_{n,m-1}$. 
    Therefore, $(1)_{m-1}$ and $(j)_m$, $j \in \{1,2,3,4\}$ together imply $(1)_{m+1}$. 
    This completes the induction process.
\end{proof}

\begin{definition}
\label{def:enriched-butterfly}
    For $n \in \Z$, we define the $n$\textsuperscript{th} \emph{hoglet butterfly pair} to be the pair of maps
\[
    \butterfly_{n} = \left\{ \Fext^{Q_n+Q_{n-1}}: \Uext'_n \to \Wext_n, \quad \Fext^{Q_n}: \Uext''_n \to \Wext_n \right\}.
\]
For $m \geq n$, we define the \emph{depth $m$ enrichment} $\butterfly_n^m$ of the $n$\textsuperscript{th} hoglet butterfly to be the collection of maps defined as follows. 
Denote $\butterfly_n^{n-1} := \butterfly_n$ for convenience.
Inductively, for $m \geq n$, 
\[
    \butterfly_n^m := \begin{cases}
        \butterfly_n^{m-1} & \text{ if } a_{m+1}<\infty,\\
        \butterfly_n^{m-1} \cup \left\{ \Fext^{(\iinfty-k)Q_m}: \Yext_{n,m}^{\iinfty-k} \to \Wext_n \right\}_{k \geq 0} & \text{ if } a_{m+1}= \infty. 
    \end{cases}
\]
Also, we define the $n$\textsuperscript{th} \emph{fully enriched hoglet butterfly} to be the set of maps
\[
\butterfly_n^{\infty} := \bigcup_{m\geq n} \butterfly_n^{m}.
\]
For every $m \in \{n,n+1,\ldots,\infty\}$, $\butterfly_n^{m}$ generates a semigroup $\Butterfly_n^{m}$ under composition. 
\end{definition}

By Lemma \ref{lem:non-esc-parab} (4), $\Butterfly_n^{m}$ always contains the map 
\[
\Fext^{Q_{m+1}}: \Yext_{n,m+1} \to \Wext_n.
\]

\begin{lemma}[Uniqueness of parabolic cycles]
\label{lem:uniqueness-parabolic-cycle}
    For every $n \in \Z$ and every parabolic depth $m \in \Z$ greater than or equal to $n$, 
    every parabolic fixed point of $\Fext^{Q_m}$ that is periodic under $\butterfly_n^{m-1}$ is eventually mapped to $\beta_{n,m}$ under $\butterfly_n^{m-1}$.
\end{lemma}

Here, the phrase ``under $\butterfly_{n}^{m-1}$'' is to be interpreted as ``under a finite composition of maps in $\butterfly_{n}^{m-1}$''.

\begin{proof}
    The case $m=n$ is trivial because $\beta_{n,n}$ is the unique fixed point of $\Fext^{Q_n}$ on the closure of $\Uext''_n$. 
    To proceed, let us assume $m > n$.

    Pick a parabolic fixed point $\beta$ of $\Fext^{Q_m}$ that is periodic under $\butterfly_n^{m-1}$. 
    If $\beta$ lies on the interval $\parRess \Yext_{n,m}$ of $\Yext_{n,m}$, then $\beta$ is equal to $\beta_{n,m}$ because $\beta_{n,m}$ is the unique parabolic fixed point of $\Fext^{Q_m}$ on $\parRess \Yext_{n,m}$. 
    So let us assume otherwise. 
    We need to show that a finite composition of maps in $\butterfly_n^{m-1}$ sends $\beta$ to the interval $\parRess \Yext_{n,m}$. 

    Let 
    \[
        k := \max \big\{ k \in \Z \: : \: n \leq k \leq m-1, \: \beta \in \parRess \Yext_{n,k} \big\}.
    \]
    Denote $\zeta_k = \beta$.
    We split the argument into two cases.
    \vspace{0.1in}
    
    \noindent \underline{Case 1:} Suppose $a_{k+1} < \infty$. \\
    According to Figure \ref{fig:interval-orbit}, there exists some integer $j \in \{1,\ldots,a_{k+1}+2\}$ such that the $j$\textsuperscript{th} iterate of $\Fext^{Q_k}: \Yext_{n,k} \to \Wext_n$ sends $\zeta_k$ to a point $\zeta_{k+1}$ on the interval $\parRess \Yext_{n,k+1}$.
    \vspace{0.1in}
    
    \noindent \underline{Case 2:} Suppose $a_{k+1} = \infty$. \\
    Since $k \neq m$, the point $\zeta_k$ is not equal to the fixed point $\beta_{n,k}$ of $\Fext^{Q_k}: \Yext_{n,k} \to \Wext_n$. 
    Again there exist some $j_1 \in \{1,2\}$ and $j_2 \in \{1,2,\ldots\} \cup \{\iinfty, \iinfty-1,\iinfty-2,\ldots \}$ such that the composition of $\Fext^{j_1 Q_k}: \Yext^{j_1}_{n,k} \to \Wext_n$ followed by $\Fext^{j_2 Q_k} : \Yext^{j_2}_{n,k} \to \Wext_n$ sends $\zeta_k$ to a point $\zeta_{k+1}$ on the interval $\parRess \Yext_{n,k+1}$.
    \vspace{0.1in}

    If $k+1 < m$, we can apply the same argument to find a map in the semigroup $\Butterfly_n^{m-1}$ sending $\zeta_{k+1}$ to a point $\zeta_{k+2}$ on the interval $\parRess \Yext_{n,k+2}$. By repeating this argument, we conclude that there exists a map in $\Butterfly_n^{m-1}$ that sends $\beta$ to a point on $\parRess \Yext_{n,m}$.
\end{proof}

\subsection{The non-escaping set}
\label{ss:non-esc-set-butterfly}

\begin{definition}
    For $n \in \Z$, we define the \emph{domain} of $\butterfly_n$ to be
\[
    \Dom(\butterfly_n) = \Uext'_n \cup \Uext''_n,
\]
    and the \emph{non-escaping set} $\Lambda_n$ of $\butterfly_n$ to be the set of points in the domain of $\butterfly_n$ that do not escape under iteration of the two-to-one map $\Fext^{Q_n+Q_{n-1}}|_{\Uext'_n} \cup \Fext^{Q_n}|_{\Uext''_n}$. 
\end{definition}

\begin{lemma}[Structure of $\Lambda_n$]
\label{lem:nowhere-dense}
    Let $n \in \Z$. The non-escaping set $\Lambda_n$ of $\butterfly_n$ has no wandering domains. The interior $\textnormal{int}(\Lambda_n)$ of $\Lambda_n$ can be described as follows.
    \begin{enumerate}
        \item If depth $m$ is non-parabolic for all $m \geq n$, then $\textnormal{int}(\Lambda_n)$ is empty.
        \item Else, if $m$ is the smallest parabolic depth greater than or equal to $n$, then $\textnormal{int}(\Lambda_n)$ is the union of the immediate parabolic basin $\Dext_{n,m}$ of $\beta_{n,m}$ and its iterated preimages under $\butterfly_n$.
    \end{enumerate}
\end{lemma}

\begin{proof}
    Suppose that $\textnormal{int}(\Lambda_n)$ is non-empty. 
    Let us pick a connected component $\mathtt{E}$ of $\textnormal{int}(\Lambda_n)$. 
    Since the forward orbit of $\mathtt{E}$ under $\butterfly_n$ remains in $\Uext'_n \cup \Uext''_n$, then in the dynamical plane of $\Fbold$, the domain $\Psi^{-1}(\mathtt{E})$ is contained in the Fatou set of $(\Fbold^P)_{P \in \Tbold_{\Arch,n}}$ outside of $\Hbold$. 
    By Theorems \ref{thm:no-wandering-domains} and \ref{thm:parabolic-fatou-set}, the domain $\mathtt{E}$ must be contained in an attracting basin of a parabolic periodic point $\beta$ of $\butterfly_n$.
    Let $m \geq n$ be such that $\Fext^{Q_m}(\beta) = \beta$. 
    Since the period $Q_m$ is a linear combination of $Q_n$ and $Q_n+Q_{n-1}$, then $m$ is the smallest parabolic depth greater than or equal to $n$.
    
    By forward iteration under $\butterfly_n$, we can assume without loss of generality that $\mathtt{E}$ is a subset of the immediate attracting basin of $\beta$. 
    By Lemma \ref{lem:non-esc-parab}, the map $\Fext^{Q_m}: \Yext_{n,m} \to \Wext_n$ is a finite composition of maps in $\butterfly_n$ and its non-escaping set is precisely $\Dext_{n,m}$. Hence, $\Dext_{n,m}$ is a connected component of $\Lambda_n$. 
    The commuting pair
    \[
        \Fext^{Q_n+Q_{n-1}} : [0,\crit_{-Q_n}] \to [\crit_{Q_n+Q_{n-1}}, \crit_{Q_{n-1}}], \quad
        \Fext^{Q_n} : [\crit_{-Q_n}, \crit_{Q_{n-1}}] \to [0, \crit_{Q_n+Q_{n-1}}]
    \]
    contains a unique cycle of periodic points which consists of $\beta_{n,m}$ and its orbit under the commuting pair. 
    By Lemma \ref{lem:uniqueness-parabolic-cycle}, the periodic point $\beta$ has to be eventually mapped under $\butterfly_n$ to the point $\beta_{n,m}$, and $\mathtt{E}$ is eventually mapped into $\Dext_{n,m}$.
    Therefore, $\mathtt{E}$ must be an iterated preimage of $\Dext_{n,m}$.
\end{proof}

The lemma above can be enhanced by bringing in enrichments into the discussion. 

\begin{definition}
    For $m \in \{n, n+1, \ldots, \infty\}$, define the \emph{non-escaping set} $\Lambda_n^m$ of $\butterfly_n^m$ to be the largest subset of $\Dom(\butterfly_{n})$ with the property that 
\[
    h(\Lambda_n^m \cap \Dom(h)) = \Lambda_n^m \quad \text{ for every map } h \in \butterfly_n^m. 
\]
\end{definition}

The non-escaping set $\Lambda_n^m$ is always invariant under $\Butterfly_n^m$.
Moreover, we have the following nesting property:
\[
    \Lambda_n \supset \Lambda_n^n \supset \Lambda_n^{n+1} \supset \Lambda_n^{n+2} \supset \ldots \supset \Lambda_n^\infty.
\]

\begin{lemma}[Structure of $\Lambda_n^m$]
\label{lem:nowhere-dense-02}
    Let $n \in \Z$ and $m \in \Z \cup \{\infty\}$ such that $m \geq n$.
    \begin{enumerate}
        \item If every depth greater than $m$ is non-parabolic, then $\Lambda_n^{m}$ is nowhere dense.
        \item Else, if $k \in \Z$ is the smallest parabolic depth greater than $m$, then the interior of $\Lambda_n^m$ is the union of the iterated preimages of the immediate parabolic basin $\Dext_{n,k}$ under $\butterfly_n^m$.
    \end{enumerate}
\end{lemma}

\begin{proof}
    Let $\mathtt{E}$ be a connected component of $\textnormal{int}(\Lambda_n^{m})$.
    Then, $\Psi^{-1}(\mathtt{E})$ is contained in the Fatou set of $(\Fbold^P)_{P \in \Tbold_{\Arch,m}}$ outside of the Mother Hedgehog. 
    By Theorems \ref{thm:no-wandering-domains} and \ref{thm:parabolic-fatou-set}, $\mathtt{E}$ can only be contained in the attracting basin of a parabolic periodic point of $\Fext$ with some period $Q_k$ with $k \geq m+1$. 
    Since $Q_k$ has to be an integral linear combination of $Q_n, Q_{n+1}, \ldots, Q_{m+1}$, then $k$ has to be the smallest parabolic depth greater than $m$. 
    According to Lemma \ref{lem:non-esc-parab}, the basin $\Dext_{n,k}$ is indeed a connected component of the interior of $\Lambda_n^m$.
    Then by Lemma \ref{lem:uniqueness-parabolic-cycle}, $\mathtt{E}$ must be eventually mapped onto the immediate parabolic basin $\Dext_{n,k}$.
\end{proof}


\section{Rigidity of bi-infinite towers}
\label{sec:qc-rigidity}

In this section, we will prove the following theorem.

\begin{theorem}[Affine Rigidity of neutral cascades]
\label{thm:qc-rigidity}
    For any two combinatorially equivalent neutral cascades $\Fbold = \{\Fbold^P\}_{P \in \Tbold}$ and $\Gbold = \{\Gbold^P\}_{P \in \Tbold}$, there exists a quasiconformal map $\Phi : \C \to \C$ such that 
    \[
    \Phi \circ \Fbold^P = \Gbold^P \circ \Phi \qquad \text{ for all } P \in \Tbold
    \]
    and that $\Phi$ is conformal on the interior of $\Hbold(\Fbold)$.
\end{theorem}

Theorems \ref{thm:no-wandering-domains} and \ref{thm:NILF} tell us that the conjugacy $\Phi$ has to be conformal almost everywhere, hence affine.
Altogether, these theorems imply Theorem \ref{main-thm:cascades}.

Figure \ref{fig:roadmap} gives a summary of the proof of Theorem \ref{thm:qc-rigidity}.
In \S\ref{ss:reduction-to-butterflies}, we first transfer the QC Thurston equivalence of bi-infinite towers $\fboldbar$ and $\gboldbar$ to the QC Thurston equivalence of the corresponding neutral cascades $\Fbold$ and $\Gbold$.
We pass to external coordinates in which the corresponding external cascades $\Fext$ and $\Gext$ are quasisymmetrically conjugate.
In \S\ref{ss:proof-qc-rigidity}, we show how quasiconformal conjugacy between the fully enriched hoglet butterflies of $\Fext$ and $\Gext$ implies global quasiconformal conjugacy, hence the desired theorem.
In \S\ref{ss:proof-main-lemma}, we apply the pullback argument to construct quasiconformal conjugacy between the fully enriched hoglet butterflies of $\Fext$ and $\Gext$ of the same depth using the machinery prepared in Section \ref{sec:butterfly}.

Lastly, in \S\ref{ss:two-main-theorems}, we show how Theorem \ref{main-thm:cascades} implies a quantitative version of our first main Theorem \ref{main-thm:rigidity}.

\subsection{Reduction to hoglet butterflies in external coordinates}
\label{ss:reduction-to-butterflies}

For the rest of the section, we will fix two neutral cascades $\Fbold = \{\Fbold^P\}_{P \in \Tbold}$ and $\Gbold = \{\Gbold^P\}_{P \in \Tbold}$ with the same combinatorics $\tttheta = \seq{(\varepsilon_n, \abar_n)}_{n\in\Z} \in \TheBiCpt$.
We will denote by $\Kbold$ and $\Tbold$ the continuant group and the time semigroup of $\tttheta$; we will be using the slow generating set $\{Q_n = Q_n(\tttheta)\}_{n\in\Z}$.

We will denote by $\hat{\Zbold}^{[m]}(\Fbold)$ the level $m$ pseudo-Siegel pinched half-plane of $\Fbold$ for $m \in \Z$. 
The union is the pseudo-Siegel half-plane $\hat{\Zbold}(\Fbold)$ and the intersection is the Mother Hedgehog $\Hbold(\Fbold)$.
We will also denote by $\hat{\Zbold}^{[m]}(\Gbold)$ and $\Hbold(\Gbold)$ the analogous objects for $\Gbold$.

\begin{lemma}[QC Thurston equivalence on $\hat{\Zbold}$]
\label{lem:qc-thurston-equivalence-02}
    There exist a $\threshold$-uniform constant $K>1$ and a global $K$-quasiconformal map $\Phi_1: (\C,0) \to (\C,0)$ such that
    \begin{enumerate}
        \item for $m \in \Z$, $\Phi_1$ sends $\hat{\Zbold}^{[m]}(\Fbold)$ onto $\hat{\Zbold}^{[m]}(\Gbold)$;
        \item $\Phi_1$ conjugates $\Fbold : \Hbold(\Fbold) \to \Hbold(\Fbold)$ and $\Gbold : \Hbold(\Gbold) \to \Hbold(\Gbold)$, and it is conformal on the interior of $\Hbold(\Fbold)$.
    \end{enumerate}
\end{lemma}

\begin{proof}
    The neutral cascade $\Fbold$ comes with a bi-infinite orbit $\fboldbar = \seq{f_n}_{n\in\Z} \in \attr$ with combinatorics $\tttheta$, a bi-infinite nest of infinite sectors $\{\Sbold_{n}(\Fbold)\}_{n\in \Z}$, and a sequence of conformal gluing maps $\{ \pphi_{n,\Fbold}: (\Sbold^n(\Fbold),0) \to (\D, v_0(f_n)) \}_{n\in\Z}$ satisfying Proposition \ref{prop:trans-sectors}.
    For every $n \in \Z$, let $S_{n}(\fboldbar)$ be the first renormalization sector of $f_n$ and let $\psi_{n,\fboldbar}$ be the conformal gluing map for $S_{n}(\fboldbar)$ projecting the first return map $f_n^{[1]}$ (mod $f_n$) to $f_{n+1}$.
    We have for all $n \in \Z$, 
    \[
    \psi_{n,\fboldbar} \circ \pphi_{n,\Fbold} = \pphi_{n+1,\Fbold} \qquad \text{ on } \Sbold^{n+1}(\Fbold).
    \]
    On the dynamical plane of $f_n$, we also have a semigroup of maps $\mathcal{F}_n$ which is essentially generated by $f_n$ and pre-renormalizations $f_n^{[m]}$, $m \geq 1$.
    The Mother Hedgehog $H_n(\fboldbar)$ is the fill of the postcritical set of $\mathcal{F}_n$.
    Similarly, $\Gbold$ comes with a bi-infinite tower $\gboldbar = \seq{g_n}_{n\in\Z}$, a nest of sectors $\Sbold^n(\Gbold)$ with gluing maps $\pphi_{n,\Gbold}$, first renormalization sectors $S_{n}(\gboldbar)$ with gluing maps $\psi_{n,\gboldbar}$, semigroup $\mathcal{G}_n$, and the Mother Hedgehog $H_n(\gboldbar)$ for $n \in \Z$.

    According to Theorem \ref{thm:qc-thurston-equivalence},
    there exists an $\threshold$-uniform constant $K'>1$ and a sequence of $K'$-quasiconformal maps 
    \[
    h_n: (\D, v_0(f_n)) \to (\D, v_0(g_n)), \qquad n \in \Z,
    \]
    where each $h_n$ conjugates $\mathcal{F}_n|_{H_n(\fboldbar)}$ and $\mathcal{G}_n|_{H_n(\gboldbar)}$.
    These can be arranged such that for all $n \in \Z$,
    \begin{itemize}
        \item $h_n$ is conformal on the interior of $H_n(\fboldbar)$,
        \item for all $m \geq -1$, $h_n$ sends $\hat{Z}^{[m]}_n(\fboldbar))$ onto $\hat{Z}^{[m]}_n(\gboldbar)$, and 
        \item on $S_n(\fboldbar)$, we have $h_{n+1} \circ \psi_{n,\fboldbar} = \psi_{n,\gboldbar} \circ h_n$.
    \end{itemize} 
    For every $n\in\Z$, define the quasiconformal map
    \[
        \mathbf{h}_n: \hat{\Zbold}(\Fbold) \cap \Sbold^n(\Fbold) \to \hat{\Zbold}(\Gbold) \cap \Sbold^n(\Gbold), \quad 
        \mathbf{h}_n = \pphi_{n,\gboldbar}^{-1} \circ h_n \circ \pphi_{n,\fboldbar}.
    \]
    This is a well-defined $K'$-quasiconformal map fixing $0$ because for every $n$,
    \[
        \hat{\Zbold}(\Fbold) \cap \Sbold^n(\Fbold) 
        = \pphi_{n,\fboldbar}^{-1} \big( \hat{Z}_{n}(\fboldbar) \big) \quad \text{and} \quad \hat{\Zbold}(\Gbold) \cap \Sbold^n(\Gbold) 
        = \pphi_{n,\gboldbar}^{-1}(\hat{Z}_{n} \big(\gboldbar) \big).
    \]
    On the set $\hat{\Zbold}(\Fbold) \cap \Sbold^{n+1}(\Fbold)$, we have
    \begin{align*}
    \mathbf{h}_{n+1} 
    &= \pphi_{n+1,\gboldbar}^{-1} \circ h_{n+1} \circ \pphi_{n+1,\fboldbar} \\
    &= \pphi_{n+1,\gboldbar}^{-1} \circ \psi_{n,\gboldbar} \circ h_n \circ \psi_{n,\fboldbar}^{-1} \circ \pphi_{n+1,\fboldbar} \\
    &= \pphi_{n,\gboldbar}^{-1} \circ h_n \circ \pphi_{n,\fboldbar} = \mathbf{h}_n.
    \end{align*}
    The equation above implies that $\mathbf{h}_n$'s glue to a $K'$-quasiconformal map
    \[
        \Phi: \hat{\Zbold}(\Fbold) \to \hat{\Zbold}(\Gbold).
    \]
    and it automatically satisfy properties (1) and (2) of $\Phi$.
    Since both $\hat{\Zbold}(\Fbold)$ and $\hat{\Zbold}(\Gbold)$ are $\threshold$-uniformly quasiconformal, $\Phi$ extends to a global $\threshold$-uniformly quasiconformal map $\Phi: \C \to \C$.
\end{proof}

Next, consider the uniformizations 
\[
\Psi_{\Fbold}: \big( \C \backslash \Hbold(\Fbold), 0 \big) \to (\UHP,0)
\quad \text{ and } \quad 
\Psi_{\Gbold}: \big( \C \backslash \Hbold(\Gbold), 0 \big) \to (\UHP,0)
\]
associated to $\Fbold$ and $\Gbold$. 
These maps bring $\Fbold$ and $\Gbold$ to external cascades $\Fext=(\Fext^P)_{P \in \Tbold}$ and $\Gext=(\Gext^P)_{P \in \Tbold}$ respectively.
(See \S\ref{ss:trans-external-coordinates}.)
Consider the $K$-quasiconformal map
\[
    \Phi_2 := \Psi_{\Gbold} \circ \Phi_1|_{\C \backslash \Hbold(\Fbold)} \circ \Psi_{\Fbold}^{-1} : \UHP \to \UHP.
\]

\begin{lemma}
\label{lem:qs-rigidity}
    The map $\Phi_2$ extends to a uniformly quasisymmetric homeomorphism $\Phi_2: (\R,0) \to (\R,0)$ of the real line conjugating $\Fext|_{\R}$ and $\Gext|_{\R}$ for all $P \in \Tbold$.
\end{lemma}

\begin{proof}
    As a quasiconformal self-map of $\UHP$, $\Phi_2$ uniquely extends to a quasisymmetric homeomorphism of $\R$.
    Since the critical orbit on the Mother Hedgehog is accessible from its complement, then $\Phi_2|_{\R}$ has to send $\crit_{P,\Fbold}$ to $\crit_{P,\Gbold}$ for all time $P \in \Tbold$. 
    By virtue of Real Bounds (Proposition \ref{prop:R-apb-transcendental}), these critical orbits are dense on $\R$ and so $\Phi_2$ is indeed a conjugacy between $\Fext$ and $\Gext$ on $\R$.
    Moreover, since $\Phi_2: \R \to \R$ is completely determined by the critical orbit, it is independent of $\threshold$.
\end{proof}

For $n \in \Z$, consider for $\Fext$ the domains $\Uext'_{n,\Fext}$, $\Uext''_{n,\Fext}$, $\Wext_{n,\Fext}$, the collection of enriched hoglet butterflies $\butterfly_{n,\Fext}^m$ and the corresponding semigroups $\Butterfly_{n,\Fext}^m$ for $m \in \{n,n+1,\ldots\infty\}$ as introduced in \S\ref{ss:butterfly}--\ref{ss:enriched-butterfly}.
Similarly, we consider
$\Uext'_{n,\Gext}$, $\Uext''_{n,\Gext}$, $\Wext_{n,\Gext}$, $\butterfly_{n,\Gext}^m$, and $\Butterfly_{n,\Gext}^m$ associated to $\Gext$.

\begin{lemma}
\label{lem:main-lemma}
    There exists a universal constant $K \geq 1$ such that for every $n \in \Z$, there exists a $K$-quasiconformal map $\Phi_{2,n}: \UHP \to \UHP$ where
    \begin{enumerate}
        \item the continuous extension of $\Phi_{2,n}$ to the real line is equal to $\Phi_2: \R \to \R$;
        \item $\Phi_{2,n}$ sends the domains $\Uext'_{n,\Fext}$, $\Uext''_{n,\Fext}$, $\Wext_{n,\Fext}$ onto the domains $\Uext'_{n,\Gext}$, $\Uext''_{n,\Gext}$, $\Wext_{n,\Gext}$ respectively;
        \item $\Phi_{2,n}$ conjugates $\Butterfly_{n,\Fext}^\infty$ with $\Butterfly_{n,\Gext}^\infty$.
    \end{enumerate}
\end{lemma}

We will prove this lemma later in \S\ref{ss:proof-main-lemma}.

\subsection{Proof of Theorem \ref{thm:qc-rigidity} assuming Lemma \ref{lem:main-lemma}}
\label{ss:proof-qc-rigidity}

    Fix $n \in \Z$.
    Denote by $\Ubold'_{n,\Fbold}$, $\Ubold''_{n,\Fbold}$, $\Wbold_{n,\Fbold}$, $\Ubold'_{n,\Gbold}$, $\Ubold''_{n,\Gbold}$, $\Wbold_{n,\Gbold}$ the lift of
    the domains $\Uext'_{n,\Fext}$, $\Uext''_{n,\Fext}$, $\Wext_{n,\Fext}$, $\Uext'_{n,\Gext}$, $\Uext''_{n,\Gext}$, $\Wext_{n,\Gext}$ under $\Psi_\Fbold$ and $\Psi_\Gbold$ respectively. 
    Consider the quasiconformal maps $\Phi_1: \C \to \C$ and $\Phi_{2,n}: \UHP \to \UHP$ from Lemmas \ref{lem:qc-thurston-equivalence-02} and \ref{lem:main-lemma}.
    Define
    \[
        \Phi_{3,n}: \C \to \C, \qquad
        \Phi_{3,n} = \begin{cases}
            \Phi_{1} & \text{ on } \Hbold(\Fbold), \\
            \Psi_{\Gbold}^{-1} \circ \Phi_{2,n} \circ \Psi_{\Fbold} & \text{ on } \C \backslash \Hbold(\Fbold).
        \end{cases}
    \]
    By design, $\Phi_{3,n}$ is continuous on the boundary of $\Hbold(\Fbold)$ and so it is a global quasiconformal map.
    Moreover,
    \begin{enumerate}[label=\textnormal{(\roman*)}]
        \item \label{equivariance-hedgehog} $\Phi_{3,n}$ is a conjugacy between $\Fbold|_{\Hbold(\Fbold)}$ and $\Gbold|_{\Hbold(\Gbold)}$;
        \item \label{equivariance-butterfly} $\Phi_{3,n}$ sends $\Ubold'_{n,\Fbold}$, $\Ubold''_{n,\Fbold}$, $\Wbold_{n,\Fbold}$ onto $\Ubold'_{n,\Gbold}$, $\Ubold''_{n,\Gbold}$, $\Wbold_{n,\Gbold}$ respectively, and it is a conjugacy between the semigroups $\Psi_{\Fbold}^{-1} \circ \Butterfly_{n,\Fext}^\infty \circ \Psi_{\Fbold}$ and $\Psi_{\Gbold}^{-1} \circ \Butterfly_{n,\Gext}^\infty \circ \Psi_{\Gbold}$.
    \end{enumerate}
    Below, we will spread around $\Phi_{3,n}$ to a neighborhood of the depth $n$ pseudo-Siegel pinched half-plane $\hat{\Zbold}^n(\Fbold)$ (i.e. the union of $\Hbold(\Fbold)$ and all depth $\geq n+1$ fjords).

    Let $\widetilde{\Wbold}_{n,\Fbold}$ be the open simply connected domain bounded by the internal rays $I_{-Q_{n-1},\Fbold}$, $I_{-2Q_n+Q_{n-1}}$, and the arc $\partial \Wbold_{n,\Fbold} \backslash \Hbold$ with the internal rays $I_{0,\Fbold}$ and $I_{Q_{n-1},\Fbold}$ removed.
    Observe that $\widetilde{\Wbold}_{n,\Fbold}$ contains the domain $\Wbold_{n,\Fbold}$.
    Let $\Sigma^\dagger_{n,\Fbold}$ be the closure of the unique lift of $\widetilde{\Ubold'}_{n,\Fbold}$ under $\Fbold^{Q_n+Q_{n-1}}$ that contains $\Ubold'_{n,\Fbold}$; it is an unbounded closed topological disk and has a unique critical value of $\Fbold^{Q_n+Q_{n-1}}$ at $0$.

    Let $\Sigma'_{n,\Fbold}$ be the closure of the unique lift of $\Sigma^\dagger_{n,\Fbold} \backslash I_{0,\Fbold}$ under $\Fbold^{Q_n+Q_{n-1}}$ that contains both $I_{-Q_n-Q_{n-1},\Fbold}$ and $I_{-2Q_n-Q_{n-1},\Fbold}$ on its boundary.
    Similarly, let $\Sigma''_{n,\Fbold}$ be the closure of the unique lift of $\Sigma^\dagger_{n,\Fbold} \backslash I_{0,\Fbold}$ under $\Fbold^{Q_n}$ that contains both $I_{-2Q_n-Q_{n-1},\Fbold}$ and $I_{-Q_n,\Fbold}$ on its boundary.
    Both $\Sigma'_{n,\Fbold}$ and $\Sigma''_{n,\Fbold}$ are contained in $\Sigma^\dagger_{n,\Fbold}$.
    
    Define
    \[
        \Sigma'_{n,\Fbold}(P) := \Fbold^P( \Sigma'_{n,\Fbold} ) \qquad
        \text{ for } P \in \Tbold \text{ with } P < Q_n + Q_{n-1}
    \]
    and
    \[
        \Sigma''_{n,\Fbold}(P) := \Fbold^P( \Sigma''_{n,\Fbold} ) \qquad
        \text{ for } P \in \Tbold \text{ with } P < Q_n.
    \]
    The collection
    \[
        \boldsymbol{\Sigma}_{n,\Fbold} := \Sigma^\dagger_{n,\Fbold} \cup  \left\{ \Sigma'_{n,\Fbold}(P) \right\}_{P \in \Tbold, P< Q_n+Q_{n-1}} 
        \cup
        \left\{ \Sigma''_{n,\Fbold}(P) \right\}_{P \in \Tbold, P<Q_n} 
    \]
    forms a triangulation of a neighborhood of $\hat{\Zbold}^n (\Fbold)$.
    Any two distinct triangles in $\boldsymbol{\Sigma}_{n,\Fbold}$ are either disjoint or meet along an arc that is the concatenation of the internal ray $I_{-P,\Fbold}$ of $\Hbold(\Fbold)$ and the internal ray of a primary hoglet of generation $P$ for some $P \in \Tbold$, $P < 2Q_n+Q_{n-1}$.
    See Figure \ref{fig:renorm-tiling-cascade}.

\begin{figure}
        \centering
    \begin{tikzpicture}
    \node[anchor=south west, inner sep=0] (image) at (0,0) {\includegraphics[width=1\linewidth]{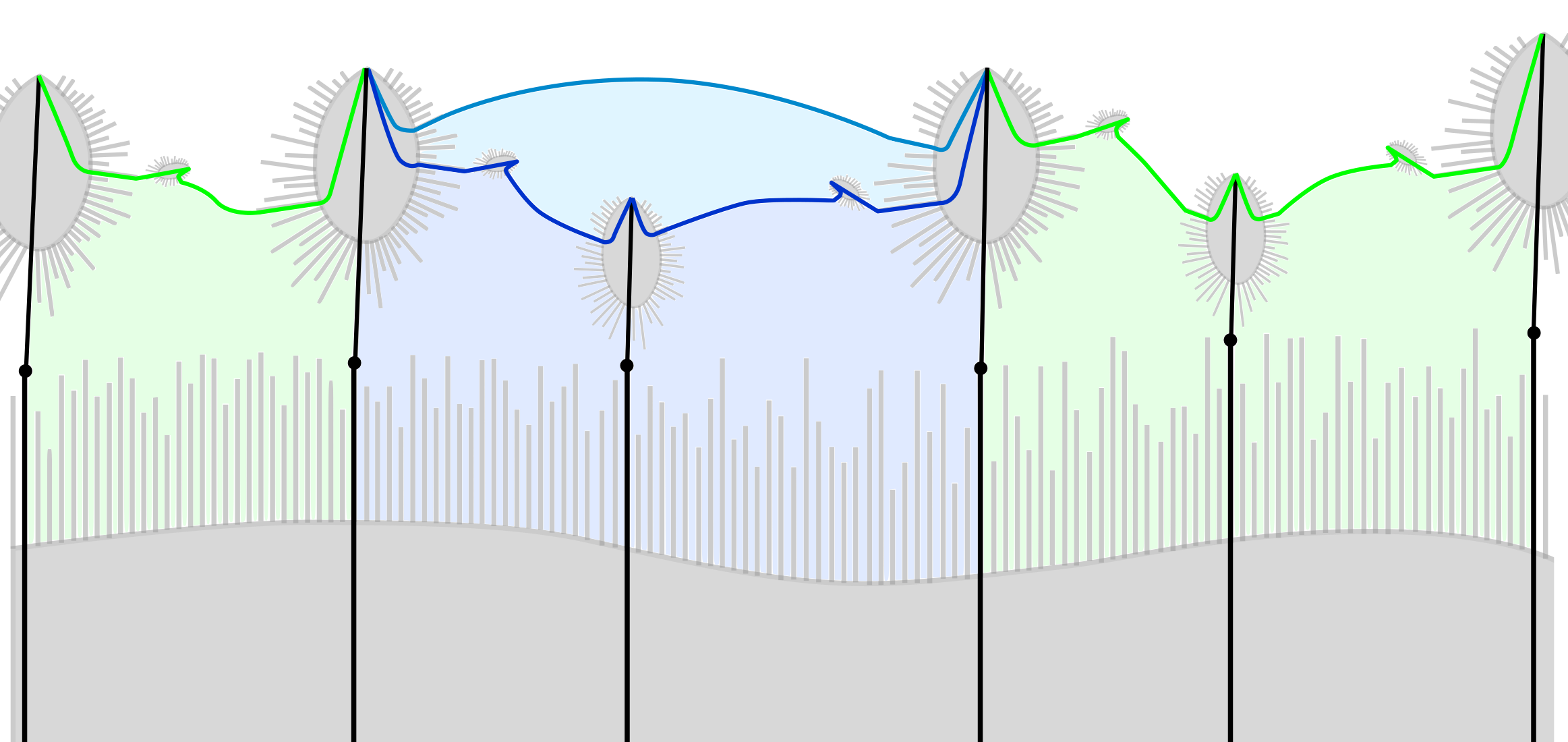}};
    \begin{scope}[
        x={(image.south east)},
        y={(image.north west)}
    ]
        \node [black!50!blue] at (0.415,0.81) {\scalebox{0.9}{$\Sigma^\dagger_{n,\Fbold}$}};
        \node [black!60!green] at (0.12,0.225) {\scalebox{0.85}{$\Sigma'_{n,\Fbold}(Q_n)$}};
        \node [black!70!blue] at (0.31,0.225) {\scalebox{0.9}{$\Sigma'_{n,\Fbold}$}};
        \node [black!70!blue] at (0.515,0.225) {\scalebox{0.9}{$\Sigma''_{n,\Fbold}$}};
        \node [black!60!green] at (0.705,0.225) {\scalebox{0.85}{$\Sigma'_{n,\Fbold}(Q_{n-1})$}};
        \node [black!60!green] at (0.885,0.225) {\scalebox{0.85}{$\Sigma''_{n,\Fbold}(Q_{n-1})$}};
        
        \node [black] at (0.64,-0.04) {\scalebox{0.85}{$I_{-Q_n, \Fbold}$}};
        \node [black] at (0.41,-0.04) {\scalebox{0.85}{$I_{-2Q_n-Q_{n-1},\Fbold}$}};
        \node [black] at (0.23,-0.04) {\scalebox{0.85}{$I_{-Q_n-Q_{n-1},\Fbold}$}};
        \draw [gray!50!black,-latex] (0.3,0.65) .. controls (0.32,0.95) and (0.34,0.95) .. (0.36,0.84);
        \node [black] at (0.33,0.96) {\scalebox{0.9}{$\Fbold^{Q_n+Q_{n-1}}$}};
        \draw [gray!50!black,-latex] (0.52,0.65) .. controls (0.5,0.95) and (0.48,0.95) .. (0.46,0.84);
        \node [black] at (0.52,0.96) {\scalebox{0.9}{$\Fbold^{Q_n}$}};
    \end{scope}
\end{tikzpicture}
        \caption{The triangulation $\mathbf{\Sigma}_{n, \Fbold}$ obtained by spreading around $\Sigma^{\dagger}_{n,\Fbold}$.}
        \label{fig:renorm-tiling-cascade}
\end{figure}
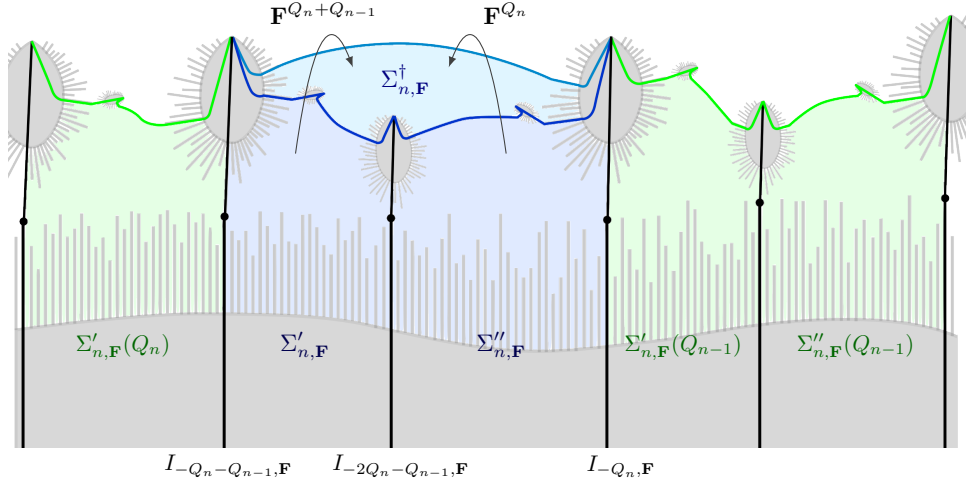

    In the dynamical plane of $\Gbold$, we define internal rays of $\Hbold(\Gbold)$ to be compatible with internal rays of $\Hbold(\Fbold)$ via $\Phi_{3,n}$. 
    Then, we define the triangulation 
    \[
        \boldsymbol{\Sigma}_{n,\Gbold} := \Sigma^\dagger_{n,\Gbold} \cup  \left\{ \Sigma'_{n,\Gbold}(P) \right\}_{P \in \Tbold, P< Q_n+Q_{n-1}} 
        \cup
        \left\{ \Sigma''_{n,\Gbold}(P) \right\}_{P \in \Tbold, P<Q_n} 
    \]
    of $\hat{\Zbold}^n(\Gbold)$ in a way analogous to $\boldsymbol{\Sigma}_{n,\Fbold}$.
    On $\Sigma^\dagger_{n,\Fbold}$, we define $\Phi_{4,n} := \Phi_{3,n}$.
    On each tile $\Sigma'_{n,\Fbold}(P)$ and $\Sigma''_{n,\Fbold}(P)$, we define $\Phi_{4,n}$ to be the unique map such that the two diagrams below are commutative.
\[
        \begin{tikzcd}[column sep=2.3cm,row sep=large]
	\Sigma'_{n,\Fbold}(P) & \Sigma^\dagger_{n,\Fbold} & \Sigma'_{n,\Fbold}(-P) & \Sigma^\dagger_{n,\Fbold} \\
	\Sigma'_{n,\Gbold}(P) & \Sigma^\dagger_{n,\Gbold} & \Sigma'_{n,\Gbold}(-P) & \Sigma^\dagger_{n,\Gbold}
	\arrow["\Fbold^{Q_n+Q_{n-1}-P}", from=1-1, to=1-2]
	\arrow["\Fbold^{Q_n-P}", from=1-3, to=1-4]
	\arrow["\Gbold^{Q_n+Q_{n-1}-P}"', from=2-1, to=2-2]
	\arrow["\Gbold^{Q_n-P}"', from=2-3, to=2-4]
	\arrow["\Phi_{4,n}"', from=1-1, to=2-1]
	\arrow["\Phi_{4,n}", from=1-2, to=2-2]
	\arrow["\Phi_{4,n}"', from=1-3, to=2-3]
	\arrow["\Phi_{4,n}", from=1-4, to=2-4]
        \end{tikzcd}
\]
    Properties \ref{equivariance-hedgehog} and \ref{equivariance-butterfly} of $\Phi_{3,n}$ imply that $\Phi_{4,n}$ glues continuously to a uniformly quasiconformal map 
    \[
        \Phi_{4,n} : \cup \boldsymbol{\Sigma}_{n,\Fbold} \to \cup \boldsymbol{\Sigma}_{n,\Gbold}.
    \]
    that conjugates $\Fbold$ and $\Gbold$.

    As $n \to -\infty$, $\cup \mathbf{\Sigma}_{n,\Fbold}$ converges to the whole plane and $\Phi_{4,n}$ converges in subsequence to a global quasiconformal conjugacy $\Phi: \C \to \C$ between $\Fbold$ and $\Gbold$.
    This concludes the proof of the theorem.

\subsection{The proof of Lemma \ref{lem:main-lemma}}
\label{ss:proof-main-lemma}

In this subsection, we will freeze the index $n \in \Z$. 
For convenience, we will drop the subscript $n$, e.g. 
$\butterfly_{\Fext} = \butterfly_{n,\Fext}$, $\Butterfly^m_{\Fext} = \Butterfly^m_{n,\Fext}$ for $m \in \{n,n+1,\ldots,\infty\}$, 
$\Wext_\Fext = \Wext_{n,\Fext}$, $\Uext'_\Fext = \Uext'_{n,\Fext}$, $\Uext''_{\Fext} = \Uext''_{n,\Fext}$, etc. 

For any quasiconformal self map $h$ of the upper half plane $\UHP$, we say that $h$ is \emph{suitable} if its continuous extension onto the real line is equal to the quasisymmetric map $\Phi_2: \R \to \R$ from Lemma \ref{lem:qs-rigidity}.
To prove Lemma \ref{lem:main-lemma}, we will construct a suitable uniformly quasiconformal self-map of $\UHP$ that conjugates the two fully enriched butterfly semigroups $\Butterfly_{\Fext}^\infty$ and $\Butterfly_{\Gext}^\infty$.

\subsubsection{QC Thurston equivalence on butterflies: an equivariant map $h_1 : \UHP \to \UHP$}\label{ss:QC.Th.eq.batt}
Recall from Section \ref{sec:butterfly} that the boundary of the domain of $\butterfly_{\Fext}$ is surrounded by the collection of hoglets
\[
    \mathscr{B}_{\Fext} = 
        \big\{
        \Bubb_{Q_n+Q_{n-1}, \Fext}, \; \Bubb_{Q_{n},\Fext}, \; \Bubb_{Q_n-Q_{n-1},\Fext}
        \big\}.
\]
Every hoglet in $\mathscr{B}_{\Fext}$ is contained in a unique pseudo-bubble of the same generation; we will denote the collection of such pseudo-bubbles by $\hat{\mathscr{B}}_{\Fext}$. 
Recall from Lemma \ref{lem:primary-bubble-disjoint} that pseudo-bubbles in $\hat{\mathscr{B}}_{\Fext}$ are pairwise disjoint.
Similarly, we define $\mathscr{B}_{\Gext}$ and $\hat{\mathscr{B}}_{\Gext}$ for $\Gext$.

Consider the QC Thurston equivalence $\Phi_1$ from Lemma \ref{lem:qc-thurston-equivalence-02}. 
After lifting $\Phi_1$ and passing to external coordinates, we obtain the quasiconformal map 
\[
    h_{-1} : \bigcup \hat{\mathscr{B}}_{\Fext} \to \bigcup \hat{\mathscr{B}}_{\Gext}
\]
The map $h_{-1}$ sends any pseudo-bubble $\hat{\Gamma}_{\Fext} \in \hat{\mathscr{B}}_{\Fext}$ of $\Fext$ of some generation, say $P$, onto its counterpart $\hat{\Gamma}_{\Gext} \in \hat{\mathscr{B}}_{\Gext}$ with the same generation $P$, and it is the unique map such that the following diagram commutes.
\begin{center}
    \begin{tikzcd}[column sep = huge, row sep = large]
 \hat{\Gamma}_{\Fext} \arrow[ r , "\Fbold^P \circ \Psi_\Fbold"] 
 \arrow[d, "h_{-1}"']
 & \hat{\Zbold}(\Fbold) \arrow[d, "\Phi_1"] \\
 \hat{\Gamma}_{\Gext} \arrow[r, "\Gbold^P \circ \Psi_\Gbold"']
 & \hat{\Zbold}(\Gbold)
    \end{tikzcd}
\end{center}

The pullback argument starts by selecting a suitable quasiconformal map $h_0 : \UHP \to \UHP$ sending $\partial \Wext_{\Fext}$ to $\partial \Wext_{\Gext}$.
This is possible because both $\parH \Wext_{\Fext}$ and $\parH \Wext_{\Gext}$ are uniform quasiarcs meeting the real line at some definite angle.
Then, we lift $h_0: \Wext_{\Fext} \to \Wext_{\Gext}$ to a quasiconformal map $h_1: \Dom(\butterfly_{\Fext}) \to \Dom(\butterfly_{\Gext})$ such that the following diagram commutes.
\begin{center}
    \begin{tikzcd}[column sep = huge, row sep = large]
 \Uext'_{\Fext} \arrow[ r , "\Fext^{Q_{n+1}}"] \arrow[d, "h_1"']
 & \Wext_{\Fext} \arrow[d, "h_0"] 
 & \Uext''_{\Fext} \arrow[ l , "\Fext^{Q_{n}}"'] \arrow[d, "h_1"]
 \\
 \Uext'_{\Gext} \arrow[r, "\Gext^{Q_{n+1}}"']
 & \Wext_{\Gext}
 & \Uext''_{\Gext} \arrow[l, "\Gext^{Q_{n}}"]
    \end{tikzcd}
\end{center}
Denote
\[
    \Omega_{\Fext} := \Wext_{\Fext} \backslash 
    \left( \overline{\Dom(\butterfly_{\Fext})} \cup \bigcup \hat{\mathscr{B}}_{\Fext} \right).
\]
Let us extend $h_1$ to a suitable quasiconformal self-map of $\UHP$ by setting
\[
    h_1 := \begin{cases}
        h_0 & \text{ on } \UHP \backslash \overline{\Wext_{\Fext}}, \\
        h_{-1} & \text{ on } \bigcup \hat{\mathscr{B}}_{\Fext} \backslash 
        \Dom(\butterfly_{\Fext}), \\
        \textnormal{QC interpolation} & \text{ on } \Omega_{\Fext}.
    \end{cases}
\]
Observe that $h_1$ is continuous on the boundary of $\Dom(\butterfly_{\Fext})$ because the equivariance of $\Phi_1$ induces the equivariance of $h$ on the hoglets $\bigcup \mathscr{B}_{\Fext}$. 
The existence of the quasiconformal interpolation between $\Omega_{\Fext}$ and its counterpart $\Omega_{\Gext}$ for $\Gext$ follows from the fact that $h_1$ is quasisymmetric on the boundary of $\Omega_\Fbold$ and that both $\partial \Omega_\Fbold$ and $\partial \Omega_\Gbold$ are uniform quasicircles (Proposition \ref{prop:uniform-qc-bounded}).

\subsubsection{The Pullback Argument}\label{sss:PA} For $k \geq 1$, let us denote 
\[
\Dom_k(\butterfly_{\Fext}) := (\butterfly_{\Fext})^{-k}(\Wext_{\Fext})
\]
where $\butterfly_{\Fext}$ is treated as the two-to-one map $\Fext^{Q_n+Q_{n-1}}|_{\Uext'_n} \cup \Fext^{Q_n}|_{\Uext''_n}$. 
Similarly, we define the sets $\Dom_k(\butterfly_{\Gext})$, $k \geq 1$. 
Let us lift $h_1$ to an infinite sequence of uniformly quasiconformal self-homeomorphisms $h_2, h_3,\ldots $ of $\UHP$ where for every $j \geq 1$, we have 
\[
h_{j+1} = h_j \qquad \text{ on } \Wext_{\Fbold} \backslash \Dom_j(\butterfly_{\Fext})
\]
and the following diagram commutes.
\begin{center}
    \begin{tikzcd}[column sep = huge, row sep = large]
 \Dom_j(\butterfly_{\Fext}) \arrow[ r , "\butterfly_{\Fext}"] \arrow[d, "h_{j+1}"']
 & \Dom_{j-1}(\butterfly_{\Fext}) \arrow[d, "h_j"] 
 \\
 \Dom_j(\butterfly_{\Gext}) \arrow[r, "\butterfly_{\Gext}"']
 & \Dom_{j-1}(\butterfly_{\Gext})
    \end{tikzcd}
\end{center}
Note that each $h_j$ is again continuous on $\UHP$ because the initial map $h_1$ is equivariant on the boundary of $\Dom(\butterfly_{\Fext})$. 
Let $K \geq 1$ be the QC dilatation of $h_1$. 
Since we are pulling back by a conformal map, the dilatation of each $h_j$ is also $K$. 
Moreover, $h_j$ is suitable.

By the compactness of normalized $K$-QC maps, as $j \to \infty$,
$h_j$ converges to a suitable $K$-QC map $h_\infty: \UHP \to \UHP$ such that the following diagram commutes.
\begin{center}
    \begin{tikzcd}[column sep = huge, row sep = large]
 \Dom(\butterfly_{\Fext}) \backslash \Lambda_{\Fext} \arrow[ r , "\butterfly_{\Fext}"] \arrow[d, "h_{\infty}"']
 & \Wext_{\Fext} \backslash \Lambda_{\Fext} \arrow[d, "h_{\infty}"] 
 \\
 \Dom(\butterfly_{\Gext}) \backslash \Lambda_{\Gext} \arrow[r, "\butterfly_{\Gext}"']
 & \Wext_{\Gext} \backslash \Lambda_{\Gext}
    \end{tikzcd}
\end{center}
Here, 
\[
\Lambda_{\Fext} = \cap_k \Dom_k(\butterfly_{\Fext})
\quad \text{ and } \quad
\Lambda_{\Gext} = \cap_k \Dom_k(\butterfly_{\Gext})
\]
are the non-escaping sets of $\butterfly_\Fext$ and $\butterfly_\Gext$ respectively.

If $n, n+1, n+2,\ldots$ are all non-parabolic depths, then by Lemma \ref{lem:nowhere-dense}, $\Lambda_{\Fext}$ has no interior. 
In this case, by continuity, $h_\infty$ must be a conjugacy between $\butterfly_{\Fext}$ and $\butterfly_{\Gext}$ on their whole domains and we are done. 
Otherwise, let us consider the smallest parabolic depth $s_1$ that is greater or equal to $n$.

\subsubsection{Resolution of the first parabolic enrichment}

Consider the parabolic fixed point $\beta_{\Fext,s_1}$ of $\Fext^{Q_{s_1}}$ contained in the principal fjord of depth $s_1$,
and let $\Dext_{\Fext,s_1}$ be the corresponding immediate parabolic basin.
Let us label the connected components of the interior of $\Lambda_{\Fext}$ by 
\[
    E^0_{\Fext} = \Dext_{\Fext,s_1}, \quad E^1_{\Fext}, \quad E^2_{\Fext}, \quad E^3_{\Fext}, \quad \ldots.
\]
According to Lemma \ref{lem:nowhere-dense}, for every $j \geq 1$, there exists a unique map $\rho_{\Fext,j}$ in the enriched semigroup $\Butterfly_{\Fext}^{s_1}$ with the smallest generation such that $\rho_{\Fext,j}$ sends $E^j_{\Fext}$ onto $E^0_{\Fext}$. 
In the same way, we can also enumerate the components of the interior of $\Lambda_{\Gext}$ by $E^j_{\Gext}$, $j \geq 0$ and the corresponding maps $\rho_{\Gext, j}$ such that for all $j \geq 0$, 
\[
    \rho_{\Gext,j} \circ h_\infty = h_\infty \circ \rho_{\Fext,j} \qquad \text{ on } E^j_{\Fext}.
\]

Recall from \S\ref{ss:enriched-butterfly} that the enriched butterfly $\butterfly_{\Fext}^{s_1}$ is defined by adding into $\butterfly_{\Fext}$ the Lavaurs maps 
\[
\Fext^{Q_{s_1+1} - Q_{s_1-1} - kQ_{s_1}}: \Yext_{\Fext,s_1}^{\iinfty-k} \to \Wext_{\Fext}, \qquad k \in \Z.
\]
We have
\[
    \Yext_{\Fext,s_1}^{\iinfty} \;\subset\;
    \Yext_{\Fext,s_1}^{\iinfty-1} \;\subset\;
    \Yext_{\Fext,s_1}^{\iinfty-2} \;\subset\;
    \ldots 
    \quad \text{ and } \quad
    \Dext_{\Fext,s_1} = \bigcup_{k = 0}^\infty \Yext_{\Fext,s_1}^{\iinfty-k}.
\]
Let us define a sequence of $K$-QC maps $h_{k,1}: \UHP \to \UHP$, $k \geq 0$ as follows. 
We set $h_{0,1}$ to be equal to $h_\infty$. For $k \geq 1$, $h_{k,1}$ is constructed out of $h_{k-1,0}$ as follows.
\begin{itemize}
    \item On $E^0_{\Fext}$, $h_{k,1}$ sends $E^0_{\Fext}$ onto $E^0_{\Gext}$, and it is the lift of $h_{k-1,1}$ under the corresponding Lavaurs map. More precisely, for any $z \in \Dext_{\Fext,s_1}$, let $k=k(z) \geq 0$ be the smallest natural number such that the domain $\Yext_{\Fext,s_1}^{\iinfty-k}$ contains $z$. Then, $h_{k,1}(z)$ is the unique point in $\Yext_{\Gext,s_1}^{\iinfty-k}$ such that 
    \[
    \Gext^{Q_{s_1+1}-Q_{s_1-1}-kQ_{s_1}} \circ h_{k,1}(z) = h_{k-1,1} \circ \Fext^{Q_{s_1+1}-Q_{s_1-1}-kQ_{s_1}}(z).
    \]
    \item On $E^j_{\Fext}$ for $j \geq 1$, $h_{k,1}$ sends $E^j_{\Fext}$ onto $E^j_{\Gext}$, and it is the unique map such that 
    \[
    \rho_{\Gext, j} \circ h_{k,1} = h_{k,1} \circ \rho_{\Fext, j}.
    \]
    \item Elsewhere on $\UHP \backslash \textnormal{int}\Lambda_{\Fext}$, we set $h_{k,1} = h_{k-1,1}$. 
\end{itemize}
By equivariance, $h_{k,1}$ glues to a continuous and therefore $K$-QC self-map of $\UHP$ that is suitable.
As $k \to \infty$, $h_{k,1}$ converges in subsequence to a suitable $K$-QC map $h_{\infty, 1}: \UHP \to \UHP$ that conjugates $\butterfly_{\Fext}^{s_1}$ and $\butterfly_{\Gext}^{s_1}$ away from their non-escaping sets $\Lambda_{\Fext}^{s_1}$ and $\Lambda_{\Gext}^{s_1}$.

\subsubsection{Resolution of all parabolic enrichments}

In general, consider the sequence of all parabolic depths 
\[
    s_1, \quad s_2, \quad s_3, \quad  \ldots
\]
that are greater or equal to $n$, labeled in increasing order. 
By adapting the argument above, we construct a sequence of $K$-QC suitable maps $h_{\infty,j}: \UHP \to \UHP$ such that for $j \geq 1$,
\begin{itemize}
    \item $h_{\infty,j+1} = h_{\infty,j}$ outside of the interior of the non-escaping set $\Lambda_{\Fext}^{s_j}$ of $\butterfly_{\Fext}^{s_j}$, and
    \item $h_{\infty,j}$ conjugates the depth $s_j$ enriched butterflies $\butterfly_{\Fext}^{s_j}$ and $\butterfly_{\Gext}^{s_j}$ away from $\Lambda_{\Fext}^{s_j}$ and $\Lambda_{\Gext}^{s_j}$ respectively.
\end{itemize}
As $j$ increases, $h_{\infty,j}$ converges in subsequence to a suitable quasiconformal map $h$ that conjugates $\butterfly_{\Fext, \infty}$ and $\butterfly_{\Gext, \infty}$ away from their non-escaping sets $\Lambda_{\Fext}^{\infty}$ and $\Lambda_{\Gext}^{\infty}$. 
By Lemma \ref{lem:nowhere-dense-02}, both $\Lambda_{\Fext}^{\infty}$ and $\Lambda_{\Gext}^{\infty}$ are nowhere dense. 
Hence, by continuity, $h$ is a conjugacy between $\Butterfly_{\Fext}^\infty$ and $\Butterfly_{\Gext}^{\infty}$ on their whole domains.
This concludes the proof of Lemma \ref{lem:main-lemma}.

\subsection{Back to Theorem \ref{main-thm:rigidity}}
\label{ss:two-main-theorems}

\begin{definition}
    For $\varepsilon>0$, we say that two bi-infinite towers $\fboldbar = \seq{f_n}_{n\in\Z}$ and $\gboldbar = \seq{g_n}_{n\in\Z}$ in $\attr$ are $\varepsilon$-\emph{conformally equivalent} if there exists a sequence of conformal maps $\phi_n: X_n \to Y_n$ conjugating $f_n$ and $g_n$, where $X_n$ (resp. $Y_n$) is an open disk containing the $\varepsilon$-neighborhood of the $n$\textsuperscript{th} top pseudo-Siegel disk $\hat{Z}_{n}(\fboldbar)$ (resp. $\hat{Z}_n(\gboldbar)$).
\end{definition}

The following theorem is a quantitative version of Theorem \ref{main-thm:rigidity}.

\begin{theorem}
\label{main-thm:rigidity-precise}
    There exists a universal constant $\boldsymbol{\epsilon} > 0$ such that 
    any two combinatorially equivalent bi-infinite towers $\fboldbar$ and $\gboldbar$ in $\attr$ are $\boldsymbol{\epsilon}$-conformally equivalent.
\end{theorem}

The proof is essentially to transfer the affine rigidity of neutral cascades back to the level of towers of neutral maps.

\begin{proof}
    Let $\fboldbar = \seq{f_n}_{n\in\Z}$ and $\gboldbar = \seq{g_n}_{n\in\Z}$ be two combinatorially equivalent bi-infinite towers in $\attr$.
    In this subsection, we will show that $\fboldbar$ and $\gboldbar$ are $\varepsilon$-conformally equivalent where $\varepsilon$ is independent of $\fboldbar$ and $\gboldbar$, thereby proving Theorem \ref{main-thm:rigidity}.

    For $n \in \Z$, let $H_{n,\fboldbar}$ be the $n$\textsuperscript{th} Mother Hedgehog of $\fboldbar$ and let $S_{n,\fboldbar}^m$, $m \geq 1$ be the associated nest of sectors in the dynamical plane of $f_n$.
    For $n \in \Z$, there exists an associated gluing map $\psi_{n,\fboldbar} : S_{n,\fboldbar}^1 \to \D$ that identifies the sides of $S_{n,\fboldbar}^1$ via $z \sim f_n(z)$ and essentially sends $H_{n,\fboldbar}$ to $H_{n+1,\fboldbar}$.
    We will denote the slit of $\psi_{n,\fboldbar}$ by $\gamma_{n+1, \fboldbar}$.
    Recall from Theorem \ref{thm:qc-thurston-equivalence} that $\fboldbar$ is QC Thurston equivalent to $\gboldbar$.
    Via this equivalence, $\gboldbar$ admits analogous objects $H_{n,\gboldbar}$, $S_{n,\gboldbar}$, $\psi_{n,\gboldbar}$, $\gamma_{n,\gboldbar}$, etc.

    According to Theorem \ref{main-thm:cascades}, both $\fboldbar$ and $\gboldbar$ admit the same normalized neutral cascade $\Fbold = (\Fbold^P)_{P \in \Tbold}$.
    By Proposition \ref{prop:trans-sectors}, there exist nests of infinite sectors $\Sbold_{\fboldbar}^m$, $m \in \Z$ and $\Sbold_{\gboldbar}^m$, $m \in \Z$ in the dynamical plane of $\Fbold$ together with gluing maps $\pphi_{n,\fboldbar} : \Sbold_{\fboldbar}^m \to \D \backslash \gamma_{m, \fboldbar}$ and $\pphi_{n,\gboldbar} : \Sbold_{\gboldbar}^m \to \D \backslash \gamma_{m, \gboldbar}$ respectively.

    Let $\Hbold$ be the Mother Hedgehog and $\hat{\Zbold}$ be the top pseudo-Siegel half-plane of $\Fbold$.
    Denote the sides of $\Sbold^m_{\fboldbar}$ by $\Upsilon^m_{\fboldbar,1}$ and $\Upsilon^m_{\fboldbar,2}$.
    Recall that for $j \in \{1,2\}$, $\Upsilon^m_{\fboldbar,j}$ is the concatenation of an internal ray $I^m_{\fboldbar,j}$ of $\Hbold$ and a subarc $\delta^m_{\fboldbar,j}$ of an internal ray of a hoglet of $\Fbold$.
    We also know that $\delta^m_{\fboldbar,j}$ is a quasiarc meeting $\Hbold$ with definite angles (cf. Pseudo-Bubble Bounds) and satisfies $\diam(\delta^m_{\fboldbar,j}) \asymp |C_{Q_{[m]}}|$ (cf. Theorem \ref{thm:sectorial-bounds} (1)).
    Similarly, the sides of the sector $\Sbold_{\gboldbar}^m$ satisfy analogous properties.
    Since we select the sectors to respect the QC Thurston equivalence, this would mean that the sides of $\Sbold_{\gboldbar}^m$ coincide with that of $\Sbold_{\fboldbar}^m$ within the Mother Hedgehog.
    The difference in part of the sides that are outside of $\Hbold$ are contained in the bubbles.
    
    There exists a universal constant $\boldsymbol{\epsilon} > 0$ such that the domain of $f_n$ contains the $(2\boldsymbol{\epsilon})$-neighborhood of the top pseudo-Siegel disk $\hat{Z}_{n,\fboldbar}$ of $\fboldbar$.
    Pick $\varepsilon \in (0,\boldsymbol{\epsilon}]$ and an open quasidisk $X_n$ in the domain of $f_n$ that contains the $\varepsilon$-neighborhood of $\hat{Z}_{n,\fboldbar}$ such that $\partial X_n$ intersects the slit $\gamma_{n,\fboldbar}$ transversally at one point.
    Then,
    \[
        \mathbf{X}_n := \pphi_{n,\fboldbar}^{-1}(X_n \backslash \gamma_n)
    \]
    is a domain in $\Sbold_{\fboldbar}^n$.
    Since the sectors for $\gboldbar$ are constructed respecting QC Thurston equivalence, we have that 
    \[
        \mathbf{X}_n \cap \Hbold = \Sbold_{\fboldbar}^n \cap \Hbold = \Sbold_{\gboldbar}^n \cap \Hbold.
    \]
    Beyond $\Hbold$, points in $\mathbf{X}_n \backslash \Hbold$ may not be contained in $\Sbold_{\gboldbar}^n$.
    Nonetheless, the observation in the last paragraph implies that, assuming that $\varepsilon$ is sufficiently small, the set $\mathbf{X}_n \backslash \Sbold_{\gboldbar}^n$ is contained in two primary pseudo-bubbles that are dynamically related by $\Fbold^{[n-1]}$.
    In particular, $\mathbf{X}_n$ can be modified to a domain $\mathbf{Y}_n$ that is contained in $\Sbold_{\gboldbar}^n$ by replacing every point $z$ in $\mathbf{X}_n \backslash \Sbold_{\gboldbar}^n$ with the image of $z$ under either $\Fbold^{Q_{[n-1]}}$ or an inverse branch of $\Fbold^{Q_{[n-1]}}$.
    
    Lastly, we push $\mathbf{Y}_n$ onto the dynamical plane of $g_n$ under $\pphi_{n,\gboldbar}$ and obtain a disk neighborhood $V_n$ of $H_{n,\gboldbar}$.
    In this way, we also obtain a conformal isomorphism $\phi_n : U_n \to V_n$ such that $\phi_n(z) = \pphi_{n,\gboldbar} \circ \pphi_{n,\fboldbar}^{-1}(z)$ for points $z$ such that $\pphi_{n,\fboldbar}(z)$ is in $\mathbf{Y}_n$.
    Then, by design, $\phi_n$ is the desired conformal conjugacy between $f_n$ and $g_n$.
\end{proof}

\section{Rigidity of backward parabolic towers}
\label{sec:backward-rigidity}

Let us fix two bi-infinite sequences $\tttheta = \seq{ (\varepsilon_n, \abar_n) } _{n \in \Z}$ and $\tttheta' = \seq{ (\varepsilon'_n, \abar'_n) } _{n \in \Z}$ in $\TheBiCpt$ such that 
\begin{align}
\label{eqn:backward-equivalence}
    \abar_1 = \abar'_1 = \infty
    \qquad \text{ and } \qquad 
    (\varepsilon'_n, \abar'_n) = (\varepsilon_n, \abar_n) \quad \text{ for all } n \leq 1.
    \tag{$\spadesuit$}
\end{align}
Consider the sub-semigroup $\Tbold_{\Arch,0}$ of $\Tbold_{\tttheta}$ (from Section \ref{sec:hairiness-NILF}) that is generated by the elements $Q_0,Q_{-1}, Q_{-2},\ldots$ and all parabolic enrichments up to depth $-1$.
The corresponding sub-semigroup for $\tttheta'$ is isomorphic to $\Tbold_{\Arch,0}$, so we will simply use the same notation for both of them.
Our goal in this section is to prove the following theorem.

\begin{theorem}[Backward rigidity of parabolic bi-infinite towers]
\label{thm:parabolic-backward-rigidity}
    Let $\Fbold$ and $\Gbold$ be the (unique normalized) neutral cascades with combinatorics $\tttheta$ and $\tttheta'$ respectively satisfying \textnormal{(\ref{eqn:backward-equivalence})}.
    Then, there exists an affine map $\Phi: \C \to \C$ such that
    \[
        \Phi \circ \Fbold^P = \Gbold^P \circ \Phi \qquad \text{ for all } P \in \Tbold_{\Arch,0}.
    \]
\end{theorem}

The proof will mimic the proof of Theorem \ref{thm:qc-rigidity} in the previous section.
Repeated details will be spared and new additions will be highlighted.
Following the discussion in \S\ref{ss:two-main-theorems}, we also know that Theorem \ref{thm:parabolic-backward-rigidity} implies Theorem \ref{main-thm:backward-rigidity}.

\begin{remark}
    It may be possible that the conclusion of Theorem \ref{thm:parabolic-backward-rigidity} also holds under the condition that 
    \begin{itemize}
        \item $(\varepsilon'_n, \abar'_n) = (\varepsilon_n, \abar_n)$ for all $n \leq -1$,
        \item $(\varepsilon_1,\abar_1) = (+ \infty)$ and $(\varepsilon'_1,\abar_1) = (-, \infty)$, and
        \item either $(\varepsilon_0, \abar_0) = (-,2)$ and $(\varepsilon'_0, \abar'_0) = (+,2)$, or $\varepsilon'_0 = \varepsilon_0 = +$ and $\abar'_0 = \abar_0 + \varepsilon_0$.
    \end{itemize}
    If this is true, it means that $\Fbold$ and $\Gbold$ have bi-infinite tower representatives $\seq{f_n}_{n\in \Z}$ and $\seq{g_n}_{n\in \Z}$ respectively in $\attr$ such that $f_0$ and $g_0$ are equal to a real-symmetric parabolic map.
    The methods in this paper are not enough to detect such symmetry.
\end{remark}

\subsection{The Mother Flower $\check{\Hbold}$}

From now on, we will fix a neutral cascade $\Fbold$ with combinatorics $\tttheta$ associated to a bi-infinite tower $\fboldbar = \seq{f_n}_{n\in\Z}$ in $\overline{\towsec}$ corresponding to $\Fbold$.
We will use the standard notation summarized in \S\ref{sss:cascade}.

The primary modification is the introduction of the Mother Flower, which plays the role of the Mother Hedgehog in this backward/parabolic setting and is built only from the sub-semigroup $\Tbold_{\Arch,0}$.
Let us go over the details now.

\begin{definition}
    Denote $\check{\Zbold}^{[0]}[0] := \Zbold^{[0]} \cap \Sbold^1$, and for $P \in \Tbold_{\Arch,0}$, denote 
    \begin{itemize}
        \item $\check{\Zbold}^{[0]}[P] := \Fbold^P \big(\check{\Zbold}^{[0]}[0] \big)$,
        \item $\check{\Zbold}^{[0]}[-P] :=$ the closure of the lift under $\Fbold^{-P}$ of the interior of $\check{\Zbold}_0^{[0]}[0]$ that intersects with the Mother Hedgehog of $\Fbold$.
    \end{itemize}
    Then, we define the \emph{Mother Flower} of $\Fbold$ to be 
    \[
        \check{\Hbold} \equiv \check{\Hbold}(\Fbold) := \bigcup_{P \in \Tbold_{\Arch,0} \cup \{0\} \cup -\Tbold_{\Arch,0}} \check{\Zbold}^{[0]}[P].
    \]
\end{definition}

By design, we have:

\begin{lemma}
\label{lem:modified-mother-hedgehog}
    The map $\Fbold^{[0]}$ restricts to a self-homeomorphism of $\check{\Hbold}$.
\end{lemma}

Recall that in external coordinates, for every near-parabolic level $n$, $\Fjord^{[n]}$ denotes the union of all level $n$ fjords.

\begin{definition}
    For $m \leq 0$, we define the \emph{level $m$ pseudo-Siegel flower} $\check{\Zbold}^{[m]}$ of $\Fbold$ to be
    \[
        \check{\Zbold}^{[m]} \equiv \check{\Zbold}^{[m]}(\Fbold) = 
        \begin{cases}
            \displaystyle \check{\Hbold} & \textnormal{ if } m = 0,\\
            \displaystyle \check{\Hbold} \cup \bigcup_{n = m+1}^{0} \Psi^{-1}(\Fjord^{[n]}) & \textnormal{ if } m \leq -1.
        \end{cases}
    \]
    The \emph{top pseudo-Siegel flower} is
    \[
        \check{\Zbold} = \bigcup_{m \leq -1} \check{\Zbold}^{[m]}.
    \]
\end{definition}

By almost invariance (Lemma \ref{lem:fjord-almost-invariance}), 
the modified $\check{\Zbold}^{[m]}$ is almost identical to the original $\hat{\Zbold}^{[m]}$ for all $m \leq 0$.
The boundary of $\check{\Zbold}$ remains a uniform quasi-line.

For all $n \leq 0$ and $m \geq -1$ with $n+m \leq 0$, we will also replace the level $m$ pseudo-Siegel pinched disk $\hat{Z}_n^{[m]}$ of $f_n$ with the modified version $\check{Z}_n^{[m]} = \check{Z}_n^{[m]}(\fboldbar)$ that projects onto $\check{\Zbold}^{[m+n]}$ under the gluing map $\pphi_n$ described in Proposition \ref{prop:trans-sectors}.
We still have that $\check{Z}_n^{[0]}$ and $\check{Z}_n^{[-1]}$ are uniform quasidisks.
The modification gives us the new property that for $m \geq 0$, $\check{Z}_{-m}^{[m]}$ is completely invariant under $f_{-m}^{[m]}$.

In the dynamical plane of $\Fbold$, we enlarge every hoglet of generation $P \in \Tbold_{\Arch,0}$ to what shall be called the \emph{floret} of the same generation $P$, which is the appropriate lift of $\check{\Hbold}$ under $\Fbold^P$. 
The lift of the top pseudo-Siegel flower $\check{\Zbold}$ will again be referred to as \emph{pseudo-bubbles}.
They admit Pseudo-Bubble Bounds in the sense of Proposition \ref{prop:bubble-bounds}.

\subsection{External coordinate analysis}

We define the new external coordinates in terms of the exterior of $\check{\Hbold}$.
More precisely, consider the Riemann mapping
\[
    \check{\Psi} = \check{\Psi}_{\Fbold} : \C \backslash \check{\Hbold} \to \UHP
\]
normalized to fix $0$, $\varepsilon_0$, and $\infty$.
Then, the new external cascade is defined by
\[
    \check{\Fext}^P := \check{\Psi} \circ \Fbold^P \circ \Psi^{-1} : \check{\Psi}(\Fbold^{-P}(\C \backslash \check{\Hbold}) \to \UHP
\]
for $P \in \Tbold_{\Arch,0}$.
Again, connected components of the domain of $\check{\Fext}^P$ are called lakes of generation $P$.
With essentially the same proof, we have the following properties:
\begin{enumerate}
    \item Real Bounds for the critical orbit of $\check{\Fext}$ along the real line, mimicking Proposition \ref{prop:R-apb-transcendental};
    \item Pseudo-Bubble Bounds, mimicking \ref{prop:bubble-bounds};
    \item Floret-Butterfly Bounds, mimicking what was constructed in Section \ref{sec:butterfly}.
\end{enumerate}
Let us elaborate on the last point.
For $n < 0$, we define the depth $n$ butterfly
\[
    \big( \check{\Fext}^{Q_{n+1}}:\check{\Uext}'_{n, \check{\Fext}} \to \check{\Wext}_{n, \check{\Fext}}, \quad  \Fext^{Q_{n}}:\check{\Uext}''_{n, \check{\Fext}} \to \check{\Wext}_{n, \check{\Fext}} \big)
\]
where $\check{\Wext}_{n,\check{\Fext}}$ is constructed using a circular arc analogous to $\Wext_{n,\Fext}$, and $\check{\Uext}'_{n,\check{\Fext}}$ and $\check{\Uext}''_{n,\check{\Fext}}$ are the corresponding preimages.
It also comes with the collection of enriched butterfly maps $\check{\butterfly}_{n,\check{\Fext}}^m$, $n \leq m \leq -1$ associated to $\check{\Fext}$.
These objects admit all the properties described in Section \ref{sec:butterfly}.

\subsection{Constructing quasiconformal conjugacy}

Let $\Gbold$ be another neutral cascade with combinatorics $\tttheta'$ having the same backward tail as $\tttheta$ in the sense of (\ref{eqn:backward-equivalence}). 
Let $\gboldbar = \seq{g_n}_{n\in\Z}$ be a bi-infinite tower in $\overline{\towsec}$ corresponding to $\Gbold$.

This theorem is a variant of Theorem \ref{thm:qc-thurston-equivalence}.

\begin{theorem}
    There exist a universal constant $K>1$ and a sequence of $K$-quasiconformal maps $\{h_{-n} : \D \to \D\}_{n \geq 0}$ such that for any $n \geq 0$,
    \begin{enumerate}
        \item $h_{-n}$ sends the critical value of $f_{-n}$ to the critical value of $g_{-n}$;
        \item $h_{-n}$ sends $\check{Z}^{[m]}_{-n}(\fboldbar)$ onto $\check{Z}^{[m]}_{-n}(\gboldbar)$ for every $m \in \{-1,0,\ldots, n\}$;
        \item $h_{-n}$ is conformal on the interior of $\check{Z}^{[n]}_{-n}(\fboldbar)$;
        \item $h_{-n}$ is a conjugacy between the semigroups $\mathcal{F}_n|_{\check{Z}^{[n]}_{-n}(\fboldbar)}$ and $\mathcal{G}_n|_{\check{Z}^{[n]}_{-n}(\gboldbar)}$;
        \item for $m \geq 1$, $\psi_{n,\gboldbar} \circ h_n = h_{n+1} \circ \psi_{n,\fboldbar}$ on the sector $S_{n,1}(\fboldbar)$.
    \end{enumerate}
\end{theorem}

\begin{proof}
    In the interior of $\check{Z}_0^{[0]}(\fboldbar)$, the dynamics of $f_0$ is a simple parabolic one, conjugate to a translation of the upper half plane. 
    Thus, there exists a unique conjugacy $h_0: \check{Z}_0^{[0]}(\fboldbar) \to \check{Z}_0^{[0]}(\gboldbar)$ between $f_0$ and $g_0$ such that it is conformal on the interior and it sends the critical value of $f_0$ to the critical value of $g_0$.
    Using the uniform quasiconformality of both $\check{Z}_0^{[0]}(\fboldbar)$ and $\check{Z}_0^{[-1]}(\fboldbar)$, we extend $h_0$ to a uniformly quasiconformal map from $\D$ to $\D$ that sends $\partial \check{Z}_0^{[-1]}(\fboldbar)$ to $\partial \check{Z}_0^{[-1]}(\gboldbar)$.
    To construct $h_{-1}$, $h_{-2}$, $\ldots$, we lift and propagate in a way analogous to the proof of Theorem \ref{thm:qc-thurston-equivalence}.
\end{proof}

Similar to Lemma \ref{lem:qc-thurston-equivalence-02}, this theorem implies:

\begin{lemma}
    There exists a universal constant $K>1$ and a global $K$-quasiconformal map $\check{\Phi}_1: (\C,0) \to (\C,0)$ such that
    \begin{enumerate}
        \item for $m \leq 0$, $\check{\Phi}_1$ sends $\check{\Zbold}^m(\Fbold)$ onto $\check{\Zbold}^m(\Gbold)$;
        \item $\check{\Phi}_1$ is a conformal conjugacy between $\Fbold: \textnormal{int}(\check{\Hbold}(\Fbold)) \to \textnormal{int}(\check{\Hbold}(\Fbold))$ and $\Gbold: \textnormal{int}(\check{\Hbold}(\Gbold)) \to \textnormal{int}(\check{\Hbold}(\Gbold))$.
    \end{enumerate}
\end{lemma}

The map $\check{\Phi}_1$ above induces the $K$-quasiconformal map
\[
    \check{\Phi}_2 := \check{\Psi}_{\Gbold} \circ 
    \check{\Phi}_1|_{\C \backslash \check{\Hbold}(\Fbold)} \circ 
    \check{\Psi}_{\Fbold}^{-1} : \UHP \to \UHP
\]
in external coordinates.
Similar to Lemma \ref{lem:qs-rigidity}, we have:

\begin{lemma}
    The map $\check{\Phi}_2$ extends to a $K$-quasisymmetric homeomorphism $\check{\Phi}_2: (\R,0) \to (\R,0)$ conjugating $\check{\Fext}^P|_{\R}$ and $\check{\Gext}^P|_{\R}$ for all $P \in \Tbold_{\Arch,0}$.
\end{lemma}

For $n <0$, consider the butterfly domains 
$\check{\Uext}'_{n,\check{\Fext}}$, $\check{\Uext}''_{n,\check{\Fext}}$, $\check{\Wext}_{n,\check{\Fext}}$ associated to $\check{\Fext}$
as well as the collection of enriched butterfly maps $\check{\butterfly}_{n,\check{\Fext}}^m$, $n \leq m \leq -1$ associated to $\check{\Fext}$.
Similarly, we have 
$\check{\Uext}'_{n,\check{\Gext}}$, $\check{\Uext}''_{n,\check{\Gext}}$, $\check{\Wext}_{n,\check{\Gext}}$, and $\check{\butterfly}_{n,\check{\Gext}}^m$ associated to $\check{\Gext}$.

\begin{lemma}
    There exists an absolute constant $K \geq 1$ such that for every $n \in \Z$, there exists a $K$-quasiconformal map $\check{\Phi}_{2,n}: \UHP \to \UHP$ where
    \begin{enumerate}
        \item the continuous extension of $\check{\Phi}_{2,n}$ to the real line is equal to $\check{\Phi}_2: \R \to \R$;
        \item $\check{\Phi}_{2,n}$ sends the domains $\check{\Uext}'_{n,\check{\Fext}}$, $\check{\Uext}''_{n,\check{\Fext}}$, $\check{\Wext}_{n,\check{\Fext}}$ onto the domains $\check{\Uext}'_{n,\check{\Gext}}$, $\check{\Uext}''_{n,\check{\Gext}}$, $\check{\Wext}_{n,\check{\Gext}}$ respectively;
        \item $\check{\Phi}_{2,n}$ conjugates $\check{\Butterfly}_{n,\check{\Fext}}^{-1}$ with $\check{\Butterfly}_{n,\check{\Gext}}^{-1}$.
    \end{enumerate}
\end{lemma}

Similar to Lemma \ref{lem:main-lemma}, this lemma implies the existence of a quasiconformal conjugacy $\check{\Phi}: \C \to \C$ between $(\Fbold^P)_{P \in \Tbold_{\Arch,0}}$ and $(\Gbold^P)_{P \in \Tbold_{\Arch,0}}$ that coincides with $\check{\Phi}_1$ on the Mother Flower $\check{\Hbold}(\Fbold)$.
Observe that $\check{\Phi}$ is conformal everywhere on the grand orbit of the interior of $\check{\Hbold}(\Fbold)$ under $(\Fbold^P)_{P \in \Tbold_{\Arch,0}}$.
Theorems \ref{thm:parabolic-fatou-set} and \ref{thm:NILF} imply that $\check{\Phi}$ is conformal almost everywhere, hence affine.
This concludes the proof of Theorem \ref{thm:parabolic-backward-rigidity}.

\subsection{Generalization}

Observe that the proof of Theorem \ref{thm:parabolic-backward-rigidity} does not involve any of the forward renormalization data at all.
In this subsection, we will formulate a general version of Theorem \ref{main-thm:backward-rigidity}.

Consider a backward tower of neutral renormalizations $\fbold = \seq{ f_n }_{n\leq 0}$ with backward combinatorics $\ttheta = \seq{ (\varepsilon_n, \abar_n) }_{n\leq 1}$.
Suppose that $f_0$ has a parabolic fixed point at $0$.
Then $\ttheta$ induces a natural time semigroup $\Tbold_{\Arch,0}$.
We will impose the following hypothesis.
\begin{enumerate}
    \item $f_0$ has an invariant Mother Flower $\check{H}$ where $\partial \check{H}$ contains the critical orbit and $\text{int} (\check{H})$ is contained in the basin of attraction of $0$.
    \item Each $f_n$ has pseudo-Siegel flowers $\check{Z}_n^{[m]}$ for $-1 \leq m \leq -n + N$ where $\check{H} = \check{Z}_0^{[N]}$ for some $N \geq 0$. 
    These objects are well-contained in the domain of $f_n$ and come with Pseudo-Siegel Bounds.
    \item $\fbold$ induces a $\sigma$-proper cascade $\Fbold = (\Fbold^P)_{P \in \Tbold_{\Arch,0}}$ with Mother Flower $\check{\Zbold}^{[N]} = \check{\Hbold}$ and pseudo-Siegel flowers $\check{\Zbold}^{[m]}$, $m \leq N$, where $\check{\Zbold} = \cup_{m\leq N} \check{\Zbold}^{[m]}$ is uniformly QC.
    \item $\Fbold$ satisfies Pseudo-Bubble Bounds for the preimages of $\check{\Zbold}$.
\end{enumerate}

\begin{theorem}[Rigidity of parabolic backward-infinite towers]
\label{thm:backward-rigidity-general}
    Let $\fbold = \seq{ f_n }_{n\leq 0}$ and $\gbold = \seq{ g_n }_{n\leq 0}$ be two backward towers. 
    Suppose that $f_0$ and $g_0$ have the same rational rotation number at $0$, that $\fbold$ and $\gbold$ have equivalent backward combinatorics, that they satisfy hypotheses $(1)$--$(4)$.
    Then, for all $n \leq 0$, $f_n$ and $g_n$ are conformally conjugate on a definite neighborhood of their Mother Flowers.
\end{theorem}

We expect to apply this theorem in the upcoming hyperbolicity program.

\section{The Renormalization Horseshoe and further applications}
\label{sec:horseshoe}

In this final section, we prove Corollary \ref{main-corollary} and the uniform continuity of Mother Hedgehogs on the combinatorics, discuss how cylinder renormalization can be derived, and elaborate on a strategy to proving uniform hyperbolicity of renormalization.

\subsection{Uniform continuity}
\label{ss:uniform-continuity}

Let $\boldsymbol{\epsilon}>0$ be from Theorem \ref{main-thm:rigidity-precise}.
The notion of $\boldsymbol{\epsilon}$-conformal equivalence induces a closed equivalence relation on $\attr$ which we will denote by $\underset{\textnormal{conf}}{\sim}$.
The quotient
\[
    \Hors := \attr/_{\underset{\textnormal{conf}}{\sim}}
\]
is a well-defined compact Hausdorff space, which we will refer to as the 
\emph{full horseshoe} of \emph{neutral renormalization}.

Consider the map
    \[
        \textnormal{comb} : \attr \to \TheBiCpt, \qquad \fboldbar \mapsto \tttheta(\fboldbar),
    \]
which is a surjective continuous function by virtue of Proposition \ref{prop:continuity-of-combinatorics}.
The ``first isomorphism theorem'' for topological spaces, together with Theorem \ref{main-thm:rigidity-precise}, gives us:

\begin{corollary}
    The induced combinatorics map
    \[
    \Hors \to \TheBiCpt, \qquad [\fboldbar] \mapsto \tttheta(\fboldbar)
    \]
    is a homeomorphism.
\end{corollary}

This settles Corollary \ref{main-corollary}.
For any $\tttheta \in \TheBiCpt$, let $\Fbold_{\tttheta}$ denote the unique normalized neutral cascade with combinatorics $\tttheta$.

\begin{proposition}
\label{prop:uniform-continuity}
    The Mother Hedgehog $\Hbold(\Fbold_{\tttheta})$ varies continuously in $\tttheta \in \TheBiCpt$ in the Carath\'eodory topology.
\end{proposition}

The proof is a combination of pre-compactness and rigidity.

\begin{proof}
    Let $\{\tttheta_m\}_{m\geq 1}$ be a sequence in $\TheBiCpt$ converging to $\tttheta_\infty \in \TheBiCpt$.
    For $m \geq 1$, let $\fboldbar_m = \seq{ f_{n,m} }_{n\geq 1}$ be a bi-infinite tower in $\attr$ with combinatorics $\tttheta_m$.
    Up to subsequence, as $m \to \infty$, $\fboldbar_m$ converges to a limit $\fboldbar_\infty \in \attr$ and it must have combinatorics $\tttheta$.
    For $m \in \N$, consider the conformal gluing maps $\pphi_{n,m}: (\Sbold^n_m, 0) \to (\D, v_0(f_{n,m}))$ from the $n$\textsuperscript{th} sector of the neutral cascade
    $\Fbold_{\tttheta_m}$ to the dynamical plane of $f_{n,m}$.
    With respect to the Carath\'eodory topology of conformal maps between pointed disks, for all $n \in \Z$, the collection of maps $\{ \pphi_{n,m} \}_{m \in \N \cup \{\infty\}}$ is pre-compact by virtue of Theorem \ref{thm:renorm-limits} and Proposition \ref{prop:trans-sectors}.
    Hence, up to further subsequence, we have that for all $n \in \Z$, as $m \to \infty$,
    \begin{itemize}
        \item the conformal maps $\pphi_{n,m}$ converge to a conformal gluing map $\pphi_{n,\infty}: (\Sbold^n_\infty, 0) \to (\D, v_0(f_{n,m})$, and
        \item the $n$\textsuperscript{th} Mother Hedgehog $H_{n,m}$ of $\fboldbar_m$ converges to the Mother Hedgehog $H_{n,\infty}$ of $\fboldbar_\infty$.
    \end{itemize}
    By Theorem \ref{prop:trans-mother}, these imply that $\Hbold(\Fbold_{\tttheta_m})$ converges to $\Hbold(\Fbold_{\tttheta})$ as $m \to \infty$.
\end{proof}

\subsection{On the fixed point $\mathbf{b}_n$ and cylinder renormalization}
\label{ss:cylinder}

In an upcoming work, we will prove that for every $\tttheta \in \TheBiCpt$, there exists a bi-infinite neutral renormalization tower $\fboldbar = \seq{f_n}_{n \geq 1}$ with combinatorics $\tttheta$ such that for every $n \in \Z$,
\begin{enumerate}[label = \textnormal{(\alph*)}]
    \item \label{cyl-a} $f_n$ admits a fixed point $\mathbf{b}_n$ of $f_n$ such that
    \begin{itemize}
        \item if $n$ is a near-parabolic level, then $\mathbf{b}_n$ is the fixed point described in Theorem \ref{thm:beta-fixed-points}, and
        \item otherwise, $\mathbf{b}_n$ is outside of the top pseudo-Siegel disk $\hat{Z}_n$ of $f_n$ and $|f_n'(\mathbf{b}_n)|> 1+ \varepsilon$ for some uniform constant $\varepsilon > 0$;
    \end{itemize}
    \item \label{cyl-b} there exists a \emph{renormalization croissant} $\mathcal{C}_n$ that contains the critical value of $f_n$ and is made of an arc $\delta_n$ with ends $0$ and $\mathbf{b}_n$ and its image $f_n(\delta_n)$,
    \item \label{cyl-c} there exists a conformal gluing map $\psi_n : \mathcal{C}_n \to \C^*$ that projects 
    \begin{itemize}
        \item the first return map of $f_n$ back to $\mathcal{C}_n$ to the map $f_{n+1}$,
        \item the Mother Hedgehog $H_n$ of $f_n$ restricted to $\mathcal{C}_n$ onto the Mother Hedgehog $H_{n+1}$ of $f_{n+1}$,
        \item the zeroth pseudo-Siegel disk $\hat{Z}_n^{[0]}$ of $f_n$ restricted to $\mathcal{C}_n$ onto the top pseudo-Siegel disk $\hat{Z}_n^{[-1]}$ of $f_{n+1}$,
    \end{itemize}
    \item \label{cyl-d} there is a ``definite separation'' between $\mathcal{C}_n$ and the infinite slit of $\psi_{n-1}$ away from the fixed point at $0$, and
    \item \label{cyl-e} $f_n$ admits a pacman structure with uniform improvement of domain; compare with \cite{DLS}.
\end{enumerate}

Similar to our sector renormalization, the arc $\delta_n$ can be chosen to start with an internal ray of $H_n$. To go beyond, we perform the construction in the level of the external cascade $\Fext = (\Fext^P)_{P \in \Tbold}$.
Let us consider the increasing sequence $\{t_n\}_{n \in \Z}$ where $Q_{t_n} = Q_{[n]}$ for all $n$.
The fixed point $\mathbf{b}_n$ of $f_n$ corresponds to the point $\beta_{t_n}$ described in the end of \S\ref{ss:parabolic-levels}.
The point $\beta_{t_n}$ is ultimately the fixed point of the conformal map $\Fext^{[n]}: \lake_{Q_{[n]}} \to \UHP$ on the principal lake $\lake_{Q_{[n]}}$ of generation $Q_{[n]}$.
The construction of the outer part of the croissant will rely on the structure of principal lakes which we have worked out in \S\ref{ss:lake}.

The operation above is a refinement of sector renormalization called \emph{cylinder renormalization} $\renorm_{\textnormal{cyl}}$.
Unlike sector, the cylinder renormalization of a map is unique up to linear conjugacy.
Hence, such a bi-infinite tower is unique once we normalize some point, e.g. the critical value to be $1$.
Denote such normalized bi-infinite towers by $\fboldbar_{\tttheta} = \seq{ f_{n,\tttheta}}_{n\in\Z}$ for $\tttheta \in \TheBiCpt$; clearly, it satisfies $f_{n+1,\tttheta} = f_{n,\shift (\tttheta)}$ for all $n$.

For every $\tttheta \in \TheBiCpt$, let $U_{\tttheta}$ be a definite open disk neighborhood of the Mother Hedgehog $H_{\tttheta}$ of $f_{0,\tttheta}$.
For $\tau >0$, let $\mathcal{B}^*(\tttheta,\tau)$ be the space of holomorphic maps $g$ on a neighborhood of $U_{\tttheta}$ that have a neutral fixed point at $0$ with the same multiplier as $f_{0,\tttheta}$ and 
\[
\sup_{z \in U_{\tttheta}} |f(z)-g(z)| < \tau.
\]
It is naturally equipped with the sup norm, forming a natural Banach neighborhood of $f_{0,\tttheta}$.
We claim the neighborhoods $U_{\tttheta}$ can be chosen such that the following properties hold.
\begin{itemize}
    \item For sufficiently small $\tau>0$, 
    \[
        \mathcal{B}^*(\tau) = \bigsqcup_{\tttheta \in \TheBiCpt} \mathcal{B}^*_{\tttheta}(\tau)
    \]
    is a trivial Banach analytic bundle over $\TheBiCpt$.
    \item There exist uniform constants $\ttau > \ttau' > 0$ such that the cylinder renormalization operator $\renorm_{\textnormal{cyl}}: f_{0,\tttheta} \mapsto f_{0,\shift(\tttheta)}$ extends to a compact analytic operator from $\mathcal{B}^*(\tttheta,\ttau')$ into $\mathcal{B}^*(\shift(\tttheta),\ttau)$.
\end{itemize}
Together, these properties imply that the cylinder renormalization operator acting on $\Hors$ extends to a compact analytic skew product
\[
    \renorm_{\textnormal{cyl}}: \mathcal{B}^*(\ttau') \to \mathcal{B}^*(\ttau).
\]
The proof of these claims will be supplied in detail in a follow-up paper.
Ultimately, it follows from the uniform continuity of Mother Hedgehogs (Proposition \ref{prop:uniform-continuity}).

\subsection{The $\boldsymbol\alpha'$-hyperbolicity}\label{ss:alpha'.hyperb}

Consider a neutral cascade $\Fbold = (\Fbold^{P})_{P \in \Tbold}$ and the corresponding external cascade $\Fext = (\Fext^P)_{P \in \Tbold}$. 
Also, let $\{Q_{[n]}\}_{n\in\Z}$ be the generators of the time semigroup $\Tbold$.
This subsection discusses how alpha-points $\boldsymbol{\alpha}_{Q_{[n]}}$ of $\Fbold$ associated to principal primary hoglets come with some uniform hyperbolicity.
This property is expected to have applications in the future hyperbolicity program when perturbations of maps are explored; see \S\ref{ss:uniform-hyperbolicity}.

\begin{proposition}
    There exist a uniform constant $\boldsymbol{\lambda}>1$ (independent of $\Fbold$) such that for all $n \in \Z$, we have a definite expansion of $\Fext^{Q_{[n]}-Q_{[n-1]}}$ at the alpha-point $\alpha_{Q_{[n]}}$ with respect to the hyperbolic metric of $\UHP$:
    \[
        \| (\Fext^{Q_{[n]}-Q_{[n-1]}})'(\alpha_{Q_{[n]}}) \| \geq \boldsymbol{\lambda}.
    \]
\end{proposition}

A variant of this proposition was already used in Claim 3 in the proof of Proposition \ref{prop:lake-bounds-0}.
We provide the proof again for the sake of completeness.

\begin{proof}
    The alpha-point $\alpha_{Q_{[n]}}$ is contained in the principal lake $\lake_{Q_{[n]}-Q_{[n-1]}}$ of generation $Q_{[n]}-Q_{[n-1]}$, which avoids the principal hoglet $\Bubb_{Q_{[n]}-Q_{[n-1]}}$.
    The hyperbolic distance between $\alpha_{Q_{[n]}}$ and $\Bubb_{Q_{[n]}-Q_{[n-1]}}$ is uniformly bounded above by virtue of Real Bounds, Pseudo-Bubble Bounds, and Corollary \ref{cor:location-of-alpha}.
    Then, the proposition follows from Schwarz Lemma applied to the map $\Fext^{Q_{[n]}-Q_{[n-1]}}: \lake^{Q_{[n]}-Q_{[n-1]}} \to \UHP$.
\end{proof}

Here is a quantitative version of the proposition in the dynamical plane $\Fbold$.

\begin{proposition}
    There exist uniform constants $\mathbf{r}>0$ and $\boldsymbol{\epsilon}>0$ such that for any $r \in (0,\mathbf{r})$, there exist two sequences of open disks $\{ D_n \}_{n \in \Z}$ and $\{ D'_n \}_{n \in \Z}$ where for every $n \in \Z$,
    \begin{enumerate}
        \item $\aalpha_{Q_{[n]}} \in D_n \Subset D'_n$;
        \item $(D'_n, \aalpha_{Q_{[n]}})$ is a pointed disk in $\C \backslash \UHP$ with uniformly bounded shape and outer radius $\asymp r |C_{Q_{[n]}}|$;
        \item $\modu(D'_n \backslash D_n) \geq \boldsymbol{\epsilon}$;
        \item $\Fbold^{Q_{[n]}-Q_{[n-1]}}$ sends $(D_n, \aalpha_{Q_{[n]}})$ conformally onto $(D'_{n-1}, \aalpha_{Q_{[n-1]}})$.
    \end{enumerate}
\end{proposition}

\begin{proof}
    For $n \in \Z$, we define $D'_n$ to be the round disk with radius $r$ and center $\aalpha_{Q_{[n]}}$ with respect to the hyperbolic metric of $\C \backslash \Hbold$, and $D_n$ to be the lift of $D'_{n+1}$ under $\Fbold^{Q_{[n+1]}-Q_{[n]}}$ about the point $\aalpha_{Q_{[n]}}$.
    Properties (1), (3), and (4) follow immediately from the previous proposition if we take $r$ to be uniformly sufficiently small.
    The only property left to check is that the size of $D'_n$ is $\asymp r |C_{Q_{[n]}}|$.
    This follows from Corollary \ref{cor:location-of-alpha} which guarantees that the distance from the alpha-point $\aalpha_{Q_{[n]}}$ to $\Hbold$ is $\asymp |C_{Q_{[n]}}|$.
\end{proof}

Let $\seq{f_n}_{n \in \Z}$ be the bi-infinite tower associated to $\Fbold$ that is described in \S\ref{ss:cylinder}.
Let $\psi_n$'s be the associated gluing maps described in \ref{cyl-c}.
For every $n \in \Z$, in the dynamical plane of $f_n$, we expect to be able to locate the expected preimage of $0$, which we will denote by $\alpha'_n$.
Each $\alpha'_n$ does not exist in the domain of $f_n$ because the point $\aalpha_{Q_{[n]}}$ is on the boundary of $\Dom(\Fbold^{[n]})$.
However, we claim that the disks $D_n$ and $D_n'$ around $\aalpha_{Q_{[n]}}$ can be transferred to disks $E_n$ and $E_n'$ living on the dynamical plane of the $f_n$ with the property that for all $n \in \Z$,
\begin{itemize}
    \item $E_n'$ well-contains $E_n$ and lives in the dynamical plane of $f_n$,
    \item $E_{n+1}$ lifts under $\psi_n$ to a disk which we will simply denote by $\psi_{n}^{-1}(E_{n+1})$, and 
    \item $f_n^{\qq^n_{[1]}-1}$ is univalent on $\psi_{n}^{-1}(E_{n+1})$ and its image is equal to $E'_n$.
\end{itemize}
As $n \to \infty$, the disks
\[
    \psi_{-1} \circ \left( f_{-1}^{\qq^{-1}_{[1]}-1} \right)^{-1} \circ \ldots \circ \psi_{-n+1} \circ \left( f_{-n+1}^{\qq^{-n+1}_{[1]}-1} \right)^{-1}\circ \psi_{-n} \circ \left( f_{-n}^{\qq^n_{[1]}-1} \right)^{-1}(E_{-n})
\]
shrink to a singleton, which is the desired point $\alpha'_0$.

\subsection{Uniform hyperbolicity of renormalization}
\label{ss:uniform-hyperbolicity}

In this final subsection, we propose a strategy to prove the uniform hyperbolicity of the full horseshoe of neutral renormalization.
The goal is to show that the cylinder renormalization skew product $\renorm_{\textnormal{cyl}}: \mathcal{B}^*(\ttau') \to \mathcal{B}^*(\ttau)$ 
is uniformly contracting on fibers.
\begin{enumerate}
    \item Suppose for a contradiction that the top Lyapunov exponent $\boldsymbol{\chi}$ is non-negative.
    Apply an adaptation of the Oseledets-Pesin Theory to Banach analytic bundle maps to show that there exist some $\tttheta \in \TheBiCpt$ and some infinite backward orbit 
    \[
    g_{n} \in \mathcal{B}^*_{\shift^{n}(\tttheta)}(\ttau'), \quad n \leq 0
    \]
    shadowing the bi-infinite orbit $\fboldbar = \seq{f_n = f_{n,\tttheta}}_{n \in \Z}$ in the horseshoe.
    \item Argue that the backward orbit $\{g_n\}_{n\leq 0}$ admits the structure of a transcendental cascade $\mathbf{G}$, a commutative semigroup of $\sigma$-proper maps onto $\C$ parametrized by the sub-semigroup $\Tbold_{\Arch,0}$ of $\Tbold_{\tttheta}$.
    In the periodic-type setting, this step is called the Key Lemma in \cite{DLS}.
    In our setting, the general Key Lemma should follow from Hoglet Butterfly Bounds together with the following observation. 
    The disks $E_n$ and $E'_n$ associated to $f_n$ that are described in the previous subsection persist under perturbation.
    In particular, in the dynamical plane of $g_n$ for $n \leq 0$, these disks will recover the alpha-points $\alpha'_n(g_n)$ of $g_n$.
    \item By a slight perturbation, assume that both $f_0$ and $g_0$ are parabolic with the same rotation number. 
    Ultimately, we would want to apply Theorem \ref{thm:backward-rigidity-general} to $\seq{g_n}_{n \leq 0}$ and $\seq{f_n}_{n \leq 0}$ to obtain a contradiction.
    To do this, we argue that $\mathbf{G}$ has a controlled Mother Flower with Pseudo-Siegel Bounds and Pseudo-Bubble Bounds.
\end{enumerate}
We aim to realize this strategy in the near future.
Uniform hyperbolicity is expected to have strong consequences such as the universal scaling law of the Mandelbrot set along the main cardioid, extending the periodic-type parameter self-similarity in \cite{DLS}.


\bibliographystyle{alpha}
 

\end{document}